\documentclass[twoside]{aomart}
\usepackage[T1]{fontenc}
\usepackage[utf8]{inputenc}
\usepackage{lmodern,amsmath,amssymb,mathrsfs,booktabs,needspace,array}
\usepackage{xcolor,microtype,geometry,indentfirst}
\usepackage{tikz}
\usetikzlibrary{arrows.meta,calc,positioning}
\definecolor{ContentsBlue}{RGB}{20,65,165}
\definecolor{CitationGreen}{RGB}{0,115,45}
\definecolor{ReferenceRed}{RGB}{180,25,35}
\hypersetup{colorlinks=true,linkcolor=ReferenceRed,citecolor=CitationGreen,
  urlcolor=ContentsBlue,linktoc=all,
  pdftitle={Pythagorean pairs in boxes of prime powers}}
\usepackage[capitalise,nameinlink,noabbrev]{cleveref}
\copyrightnote{}
\doinumber{}
\fancypagestyle{firstpage}{%
  \fancyhf{}%
  \fancyfoot[LE,RO]{\footnotesize\thepage}%
  }
\newtheoremstyle{paperplain}{.65\baselineskip}{.65\baselineskip}
  {\itshape}{0pt}{\normalfont\bfseries}{.}{.5em}{}
\newtheoremstyle{paperdefinition}{.65\baselineskip}{.65\baselineskip}
  {\normalfont}{0pt}{\normalfont\bfseries}{.}{.5em}{}
\theoremstyle{paperplain}
\newtheorem{theorem}{Theorem}[section]
\newtheorem{proposition}[theorem]{Proposition}
\newtheorem{lemma}[theorem]{Lemma}
\newtheorem{corollary}[theorem]{Corollary}
\theoremstyle{paperdefinition}

\makeatletter
\@namedef{subjclassname@2020}{2020 Mathematics Subject Classification}
\renewcommand\section{\@startsection{section}{1}{\z@}
  {1.2\baselineskip plus .3\baselineskip minus .2\baselineskip}
  {.6\baselineskip}{\normalfont\large\bfseries\raggedright}}
\renewcommand\subsection{\@startsection{subsection}{2}{\z@}
  {.9\baselineskip plus .2\baselineskip minus .1\baselineskip}
  {-.5em}{\normalfont\bfseries}}
\renewcommand\subsubsection{\@startsection{subsubsection}{3}{\z@}
  {.7\baselineskip}{-.5em}
  {\normalfont\itshape}}
\renewenvironment{proof}[1][\proofname]{\par\pushQED{\qed}%
  \normalfont\topsep6\p@\@plus3\p@\relax\trivlist
  \item[\hskip\labelsep\itshape #1\@addpunct{.}]\ignorespaces
}{\popQED\endtrivlist\@endpefalse}
\makeatother
\crefname{theorem}{Theorem}{Theorems}
\crefname{lemma}{Lemma}{Lemmas}
\crefname{proposition}{Proposition}{Propositions}
\crefname{corollary}{Corollary}{Corollaries}
\crefname{section}{Section}{Sections}
\crefname{subsection}{Section}{Sections}
\crefname{equation}{Equation}{Equations}
\Crefname{theorem}{Theorem}{Theorems}
\Crefname{lemma}{Lemma}{Lemmas}
\Crefname{proposition}{Proposition}{Propositions}
\Crefname{corollary}{Corollary}{Corollaries}
\Crefname{section}{Section}{Sections}
\Crefname{subsection}{Section}{Sections}
\Crefname{equation}{Equation}{Equations}
\newcommand{\E}{\mathbb E}
\newcommand{\Z}{\mathbb Z}
\newcommand{\N}{\mathbb N}
\newcommand{\Q}{\mathbb Q}
\newcommand{\R}{\mathbb R}
\newcommand{\C}{\mathbb C}
\newcommand{\T}{\mathbb T}
\newcommand{\Pp}{\mathbb P}
\newcommand{\M}{\mathcal M}
\newcommand{\TV}{d_{\mathrm{TV}}}
\newcommand{\Exp}{\operatorname{Exp}}

\numberwithin{equation}{section}
\allowdisplaybreaks[1]
\title[Pythagorean pairs in boxes of prime powers]{Pythagorean pairs in boxes of prime powers}

\newcommand*{\AuthorName}{Rongzhong Xiao}
\newcommand*{\AuthorAffiliation}{School of Mathematical Sciences, University of Science and Technology of China, Hefei, Anhui, 230026, PR China}
\newcommand*{\AuthorEmail}{xiaorz@mail.ustc.edu.cn}
\ifx\AuthorName\empty\else\author{\AuthorName}\fi
\ifx\AuthorAffiliation\empty\else\address{\AuthorAffiliation}\fi
\ifx\AuthorEmail\empty\else\email{\AuthorEmail}\fi

\date{}
\thanks{This research is supported by the National Natural Science Foundation of China (No. 12426201), and by the Postdoctoral Fellowship Program of CPSF under Grant Number GZB20260746.}
\begin{document}
\begin{abstract}
Let $p_i$ be the $i$th prime, and let $f(N)$ be the largest cardinality
of a subset of \[
\left\{\prod_{i=1}^N p_i^{a_i}:0\le a_i<N,\ a_i\in\Z\right\}
\]
containing no two distinct legs of an integer right triangle.
We prove
\[
 f(N)
 \ll N^N (\log\log N)^{-1/1296}.
\]
\end{abstract}
\maketitle
\begingroup
\color{ContentsBlue}
\hypersetup{linkcolor=ContentsBlue}
\tableofcontents
\par\removelastskip
\endgroup
\section{Introduction}
\label{sec:introduction}

\subsection{History and motivations}

The Pythagorean equation $x^2+y^2=z^2$ is a central nonlinear question
in partition regularity.  Erd\H{o}s and
Graham \cite[Section~3, Problem~6]{GrahamFavorite} asked whether every
finite colouring of $\N$ contains a monochromatic solution.  

The finite two-colour
problem was first approached computationally: Cooper and Overstreet
\cite{CooperOverstreet} found a two-colouring of
$\{1,\ldots,7664\}$ with no monochromatic triple by a SAT computation.
Heule, Kullmann, and Marek \cite[Theorem~1]{HeuleKullmannMarek} then gave
a certified cube-and-conquer proof (with a formally checkable certificate
of roughly 200 terabytes) that $\{1,\ldots,7824\}$ admits such a colouring,
whereas every two-colouring of $\{1,\ldots,7825\}$ contains a
monochromatic Pythagorean triple.  Two independent 2026 preprints give
purely theoretic proofs of the same two-colour theorem: Leng \cite{Leng2026} uses an
analytic and ergodic argument, while Wesley \cite{Wesley2026} gives an
elementary argument based on linear algebra and counting.  Thus the
finite two-colour problem is settled, whereas the corresponding question
for three or more colours remains open.

A second line of progress uses multiplicative harmonic analysis.
Frantzikinakis and Host \cite[Corollary~2.8 and Problem~3]{FrantzikinakisHost}
developed higher-order Fourier tools for bounded multiplicative functions.
They proved that the following holds: Suppose that $a$, $b$, $c$, and $a+b$ are squares of positive integers. For the equation $ax^2+by^2=cz^2$,
every
finite colouring of $\N$ admits a solution $x,y,z\in\N$ with
$x\ne y$ and $x,y$ of the same colour. They also isolated the
two-leg Pythagorean question, where $x,y$ have the same colour and
$x^2+y^2=z^2$ for some $z$.  Frantzikinakis, Klurman, and Moreira
\cite[Theorems~1.1--1.2]{FKM} subsequently proved that every finite
colouring of $\N$ contains such distinct same-coloured legs and established
the analogous statements for the other two pairs of entries.  Their
density theorem gives the same conclusion inside every set of positive
multiplicative density.\footnote{Ordinary natural density is insufficient: the
odd integers have density $1/2$, but the sum of two odd squares is
$2\pmod 4$ and hence is not a square.  Multiplicative F\o lner sequences
retain the local information lost by ordinary density.}

In a sequel, Frantzikinakis, Klurman, and Moreira
\cite[Theorems~1.2--1.4 and 1.9--1.10]{FKMgeneralized} extended this
framework to generalized Pythagorean equations
$ax^2+by^2=cz^2$.  They proved unconditional partition and density
regularity for the relevant pairs when $ac$ or $bc$ is a square; in the
complementary $(a+b)c$-square case they obtained corresponding results
under an Elliott-type correlation conjecture, together with an unconditional
theorem for pretentious multiplicative level sets.  These results place the
Pythagorean pair theorem within a broader theory of homogeneous quadratic
forms.

Recent work broadens this multiplicative picture.  Azevedo and Moreira
\cite{AzevedoMoreira} showed that finitely many completely multiplicative
functions can be simultaneously close to $1$ on the three entries of a
Pythagorean triple, and treated square coefficients satisfying Rado's
condition.  Frantzikinakis and Mountakis \cite{FrantzikinakisMountakis} proved a
corresponding recurrence theorem for generalized triples
$ax^2+by^2=cz^2$ in pretentious systems. Leng, McNamara, and Mountakis \cite{LengMcNamaraMountakis}
obtained almost-Pythagorean quadruples $x^2+y^2w^2=z^2$ in
multiplicatively dense sets, with $w$ allowed to lie in a second dense set.

Here, our aim is quantitative and finite.  We replace an arbitrary
multiplicative F\o lner sequence by the explicit prime-power boxes below
and obtain a decay rate for the largest subset containing no two distinct
legs of an integer right triangle.
\subsection{Prime boxes and the main result}

For a positive integer $N$, let $p_1<\cdots<p_N$ be the first $N$
primes and define
\begin{equation*}
 \Phi_N=\left\{\prod_{i=1}^N p_i^{a_i}:
              a_i\in\{0,1,\ldots,N-1\}\right\}.
\end{equation*}
Unique factorization gives $|\Phi_N|=N^N$. These sets are products of finite geometric progressions whose exponent
vectors form the integer box $\{0,\ldots,N-1\}^N$.

The boxes form a multiplicative F\o lner sequence: multiplication by
any fixed positive integer changes only a boundary proportion tending to
zero.  For a multiplicative F\o lner sequence $\Psi=(\Psi_K)$, define
$$\overline d_\Psi(A)=\limsup_{K\to\infty}|A\cap\Psi_K|/|\Psi_K|.$$
The density theorem of Frantzikinakis, Klurman and Moreira \cite[Theorem~1.2(1), with $a=b=c=1$]{FKM} states that
positive multiplicative density forces distinct $x,y$ with
$x^2+y^2$ a square.

We study the following finite extremal formulation of the two-leg
density problem, with both the prime support and the exponent range
specified by $N$. Call
$A\subseteq\Phi_N$ admissible if there are no distinct $a,b\in A$
and $c\in\N$ with $a^2+b^2=c^2$, and put
\begin{equation*}
 f(N)=\max\{|A|:A\subseteq\Phi_N\text{ is admissible}\}.
\end{equation*}
The integer $c$ is not required to belong to $A$ or to $\Phi_N$. The fact that the sum of two odd squares is $2\pmod 4$ implies that $f(N)\ge N^{N-1}$.
The theorem below bounds the density of every
admissible subset of a sufficiently large box, with the threshold
uniform over the choice of that subset.

\begin{theorem}
\label{thm:main}
There are absolute constants $C_{\rm main},C_1,C_2,N_0\ge1$
with the following properties. For $0<\delta<1/10$ and every integer $N\ge1$,
\begin{equation}
 N\ge\exp\left(\exp\!\left(C_{\rm main}\delta^{-1295}\log^2(2/\delta)\right)\right)
 \quad\Longrightarrow\quad f(N)\le2\delta N^N.
 \label{eq:main-threshold}
\end{equation}
For every integer $N\ge N_0$,
\begin{equation}
 \frac{f(N)}{N^N}
 \le C_1\left(\frac{(\log\log\log N)^2}{\log\log N}\right)^{1/1295}
 \le C_2(\log\log N)^{-1/1296}.
 \label{eq:main-rate}
\end{equation}
For each $0<b<1/1295$, there are $C_b,N_b>0$ such that
\begin{equation}
 f(N)\le C_b N^N(\log\log N)^{-b}
 \qquad\text{for every integer }N\ge N_b.
 \label{eq:main-other-powers}
\end{equation}
\end{theorem}

In particular, $b=1/1296$ is an admissible numerical exponent.
The first inequality in \eqref{eq:main-rate} records the more precise
logarithmic correction. The constants in the theorem are absolute
where stated, but their numerical values are not evaluated.

\subsection{Notation}
\label{intro:notation}

For $x\ge1$, write $[x]=\{1,\ldots,\lfloor x\rfloor\}$,
$e(t)=e^{2\pi it}$, and $\T=\R/\Z$.  The distance to the nearest
integer is $\|t\|_{\R/\Z}$.  A finite-set average is normalized:
\begin{equation*}
 \E_{x\in S}F(x)=|S|^{-1}\sum_{x\in S}F(x).
\end{equation*}
We write $\mathbf1_S$ for an indicator and $A\ll B$ when
$A\le CB$ with the stated dependence of $C$; all logarithms are
natural.  The iterated logarithms are
$\log_2x=\log\log x$, $\log_3x=\log\log\log x$, and
$\Exp_2(x)=\exp(\exp x)$.

Let $\mathcal M$ be the set of all completely multiplicative functions
$g:\N\to\C$ with $|g(n)|=1$.  Their unique extension to
$\Q_{>0}$ is
\begin{equation}
 g(a/b)=g(a)\overline{g(b)}.
 \label{eq:rational-extension}
\end{equation}
For $r\in\Q_{>0}$, $v_p(r)$ denotes its (possibly negative) prime
valuation.  For a primitive Dirichlet character $\chi$,  denote its conductor
by $q_\chi$.  For real $X\ge2$, set
\begin{equation*}
 \mathbb D_X(g,\chi n^{it})^2
 =\sum_{\substack{p\le X\\p\ \mathrm{prime}}}\frac{1-\Re(g(p)\overline{\chi(p)}p^{-it})}{p}.
\end{equation*}

\section{Proof outline}
\label{sec:framework}

We explain the proof of \Cref{thm:main} before developing its
ingredients. Fix $0<\delta<1/10$. The construction must work
simultaneously for every admissible set and every multiplicative
character. Its quantitative objective is a finite measure whose
support has bounded rational height. We describe this objective
first, then the arithmetic averages, the analytic comparison, and
the choice of scales. 

\subsection{Fourier transfer}

A positive rational $r$ is called a Pythagorean ratio if $1+r^2$
is the square of a rational number. Write $r=u/v$ in the lowest terms,
with $u,v\in\N$, and define its height by
$
 \operatorname{ht}(r)=\max(u,v).
$
For a finitely supported probability measure $\mu$ on these ratios,
we seek the uniform lower bound
\begin{equation}
 \Re\sum_r\mu(r)g(r)\ge-\delta
 \qquad\text{for every }g\in\mathcal M.
 \label{fw:measure-objective}
\end{equation}

This estimate controls $f(N)$ through translation of exponent
vectors. A ratio $r$ acts on exponent
vectors by translation through
$d(r)=(v_{p_1}(r),\ldots,v_{p_N}(r))$. If $\operatorname{ht}(r)\le H$
and $N\ge H$, every prime in $r$ occurs among the first $N$ primes,
and
\begin{equation*}
 \sum_{i=1}^N|v_{p_i}(r)|\le\frac{2\log H}{\log2}.
\end{equation*}
The proportion of exponent vectors for which translation leaves
the box is at most this last quantity divided by $N$.
On the other vectors, two points of an admissible set cannot differ
by $d(r)$: their corresponding integers would be distinct legs of
an integer right triangle.

Finite Fourier analysis on $(\Z/N\Z)^N$ supplies the other side of
the estimate. If an admissible set has density $\alpha$, its zero
Fourier coefficient contributes $\alpha^2$. The other frequencies
contribute at least $-\delta\alpha$ by
\eqref{fw:measure-objective}. Thus \Cref{tr:finite-transfer} proves
\begin{equation*}
 \alpha^2-\delta\alpha
 \le\frac{2\log H}{N\log2}.
\end{equation*}
For $N\ge2\log H/(\delta^2\log2)$, the right side is at most
$\delta^2$, and the positive root of the resulting quadratic is
less than $2\delta$. The remaining task is therefore to construct
$\mu$ satisfying \eqref{fw:measure-objective} with a suitable $H$.

\subsection{Local averages}

Euclid's identity
\begin{equation*}
 (m^2-n^2)^2+(2mn)^2=(m^2+n^2)^2
 \qquad(m>n>0)
\end{equation*}
produces the ratios
\begin{equation*}
 R(m,n)=\frac{m^2-n^2}{2mn}.
\end{equation*}
The measure will be an average of weighted distributions of these
ratios. The two weights have different purposes. A fixed nonnegative
function $w$ of two real variables controls the factor $r^{it}$.
A nonnegative periodic function $W_0$ controls values at small
primes. The function $w$ has integral one and compact support in
$\{(s,t):s>t>0,\ s+t<C_0-1\}$, where $C_0\ge2$ is a fixed integer.
The period $D_0$ of $W_0$ depends only on $\delta$, and the mean of
$W_0$ modulo $D_0$ is one. Define
\begin{equation*}
 \begin{aligned}
 Z_X&=X^{-2}\sum_{m,n\in\Z}w(m/X,n/X)W_0(m,n),\\
 K_X(g)&=\frac{X^{-2}}{Z_X}
       \sum_{m>n>0}w(m/X,n/X)W_0(m,n)g(R(m,n)).
 \end{aligned}
\end{equation*}
These expressions are used only when $Z_X>0$. The normalization
estimate in \eqref{ar:normalization} gives $Z_X\ge1/2$ for every
sufficiently large $X$. Consequently $K_X(1)=1$, so each $K_X$
is the transform of a probability measure.

The construction of $W_0$ in \Cref{ar:section} uses
\begin{equation*}
 B=(200/\delta)^2,\qquad T_0=\lceil100/\delta\rceil,
 \qquad D_0=\prod_{\substack{p\le B\\p\ \mathrm{prime}}}p^{2T_0+5}.
\end{equation*}
At each prime $p\le B$, a congruence condition fixes the valuation
of $R$ to be $2j$ and constrains its unit part to be one modulo
$p^3$. The same integer $j$ is used at all these primes. Its
probabilities are proportional to $T_0+1-|j|$ for
$0<|j|\le T_0$. To describe the relevant identity, let $z$ be
any complex number of modulus one. Then
\begin{equation*}
 \frac1{T_0}\left\{
       \frac1{T_0+1}\left|\sum_{j=0}^{T_0}z^j\right|^2-1
       \right\}\ge-\frac1{T_0}.
\end{equation*}
In the application, $\displaystyle z=\prod_{\substack{p\le B\\p\ \mathrm{prime}}}z_p^2$, where the corrected
prime phases $z_p$ are defined in \eqref{ar:prime-phases}.
That definition separates the factor $r^{it}$ and the local unit
characters as in \eqref{ar:character-factorization}.
The identity explains the use of one common integer at all small
primes. The conditional unit distribution also eliminates every
local character that is nontrivial on the prescribed unit subgroup.

For the remaining primes, the local principal-character factors
are nonnegative, while a nonprincipal factor at a prime $p>B$
has modulus at most $2/\sqrt p$. The real weight is chosen so that
\begin{equation*}
 \int_{s>t>0}w(s,t)R(s,t)^{i\tau}\,ds\,dt
 =\left(\frac{\sin(\tau/2)}{\tau/2}\right)^4\ge0,
\end{equation*}
with value one at $\tau=0$.
\Cref{ar:matching} compares the finite integer average with the
product of these local and real averages, including all unit-character
data. It yields \Cref{ar:model-positivity}: there is an absolute
constant $C$ and an integer $U$, depending only on $\delta$ and
satisfying $2\le U\ll\delta^{-3}$, such that every
$M_1\ge\exp(C\delta^{-3})$ satisfies
\begin{equation}
 \Re K_X(\widetilde g)\ge-\delta/50
 \quad\text{if}\quad y\ge M_1,\quad X\ge(M_1y)^U,
 \label{fw:model-bound}
\end{equation}
whenever $\widetilde g\in\mathcal M$ agrees with
$\chi(p)p^{it}$ at all primes $p>y$, with $q_\chi,|t|\le M_1$.
The values at primes at most $y$ remain arbitrary.

\subsection{The inverse theorem}

For $g\in\mathcal M$, complete multiplicativity factors the kernel
integrand into four functions of linear forms:
\begin{equation}
 g(R(m,n))=\overline{g(2)}\,
 g(m-n)g(m+n)\overline{g(m)}\,\overline{g(n)}.
 \label{fw:four-forms}
\end{equation}
Three applications of Cauchy--Schwarz control this average by the
$U^3$ norm. A finite Fourier expansion of $W_0$ and a subdivision
of the support of $w$ give \Cref{tr:uniform-estimate}. In particular,
there exist $\epsilon_\delta>0$ and $X_\delta\ge1$ with
$\log(2X_\delta)\ll\delta^{-3}$ such that, for every integer
$X\ge X_\delta$,
\begin{equation}
 \log(2/\epsilon_\delta)\ll\delta^{-3},\qquad
 \|g\|_{U^3[C_0X]}<\epsilon_\delta
 \ \Longrightarrow\ |K_X(g)|\le\delta/50.
 \label{fw:uniform-alternative}
\end{equation}
Thus only functions with a larger uniformity norm require a
comparison with the models in \eqref{fw:model-bound}.

The quantitative inverse theorem proved in \Cref{inv:section} has
the following form. Put $L=2+\log(2/\epsilon)$. If
$$\|g\|_{U^3[Z]}\ge\epsilon\quad\text{and}\quad 
Z\ge\Exp_2(CL^{431}),$$ then some primitive character $\chi$ and
real number $t$ satisfy
\begin{equation}
 q_\chi,|t|\le\exp(CL^{431}),\qquad
 \mathbb D_Z(g,\chi n^{it})^2\le CL^{431}.
 \label{fw:inverse-output}
\end{equation}
The proof is divided into two parts. In \Cref{sec:scalar}, a
correlation with a scalar quadratic phase is converted into
rational approximations to its coefficients. After these
coefficients are controlled, a mean-value theorem for multiplicative
functions gives the character and real parameter. In
\Cref{inv:section}, the $U^3$ inverse theorem first supplies a
quadratic nilsequence. The group and orbit are constructed explicitly.
Further correlations reduce this nilsequence to the scalar case.

The power $431$ is obtained from two bounds in the initial inverse
theorem and the subsequent scalar reduction. If $d$ is the dimension
of the constructed group, $M$ is the Lipschitz bound, and $\eta$ is
the correlation, then
\begin{equation*}
 d\ll L^{43},\qquad
 \log(10dM/\eta)\ll L^{87},\qquad
 (d+1)^8\log(10dM/\eta)\ll L^{431}.
\end{equation*}
The interval thresholds are checked in
\Cref{thm:mult_inverse}; they have the same numerical power.

\subsection{Concentration}

Fix ranges $q_\chi\le Q$ and $|t|\le T_1$ that include all
parameters produced by \eqref{fw:inverse-output} at
$\epsilon=\epsilon_\delta$. For each $g\in\mathcal M$ and real
$z\ge2$, define
\begin{equation*}
 E_g(z)=\min_{\substack{\chi\text{ primitive},\ q_\chi\le Q\\|t|\le T_1}}
       \mathbb D_z(g,\chi n^{it})^2.
\end{equation*}
This is a nonnegative, nondecreasing function of $z$. A minimizing
pair exists because the character set is finite and the parameter
interval is compact. If $(\chi,t)$ minimizes the sum at $z$ and
$2\le y\le z$, its contribution from primes in $(y,z]$ is at most
$E_g(z)-E_g(y)$; see \Cref{main:minimized-energy}.

For $y\ge Q$, define $\widetilde g(p)=g(p)$ for $p\le y$ and
$\widetilde g(p)=\chi(p)p^{it}$ for $p>y$. Both choices have modulus
one. Their completely multiplicative extension is therefore in
$\mathcal M$. The quotient $h=g/\widetilde g$ is one at every
prime at most $y$.
\Cref{con:section} applies the Tur\'an--Kubilius inequality to
the additive function whose value at a prime is the argument of
$h(p)$. It proves that the values of $h$ are close, in mean square,
to a single complex number of modulus one on all four relevant
linear forms. In \eqref{fw:four-forms}, that number occurs twice
and its conjugate occurs twice, so their product is one.

The congruence weight is invariant under multiplying both
coordinates by a unit modulo $D_0$. This invariance gives the
comparison coefficient $b_\delta$ with
\begin{equation}
 |K_X(g)-K_X(\widetilde g)|
 \le b_\delta\sqrt{E_g(C_0X)-E_g(y)},
 \qquad b_\delta^2\ll\log(2/\delta),
 \label{fw:concentration}
\end{equation}
Here the pair $(\chi,t)$ minimizes the sum at $z=C_0X$,
and the scale assumptions are those in \Cref{con:comparison}.
The estimate follows by averaging over the weighted residue classes.

\subsection{Scales and height}

Set $\ell=\log(2/\delta)$. Substitution of
\eqref{fw:uniform-alternative} into \eqref{fw:inverse-output}
allows a bound $C_1\ll\delta^{-1293}$ for the minimized prime sum,
because $1293=3\cdot431$. Choose a sufficiently large initial
parameter $M_1$ with
\begin{equation}
 \log\log M_1\ll\delta^{-1293}.
\label{fw:initial-scale}
\end{equation}
Take $M_1$ to be an integer dominating $Q,T_1,D_0$, the uniformity
threshold $X_\delta$, and the inverse-theorem threshold at
$\epsilon_\delta$. The verification that these requirements are
consistent with \eqref{fw:initial-scale} is given in
\eqref{main:initial-parameter}--\eqref{main:scale-sequence}.
Starting from $Z_0=M_1^{a_0}$, with a fixed integer $a_0\ge2$,
define integer scales by
\begin{equation}
 X_j=\left\lceil(M_1Z_{j-1})^U\right\rceil,\qquad
 Z_j=C_0X_j\qquad(1\le j\le k).
 \label{fw:scales}
\end{equation}
The scales are chosen independently of $g$. If $E_g(Z_j)>C_1$,
the inverse theorem implies the small-kernel estimate
\eqref{fw:uniform-alternative}. If $E_g(Z_j)\le C_1$, use
\eqref{fw:model-bound} and \eqref{fw:concentration}. The result
in the second case is
\begin{equation*}
 \Re K_{X_j}(g)\ge-\delta/50
         -b_\delta\sqrt{E_g(Z_j)-E_g(Z_{j-1})}.
\end{equation*}

To sum the errors, put $F_j=\min(C_1,E_g(Z_j))$. This sequence
is nondecreasing and lies in $[0,C_1]$. Cauchy--Schwarz gives
\begin{equation*}
 \frac1k\sum_{j=1}^k\sqrt{F_j-F_{j-1}}
 \le\left(\frac{F_k-F_0}{k}\right)^{1/2}
 \le\sqrt{C_1/k}.
\end{equation*}
The two cases at each scale are combined in
\Cref{main:capped-average}. Taking $k$ to be a sufficiently large
constant times $b_\delta^2C_1\delta^{-2}$, rounded up, yields
\begin{equation*}
 k\ll\delta^{-1295}\ell,
 \qquad
 \Re\left(\frac1k\sum_{j=1}^kK_{X_j}(g)\right)\ge-\delta.
\end{equation*}
The average of the corresponding probability measures is the
required measure $\mu_\delta$.

Every ratio in its support has height at most $Z_k^2$. On taking
logarithms in \eqref{fw:scales}, one obtains a linear recurrence
for $\log Z_j$, with multiplier $U$. Its solution is written out
in \eqref{main:solved-recurrence}. Since
$\log U\ll\ell$, it gives
\begin{equation*}
 \log\log\operatorname{ht}(r)
 \le k\log U+O(\delta^{-1293})
 \ll\delta^{-1295}\ell^2
 \qquad(r\in\operatorname{supp}\mu_\delta).
\end{equation*}
The two factors of $\ell$ have separate sources: one comes from
$b_\delta^2$, and the other comes from $\log U$.
\Cref{main:finite-measure} verifies all thresholds in this
calculation. The finite Fourier estimate then proves
\eqref{eq:main-threshold}, and its elementary inversion proves
\eqref{eq:main-rate} and \eqref{eq:main-other-powers}.

\subsection{Relations between the results}
\label{framework:result-graph}

The four panels below summarize the principal proof dependencies of the
results in the paper. An arrow \(A\to B\) means that the proof of \(B\)
uses \(A\), together with the hypotheses and other inputs required in
that proof. It does not assert that \(A\) alone implies \(B\), or that
the two statements are equivalent. Directed paths record successive
uses of results.

Repeated boxes represent the same result and connect the panels.
Their pale shading distinguishes these repeated nodes. The double
dashed boxes mark \Cref{sc:halasz-external} and \Cref{sc:LSS-5.2} as external inputs; other external inputs
and elementary estimates are cited in the proofs rather than drawn
as additional nodes. An edge can also record an estimate established
within the indicated result's proof, such as the prime or
normalization estimates. Crossing arrows have no junction unless
they meet a result box.

\begingroup
\raggedbottom
\newcommand{\resultgraphbox}[2]{\Cref{#1}}
\tikzset{
  result graph node/.style={
    draw=black!65, rounded corners=2pt, line width=.45pt,
    align=center, text width=2.65cm, minimum height=.75cm,
    inner xsep=3pt, inner ysep=4pt,
    font=\fontsize{9}{10.6}\selectfont, fill=white,
    execute at begin node={\hyphenpenalty=10000\exhyphenpenalty=10000}},
  result graph shared/.style={fill=blue!4},
  result graph external/.style={
    draw=black, double, double distance=.7pt,
    densely dashed, fill=black!5},
  result graph arrow/.style={
    -{Stealth[length=4.5pt,width=3.5pt]}, draw=black!70,
    line width=.5pt, rounded corners=3pt}
}

\par\medskip\noindent\begin{minipage}{\linewidth}
\hypertarget{resultgraph:arithmetic}{\textbf{A. Sections 3 and 4}}\par\smallskip
\centering
\begin{tikzpicture}[x=1cm,y=1cm]
\node[result graph node] (a41) at (1.7,0)
  {\resultgraphbox{ar:shells}{Lemma}};
\node[result graph node] (a42) at (0,-1.9)
  {\resultgraphbox{ar:local-factor-bounds}{Lemma}};
\node[result graph node] (a43) at (3.4,-1.9)
  {\resultgraphbox{ar:Fejer}{Lemma}};
\node[result graph node,result graph shared] (a34) at (0,-3.9)
  {\resultgraphbox{sc:prime-products}{Lemma}};
\node[result graph node,result graph shared] (a44) at (3.4,-3.9)
  {\resultgraphbox{ar:weight-cost}{Lemma}};
\node[result graph node,result graph shared] (a45) at (7.9,0)
  {\resultgraphbox{ar:quadrature}{Lemma}};
\node[result graph node] (a47) at (11.5,0)
  {\resultgraphbox{ar:joint-valuation-data}{Lemma}};
\node[result graph node] (a46) at (7.9,-1.9)
  {\resultgraphbox{ar:spatial-quadrature}{Lemma}};
\node[result graph node,result graph shared] (a34b) at (11.5,-1.9)
  {\resultgraphbox{sc:prime-products}{Lemma}};
\node[result graph node] (a48) at (11.5,-3.9)
  {\resultgraphbox{ar:joint-valuation-approximation}{Theorem}};
\node[result graph node] (a49) at (7.9,-5.7)
  {\resultgraphbox{ar:matching}{Theorem}};
\node[result graph node,result graph shared] (a410) at (4.4,-7.6)
  {\resultgraphbox{ar:model-positivity}{Theorem}};
\draw[result graph arrow] (a41) -- (a42);
\draw[result graph arrow] (a41) -- (a43);
\draw[result graph arrow] (a41.east) -- (5.2,0) |- (a44.east);
\draw[result graph arrow] (a34) -- (a44);
\draw[result graph arrow] (a34b) -- (a48);
\draw[result graph arrow] (a45) -- (a46);
\draw[result graph arrow] (a47.east) -- (13.1,0) |- (a48.east);
\draw[result graph arrow] (a45.west) -- (6.2,0) |- (a49.west);
\draw[result graph arrow] (a46) -- (a49);
\draw[result graph arrow] (a48) -- (a49);
\draw[result graph arrow] (a42.west) -- (-1.6,-1.9) |- (a410.west);
\draw[result graph arrow] (a43.south) -- (3.4,-2.75) -- (1.7,-2.75) -- (1.7,-6.65) -- (a410.north west);
\draw[result graph arrow] (a44) -- (a410);
\draw[result graph arrow] (a49) -- (a410);
\draw[result graph arrow] (a45.west) -- (6.2,0) -- (6.2,-6.65) -- (a410.north east);
\end{tikzpicture}\par\smallskip
\end{minipage}\par\bigskip

\par\medskip\noindent\begin{minipage}{\linewidth}
\hypertarget{resultgraph:scalar}{\textbf{B. Sections 3 and 6}}\par\smallskip
\centering
\begin{tikzpicture}[x=1cm,y=1cm]
\node[result graph node,result graph shared] (b34) at (0,0)
  {\resultgraphbox{sc:prime-products}{Lemma}};
\node[result graph node] (b35) at (3.8,0)
  {\resultgraphbox{tr:rough-number}{Lemma}};
\node[result graph node] (b64) at (7.6,0)
  {\resultgraphbox{sc:congruence-counts}{Lemma}};
\node[result graph node,result graph external] (b69) at (11.4,0)
  {\resultgraphbox{sc:halasz-external}{Theorem}};
\node[result graph node] (b61) at (1.9,-1.8)
  {\resultgraphbox{sc:decomposition}{Lemma}};
\node[result graph node] (b65) at (7.6,-1.8)
  {\resultgraphbox{sc:rational-approx}{Proposition}};
\node[result graph node] (b610) at (11.4,-1.8)
  {\resultgraphbox{sc:residue-mean}{Lemma}};
\node[result graph node,result graph shared] (b62) at (1.9,-3.6)
  {\resultgraphbox{prop:bsz_dense}{Proposition}};
\node[result graph node] (b68) at (7.6,-3.6)
  {\resultgraphbox{sc:circle-lemma}{Lemma}};
\node[result graph node,result graph external] (b67ext) at (11.4,-3.6)
  {\resultgraphbox{sc:LSS-5.2}{Theorem}};
\node[result graph node] (b63) at (1.9,-5.4)
  {\resultgraphbox{sc:cross-progression}{Lemma}};
\node[result graph node,result graph shared] (b611) at (7.6,-5.4)
  {\resultgraphbox{thm:scalar_inverse}{Theorem}};
\node[result graph node] (b66) at (0,-7.2)
  {\resultgraphbox{sc:primitive-congruence}{Lemma}};
\node[result graph node,result graph shared] (b67) at (3.8,-7.2)
  {\resultgraphbox{lem:linear_dense}{Lemma}};
\draw[result graph arrow] (b34) -- (b35);
\draw[result graph arrow] (b34) -- (b61);
\draw[result graph arrow] (b35) -- (b61);
\draw[result graph arrow] (b61) -- (b62);
\draw[result graph arrow] (b62) -- (b63);
\draw[result graph arrow] (b64) -- (b65);
\draw[result graph arrow] (b66) -- (b67);
\draw[result graph arrow] (b69) -- (b610);
\draw[result graph arrow] (b62) -- (b611);
\draw[result graph arrow] (b63) -- (b611);
\draw[result graph arrow] (b67ext.west) -- (10.0,-3.6) -- (10.0,-4.55) -- (b68.east);
\draw[result graph arrow] (b65.east) -- (9.35,-1.8) -- (9.35,-4.55) -- (b611.north east);
\draw[result graph arrow] (b68) -- (b611);
\draw[result graph arrow] (b610.east) -- (13.1,-1.8) |- (b611.east);
\end{tikzpicture}\par\smallskip
\end{minipage}\par\bigskip

\par\medskip\noindent\begin{minipage}{\linewidth}
\hypertarget{resultgraph:nilsequence}{\textbf{C. Sections 3, 6 and 7}}\par\smallskip
\centering
\begin{tikzpicture}[x=1cm,y=1cm]
\node[result graph node] (c71) at (0,0)
  {\resultgraphbox{inv:progression-bohr}{Lemma}};
\node[result graph node] (c73) at (3.8,0)
  {\resultgraphbox{inv:realization}{Lemma}};
\node[result graph node] (c75) at (7.6,0)
  {\resultgraphbox{inv:fejer}{Lemma}};
\node[result graph node] (c76) at (11.4,0)
  {\resultgraphbox{inv:bracket}{Lemma}};
\node[result graph node] (c72) at (0,-1.8)
  {\resultgraphbox{inv:bohr-rank}{Proposition}};
\node[result graph node] (c77) at (9.5,-1.8)
  {\resultgraphbox{inv:vertical-kernel}{Theorem}};
\node[result graph node,result graph shared] (c31) at (0,-3.6)
  {\resultgraphbox{tr:interval-embedding}{Lemma}};
\node[result graph node,result graph shared] (c74) at (3.8,-3.6)
  {\resultgraphbox{inv:initial-inverse}{Corollary}};
\node[result graph node,result graph shared] (c62) at (7.6,-3.6)
  {\resultgraphbox{prop:bsz_dense}{Proposition}};
\node[result graph node,result graph shared] (c67) at (11.4,-3.6)
  {\resultgraphbox{lem:linear_dense}{Lemma}};
\node[result graph node] (c78) at (9.5,-5.4)
  {\resultgraphbox{inv:dense-constraints}{Lemma}};
\node[result graph node,result graph shared] (c611) at (1.9,-7.2)
  {\resultgraphbox{thm:scalar_inverse}{Theorem}};
\node[result graph node] (c79) at (9.5,-7.2)
  {\resultgraphbox{inv:scalar-reduction}{Theorem}};
\node[result graph node,result graph shared] (c710) at (5.7,-9)
  {\resultgraphbox{thm:mult_inverse}{Theorem}};
\draw[result graph arrow] (c71) -- (c72);
\draw[result graph arrow] (c72) -- (c74);
\draw[result graph arrow] (c73) -- (c74);
\draw[result graph arrow] (c31) -- (c74);
\draw[result graph arrow] (c75) -- (c77);
\draw[result graph arrow] (c76) -- (c77);
\draw[result graph arrow] (c62) -- (c78);
\draw[result graph arrow] (c67) -- (c78);
\draw[result graph arrow] (c77) -- (c78);
\draw[result graph arrow] (c77.east) -- (13.1,-1.8) |- (c79.east);
\draw[result graph arrow] (c78) -- (c79);
\draw[result graph arrow] (c75.west) -- (5.65,0) -- (5.65,-6.3) -- (c79.north west);
\draw[result graph arrow] (c74.south) -- (3.8,-8.1) -- (c710.north west);
\draw[result graph arrow] (c79) -- (c710);
\draw[result graph arrow] (c611) -- (c710);
\end{tikzpicture}\par\smallskip
\end{minipage}\par\bigskip

\par\medskip\noindent\begin{minipage}{\linewidth}
\hypertarget{resultgraph:assembly}{\textbf{D. Sections 1, 3, 4, 5, 7 and 8}}\par\smallskip
\centering
\begin{tikzpicture}[x=1cm,y=1cm]
\node[result graph node,result graph shared] (d34) at (0,0)
  {\resultgraphbox{sc:prime-products}{Lemma}};
\node[result graph node,result graph shared] (d44) at (3.8,0)
  {\resultgraphbox{ar:weight-cost}{Lemma}};
\node[result graph node,result graph shared] (d31) at (7.6,0)
  {\resultgraphbox{tr:interval-embedding}{Lemma}};
\node[result graph node] (d32) at (11.4,0)
  {\resultgraphbox{tr:four-forms}{Lemma}};
\node[result graph node] (d51) at (0,-1.8)
  {\resultgraphbox{con:prime-estimates}{Lemma}};
\node[result graph node,result graph shared] (d45) at (3.8,-1.8)
  {\resultgraphbox{ar:quadrature}{Lemma}};
\node[result graph node] (d81) at (7.6,-1.8)
  {\resultgraphbox{tr:uniform-estimate}{Proposition}};
\node[result graph node,result graph shared] (d710) at (11.4,-1.8)
  {\resultgraphbox{thm:mult_inverse}{Theorem}};
\node[result graph node] (d52) at (0,-3.6)
  {\resultgraphbox{con:common-scalar}{Lemma}};
\node[result graph node] (d53) at (3.8,-3.6)
  {\resultgraphbox{con:invariant-weight}{Lemma}};
\node[result graph node,result graph shared] (d410) at (7.6,-3.6)
  {\resultgraphbox{ar:model-positivity}{Theorem}};
\node[result graph node] (d82) at (11.4,-3.6)
  {\resultgraphbox{main:minimized-energy}{Lemma}};
\node[result graph node] (d54) at (0,-5.4)
  {\resultgraphbox{con:homogeneous-comparison}{Proposition}};
\node[result graph node] (d55) at (3.8,-5.4)
  {\resultgraphbox{con:comparison}{Corollary}};
\node[result graph node] (d84) at (7.6,-5.4)
  {\resultgraphbox{main:finite-measure}{Proposition}};
\node[result graph node] (d83) at (11.4,-5.4)
  {\resultgraphbox{main:capped-average}{Lemma}};
\node[result graph node] (d33) at (3.8,-7.5)
  {\resultgraphbox{tr:finite-transfer}{Proposition}};
\node[result graph node] (d11) at (7.6,-7.5)
  {\resultgraphbox{thm:main}{Theorem}};
\draw[result graph arrow] (d34) -- (d51);
\draw[result graph arrow] (d51) -- (d52);
\draw[result graph arrow] (d52) -- (d53);
\draw[result graph arrow] (d52) -- (d54);
\draw[result graph arrow] (d53) -- (d54);
\draw[result graph arrow] (d45.south west) -- (1.9,-2.65) -- (1.9,-4.5) -- (d54.north east);
\draw[result graph arrow] (d54) -- (d55);
\draw[result graph arrow] (d51.west) -- (-1.6,-1.8) -- (-1.6,-6.4) -- (3.8,-6.4) -- (d55.south);
\draw[result graph arrow] (d44.east) -- (5.7,0) -- (5.7,-4.5) -- (d55.north east);
\draw[result graph arrow] (d31) -- (d81);
\draw[result graph arrow] (d32) -- (d81);
\draw[result graph arrow] (d44) -- (d81);
\draw[result graph arrow] (d45) -- (d81);
\draw[result graph arrow] (d710.east) -- (13.1,-1.8) -- (13.1,-6.5) -- (9.2,-6.5) -- (d84.south east);
\draw[result graph arrow] (d81.east) -- (9.45,-1.8) -- (9.45,-4.55) -- (d84.north east);
\draw[result graph arrow] (d410) -- (d84);
\draw[result graph arrow] (d55) -- (d84);
\draw[result graph arrow] (d82) -- (d84);
\draw[result graph arrow] (d83) -- (d84);
\draw[result graph arrow] (d44.east) -- (5.7,0) -- (5.7,-4.5) -- (d84.north west);
\draw[result graph arrow] (d84) -- (d11);
\draw[result graph arrow] (d33) -- (d11);
\end{tikzpicture}\par\smallskip
\end{minipage}\par\bigskip

\endgroup

\subsection{Reading guide}

\Cref{tr:section} contains the finite Fourier reduction and the
elementary estimates for primes. Its four-form inequality explains
the occurrence of $U^3$. Its transfer proposition is the final
step from a measure on ratios to the extremal bound.

\Cref{ar:section} develops the local arithmetic in the order in
which the finite average is constructed. It begins with valuations
and unit parts, defines the weights, proves the required quadrature
estimates, and then establishes the joint distribution comparison.
The final result of that section is the lower bound for finite
character models. \Cref{con:section} proves the mean-square
comparison needed to replace a general function by such a model.

\Cref{sec:scalar,inv:section} contain the inverse-theorem argument.
The first treats scalar quadratic exponentials, and the second
reduces the general $U^3$ correlation to that case. A reader
checking the numerical exponent should follow the rank estimates,
the explicit group construction, and the scalar reduction in
\Cref{inv:section}, ending with \Cref{thm:mult_inverse}.

\Cref{main:section} makes all final parameters explicit in their
order of choice. It first establishes the periodic uniformity
estimate, then proves the monotone-sum averaging inequality,
constructs the measure, calculates its height, and completes the
proof of \Cref{thm:main}. A reader interested first in how the
estimates fit together may read this section after the present
outline, referring back to the precise results cited there.

\section{Fourier transfer}
\label{tr:section}

Recall that $\mathcal M$ is the class of completely multiplicative
functions with values in $\mathbb S^1$, defined in \Cref{intro:notation}.
We first describe their restrictions to finitely generated subgroups
of $\Q_{>0}$. This identifies the precise finite Fourier transform
that will be used in the transfer argument.
If $p_1,\ldots,p_s$ are distinct primes, put
\begin{equation*}
 v(r)=(v_{p_1}(r),\ldots,v_{p_s}(r))\in\Z^s
\end{equation*}
for rational numbers whose numerator and denominator have no prime
divisors outside this list. Every
$\theta=(\theta_1,\ldots,\theta_s)\in\T^s$ determines an element
$g_\theta\in\mathcal M$ by
\begin{equation*}
 g_\theta(p_i)=e(\theta_i),\qquad
 g_\theta(p)=1\ (p\notin\{p_1,\ldots,p_s\}),\qquad
 g_\theta(r)=e(\theta\cdot v(r)).
\end{equation*}
Conversely, the restriction of any $g\in\mathcal M$ to this
subgroup of $\Q_{>0}$ has this form. Thus a finitely supported
probability measure $\mu$ on positive rationals has a finite-dimensional
Fourier transform
\begin{equation*}
 \widehat\mu(\theta)=\sum_r\mu(r)e(\theta\cdot v(r))
                       =\sum_r\mu(r)g_\theta(r).
\end{equation*}

\subsection{Interval uniformity}

For a finite abelian group $G$, all averages over $G$ are normalized
counting averages. For $\omega\in\{0,1\}^3$, write
$|\omega|=\omega_1+\omega_2+\omega_3$ and
$\omega\cdot h=\omega_1h_1+\omega_2h_2+\omega_3h_3$.
If $F:G\to\C$, define its Gowers norm $\|\cdot\|_{U^3(G)}$:
\begin{equation*}
 \|F\|_{U^3(G)}^8
 =\E_{x,h_1,h_2,h_3\in G}
   \prod_{\omega\in\{0,1\}^3}
      \mathcal C^{|\omega|}F(x+\omega\cdot h),
\end{equation*}
where $\mathcal C z=\overline z$. To verify that the right side is
nonnegative, put $\Delta_hF(x)=F(x)\overline{F(x+h)}$.
For fixed $h_1,h_2$, let $H=\Delta_{h_1}\Delta_{h_2}F$.
The average over $x,h_3$ is
$$\E_{x,h_3}H(x)\overline{H(x+h_3)}=|\E_xH(x)|^2,$$ because
$(x,x+h_3)$ is uniform on $G^2$. Therefore
\begin{equation}
 \|F\|_{U^3(G)}^8
 =\E_{h_1,h_2\in G}
    \left|\E_{x\in G}\Delta_{h_1}\Delta_{h_2}F(x)\right|^2.
 \label{tr:norm-nonnegative}
\end{equation}

For a positive integer $Z$, write $[Z]=\{1,\ldots,Z\}$ and let
\begin{equation*}
 \mathcal Q_Z=
 \{(x,h_1,h_2,h_3)\in\Z^4:
                   x+\omega\cdot h\in[Z]\text{ for every }
                   \omega\in\{0,1\}^3\}.
\end{equation*}
The normalized interval convention used throughout the paper is
\begin{equation}
 \|f\|_{U^3[Z]}^8
 =\frac1{|\mathcal Q_Z|}
   \sum_{(x,h)\in\mathcal Q_Z}
       \prod_{\omega\in\{0,1\}^3}
          \mathcal C^{|\omega|}f(x+\omega\cdot h).
 \label{tr:interval-norm}
\end{equation}
The following identity verifies its positivity and its compatibility
with the cyclic-group formulation.

\begin{lemma}\label{tr:interval-embedding}
Let $Z\ge1$ and $P>8Z$ be integers. For $f:[Z]\to\C$, let
$f_Z$ denote its zero extension to $\Z/P\Z$. Then
\begin{equation}
 \|f_Z\|_{U^3(\Z/P\Z)}
 =\|f\|_{U^3[Z]}
        \|\mathbf1_{[Z]}\|_{U^3(\Z/P\Z)}.
 \label{tr:embedding-identity}
\end{equation}
In particular, if $|f|\le1$, then
$0\le\|f\|_{U^3[Z]}\le1$ and
$\|f_Z\|_{U^3(\Z/P\Z)}\le\|f\|_{U^3[Z]}$.
\end{lemma}

\begin{proof}
We divide the proof into three steps.

\medskip
\noindent\textbf{Step 1. Lift supported cubes to the interval.}
Put $G=\Z/P\Z$, and identify $[Z]$ with its image in $G$.
Call a cyclic cube supported if all eight vertices belong to $[Z]$.
Cubes are counted by their parameter tuples; repeated vertices and
zero increments are allowed.

Let $(\bar x,\bar h_1,\bar h_2,\bar h_3)\in G^4$ be supported.
Choose the unique representatives $x,x_i\in[Z]$ of $\bar x$ and
$\bar x+\bar h_i$, respectively, and put
\begin{equation*}
 h_i=x_i-x,\qquad |h_i|\le Z-1 \qquad(1\le i\le3).
\end{equation*}
For each $\omega\in\{0,1\}^3$, let $y_\omega\in[Z]$ represent
$\bar x+\omega\cdot\bar h$. Then
\begin{equation*}
 x+\omega\cdot h\equiv y_\omega\pmod P,\qquad |x+\omega\cdot h-y_\omega|\le |x-y_\omega|+\sum_{i=1}^3\omega_i|h_i|\le4(Z-1)<P.
\end{equation*}
The difference is a multiple of $P$ with absolute value less than
$P$, so it is zero. Thus every integer vertex $x+\omega\cdot h$
belongs to $[Z]$, and the lifted tuple lies in $\mathcal Q_Z$.
Conversely, reduction modulo $P$ sends every tuple of $\mathcal Q_Z$
to a supported cyclic cube. Lifting recovers its initial vertex $x$
and each adjacent vertex $x+h_i$, hence each increment $h_i$.
Lifting and then reducing also recovers the original cyclic tuple.
These maps are therefore inverse bijections.

\medskip
\noindent\textbf{Step 2. Count the lifted cubes and identify the normalized average.}
Write
\begin{equation*}
 N_Z=|\mathcal Q_Z|,\qquad
 S_f=\sum_{(x,h)\in\mathcal Q_Z}
       \prod_{\omega\in\{0,1\}^3}
          \mathcal C^{|\omega|}f(x+\omega\cdot h).
\end{equation*}
Any cyclic cube with a vertex outside $[Z]$ contributes zero to
the average for $f_Z$. On supported cubes, the bijection preserves
every factor. For $\mathbf1_{[Z]}$, each supported cube contributes
one. Since a normalized average over $G^4$ has denominator $P^4$,
\begin{equation}
 \|f_Z\|_{U^3(G)}^8=\frac{S_f}{P^4},\qquad
 \|\mathbf1_{[Z]}\|_{U^3(G)}^8=\frac{N_Z}{P^4}.
 \label{tr:cube-normalization}
\end{equation}
The first quantity is real and nonnegative by
\eqref{tr:norm-nonnegative}; hence $S_f\ge0$, although individual
summands may be complex. The constant cubes $(x,0,0,0)$, with
$x\in[Z]$, give $N_Z\ge Z>0$. Thus the interval average $S_f/N_Z$
is nonnegative, and its nonnegative eighth root is well-defined.
By \eqref{tr:interval-norm} and \eqref{tr:cube-normalization},
\begin{equation*}
 \|f_Z\|_{U^3(G)}^8
 =\frac{S_f}{N_Z}\frac{N_Z}{P^4}
 =\|f\|_{U^3[Z]}^8
    \|\mathbf1_{[Z]}\|_{U^3(G)}^8.
\end{equation*}
Taking nonnegative eighth roots proves \eqref{tr:embedding-identity}.

\medskip
\noindent\textbf{Step 3. Establish the norm bounds.}
Finally, if $|f|\le1$, the triangle inequality gives
\begin{equation*}
 0\le S_f=|S_f|
 \le\sum_{(x,h)\in\mathcal Q_Z}
    \left|\prod_{\omega\in\{0,1\}^3}
        \mathcal C^{|\omega|}f(x+\omega\cdot h)\right|
 \le N_Z.
\end{equation*}
Hence $0\le\|f\|_{U^3[Z]}\le1$. The bijection also gives
$N_Z\le|G|^4=P^4$, so \eqref{tr:cube-normalization} implies
$\|\mathbf1_{[Z]}\|_{U^3(G)}\le1$. The remaining inequality
follows from \eqref{tr:embedding-identity}.
\end{proof}
\subsection{A four-form inequality}

The factorization of $g(R(m,n))$, where $R(m,n)=(m^2-n^2)/2mn$,
involves the four forms
$m-n,m+n,m,n$. We prove the finite-group inequality for these
forms here. The proof introduces one averaging variable for each
of the three factors that must be eliminated by Cauchy--Schwarz.

\begin{lemma}\label{tr:four-forms}
Let $P\ge3$ be odd. For functions
$a,b,c,d:\Z/P\Z\to\C$ of modulus at most one,
\begin{equation}
 \left|\E_{m,n\in\Z/P\Z}
       a(m-n)b(m+n)c(m)d(n)\right|
 \le\|a\|_{U^3(\Z/P\Z)}.
 \label{tr:von-neumann}
\end{equation}
\end{lemma}

\begin{proof}
We divide the proof into four steps.

\medskip
\noindent\textbf{Step 1. Factor the four-form average.}
Put $G=\Z/P\Z$ and $q=2^{-1}\in G$, which exists because $P$
is odd. All averages below are normalized finite sums: for
$F:G^k\to\C$,
\begin{equation*}
 \E_{v_1,\ldots,v_k\in G}F(v_1,\ldots,v_k)
 =\frac1{P^k}\sum_{v_1,\ldots,v_k\in G}F(v_1,\ldots,v_k).
\end{equation*}
Let $\Lambda$ denote the average inside the absolute value in
\eqref{tr:von-neumann}. For each fixed $u_1\in G$, the substitution
\begin{equation*}
 m=qu_1+u_3,\qquad n=-qu_1-u_2
\end{equation*}
is a bijection from $(u_2,u_3)\in G^2$ to $(m,n)\in G^2$.
Indeed, its inverse is $u_2=-n-qu_1$, $u_3=m-qu_1$.
Since $2q=1$, averaging also over $u_1$ gives
\begin{equation}
 \Lambda=\E_{u_1,u_2,u_3}
    A(u_1,u_2,u_3)B(u_2,u_3)C(u_1,u_3)D(u_1,u_2),
 \label{tr:four-form-factorization}
\end{equation}
where
\begin{equation*}
 \begin{aligned}
 A(u_1,u_2,u_3)&=a(u_1+u_2+u_3),\\
 B(u_2,u_3)&=b(-u_2+u_3),\\
 C(u_1,u_3)&=c(qu_1+u_3),\\
 D(u_1,u_2)&=d(-qu_1-u_2).
 \end{aligned}
\end{equation*}
Thus $B,C,D$ do not depend on $u_1,u_2,u_3$, respectively, and
have modulus at most one.

\medskip
\noindent\textbf{Step 2. Introduce the finite differences.}
To record the resulting products, write
\begin{equation*}
 \Delta_iF=F(\ldots,u_i^0,\ldots)
                 \overline{F(\ldots,u_i^1,\ldots)},
\end{equation*}
with the other variables unchanged. The superscripts $0,1$ label
two summation variables, not powers. For example,
\begin{equation*}
 \begin{aligned}
 \Delta_2\Delta_1A
=A(u_1^0,u_2^0,u_3)\overline{A(u_1^1,u_2^0,u_3)}
\overline{A(u_1^0,u_2^1,u_3)}A(u_1^1,u_2^1,u_3).
 \end{aligned}
\end{equation*}

\medskip
\noindent\textbf{Step 3. Apply three Cauchy--Schwarz inequalities.}
First write \eqref{tr:four-form-factorization} as
$$\Lambda=\E_{u_2,u_3}B\,\E_{u_1}(ACD).$$
Cauchy--Schwarz over $(u_2,u_3)$ and $|B|\le1$ give
\begin{equation*}
 \begin{aligned}
 |\Lambda|^2
 \le\left(\E_{u_2,u_3}|B|^2\right)
       \E_{u_2,u_3}\left|\E_{u_1}ACD\right|^2
 \le\E_{u_2,u_3}\left|\E_{u_1}ACD\right|^2
 =\E_{u_1^0,u_1^1,u_2,u_3}
       \Delta_1A\,\Delta_1C\,\Delta_1D=:S_1.
 \end{aligned}
\end{equation*}
The last equality expands the squared modulus into a double sum
over $u_1^0,u_1^1$. In particular, $S_1\ge0$.

Next, $\Delta_1C$ does not depend on $u_2$ and has modulus at most
one. Taking the average over $u_2$ first, Cauchy--Schwarz over
$(u_1^0,u_1^1,u_3)$ gives
\begin{equation*}
 \begin{aligned}
 S_1^2
 \le\E_{u_1^0,u_1^1,u_3}
       \left|\E_{u_2}\Delta_1A\,\Delta_1D\right|^2
 =\E_{u_1^0,u_1^1,u_2^0,u_2^1,u_3}
       \Delta_2\Delta_1A\,\Delta_2\Delta_1D=:S_2.
 \end{aligned}
\end{equation*}
Again $S_2\ge0$, since it is an average of squared moduli.
Finally, $\Delta_2\Delta_1D$ does not depend on $u_3$ and has
modulus at most one. Taking the average over $u_3$ first,
Cauchy--Schwarz over $(u_1^0,u_1^1,u_2^0,u_2^1)$ yields
\begin{equation*}
 \begin{aligned}
 S_2^2
 \le\E_{u_1^0,u_1^1,u_2^0,u_2^1}
       \left|\E_{u_3}\Delta_2\Delta_1A\right|^2
 =\E_{u_1^0,u_1^1,u_2^0,u_2^1,u_3^0,u_3^1}
       \prod_{\omega\in\{0,1\}^3}\mathcal C^{|\omega|}
         a\left(\sum_{i=1}^3u_i^{\omega_i}\right)=:S_3.
 \end{aligned}
\end{equation*}
Each index equal to $1$ contributes a conjugation, producing
the exponent $|\omega|$.

\medskip
\noindent\textbf{Step 4. Identify the final $U^3$ average.}
To identify this last sum, set
\begin{equation*}
 x=u_1^0+u_2^0+u_3^0,\qquad
 h_i=u_i^1-u_i^0\quad(1\le i\le3).
\end{equation*}
For each prescribed $(x,h_1,h_2,h_3)\in G^4$, choose $u_1^0,u_2^0$
arbitrarily. The remaining variables are then uniquely determined by
\begin{equation*}
 u_3^0=x-u_1^0-u_2^0,\qquad
 u_i^1=u_i^0+h_i\quad(1\le i\le3).
\end{equation*}
Thus each tuple in $G^4$ has exactly $P^2$ preimages in $G^6$.
Moreover,
\begin{equation*}
 \sum_{i=1}^3u_i^{\omega_i}
 =\sum_{i=1}^3(u_i^0+\omega_i h_i)=x+\omega\cdot h.
\end{equation*}
Consequently, the definition of the group $U^3$ norm gives
\begin{equation*}
 S_3=\frac{P^2}{P^6}\sum_{x,h_1,h_2,h_3\in G}
       \prod_{\omega\in\{0,1\}^3}
          \mathcal C^{|\omega|}a(x+\omega\cdot h)
 =\|a\|_{U^3(G)}^8.
\end{equation*}
Since $S_1,S_2\ge0$, combining the three Cauchy--Schwarz inequalities
and taking nonnegative eighth roots proves \eqref{tr:von-neumann}:
\begin{equation*}
 |\Lambda|^8\le S_1^4\le S_2^2\le S_3=\|a\|_{U^3(G)}^8.\qedhere
\end{equation*}
\end{proof}
\subsection{Transfer to the exponent box}

Call $r\in\Q_{>0}$ a Pythagorean ratio if $1+r^2$ is the square
of a rational number. Such a ratio is different from one, because
$\sqrt2$ is irrational. For a reduced fraction $r=u/v$ with
$u,v\in\N$, its height is $\max(u,v)$.

\begin{proposition}\label{tr:finite-transfer}
Let $0<\delta<1/10$ and $H\ge2$. If a finitely supported
probability measure $\mu$ is supported on Pythagorean ratios of
height at most $H$ and satisfies
\begin{equation}
 \Re\sum_r\mu(r)g(r)\ge-\delta
                  \qquad\text{for every }g\in\mathcal M.
 \label{tr:positive-measure-hypothesis}
\end{equation}
then $f(N)\le2\delta N^N$ for every integer
\begin{equation}
 N\ge\max\left\{H,\frac{2\log H}{\delta^2\log2}\right\}.
 \label{tr:cube-size}
\end{equation}
\end{proposition}

\begin{proof}
We divide the proof into four steps.
\par
\medskip
\noindent\textbf{Step 1. Encode the box and control translations.}
Fix an integer $N$ satisfying \eqref{tr:cube-size}, and let
$A\subseteq\Phi_N$ be admissible. Put
\begin{equation*}
 B_N=\{0,\ldots,N-1\}^N,\qquad
 G=(\Z/N\Z)^N,\qquad
 \pi(x)=\prod_{i=1}^N p_i^{x_i}.
\end{equation*}
Unique factorization makes $\pi:B_N\to\Phi_N$ a bijection.
Identify $B_N$ with the standard representatives of $G$, and set
\begin{equation*}
 \mathcal A=\pi^{-1}(A),\qquad
 F=\mathbf1_{\mathcal A},\qquad
 \alpha=\frac{|A|}{N^N}=\E_{x\in G}F(x).
\end{equation*}
Arguments of $F$ are evaluated modulo $N$. We will bound the same
weighted correlation from above by admissibility and from below
by \eqref{tr:positive-measure-hypothesis}.

For $r=u/v\in\operatorname{supp}\mu$ in lowest terms, every prime
dividing $u$ or $v$ is at most $H\le N$, hence occurs among the
first $N$ primes. Put $d(r)=(v_{p_i}(r))_{i=1}^N\in\Z^N$.
Since $u$ and $v$ are coprime, unique factorization gives
\begin{equation*}
 \begin{aligned}
 \sum_{i=1}^N|d_i(r)|\log p_i
 &=\sum_{i=1}^N\bigl(v_{p_i}(u)+v_{p_i}(v)\bigr)\log p_i
 =\log(uv),\\
 \|d(r)\|_1
 &\le\frac{\log(uv)}{\log2}
 \le\frac{2\log H}{\log2}.
 \end{aligned}
\end{equation*}

For $d\in\Z^N$, define the boundary set
\begin{equation*}
 E_d=\{x\in B_N:x-d\notin B_N\},
\end{equation*}
where subtraction is in $\Z^N$. In coordinate $i$, the condition
$0\le x_i-d_i<N$ fails when $x_i<d_i$ if $d_i\ge0$, and when
$x_i\ge N+d_i$ if $d_i<0$. Among $0,\ldots,N-1$, exactly
$\min(N,|d_i|)$ choices fail. Counting these failures over the
coordinates yields
\begin{equation*}
 |E_d|\le N^{N-1}\sum_{i=1}^N\min(N,|d_i|)
       \le N^{N-1}\|d\|_1.
\end{equation*}

\par
\medskip
\noindent\textbf{Step 2. Use admissibility to bound the correlation.}
If $x\notin E_{d(r)}$, then the cyclic translate $x-d(r)$ is
the ordinary integer vector in $B_N$. If both $F(x)$ and
$F(x-d(r))$ were one, the labels $a=\pi(x)$ and
$b=\pi(x-d(r))$ would lie in $A$ and satisfy
\begin{equation*}
 \frac ab=\prod_{i=1}^N p_i^{d_i(r)}=r,\qquad
 a^2+b^2=b^2(1+r^2)=(bs)^2,
\end{equation*}
where $s\in\Q_{>0}$ satisfies $1+r^2=s^2$. The labels are distinct
because $r\ne1$. Moreover, $bs$ is an integer: writing $bs=k/\ell$
in lowest terms, the integrality of $(bs)^2$ gives $\ell^2\mid k^2$,
hence $\ell=1$. This contradicts admissibility. Thus
$F(x)F(x-d(r))$ vanishes outside $E_{d(r)}$ and is at most one
on that set. Since $\mu(r)\ge0$ and $\sum_r\mu(r)=1$,
\begin{equation}
 \begin{aligned}
 0\le\mathcal K
 :=\sum_r\mu(r)\E_{x\in G}F(x)F(x-d(r))
 \le\sum_r\mu(r)\frac{|E_{d(r)}|}{N^N}
 \le\frac1N\sum_r\mu(r)\|d(r)\|_1
 \le\frac{2\log H}{N\log2}\le\delta^2.
 \end{aligned}
 \label{tr:upper-correlation}
\end{equation}
The last inequality uses \eqref{tr:cube-size}.

\par
\medskip
\noindent\textbf{Step 3. Compute the Fourier lower bound.}
For the lower bound, define
\begin{equation*}
 \widehat F(\xi)=\E_{x\in G}F(x)e(-\xi\cdot x/N)
               =\frac1{N^N}\sum_{x\in G}F(x)e(-\xi\cdot x/N).
\end{equation*}
The finite geometric-series identity, applied in each coordinate,
gives
\begin{equation}
 \E_{x\in G}e((\xi-\eta)\cdot x/N)=\mathbf1_{\xi=\eta},
 \qquad
 \sum_{\xi\in G}e(\xi\cdot z/N)=N^N\mathbf1_{z=0}.
 \label{tr:cube-orthogonality}
\end{equation}
Indeed, $\sum_{t=0}^{N-1}e(jt/N)$ equals $N$ if $N\mid j$ and
zero otherwise. Since $F$ is real-valued, expanding the two Fourier
coefficients and using \eqref{tr:cube-orthogonality} gives, for every
$d\in\Z^N$,
\begin{equation*}
 \begin{aligned}
 \sum_{\xi\in G}|\widehat F(\xi)|^2e(\xi\cdot d/N)
 &=\E_{x,y\in G}F(x)F(y)
       \sum_{\xi\in G}e(\xi\cdot(y-x+d)/N)\\
 &=N^N\E_{x,y\in G}F(x)F(y)\mathbf1_{y=x-d}\\
 &=\E_{x\in G}F(x)F(x-d).
 \end{aligned}
\end{equation*}
Taking $d=0$, using $F^2=F$, and evaluating $\widehat F(0)$
directly gives
\begin{equation}
 \sum_{\xi\in G}|\widehat F(\xi)|^2=\alpha,\qquad
 \widehat F(0)=\alpha,\qquad
 \sum_{\xi\ne0}|\widehat F(\xi)|^2=\alpha-\alpha^2.
 \label{tr:transfer-fourier-mass}
\end{equation}

For each $\xi\in G$, prescribe
\begin{equation*}
 g_\xi(p_i)=e(\xi_i/N)\quad(1\le i\le N),\qquad
 g_\xi(p)=1\quad\bigl(p\notin\{p_1,\ldots,p_N\}\bigr),
\end{equation*}
and extend by prime factorization. This defines a completely
multiplicative function of modulus one, so $g_\xi\in\mathcal M$.
For every $r\in\operatorname{supp}\mu$,
\begin{equation*}
 g_\xi(r)=\prod_{i=1}^N g_\xi(p_i)^{d_i(r)}
         =e(\xi\cdot d(r)/N).
\end{equation*}
At $\xi=0$, $g_0=1$ and $\sum_r\mu(r)g_0(r)=1$, so the zero
frequency contributes exactly $\alpha^2$. Since $\mathcal K$ is real,
interchanging the finite sums and taking real parts gives
\begin{equation}
 \begin{aligned}
 \mathcal K
 =\sum_{\xi\in G}|\widehat F(\xi)|^2
       \Re\sum_r\mu(r)g_\xi(r)
 \ge\alpha^2-\delta\sum_{\xi\ne0}|\widehat F(\xi)|^2
 =\alpha^2-\delta(\alpha-\alpha^2)
 \ge\alpha^2-\delta\alpha.
 \end{aligned}
 \label{tr:lower-correlation}
\end{equation}
The inequality uses \eqref{tr:positive-measure-hypothesis}, the nonnegative coefficients
$|\widehat F(\xi)|^2$ and \eqref{tr:transfer-fourier-mass}.

\par
\medskip
\noindent\textbf{Step 4. Solve the resulting quadratic.}
Combining \eqref{tr:upper-correlation} and
\eqref{tr:lower-correlation}, and using $\alpha\ge0$, we obtain
\begin{equation*}
 \alpha^2-\delta\alpha-\delta^2\le0,
 \qquad
 \alpha\le\frac{1+\sqrt5}{2}\,\delta<2\delta.
\end{equation*}
Hence $|A|\le2\delta N^N$. Maximizing over admissible subsets
$A\subseteq\Phi_N$ proves the proposition.
\end{proof}
\subsection{Prime estimates}

We first prove the following upper-bound sieve estimate.
Let $z\ge2$, and let
$h:\N\to[0,1]$ be multiplicative with $h(p)<1$ for $p\le z$.
Suppose that, for fixed positive constants $\kappa$ and $B$,
\begin{equation}
 \prod_{\substack{a\le p\le b\\p\ \mathrm{prime}}}(1-h(p))^{-1}
 \le\left(\frac{\log b}{\log a}\right)^\kappa
       \exp(B/\log a)
 \qquad(2\le a\le b\le z).
 \label{sc:sieve-dimension}
\end{equation}
Let $\mathcal S$ be a set of primes, assume that $h(d)=0$ when
$d$ has a prime divisor outside $\mathcal S$, and let $R$ be a
random positive integer. Then
\begin{align}
 \mathbb P\left(p\nmid R\text{ for all }p\le z,
                                      \ p\in\mathcal S\right)
 \le C_{\kappa,B}\prod_{\substack{p\le z\\p\ \mathrm{prime}}}(1-h(p))
  +\sum_{\substack{d\le z\\p\ \mathrm{prime},\ p\mid d\Rightarrow p\in\mathcal S}}
        \mu^2(d)\left|\mathbb P(d\mid R)-h(d)\right|.
 \label{sc:bsz-upper-sieve}
\end{align}
Here $\mu(d)$ is the M\"obius function: it is zero if a square of
a prime divides $d$, and it equals $(-1)^k$ if $d$ is a product of
$k$ distinct primes. The constant depends only on the two displayed
dimension parameters. We verify those parameters in the application.

\begin{proof}[Proof of \eqref{sc:bsz-upper-sieve}]
We divide the proof into four steps.

\medskip
\noindent\textbf{Step 1. Remove the degenerate case and form the Euler product.}
Multiplicativity gives $h(1)=h(1)^2$.
If $h(1)=0$, then $h(n)=h(1)h(n)=0$ for every $n\ge1$.
The $d=1$ term in the error sum in \eqref{sc:bsz-upper-sieve}
then equals $|\mathbb P(1\mid R)-h(1)|=1$.
Since the left side is a probability, the estimate follows.
We may therefore assume $h(1)=1$.

\medskip
\noindent\textbf{Step 2. Invoke the upper-bound sieve weights.}
Put
\begin{equation*}
 Q=\prod_{\substack{p\le z\\p\in\mathcal S}}p,
 \qquad V(z)=\prod_{\substack{p\le z\\p\ \mathrm{prime}}}(1-h(p)).
\end{equation*}
The event on the left of \eqref{sc:bsz-upper-sieve} is
$\gcd(R,Q)=1$. Since $h(p)=0$ outside $\mathcal S$,
multiplicativity and expansion of the finite product give
\begin{equation*}
 V(z)=\prod_{\substack{p\mid Q\\p\ \mathrm{prime}}}(1-h(p))
     =\sum_{d\mid Q}\mu(d)h(d).
\end{equation*}
This explains the Euler product, but its divisor sum need not be
supported on $d\le z$.

\medskip
\noindent\textbf{Step 3. Restrict the weights and obtain a majorant.}
We apply the construction in
\cite[Lemma~2.2, p.~16, and proof of Theorem~2.4, pp.~18--19]{FordSieve2023}
with $g(p)=h(p)$ for $p\le z$ and $g(p)=0$ for $p>z$.
Condition $(\Omega)$ in \cite[p.~18]{FordSieve2023}
is then exactly \eqref{sc:sieve-dimension}.
This construction gives real weights $(c_d)_{d\ge1}$ satisfying
\begin{equation*}
 \begin{gathered}
 c_1=1,\qquad |c_d|\le1\quad(d\ge1),\\
 c_d=0\quad\text{unless }d\le z\text{ is squarefree},\\
 \sum_{d\mid m}c_d\ge0\quad(m\ge1).
 \end{gathered}
 \end{equation*}
Since $g(d)=h(d)$ whenever $c_d\ne0$,
\cite[(2.8), p.~17, and (2.11), p.~19]{FordSieve2023} also gives
\begin{equation}
 \sum_{d\ge1}c_dh(d)=\sum_{d\ge1}c_dg(d)
 \le C_{\kappa,B}V(z),
 \label{sc:upper-sieve-original-main-term}
\end{equation}
where $C_{\kappa,B}>0$ depends only on $\kappa,B$.

Restrict these weights to divisors of $Q$ by setting
$\lambda_d=c_d\mathbf1_{d\mid Q}$. Thus
\begin{equation}
 \lambda_1=1,\qquad |\lambda_d|\le1,\qquad
 \lambda_d=0\quad\text{unless }d\mid Q\text{ and }d\le z.
 \label{sc:upper-sieve-weight-support}
\end{equation}
For $m=\gcd(n,Q)$, the divisor sum of $c_d$ equals one if $m=1$
and is nonnegative otherwise. Hence
\begin{equation}
 \mathbf1_{\gcd(n,Q)=1}
 \le\sum_{d\mid\gcd(n,Q)}c_d
 =\sum_{d\mid n}\lambda_d
 \qquad(n\ge1).
 \label{sc:upper-sieve-majorant}
\end{equation}
If $c_d\ne0$ and $d\nmid Q$, then $d$ has a prime factor outside
$\mathcal S$, so $h(d)=0$.
Thus \eqref{sc:upper-sieve-original-main-term} gives
\begin{equation}
 \sum_d\lambda_d h(d)=\sum_dc_dh(d)\le C_{\kappa,B}V(z).
 \label{sc:upper-sieve-weighted-main-term}
\end{equation}

\medskip
\noindent\textbf{Step 4. Take expectations and bound the error.}
Define the divisibility errors by
\begin{equation*}
 \varepsilon_d=\mathbb P(d\mid R)-h(d).
\end{equation*}
Taking expectations in \eqref{sc:upper-sieve-majorant} and
interchanging the finite divisor sum with expectation yields
\begin{equation}
 \begin{aligned}
 \mathbb P(\gcd(R,Q)=1)
 \le\sum_d\lambda_d\mathbb P(d\mid R)
 =\sum_d\lambda_dh(d)+\sum_d\lambda_d\varepsilon_d
 \le C_{\kappa,B}V(z)
       +\sum_{\substack{d\le z\\d\mid Q}}|\varepsilon_d|.
 \end{aligned}
 \label{sc:upper-sieve-expectation-bound}
\end{equation}
The last step uses \eqref{sc:upper-sieve-weighted-main-term},
the support in \eqref{sc:upper-sieve-weight-support}, and
$|\lambda_d\varepsilon_d|\le|\varepsilon_d|$.
For $d\le z$, the condition $d\mid Q$ is equivalent to
$\mu^2(d)=1$ and $p\ \mathrm{prime},\ p\mid d\Rightarrow p\in\mathcal S$.
Consequently,
\begin{equation*}
 \sum_{\substack{d\le z\\d\mid Q}}|\varepsilon_d|
 =\sum_{\substack{d\le z\\p\ \mathrm{prime},\ p\mid d\Rightarrow p\in\mathcal S}}
       \mu^2(d)\left|\mathbb P(d\mid R)-h(d)\right|.
\end{equation*}
Substitution into \eqref{sc:upper-sieve-expectation-bound} proves
\eqref{sc:bsz-upper-sieve}. In particular, the coefficient one on
the error sum comes from $|\lambda_d|\le1$.
\end{proof}

\begin{lemma}\label{sc:prime-products}
There is an absolute $C$ such that, for every real $x\ge2$,
\begin{equation}
 \log x\le E(x):=\prod_{\substack{p\le x\\p\ \mathrm{prime}}}(1-p^{-1})^{-1}\le C\log x.
 \label{sc:bsz-euler-bounds}
\end{equation}
For every set of primes $\mathcal S$, let $h$ be the multiplicative
function defined by $h(1)=1$ and, for each prime $p$ and integer $k\ge1$,
$h(p^k)=p^{-k}$ if $p\in\mathcal S$ and $h(p^k)=0$ otherwise.
Then $h$ satisfies \eqref{sc:sieve-dimension} with absolute constants
$\kappa,B$ independent of $\mathcal S$.
\end{lemma}

\begin{proof}
We divide the proof into three steps.
\par
\medskip
\noindent\textbf{Step 1. Establish the Chebyshev and Euler-product bounds.}
We first prove the Chebyshev bound
\begin{equation}
 \vartheta(x):=\sum_{\substack{p\le x\\p\ \mathrm{prime}}}\log p\le C_0x
 \qquad(x\ge2),\qquad C_0=4\log2.
 \label{sc:chebyshev}
\end{equation}
For each integer $n\ge1$, every prime $n<p\le2n$ divides $(2n)!$
but not $n!$, and hence divides $\displaystyle\binom{2n}{n}$. The product of
these distinct primes therefore divides the binomial coefficient.
The binomial theorem gives $\displaystyle\binom{2n}{n}\le2^{2n}$, so
\begin{equation*}
 \vartheta(2n)-\vartheta(n)
 \le\log\binom{2n}{n}\le2n\log2.
\end{equation*}
Taking $n=2^{j-1}$ and summing over $1\le j\le k$ yields
\begin{equation*}
 \vartheta(2^k)
 \le\sum_{j=1}^k2^j\log2
 =(2^{k+1}-2)\log2\le2^{k+1}\log2.
\end{equation*}
For $2^{k-1}<x\le2^k$, monotonicity gives
$\vartheta(x)\le\vartheta(2^k)\le4x\log2$, proving
\eqref{sc:chebyshev}.

Expanding the Euler factors as geometric series and using unique
factorization gives
\begin{equation*}
 \begin{aligned}
 E(x)
 =\sum_{\substack{n\ge1\\p\ \mathrm{prime},\ p\mid n\Rightarrow p\le x}}\frac1n
 \ge\sum_{n=1}^{\lfloor x\rfloor}\frac1n
 \ge\int_1^{\lfloor x\rfloor+1}\frac{dt}{t}
 =\log(\lfloor x\rfloor+1)\ge\log x.
 \end{aligned}
\end{equation*}
Every $n\le x$ occurs in the first sum, and $1/t\le1/n$ on
$[n,n+1]$.

For the upper bound, put
\begin{equation*}
 \sigma=1+\frac1{\log x},\qquad
 E_\sigma(x)=\prod_{\substack{p\le x\\p\ \mathrm{prime}}}(1-p^{-\sigma})^{-1}.
\end{equation*}
Since $\sigma>1$, geometric expansion and comparison with an integral give
\begin{equation}
 E_\sigma(x)\le\sum_{n=1}^\infty n^{-\sigma}
 \le1+\int_1^\infty t^{-\sigma}\,dt
 =1+\frac1{\sigma-1}=1+\log x.
 \label{sc:prime-convergent-product}
\end{equation}
Using $-\log(1-u)=\sum_{k\ge1}u^k/k$ and
$1-e^{-v}\le v$ for $v\ge0$, we obtain
\begin{equation}
 \begin{aligned}
 \log E(x)-\log E_\sigma(x)
 &=\sum_{\substack{p\le x\\p\ \mathrm{prime}}}\sum_{k\ge1}
   \frac{p^{-k}(1-e^{-k(\sigma-1)\log p})}{k}\\
 &\le(\sigma-1)\sum_{\substack{p\le x\\p\ \mathrm{prime}}}(\log p)\sum_{k\ge1}p^{-k}\\
 &=\frac1{\log x}\sum_{\substack{p\le x\\p\ \mathrm{prime}}}\frac{\log p}{p-1}.
 \end{aligned}
 \label{sc:bsz-rankin-euler}
\end{equation}
All series here have nonnegative terms and converge. Partial summation
and \eqref{sc:chebyshev} give
\begin{equation*}
 \begin{aligned}
 \sum_{\substack{p\le x\\p\ \mathrm{prime}}}\frac{\log p}{p}
 =\frac{\vartheta(x)}x+\int_2^x\frac{\vartheta(t)}{t^2}\,dt
 \le C_0\left(1+\log\frac x2\right).
 \end{aligned}
\end{equation*}
The identity follows by multiplying
$$1/p=1/x+\int_p^x t^{-2}\,dt$$ by $\log p$ and summing.
Since $1/(p-1)\le2/p$, \eqref{sc:bsz-rankin-euler} implies
\begin{equation*}
 \log E(x)-\log E_\sigma(x)
 \le2C_0\frac{1+\log(x/2)}{\log x}
 \le\frac{2C_0}{\log2}=8.
\end{equation*}
The last inequality uses
$$(1+\log(x/2))/\log x=1+(1-\log2)/\log x\le1/\log2.$$
Together with \eqref{sc:prime-convergent-product}, this proves
\begin{equation*}
 E(x)\le e^8(1+\log x)\le C\log x,
 \qquad C=e^8\left(1+\frac1{\log2}\right),
\end{equation*}
and completes \eqref{sc:bsz-euler-bounds}.

\par
\medskip
\noindent\textbf{Step 2. Verify the sieve-dimension estimate.}
Now fix $2\le a\le b\le z$. Since $h(p)$ is either $1/p$ or zero,
the logarithmic series gives
\begin{equation}
 \begin{aligned}
 \log\prod_{\substack{a\le p\le b\\p\ \mathrm{prime}}}(1-h(p))^{-1}
 \le\sum_{\substack{a\le p\le b\\p\ \mathrm{prime}}}\sum_{k\ge1}\frac1{kp^k}
 =\sum_{\substack{a\le p\le b\\p\ \mathrm{prime}}}\frac1p
   +\sum_{\substack{a\le p\le b\\p\ \mathrm{prime}}}\sum_{k\ge2}\frac1{kp^k}.
 \end{aligned}
 \label{sc:prime-dimension-decomposition}
\end{equation}
For the first sum, apply partial summation with
$f(t)=1/(t\log t)$, whose derivative is
$$f'(t)=-(\log t+1)/(t^2(\log t)^2).$$ This yields
\begin{equation}
 \begin{aligned}
 \sum_{\substack{a\le p\le b\\p\ \mathrm{prime}}}\frac1p
 =\frac{\vartheta(b)}{b\log b}
   -\frac{\vartheta(a)}{a\log a}
+\int_a^b
   \frac{\vartheta(t)(\log t+1)}{t^2(\log t)^2}\,dt
   +\frac{\mathbf1_{\{a\text{ is prime}\}}}{a}.
 \end{aligned}
 \label{sc:prime-reciprocal-partial-summation}
\end{equation}
The final term restores the lower endpoint when $a$ is prime.
The integral needed after applying \eqref{sc:chebyshev} is
\begin{equation}
 \begin{aligned}
 \int_a^b\frac{\log t+1}{t(\log t)^2}\,dt
 =\int_{\log a}^{\log b}\left(\frac1u+\frac1{u^2}\right)\,du
 =\log\frac{\log b}{\log a}
   +\frac1{\log a}-\frac1{\log b}.
 \end{aligned}
 \label{sc:prime-reciprocal-integral}
\end{equation}
Discarding the nonpositive lower-endpoint term in
\eqref{sc:prime-reciprocal-partial-summation} and using
\eqref{sc:chebyshev} and \eqref{sc:prime-reciprocal-integral}, we get
\begin{equation}
 \begin{aligned}
 \sum_{\substack{a\le p\le b\\p\ \mathrm{prime}}}\frac1p
 \le C_0\log\frac{\log b}{\log a}
      +\frac{C_0}{\log a}+\frac1a
 \le C_0\log\frac{\log b}{\log a}
      +\frac{C_0+1}{\log a}.
 \end{aligned}
 \label{sc:bsz-dimension-primes}
\end{equation}
Here the terms $C_0/\log b$ cancel, and $1/a\le1/\log a$.

For the terms with $k\ge2$, use $1/k\le1$, sum the geometric
series, and then sum over all integers at least $a$:
\begin{equation}
 \begin{aligned}
 \sum_{\substack{p\ge a\\p\ \mathrm{prime}}}\sum_{k\ge2}\frac1{kp^k}
 \le\sum_{n\ge\lceil a\rceil}\frac1{n(n-1)}
 =\sum_{n\ge\lceil a\rceil}
     \left(\frac1{n-1}-\frac1n\right)
 =\frac1{\lceil a\rceil-1}\le\frac2a.
 \end{aligned}
 \label{sc:prime-higher-power-tail}
\end{equation}
The last inequality follows from $\lceil a\rceil-1\ge a-1\ge a/2$.
Combining \eqref{sc:prime-dimension-decomposition}, the first bound
in \eqref{sc:bsz-dimension-primes}, and
\eqref{sc:prime-higher-power-tail} gives
\begin{equation*}
 \begin{aligned}
 \log\prod_{\substack{a\le p\le b\\p\ \mathrm{prime}}}(1-h(p))^{-1}
 \le C_0\log\frac{\log b}{\log a}
      +\frac{C_0}{\log a}+\frac3a
 \le C_0\log\frac{\log b}{\log a}
      +\frac{C_0+3}{\log a}.
 \end{aligned}
\end{equation*}
\par
\medskip
\noindent\textbf{Step 3. Conclude the lemma.}
Exponentiation proves \eqref{sc:sieve-dimension} with
$\kappa=4\log2$ and $B=4\log2+3$. These constants are independent
of $\mathcal S$.
\end{proof}
\begin{lemma}\label{tr:rough-number}
There is an absolute constant $C>0$, independent of $A,B,N$, with
the following property. Let $A,B$ be real numbers with $2\le A<B$,
let $N\ge1$ be an integer, and put $z=\lceil B\rceil$. Then
\begin{equation}
 \frac1N\#\{1\le n\le N:p\nmid n\text{ for every prime }A\le p<B\}
 \le C\frac{\log A}{\log B}+\frac zN.
 \label{sc:bsz-rough}
\end{equation}
\end{lemma}

\begin{proof}
We divide the proof into three steps.

\medskip
\noindent\textbf{Step 1. Set up the sieve probability.}
Let $\mathcal S=\{p\text{ prime}:A\le p<B\}$, and let $R$ take
each value in $\{1,\ldots,N\}$ with probability $1/N$. Thus
\begin{equation}
 \rho:=\frac1N\#\{1\le n\le N:p\nmid n\text{ for all }p\in\mathcal S\}=\mathbb P(p\nmid R\text{ for all }p\in\mathcal S).
 \label{sc:rough-count-probability}
\end{equation}
Define the multiplicative function
\begin{equation*}
 h(d)=
 \begin{cases}
 d^{-1},&\text{if every prime divisor of }d\text{ belongs to }\mathcal S,\\
 0,&\text{otherwise}.
 \end{cases}
\end{equation*}
In particular, $h(1)=1$ and $0\le h(p)\le1/2<1$.
By \Cref{sc:prime-products}, $h$ satisfies
\eqref{sc:sieve-dimension} with fixed absolute parameters
$\kappa_0,\beta_0$, independent of $\mathcal S$.
Every prime in $\mathcal S$ is less than $B\le z$, so
\eqref{sc:bsz-upper-sieve} applies and gives
\begin{equation}
 \rho\le C_{\mathrm{sieve}}\prod_{\substack{A\le p<B\\p\ \mathrm{prime}}}(1-p^{-1})
 +\sum_{\substack{d\le z\\p\ \mathrm{prime},\ p\mid d\Rightarrow p\in\mathcal S}}
       \mu^2(d)\left|\mathbb P(d\mid R)-\frac1d\right|.
 \label{sc:rough-sieve-application}
\end{equation}
Here $C_{\mathrm{sieve}}=C_{\kappa_0,\beta_0}$ is absolute and
independent of $A,B,N$. The product contains only primes in
$\mathcal S$ because $h(p)=0$ outside that set.

\medskip
\noindent\textbf{Step 2. Bound the divisibility error.}
There are exactly $\lfloor N/d\rfloor$ multiples of $d$ in
$\{1,\ldots,N\}$. Hence, for every positive integer $d$,
\begin{equation*}
 \begin{aligned}
 \mathbb P(d\mid R)&=\frac{\lfloor N/d\rfloor}{N},\\
 \left|\mathbb P(d\mid R)-\frac1d\right|
 &=\frac1N\left|\lfloor N/d\rfloor-\frac Nd\right|\le\frac1N.
 \end{aligned}
\end{equation*}
This uses $0\le t-\lfloor t\rfloor<1$ and also holds when $d>N$.
Since $\mu^2(d)\le1$ and there are at most $z$ positive integers
$d\le z$, the error in \eqref{sc:rough-sieve-application} is bounded by
\begin{equation}
 \sum_{\substack{d\le z\\p\ \mathrm{prime},\ p\mid d\Rightarrow p\in\mathcal S}}
       \mu^2(d)\left|\mathbb P(d\mid R)-\frac1d\right|
 \le\frac1N\sum_{d\le z}1=\frac zN.
 \label{sc:rough-total-error}
\end{equation}

\medskip
\noindent\textbf{Step 3. Estimate the Euler product and conclude.}
To estimate the product, write
$\displaystyle E(t^-)=\prod_{\substack{p<t\\p\ \mathrm{prime}}}(1-p^{-1})^{-1}$, with a strict upper endpoint.
Factoring the product over primes below $B$ at $A$ gives
\begin{equation}
 E(B^-)=E(A^-)\prod_{\substack{A\le p<B\\p\ \mathrm{prime}}}(1-p^{-1})^{-1},\qquad \prod_{\substack{A\le p<B\\p\ \mathrm{prime}}}(1-p^{-1})=\frac{E(A^-)}{E(B^-)}.
 \label{sc:rough-endpoint-ratio}
\end{equation}
Let $C_E$ be the absolute constant in \eqref{sc:bsz-euler-bounds}.
Its upper bound gives
\begin{equation*}
 E(A^-)\le E(A)\le C_E\log A.
\end{equation*}
If $B$ is not prime, then $E(B^-)=E(B)$. If $B$ is prime, then
$$E(B^-)=(1-B^{-1})E(B)\ge E(B)/2.$$ Thus in both cases,
\begin{equation}
 E(B^-)\ge\frac12E(B)\ge\frac12\log B.
 \label{sc:rough-product-denominator}
\end{equation}
Combining \eqref{sc:rough-endpoint-ratio}--\eqref{sc:rough-product-denominator},
we obtain
\begin{equation}
 \prod_{\substack{A\le p<B\\p\ \mathrm{prime}}}(1-p^{-1})
 \le 2C_E\frac{\log A}{\log B}.
 \label{sc:rough-product-bound}
\end{equation}
Substituting \eqref{sc:rough-total-error} and
\eqref{sc:rough-product-bound} into \eqref{sc:rough-sieve-application} yields
\begin{equation*}
 \rho\le2C_{\mathrm{sieve}}C_E\frac{\log A}{\log B}+\frac zN.
\end{equation*}
By \eqref{sc:rough-count-probability}, this proves
\eqref{sc:bsz-rough} with $C=2C_{\mathrm{sieve}}C_E$.
Both factors are absolute, so $C$ is independent of $A,B,N$.
\end{proof}

\section{Weighted Pythagorean ratios}
\label{ar:section}

Throughout this section,
\begin{equation*}
 R(m,n)=\frac{m^2-n^2}{2mn},
 \qquad m>n>0.
\end{equation*}
We average these ratios using a congruence weight at small primes
and a fixed density in the two real variables. For $y\ge1$, call a
unit-modulus completely multiplicative function $g$ a finite
character model if $g(p)=\chi(p)p^{it}$ for every prime $p>y$,
where $\chi$ is a Dirichlet character of conductor at most $y$
and $t\in\R$.
We approximate the integer average by its local factors, with an
explicit error, and bound its real part below for every such model.

\subsection{Local distributions}

Write $\mathbb S^1=\{z\in\C:|z|=1\}$. For $p$-adic
fields and topology, see
\cite[Sections~3.1--3.2 and 4.1--4.4]{GouveaPadic};
for Haar measure, see
\cite[Section~6, Proposition~I.6.3]{CasselmanHaar}.
All subsets of $p$-adic spaces considered below are Borel.

\par\smallskip\noindent\textbf{$p$-adic numbers and residue classes.}
Fix a prime $p$. Write $r\in\Q^\times$ as $r=p^k a/b$,
where $k\in\Z$, $a,b\in\Z\setminus\{0\}$, and $p\nmid ab$.
Set $v_p(r)=k$ and $|r|_p=p^{-k}$, with $|0|_p=0$.
The field $\Q_p$ is the completion of $\Q$ for this
absolute value. Extend $v_p$ to $\Q_p$ by
$v_p(x)=-\log_p|x|_p$ for $x\ne0$, and set
$v_p(0)=+\infty$. Then
\begin{equation*}
 \Z_p=\{x\in\Q_p:v_p(x)\ge0\},\qquad
 \Z_p^\times=\{x\in\Q_p:v_p(x)=0\}
              =\Z_p\setminus p\Z_p.
\end{equation*}
Every $x\in\Z_p$ has a unique expansion
$\displaystyle x=\sum_{j\ge0}x_jp^j$ with $x_j\in\{0,\ldots,p-1\}$.
For each integer $a\ge1$, its residue class modulo $p^a$
is determined by the first $a$ digits, giving
$\Z_p/p^a\Z_p\simeq\Z/p^a\Z$.
The ring $\Z_p$ is compact. The residue classes
$b+p^a\Z_p$ form a basis of open and closed sets and
generate its Borel $\sigma$-algebra.

\par\smallskip\noindent\textbf{Haar measure and random variables on $\Z_p$.}
Let $\lambda_p$ be the translation-invariant Borel
probability measure on $\Z_p$. The $p^a$ residue classes
modulo $p^a$ have equal measure, so
\begin{equation}
 \lambda_p(\Z_p)=1,\qquad
 \lambda_p(b+p^a\Z_p)=p^{-a}
 \quad(a\ge1,\ b\in\Z_p).
 \label{ar:haar-residues}
\end{equation}
These values determine $\lambda_p$. A $\Z_p$-valued random variable
$X$ is Haar uniform on $\Z_p$ if it has law $\lambda_p$,
or equivalently,
\begin{equation*}
 \Pp(X\equiv b\pmod{p^a})=p^{-a}
 \quad(a\ge1,\ b\in\Z/p^a\Z).
\end{equation*}
One construction is $\displaystyle X=\sum_{j\ge0}D_jp^j$ with independent
digits $D_j$ uniform on $\{0,\ldots,p-1\}$: a prescribed
residue fixes $a$ digits and has probability $p^{-a}$.
Also $\Pp(X=x)\le p^{-a}$ for every $a$, so each
singleton has probability zero.

A Haar pair $(M,N)$ consists of independent $\Z_p$-valued
random variables $M,N$, each with law $\lambda_p$.
Their joint law is $\lambda_p\otimes\lambda_p$.
For $A,B\subseteq\Z_p$, $a\ge1$, and $b,c\in\Z/p^a\Z$,
\begin{equation*}
 \Pp(M\in A,\ N\in B)=\lambda_p(A)\lambda_p(B),\qquad
 \Pp(M\equiv b,\ N\equiv c\pmod{p^a})=p^{-2a}.
\end{equation*}
Similarly, a Haar vector in $\Z_p^d$ has law
$\lambda_p^{\otimes d}$. Haar scalar and unrestricted Haar
pair refer to these laws before conditioning; independence
between different primes is assumed only when stated.

For $\Pp(E)>0$, conditional probability is
$\Pp(F\mid E)=\Pp(F\cap E)/\Pp(E)$.
For $X$ with law $\lambda_p$ and $B\subseteq\Z_p$ with
$\lambda_p(B)>0$, its conditional law given $X\in B$
assigns $A\subseteq\Z_p$ the probability
$\lambda_p(A\cap B)/\lambda_p(B)$.
In particular, $b+p^aX$ has the normalized restriction
of $\lambda_p$ to $b+p^a\Z_p$:
\begin{equation}
 \Pp(b+p^aX\in A)
   =p^a\lambda_p\bigl(A\cap(b+p^a\Z_p)\bigr)
 \quad(a\ge0,\ b\in\Z_p,\ A\subseteq\Z_p).
 \label{ar:haar-ball}
\end{equation}
Indeed, for $j\ge0$ and $c\in\Z_p$, the preimage of
$b+p^ac+p^{a+j}\Z_p$ is $c+p^j\Z_p$.
Both sides of \eqref{ar:haar-ball} assign this target
class probability $p^{-j}=p^ap^{-(a+j)}$; equality on
residue classes proves the formula.

\par\smallskip\noindent\textbf{Haar measure on the unit group.}
Since $\lambda_p(\Z_p^\times)=1-p^{-1}$, define
\begin{equation}
 \lambda_p^\times(A)
   =\frac{\lambda_p(A)}{1-p^{-1}}
 \quad(A\subseteq\Z_p^\times).
 \label{ar:haar-unit-measure}
\end{equation}
For $X$ with law $\lambda_p$, this is its conditional
law given $X\in\Z_p^\times$. Multiplication by a unit
permutes residue classes, hence preserves $\lambda_p$
and $\lambda_p^\times$. Thus $\lambda_p^\times$ is
normalized multiplicative Haar measure on $\Z_p^\times$.
A $\Z_p^\times$-valued random variable $U$ is Haar uniform
if it has law $\lambda_p^\times$, equivalently if
\begin{equation*}
 \Pp(U\equiv u\pmod{p^a})
   =\frac{p^{-a}}{1-p^{-1}}
   =\frac1{\varphi(p^a)}
 \quad(a\ge1,\ u\in(\Z/p^a\Z)^\times).
\end{equation*}
Here $\varphi(p^a)=(p-1)p^{a-1}$. These residue probabilities
also prove uniqueness, since the unit classes generate
the Borel sets. If $U$ has this law, then for $s\ge1$
its conditional law
on $1+p^s\Z_p$ assigns $A\subseteq1+p^s\Z_p$ mass
$\lambda_p^\times(A)/(p^{-s}/(1-p^{-1}))=p^s\lambda_p(A)$.
It is the normalized multiplicative Haar measure on this
subgroup, since multiplication by its elements preserves
the subgroup and $\lambda_p$.

\par\smallskip\noindent\textbf{Primitive pairs.}
A pair $(m,n)\in\Z_p^2$ is primitive if at least one
coordinate is a unit, equivalently if
$\min(v_p(m),v_p(n))=0$. Put
$C_p=\Z_p^2\setminus(p\Z_p)^2$. For a Haar pair $(M,N)$
under $\Pp$, let $E_p=\{(M,N)\in C_p\}$.
Independence and conditional probability give
\begin{equation}
 \Pp(E_p)=1-p^{-2},\qquad
 \Pp((M,N)\in A\mid E_p)=\frac{(\lambda_p\otimes\lambda_p)(A\cap C_p)}{1-p^{-2}}\quad(A\subseteq\Z_p^2).
 \label{ar:haar-primitive}
\end{equation}
\par\smallskip\noindent\textbf{The induced measure on the projective line.}
The projective line $\mathbb P^1(\Q_p)$ consists of nonzero
pairs $(m,n)\in\Q_p^2$ modulo common nonzero scaling,
with the quotient topology. Write $[m:n]=m/n$ when
$n\ne0$, and $[1:0]=\infty$. The map
$\iota:p\Z_p\to\mathbb P^1(\Q_p)$, $u\mapsto[1:u]$,
is a homeomorphism onto $(\Q_p\setminus\Z_p)\cup\{\infty\}$;
thus $\iota(u)=1/u$ for $u\ne0$, $\iota(0)=\infty$, and
$\mathbb P^1(\Q_p)=\Z_p\sqcup\iota(p\Z_p)$.

Let $(M,N)$ be a Haar pair under $\Pp$, with primitive
event $E_p$. Since the pair is nonzero on $E_p$, define
$\mu_p$ as the law of $[M:N]$ under
$\Pp(\,\cdot\mid E_p)$. For $A\subseteq\Z_p$ and
$B\subseteq p\Z_p$, the events $[M:N]\in A$ and
$[M:N]\in\iota(B)$, restricted to $E_p$, are respectively
\mbox{$\{N\in\Z_p^\times,\ M\in NA\}$} and
\mbox{$\{M\in\Z_p^\times,\ N\in MB\}$}.
Multiplication by a unit preserves $\lambda_p$.
Independence and division by $\Pp(E_p)=1-p^{-2}$ therefore give
\begin{equation}
 \begin{aligned}
 \mu_p(A)&=\frac{(1-p^{-1})\lambda_p(A)}{1-p^{-2}}
          =\frac p{p+1}\lambda_p(A),\\
 \mu_p(\iota(B))&=\frac{(1-p^{-1})\lambda_p(B)}{1-p^{-2}}
          =\frac p{p+1}\lambda_p(B).
 \end{aligned}
 \label{ar:projective}
\end{equation}
In particular, $\mu_p(\Z_p)=p/(p+1)$ and
$\mu_p(\iota(p\Z_p))=1/(p+1)$.
Given $M\in\Z_p^\times$ and $N\in p\Z_p$, the random
variable $u=N/M$ has law $p\,\lambda_p|_{p\Z_p}$.
The unnormalized restriction $\lambda_p|_{p\Z_p}$ has mass $1/p$.

For $m,n\in\Q_p$ with $mn\ne0$, put
$R(m,n)=(m^2-n^2)/(2mn)$ and
$R(t)=R(t,1)=(t^2-1)/(2t)$ for $t\in\Q_p^\times$.
Then $R(m,n)=R(m/n)$ and $R(p^jm,p^jn)=R(m,n)$.
Assign $R(M,N)$ the value zero on the null event $MN=0$.
Its law is the same under $\Pp$ and
$\Pp(\,\cdot\mid E_p)$. Indeed, for
$J=\min(v_p(M),v_p(N))$,
$\Pp(J=j)=p^{-2j}-p^{-2j-2}=p^{-2j}(1-p^{-2})$
for $j\ge0$.
Given $M,N\in p^j\Z_p$, the scaled pair
$(p^{-j}M,p^{-j}N)$ is a Haar pair by
\eqref{ar:haar-ball} and independence; $J=j$ means that this
scaled pair is primitive. Hence its law given $J=j$ is
\eqref{ar:haar-primitive}. Homogeneity makes the conditional
law of $R(M,N)$ independent of $j$, hence equal to its
law given $E_p=\{J=0\}$. Averaging over $j$, using
$\displaystyle \sum_{j\ge0}p^{-2j}(1-p^{-2})=1$, proves the assertion.

For $t$ with law $\mu_p$, define, outside the null set
$\{0,1,-1,\infty\}$,
\begin{equation*}
 V_p=v_p(R(t)),\qquad U_p=p^{-V_p}R(t)\in\Z_p^\times.
\end{equation*}
The unit part $U_p$ need not be an ordinary integer.

\Needspace{14\baselineskip}
\begin{samepage}
\begin{lemma}
\label{ar:shells}
With the notation above, for every odd prime $p$ and
every integer $k\ge1$,
\begin{equation}
 \mu_p(V_p=k)=\mu_p(V_p=-k)=\frac{2(p-1)}{p+1}p^{-k},\qquad
 \mu_p(V_p=0)=\frac{p-3}{p+1}.
 \label{ar:odd-shells}
\end{equation}
For $p=2$, one has $|V_2|\ge2$ almost surely, and for
every integer $h\ge2$,
\begin{equation}
 \mu_2(V_2=h)=\frac{2^{1-h}}3,\qquad
 \mu_2(V_2=-h)=\frac{2^{2-h}}3.
 \label{ar:dyadic-shells}
\end{equation}
For every nonzero integer $j$ with $\mu_p(V_p=j)>0$,
the conditional distribution of $U_p$ given $V_p=j$ is
the normalized Haar measure $\lambda_p^\times$ on
$\Z_p^\times$. Explicitly, for every integer $a\ge1$ and
every $b\in(\Z/p^a\Z)^\times$,
\begin{equation}
 \mu_p\bigl(U_p\equiv b\pmod{p^a}\mid V_p=j\bigr)
   =\frac1{\varphi(p^a)}
   =\frac1{(p-1)p^{a-1}}.
 \label{ar:shell-unit-law}
\end{equation}
\end{lemma}
\end{samepage}
\begin{proof}
We divide the proof into three steps.
\par
\medskip
\noindent\textbf{Step 1. Compute the odd-prime shells.}
We use the residue-class measures in \eqref{ar:haar-residues},
the unit measure in \eqref{ar:haar-unit-measure},
the projective measure \eqref{ar:projective}, and
compactness of $\Z_p$. We omit the four null points
throughout the valuation calculations. For $c\in\Z_p$ and every
integer $k\ge0$,
\begin{equation}
 \lambda_p\{x\in\Z_p:v_p(x-c)=k\}=\lambda_p(c+p^k\Z_p)-\lambda_p(c+p^{k+1}\Z_p)=p^{-k}-p^{-k-1}=(1-p^{-1})p^{-k}.
 \label{ar:shell-residue-mass}
\end{equation}
For $t\in\Q_p\setminus\{0,1,-1\}$, factoring the numerator
of $R$ gives
\begin{equation}
 v_p(R(t))=v_p(t-1)+v_p(t+1)-v_p(2)-v_p(t),\qquad R(1/t)=-R(t).
 \label{ar:shell-valuation-identity}
\end{equation}

Suppose first that $p$ is odd. If $t\in p\Z_p\setminus\{0\}$,
then $t-1$ and $t+1$ are units, so $V_p=-v_p(t)$.
For $t=\iota(u)$ with $u\in p\Z_p\setminus\{0\}$,
the second identity in \eqref{ar:shell-valuation-identity}
likewise gives $V_p=-v_p(u)$. For $k\ge1$, each of
these two disjoint regions therefore contributes
\begin{equation}
 \frac p{p+1}(p^{-k}-p^{-k-1})
   =\frac{p-1}{p+1}p^{-k}
 \label{ar:odd-shell-branch}
\end{equation}
to $\mu_p(V_p=-k)$, by \eqref{ar:projective} and
\eqref{ar:shell-residue-mass}.

If $t\in1+p\Z_p$, then $t$ and $t+1$ are units,
so $V_p=v_p(t-1)$. If $t\in-1+p\Z_p$, then $t$ and
$t-1$ are units, so $V_p=v_p(t+1)$. These balls are
disjoint, and each contributes \eqref{ar:odd-shell-branch}
to $\mu_p(V_p=k)$. This proves both nonzero-valuation
formulas in \eqref{ar:odd-shells}. The remaining points
lie in $\Z_p$ and satisfy $t\bmod p\notin\{0,1,-1\}$.
Both the
numerator and denominator of $R(t)$ are units there.
The $p-3$ remaining residue classes each have
$\mu_p$-measure $1/(p+1)$. Thus
\begin{equation}
 \mu_p(V_p=0)=\frac{p-3}{p+1},\qquad
 \sum_{k\ge1}\mu_p(V_p=k)
  =\sum_{k\ge1}\mu_p(V_p=-k)=\frac2{p+1}.
 \label{ar:odd-shell-total}
\end{equation}
The three probabilities in \eqref{ar:odd-shell-total}
sum to one.

\par
\medskip
\noindent\textbf{Step 2. Compute the dyadic shells.}
Now let $p=2$. If $t=2^k a$, where $k\ge1$ and
$a\in\Z_2^\times$, then $t^2-1$ is odd and
$V_2=-k-1$. For $h\ge2$, the set giving $V_2=-h$
in $2\Z_2$ is therefore $\{t:v_2(t)=h-1\}$.
The region parametrized by $\iota(u)$ contributes
the same measure, since $R(1/u)=-R(u)$. Hence
\begin{equation*}
 \mu_2(V_2=-h)
  =2\cdot\frac23\bigl(2^{-(h-1)}-2^{-h}\bigr)
  =\frac{2^{2-h}}3\qquad(h\ge2).
\end{equation*}
If $t$ is odd and $t\ne\pm1$, exactly one of $t-1,t+1$
has valuation one, while the other has valuation at least
two. Consequently,
\begin{equation}
 V_2=
 \begin{cases}
  v_2(t-1),&t\in1+4\Z_2,\quad t\ne1,\\
  v_2(t+1),&t\in-1+4\Z_2,\quad t\ne-1.
 \end{cases}
 \label{ar:dyadic-positive-valuation}
\end{equation}
The two balls in \eqref{ar:dyadic-positive-valuation}
partition $\Z_2^\times$. Each contributes
$(2/3)(2^{-h}-2^{-h-1})$ to $\mu_2(V_2=h)$, giving
\begin{equation*}
 \mu_2(V_2=h)
  =2\cdot\frac23\bigl(2^{-h}-2^{-h-1}\bigr)
  =\frac{2^{1-h}}3\qquad(h\ge2).
\end{equation*}
These cases exhaust the projective line apart from the
four null points. In particular, $|V_2|\ge2$ almost surely,
and the total probabilities are
\begin{equation*}
 \sum_{h\ge2}\mu_2(V_2=h)=\frac13,\qquad
 \sum_{h\ge2}\mu_2(V_2=-h)=\frac23.
\end{equation*}

\par
\medskip
\noindent\textbf{Step 3. Identify the conditional unit law.}
To prove \eqref{ar:shell-unit-law}, we first justify the
measure-preservation fact needed below. Let
$B=c+p^r\Z_p$ and $C=d+p^r\Z_p$, where $r\ge0$
is an integer and $c,d\in\Z_p$. Suppose that $F:B\to C$
satisfies
\begin{equation}
 v_p(F(x)-F(y))=v_p(x-y)\qquad(x,y\in B).
 \label{ar:shell-isometry}
\end{equation}
For each $a\ge r$, congruence of $x,y$ modulo $p^a$
is equivalent to congruence of $F(x),F(y)$ modulo $p^a$.
Thus $F$ induces an injection between the residue classes
of $B$ and $C$ modulo $p^a$. Both sets have $p^{a-r}$
classes, so the induced map is a bijection. Each target
residue class has exactly one source residue class as its
preimage, and their normalized Haar probabilities are
both $p^{r-a}$. For a fixed $z\in C$, the preimages of
$z+p^a\Z_p$ are nested nonempty compact balls. Their
intersection contains a point $x$, and
$F(x)\equiv z\pmod{p^a}$ for all $a\ge r$ implies
$F(x)=z$. Hence $F$ is onto and carries normalized Haar
measure on $B$ to normalized Haar measure on $C$,
because residue classes generate the Borel sets.
The same argument applies to $F:\Z_p^\times\to\Z_p^\times$
satisfying \eqref{ar:shell-isometry}: for $a\ge1$ there
are $\varphi(p^a)$ residue classes, each of probability
$1/\varphi(p^a)$ under $\lambda_p^\times$.

For positive valuations, put $r=1$ when $p$ is odd and
$r=2$ when $p=2$. For $\epsilon\in\{1,-1\}$, let
$B_\epsilon=\epsilon+p^r\Z_p$. The two balls are disjoint.
For $t,s\in B_\epsilon$, direct subtraction gives
\begin{equation*}
 R(t)-R(s)=\frac{t-s}{2}\left(1+\frac1{ts}\right).
\end{equation*}
For odd $p$, the factor multiplying $t-s$ is congruent
to $1$ modulo $p$, since $ts\equiv1\pmod p$.
For $p=2$, one has $ts\equiv1\pmod4$, so
$v_2(ts+1)=1$ and $(ts+1)/(2ts)$ is odd.
The factor is therefore a unit in both cases. Since
$R(\epsilon)=0$, the map $R$ is an isometry
$B_\epsilon\to p^r\Z_p$. The preceding argument proves
that it is onto and preserves normalized Haar measure.
By \eqref{ar:projective}, the conditional distribution
of $t$ given $t\in B_\epsilon$ is normalized additive
Haar measure on that ball. Its image under $R$ is
therefore normalized additive Haar measure on $p^r\Z_p$.
For every $j\ge r$, $a\ge1$, and
$b\in(\Z/p^a\Z)^\times$, it follows that
\begin{equation*}
 \begin{aligned}
 \mu_p\bigl(U_p\equiv b\pmod{p^a}
          \mid V_p=j,\ t\in B_\epsilon\bigr)
   =\frac{p^{-(j+a)}}{p^{-j}-p^{-j-1}}
   =\frac1{(p-1)p^{a-1}}.
 \end{aligned}
\end{equation*}
Here the numerator is the additive Haar measure of
$p^j b+p^{j+a}\Z_p$, using any integer representative
of $b$; the coset is independent of that choice.
The denominator is the measure of $p^j\Z_p^\times$.
The common normalizing factor for
$p^r\Z_p$ cancels in the quotient. Both balls give the
same conditional unit law. Since they exhaust the
event $V_p=j$, the law of total probability gives
\eqref{ar:shell-unit-law} for positive $j$.

For negative valuations, set $\delta=v_p(2)$ and
$d=2/p^\delta\in\Z_p^\times$; thus $(\delta,d)=(0,2)$
for odd $p$ and $(\delta,d)=(1,1)$ for $p=2$.
Write $t=p^k a$ with $k\ge1$ and $a\in\Z_p^\times$.
Given $v_p(t)=k$ in the region $p\Z_p$, the variable
$a=p^{-k}t$ has distribution $\lambda_p^\times$,
by \eqref{ar:projective} and the residue-class measures.
Since $V_p=-k-\delta$, its unit part is
\begin{equation*}
 U_p=p^{k+\delta}R(p^k a)
    =\frac{p^{2k}a-a^{-1}}d=:F_{p,k}(a).
\end{equation*}
The numerator is a unit, so $F_{p,k}$ maps
$\Z_p^\times$ into itself. For units $a,b$,
\begin{equation*}
 F_{p,k}(a)-F_{p,k}(b)
   =(a-b)\frac{p^{2k}+(ab)^{-1}}d.
\end{equation*}
The second factor is a unit. Thus $F_{p,k}$ satisfies
\eqref{ar:shell-isometry} on $\Z_p^\times$ and preserves
$\lambda_p^\times$. In the other region, parametrized by
$t=\iota(u)$ with $v_p(u)=k$, the identity
$R(1/u)=-R(u)$ changes only the sign of the unit part.
Multiplication by $-1$ permutes unit residue classes
modulo every $p^a$, so this region gives the same unit law.
The two regions exhaust the event $V_p=-k-\delta$.
The law of total probability therefore proves
\eqref{ar:shell-unit-law} for every negative valuation
of positive probability.
\end{proof}
For $z\in\mathbb S^1$ and a continuous character
$\eta:\Z_p^\times\to\mathbb S^1$, put
\begin{equation*}
 K_p(z,\eta)=\int z^{V_p}\eta(U_p)\,d\mu_p.
\end{equation*}
The conductor exponent of a nontrivial $\eta$ is the least
integer $a\ge1$ such that $\eta$ is trivial on $1+p^a\Z_p$.
For the trivial character the conductor exponent is zero.
Such an exponent exists for every continuous character:
continuity puts the image of $1+p^a\Z_p$, for sufficiently
large $a$, in the arc $\{e^{i\theta}:|\theta|<\pi/3\}$.
This image is a subgroup, and the only subgroup of
$\mathbb S^1$ contained in that arc is $\{1\}$, since the
powers of any nonidentity element leave the arc.

\begin{lemma}
\label{ar:local-factor-bounds}
Let $p$ be an odd prime, $z\in\mathbb S^1$, and
$\eta:\Z_p^\times\to\mathbb S^1$ a continuous character.
If $\eta$ is trivial, then $K_p(z,1)$ is real and
\begin{equation}
 K_p(z,1)=\frac{p-3}{p+1}+\frac{4(p-1)}{p+1}\sum_{k\ge1}p^{-k}\Re(z^k),\qquad K_p(z,1)\ge\frac{p^2-6p+1}{(p+1)^2}.
 \label{ar:principal-factor}
\end{equation}
In particular, $K_p(z,1)\in[0,1]$ for $p\ge7$.
If $\eta$ is nontrivial with conductor exponent $a\ge1$, then
\begin{equation}
 |K_p(z,\eta)|\le
 \begin{cases}2\sqrt p/(p+1),&a=1,\\2/(p+1),&a\ge2.
 \end{cases}
 \label{ar:nonprincipal-factor}
\end{equation}
\end{lemma}
\begin{proof}
We divide the proof into three steps.
\par
\medskip
\noindent\textbf{Step 1. Treat the principal character.}
For $\eta=1$, decompose the integral according to $V_p$:
\begin{equation*}
 K_p(z,1)=\mu_p(V_p=0)
 +\sum_{k\ge1}\bigl(z^k\mu_p(V_p=k)+z^{-k}\mu_p(V_p=-k)\bigr).
\end{equation*}
The series converges absolutely. Substituting
\eqref{ar:odd-shells} and using $z^{-k}=\overline{z^k}$ proves
the equality in \eqref{ar:principal-factor} and its reality.
For $c=\Re z\in[-1,1]$, the geometric series gives
\begin{equation*}
 f(c):=\sum_{k\ge1}p^{-k}\Re(z^k)=\Re\frac z{p-z}=\frac{pc-1}{p^2-2pc+1},\qquad f'(c)=\frac{p(p^2-1)}{(p^2-2pc+1)^2}>0.
\end{equation*}
Thus $f(c)\ge f(-1)=-1/(p+1)$, and
\begin{equation}
 K_p(z,1)\ge\frac{p-3}{p+1}-\frac{4(p-1)}{(p+1)^2}
 =\frac{p^2-6p+1}{(p+1)^2}.
 \label{ar:principal-lower-bound}
\end{equation}
Since $|z^{V_p}\eta(U_p)|=1$ almost surely and $\mu_p$ is a
probability measure, $|K_p(z,\eta)|\le1$. For $p\ge7$, the
numerator in \eqref{ar:principal-lower-bound} is positive,
which proves the asserted interval bound.

\par
\medskip
\noindent\textbf{Step 2. Reduce a nonprincipal factor to a finite character sum.}
Now suppose that $\eta$ is nontrivial. Put
$I=\int_{\Z_p^\times}\eta(u)\,d\lambda_p^\times(u)$ and choose
$u_0\in\Z_p^\times$ with $\eta(u_0)\ne1$.
Multiplicative invariance of Haar measure gives
\begin{equation*}
 I=\int_{\Z_p^\times}\eta(u_0u)\,d\lambda_p^\times(u)
   =\eta(u_0)I,\qquad I=0.
\end{equation*}
By \eqref{ar:shell-unit-law}, for every $j\ne0$,
\begin{equation*}
 \int_{\{V_p=j\}}z^{V_p}\eta(U_p)\,d\mu_p
 =z^j\mu_p(V_p=j)I=0.
\end{equation*}
Consequently only $V_p=0$ contributes. Up to a null set, this event is
$A=\{t\in\Z_p:t\bmod p\notin\{0,1,-1\}\}$, where $U_p=R(t)$,
so
\begin{equation*}
 K_p(z,\eta)=\int_A\eta(R(t))\,d\mu_p(t).
\end{equation*}
In particular, this expression is independent of $z$.

Let $a\ge1$ be the conductor exponent and let $\eta_a$ be the
induced character of $(\Z/p^a\Z)^\times$. By
\eqref{ar:projective}, every class $r+p^a\Z_p\subseteq\Z_p$
has measure $p^{1-a}/(p+1)$. On $A$, the value of $R(t)$ modulo
$p^a$ depends only on $t$ modulo $p^a$, since $2t$ is a unit.
Hence
\begin{equation}
 K_p(z,\eta)=\frac1{(p+1)p^{a-1}}
 \sum_{\substack{t\bmod p^a\\p\nmid t(t^2-1)}}
 \eta_a\bigl((t^2-1)(2t)^{-1}\bigr),
 \label{ar:finite-conductor-factor}
\end{equation}
where the inverse is taken modulo $p^a$.

If $a=1$, write $\chi=\eta_1$, of order $d>1$,\footnote{The
order of $\chi$ is the least positive integer $d$ such that
$\chi(u)^d=1$ for every $u\in\mathbb F_p^\times$; equivalently,
$\chi^d$ is the trivial character. It divides $p-1$ and equals
$1$ exactly when $\chi$ is trivial. } and extend it by $\chi(0)=0$. Then \eqref{ar:finite-conductor-factor} becomes
\begin{equation}
 K_p(z,\eta)=\frac1{p+1}
 \sum_{t\in\mathbb F_p\setminus\{0,1,-1\}}
          \chi\left(\frac{t^2-1}{2t}\right).
 \label{ar:finite-field-factor}
\end{equation}
We use the following form of Weil's character-sum bound
\cite[Theorem 5.41]{LidlNiederreiter}. Let $q$ be a prime power,
let $\xi:\mathbb F_q^\times\to\mathbb S^1$ be a multiplicative
character of order $m>1$, and set $\xi(0)=0$.
If a monic polynomial $F\in\mathbb F_q[X]$ of positive degree
is not an $m$th power and has $r$ distinct roots in its splitting
field, then, for every $\alpha\in\mathbb F_q^\times$,
\begin{equation}
 \left|\sum_{x\in\mathbb F_q}\xi(\alpha F(x))\right|
 \le(r-1)\sqrt q.
 \label{ar:Weil}
\end{equation}
For the application to \eqref{ar:finite-field-factor}, put
$P(T)=T^{d-1}(T-1)(T+1)$ and
$\gamma=2^{d-1}\in\mathbb F_p^\times$.
For $t\notin\{0,1,-1\}$,
\begin{equation*}
 \chi(\gamma P(t))=\chi(2)^{-1}\chi(t)^{-1}\chi(t-1)\chi(t+1)
 =\chi\left(\frac{t^2-1}{2t}\right).
\end{equation*}
because $\chi(u)^{d-1}=\chi(u)^{-1}$ for every unit $u$.
The polynomial $P$ is monic and has exactly three distinct roots
$0,1,-1$, since $p$ is odd. Its root at $1$ has multiplicity
one, so $P$ is not a $d$th power. The terms at $0,1,-1$ vanish
under the extension of $\chi$ by zero.
Apply \eqref{ar:Weil} with $q=p$, $\xi=\chi$, $m=d$,
$F=P$, $r=3$, and $\alpha=\gamma$.
This bounds the sum in \eqref{ar:finite-field-factor} by
$2\sqrt p$, proving the first line of
\eqref{ar:nonprincipal-factor}.

\par
\medskip
\noindent\textbf{Step 3. Handle higher conductor.}
Suppose finally that $a\ge2$. Define
$\psi(b)=\eta(1+bp^{a-1})$ for $b\in\mathbb F_p$, using any
integer representative of $b$. This is well-defined because
$\eta$ depends only on a unit modulo $p^a$. Moreover,
\begin{equation*}
 (1+bp^{a-1})(1+cp^{a-1})
 \equiv1+(b+c)p^{a-1}\pmod{p^a},
\end{equation*}
since $2a-2\ge a$. Hence $\psi$ is an additive character of
$\mathbb F_p$. It is nontrivial: otherwise $\eta$ would be
trivial on $1+p^{a-1}\Z_p$, contrary to the minimality of $a$.

Fix $t\in A$. For $h\in p^{a-1}\Z_p$, direct subtraction gives
\begin{equation*}
 R(t+h)=R(t)+hR'(t)-\frac{h^2}{2t^2(t+h)},
 \qquad R'(t)=\frac{1+t^{-2}}2.
\end{equation*}
The quantities $2,t,t+h,R(t)$ are units. After division by
$R(t)$, the remainder belongs to
$p^{2a-2}\Z_p\subseteq p^a\Z_p$. With $h=bp^{a-1}$ this gives
\begin{equation}
 \frac{R(t+bp^{a-1})}{R(t)}
 \equiv1+bp^{a-1}\frac{R'(t)}{R(t)}\pmod{p^a}.
 \label{ar:unit-Taylor}
\end{equation}

Group the sum in \eqref{ar:finite-conductor-factor} by residues
modulo $p^{a-1}$. The $p$ lifts of each such residue have equal
measure. For a representative $t\in A$, put
$\beta_t=R'(t)/R(t)\bmod p\in\mathbb F_p$.
\eqref{ar:unit-Taylor} gives their average as
\begin{equation*}
 \begin{aligned}
 \frac1p\sum_{b\in\mathbb F_p}\eta(R(t+bp^{a-1}))
 =\frac{\eta(R(t))}{p}\sum_{b\in\mathbb F_p}\psi(b\beta_t)
 =\begin{cases}\eta(R(t)),&\beta_t=0,\\0,&\beta_t\ne0.
 \end{cases}
 \end{aligned}
\end{equation*}
Indeed, when $\beta_t\ne0$, multiplication by $\beta_t$
permutes $\mathbb F_p$. The sum of $\psi$ is zero: translating
its argument by $b_0$ with $\psi(b_0)\ne1$ multiplies the sum
by $\psi(b_0)$ without changing it.
Since $R(t)$ is a unit,
\begin{equation*}
 \beta_t=0\quad\Longleftrightarrow\quad
 R'(t)\equiv0\pmod p\quad\Longleftrightarrow\quad
 t^2\equiv-1\pmod p.
\end{equation*}
Thus only the residue classes $r+p\Z_p$ with $r^2=-1$ in
$\mathbb F_p$ can contribute. There are at most two, and each
has measure $1/(p+1)$. Since $|\eta(R(t))|=1$ on $A$,
\begin{equation*}
 |K_p(z,\eta)|
 \le\sum_{\substack{r\in\mathbb F_p\\r^2=-1}}
       \mu_p(r+p\Z_p)
 \le\frac2{p+1},
\end{equation*}
where an integer representative of $r$ is understood.
This proves the second line of \eqref{ar:nonprincipal-factor}.
\end{proof}
\subsection{Weights}

For a prime $p$, an integer $j\ne0$, and an integer $s\ge1$, let
\begin{equation*}
 E_{p,j,s}=\{V_p=2j,\ U_p\in1+p^s\Z_p\}.
\end{equation*}
Its probability is
\begin{equation}
 b_{p,j,s}:=\mu_p(E_{p,j,s})=
 \begin{cases}
 \dfrac2{p+1}p^{1-s-2|j|},&p\text{ odd},\\[3pt]
 2^{2-s-2j}/3,&p=2,\ j>0,\\
 2^{3-s-2|j|}/3,&p=2,\ j<0.
 \end{cases}
 \label{ar:conditioning-mass}
\end{equation}
Indeed, $U_p\in1+p^s\Z_p$ means $U_p\equiv1\pmod{p^s}$.
The conditional unit law \eqref{ar:shell-unit-law} therefore gives
\begin{equation*}
 b_{p,j,s}=\mu_p(V_p=2j)\,\mu_p(U_p\equiv1\pmod{p^s}\mid V_p=2j)
 =\frac{\mu_p(V_p=2j)}{(p-1)p^{s-1}}.
\end{equation*}
Substituting \eqref{ar:odd-shells} and \eqref{ar:dyadic-shells}
proves \eqref{ar:conditioning-mass}. These values satisfy
$b_{p,j,s}\ge p^{-s-2|j|}>0$. Hence, for integers $T\ge1$
and $0<|j|\le T$,
\begin{equation}
 \left\|\frac{\mathbf1_{E_{p,j,s}}}{b_{p,j,s}}\right\|_\infty
 =b_{p,j,s}^{-1}\le p^{s+2|j|}\le p^{s+2T}.
 \label{ar:conditioning-density-bound}
\end{equation}

Fix integers $T\ge1$ and $s\ge1$, and a finite set of primes
$\mathcal B$. Let $J$ be a random variable taking values in
$\{\pm1,\ldots,\pm T\}$ with law
\begin{equation}
 \Pp(J=j)=\frac{T+1-|j|}{T(T+1)}
 \quad(0<|j|\le T).
 \label{ar:label-law}
\end{equation}
Given $J=j$, let $t_p\in\mathbb P^1(\Q_p)$, $p\in\mathcal B$,
be independent random variables with respective laws
$\mu_p(\,\cdot\mid E_{p,j,s})$.
Here $V_p,U_p$ are evaluated at $t_p$; in particular,
$V_p=2J$ for every $p\in\mathcal B$ almost surely.

\begin{lemma}
\label{ar:Fejer}
Let $J$ and $t_p$, $p\in\mathcal B$, be the random variables
defined above. For each $p\in\mathcal B$, let
$z_p\in\mathbb S^1$ and let $\eta_p:\Z_p^\times\to\mathbb S^1$
be a continuous character. Then
$\E[\prod_{p\in\mathcal B}z_p^{V_p}\eta_p(U_p)]$ is zero
if some $\eta_p$ has conductor exponent greater than $s$.
Otherwise this expectation is real and belongs to $[-1/T,1]$.
\end{lemma}
\begin{proof}
We divide the proof into three steps.

\medskip
\noindent\textbf{Step 1. Normalize the valuation labels and condition the unit variables.}
The probabilities in \eqref{ar:label-law} sum to one, since
\begin{equation}
 \sum_{0<|j|\le T}\frac{T+1-|j|}{T(T+1)}
 =\frac2{T(T+1)}\sum_{j=1}^T(T+1-j)=1.
 \label{ar:label-law-normalization}
\end{equation}
Put $H_{p,s}=1+p^s\Z_p$. For $0<|j|\le T$ and a Borel set
$D\subseteq\Z_p^\times$, the definition of the conditional law
and \eqref{ar:shell-unit-law} give
\begin{equation*}
 \begin{aligned}
 \Pp(U_p\in D\mid J=j)
 =\frac{\mu_p(U_p\in D\cap H_{p,s}\mid V_p=2j)}
         {\mu_p(U_p\in H_{p,s}\mid V_p=2j)}
 =\frac{\lambda_p^\times(D\cap H_{p,s})}
         {\lambda_p^\times(H_{p,s})}
 =:\nu_{p,s}(D).
 \end{aligned}
\end{equation*}
Here $\lambda_p^\times(H_{p,s})=1/((p-1)p^{s-1})>0$.
The normalized restriction $\nu_{p,s}$ is invariant under
multiplication by elements of $H_{p,s}$, so it is Haar probability
on this subgroup. In particular, this conditional unit law
does not depend on $j$.

\medskip
\noindent\textbf{Step 2. Evaluate the local character integral.}
Let $a_p$ be the conductor exponent of $\eta_p$, with $a_p=0$
for the trivial character. By definition, $\eta_p$ is trivial
on $H_{p,s}$ exactly when $a_p\le s$. Thus
\begin{equation*}
 I_{p,s}:=\int_{H_{p,s}}\eta_p(u)\,d\nu_{p,s}(u)
 =\begin{cases}1,&a_p\le s,\\0,&a_p>s.
 \end{cases}
\end{equation*}
Indeed, the integral is one in the first case. In the second,
choose $h_0\in H_{p,s}$ with $\eta_p(h_0)\ne1$.
Haar invariance gives $I_{p,s}=\eta_p(h_0)I_{p,s}$, forcing
$I_{p,s}=0$.

\medskip
\noindent\textbf{Step 3. Average over the labels and apply the Fej'er identity.}
Write $$Y=\prod_{p\in\mathcal B}z_p^{V_p}\eta_p(U_p)\quad \text{and} \quad
z=\prod_{p\in\mathcal B}z_p^2\in\mathbb S^1.$$
Given $J=j$, the variables $t_p$ are independent and
$V_p=2j$ almost surely. Hence
\begin{equation}
 \E[Y\mid J=j]=\prod_{p\in\mathcal B}z_p^{2j}\E[\eta_p(U_p)\mid J=j]=z^j\prod_{p\in\mathcal B}I_{p,s}.
 \label{ar:conditional-product-average}
\end{equation}
If some $a_p>s$, then $\E[Y\mid J=j]=0$ for every
$0<|j|\le T$. Averaging with the probabilities in
\eqref{ar:label-law} gives $\E[Y]=0$.

Suppose now that $a_p\le s$ for every $p\in\mathcal B$.
Then all $I_{p,s}=1$. Averaging \eqref{ar:conditional-product-average}
over $J$ and pairing $j$ with $-j$ gives
\begin{equation}
 \E[Y]=\frac1{T(T+1)}\sum_{0<|j|\le T}(T+1-|j|)z^j
 =\frac2{T(T+1)}\sum_{j=1}^T(T+1-j)\Re(z^j),
 \label{ar:triangular-character-average}
\end{equation}
since $z^{-j}=\overline{z^j}$. This expression is real.
Expanding a squared absolute value, we obtain
\begin{equation}
 \left|\sum_{k=0}^Tz^k\right|^2=\sum_{k,\ell=0}^Tz^{k-\ell}
 =(T+1)+2\sum_{j=1}^T(T+1-j)\Re(z^j).
 \label{ar:Fejer-square-identity}
\end{equation}
For each $-T\le j\le T$, there are $T+1-|j|$ pairs
$(k,\ell)\in\{0,\ldots,T\}^2$ with $k-\ell=j$, which proves
the second equality. Combining
\eqref{ar:triangular-character-average} and
\eqref{ar:Fejer-square-identity} yields
\begin{equation}
 \E[Y]=\frac1T\left\{
       \frac1{T+1}\left|\sum_{k=0}^Tz^k\right|^2-1\right\}
 \ge-\frac1T.
 \label{ar:Fejer-expectation-bound}
\end{equation}
Finally, $|Y|=1$ almost surely, so
$|\E[Y]|\le\E[|Y|]=1$. Together with reality and
\eqref{ar:Fejer-expectation-bound}, this proves the interval bound.
\end{proof}
Fix $0<\delta<1/10$ and define
\begin{equation}
 T_0=\lceil100/\delta\rceil,\qquad
 B=(200/\delta)^2,\qquad
 D_0=\prod_{\substack{p\le B\\p\ \mathrm{prime}}}p^{2T_0+5}.
 \label{ar:parameters}
\end{equation}
For $0<|j|\le T_0$, define a function on $\Z_p^2$ by
\begin{equation}
 W_{p,j}(m,n)=\frac{\mathbf1_{\min(v_p(m),v_p(n))=0}\mathbf1_{E_{p,j,3}}(m,n)}{(1-p^{-2})b_{p,j,3}}.
 \label{ar:local-weight}
\end{equation}
Set $W_{p,j}(m,n)=0$ whenever $mn(m^2-n^2)=0$.
This union of four lines has product Haar measure zero.
The denominator normalizes the weight to have mean one on $\Z_p^2$.
Define the global congruence weight $W_0:\Z^2\to[0,\infty)$ by
\begin{equation}
 W_0(m,n)=\sum_{0<|j|\le T_0}
 \frac{T_0+1-|j|}{T_0(T_0+1)}\prod_{\substack{p\le B\\p\ \mathrm{prime}}}W_{p,j}(m,n).
 \label{ar:congruence-weight}
\end{equation}

\begin{lemma}
\label{ar:weight-cost}
The function $W_0$ is nonnegative, is periodic modulo $D_0$ in
each variable, and satisfies
\begin{equation*}
 D_0^{-2}\sum_{a,b\bmod D_0}W_0(a,b)=1.
\end{equation*}
For all integers $m,n,c$ with $\gcd(c,D_0)=1$, we have
$W_0(cm,cn)=W_0(m,n)$. Moreover,
\begin{equation}
 \log(2D_0)+\log(2\|W_0\|_\infty)\le C\delta^{-3}.
 \label{ar:weight-size}
\end{equation}
Here $C>0$ is an absolute constant.
\end{lemma}
\begin{proof}
We divide the proof into four steps.

\medskip
\noindent\textbf{Step 1. Prove local periodicity.}
Nonnegativity follows from \eqref{ar:local-weight} and
\eqref{ar:congruence-weight}. Fix $p\le B$ and $0<|j|\le T_0$,
and put
\begin{equation*}
 F=m^2-n^2,\qquad G=2mn,\qquad
 \ell=2|j|+1,\qquad k=\ell+4.
\end{equation*}
We show that $W_{p,j}$ depends only on $(m,n)\bmod p^k$.
Primitivity is determined modulo $p$; nonprimitive pairs have
weight zero. For a primitive pair, write $\alpha=v_p(F)$ and
$\beta=v_p(G)$, with $v_p(0)=\infty$. If $p$ is odd and
$p\mid G$, exactly one coordinate is divisible by $p$, so
$F$ is a unit. If $p=2$, opposite parity makes $F$ odd,
whereas two odd coordinates give $\beta=1$. Thus
\begin{equation}
 \min(\alpha,\beta)=0\quad(p\text{ odd}),\qquad
 \min(\alpha,\beta)\le1\quad(p=2).
 \label{ar:weight-small-valuation}
\end{equation}

The residues of $m,n$ modulo $p^k$ determine those of $F,G$.
If either $F$ or $G$ is divisible by $p^{\ell+1}$, the pair
has weight zero. Indeed, when $FG\ne0$,
\eqref{ar:weight-small-valuation} gives
$|\alpha-\beta|\ge\ell>2|j|$; when $FG=0$, the defining
zero convention applies. Otherwise $\alpha,\beta\le\ell$
are determined exactly. The units $F_0=p^{-\alpha}F$ and
$G_0=p^{-\beta}G$ are determined modulo $p^{k-\alpha}$ and
$p^{k-\beta}$, respectively. Since both exponents are at
least four, these units are known modulo $p^3$. As
$R=p^{\alpha-\beta}F_0/G_0$, we have
\begin{equation*}
 [m:n]\in E_{p,j,3}
 \quad\Longleftrightarrow\quad
 \alpha-\beta=2j\quad\text{and}\quad F_0\equiv G_0\pmod{p^3}.
\end{equation*}
Both conditions are determined by the prescribed residues.
This also proves that the zero convention at $FG=0$ is
compatible with periodicity. Since $k\le2T_0+5$, each
$W_{p,j}$ depends only on the pair modulo $p^{2T_0+5}$.
Their products and sum in \eqref{ar:congruence-weight}
are therefore periodic modulo $D_0$ in each variable.

\medskip
\noindent\textbf{Step 2. Compute the mean by Haar measure and the Chinese remainder theorem.}
Put $L=2T_0+5$. A Haar pair in $\Z_p^2$ is primitive with
probability $1-p^{-2}$, and its projective point has law
$\mu_p$ given primitivity. Every residue pair modulo $p^L$
has product Haar measure $p^{-2L}$. Thus the preceding
periodicity and \eqref{ar:local-weight} give
\begin{equation}
 \begin{aligned}
 p^{-2L}\sum_{r,s\bmod p^L}W_{p,j}(r,s)
 =\int_{\Z_p^2}W_{p,j}\,d(\lambda_p\otimes\lambda_p)
 =\frac{(1-p^{-2})\mu_p(E_{p,j,3})}
         {(1-p^{-2})b_{p,j,3}}=1.
 \end{aligned}
 \label{ar:weight-local-average}
\end{equation}
For each fixed $j$, the Chinese remainder theorem\footnote{For
pairwise coprime $d_1,\ldots,d_k\ge2$ and $D=\prod_{\ell=1}^k d_\ell$,
the map $\Z/D\Z\to\prod_{\ell=1}^k\Z/d_\ell\Z$,
$a\mapsto(a\bmod d_\ell)_\ell$, is a ring isomorphism.} yields
\begin{equation}
 \begin{aligned}
 D_0^{-2}\sum_{a,b\bmod D_0}\prod_{\substack{p\le B\\p\ \mathrm{prime}}}W_{p,j}(a,b)
 &=\prod_{\substack{p\le B\\p\ \mathrm{prime}}}\left(
 p^{-2L}\sum_{r,s\bmod p^L}W_{p,j}(r,s)\right)=1.
 \end{aligned}
 \label{ar:weight-crt-average}
\end{equation}
The coefficients
$\omega_j=(T_0+1-|j|)/(T_0(T_0+1))$ sum to one by
\eqref{ar:label-law-normalization}. Averaging
\eqref{ar:weight-crt-average} with these coefficients proves
the asserted mean of $W_0$.

\medskip
\noindent\textbf{Step 3. Prove unit-scaling invariance.}
For $u\in\Z_p^\times$, simultaneous multiplication of $m,n$
by $u$ preserves primitivity and multiplies both $F$ and $G$
by $u^2$. Hence it preserves their valuations, their unit
quotient, and the exceptional set $FG=0$. Therefore
\begin{equation}
 W_{p,j}(um,un)=W_{p,j}(m,n)
 \qquad(m,n\in\Z_p,\ u\in\Z_p^\times).
 \label{ar:weight-local-scaling}
\end{equation}
If $\gcd(c,D_0)=1$, then $c\in\Z_p^\times$ for every
$p\le B$. Applying \eqref{ar:weight-local-scaling} to each
factor in \eqref{ar:congruence-weight} proves the stated
identity for $W_0$.

\medskip
\noindent\textbf{Step 4. Bound the supremum norm.}
By \eqref{ar:conditioning-density-bound} with $s=3$,
\begin{equation*}
 \|W_{p,j}\|_\infty
 \le(1-p^{-2})^{-1}b_{p,j,3}^{-1}
 \le(1-p^{-2})^{-1}p^{3+2T_0}.
\end{equation*}
Since $\omega_j\ge0$ and $\displaystyle\sum_j\omega_j=1$, it follows that
\begin{equation*}
 \|W_0\|_\infty\le\sum_{0<|j|\le T_0}\omega_j\prod_{\substack{p\le B\\p\ \mathrm{prime}}}\|W_{p,j}\|_\infty\le\prod_{\substack{p\le B\\p\ \mathrm{prime}}}(1-p^{-2})^{-1}p^{3+2T_0}\le2\prod_{\substack{p\le B\\p\ \mathrm{prime}}}p^{3+2T_0}.
\end{equation*}
For the last inequality, each factor exceeds one and
\begin{equation*}
 \prod_{n=2}^{N}(1-n^{-2})^{-1}
 =\left(\prod_{n=2}^{N}\frac n{n-1}\right)
  \left(\prod_{n=2}^{N}\frac n{n+1}\right)
 =\frac{2N}{N+1}\longrightarrow2.
\end{equation*}
Thus the product over $p\le B$ is at most two. Taking
logarithms and using the Chebyshev estimate
$\displaystyle\sum_{\substack{p\le B\\p\ \mathrm{prime}}}\log p\le4B$ from \eqref{sc:chebyshev}, we obtain
\begin{equation*}
 \log D_0\le4(2T_0+5)B,\qquad
 \log\|W_0\|_\infty\le\log2+4(2T_0+3)B.
\end{equation*}
Finally, $T_0\le101\delta^{-1}$ and $B=40000\delta^{-2}$,
so
\begin{equation*}
 \log(2D_0)+\log(2\|W_0\|_\infty)
 \le3\log2+(16T_0+32)B\le C\delta^{-3}.
\end{equation*}
This proves \eqref{ar:weight-size}.
\end{proof}
We next construct a fixed real probability density, independently
of $\delta$. Let $b=\mathbf1_{[-1/2,1/2]}$ and $h=b*b*b*b$.
Then
\begin{equation}
 h(s)=\frac16\sum_{k=0}^4(-1)^k\binom4k(s+2-k)_+^3,
 \qquad a_+=\max(a,0).
 \label{ar:spline-formula}
\end{equation}
Put $c=\mathbf1_{[0,1]}$, and write $c^{*r}$ for its
$r$-fold convolution. Since $b(s)=c(s+1/2)$, we have
$h(s)=c^{*4}(s+2)$. The value $(c*c)(x)$ is the length
of $[0,1]\cap[x-1,x]$, so
\begin{equation*}
 (c*c)(x)=x_+-2(x-1)_++(x-2)_+.
\end{equation*}
Convolving once more means integrating over $0\le a\le1$.
For $q=1,2$, the antiderivative of a positive-part power gives
\begin{equation*}
 \int_0^1(x-a-k)_+^q\,da
 =\frac{(x-k)_+^{q+1}-(x-k-1)_+^{q+1}}{q+1}.
\end{equation*}
Applying this identity to each term, first with $q=1$ and
then with $q=2$, yields
\begin{equation*}
 \begin{aligned}
 c^{*3}(x)
 &=\frac12\bigl[x_+^2-3(x-1)_+^2
                   +3(x-2)_+^2-(x-3)_+^2\bigr],\\
 c^{*4}(x)
 &=\frac16\bigl[x_+^3-4(x-1)_+^3+6(x-2)_+^3
                   -4(x-3)_+^3+(x-4)_+^3\bigr].
 \end{aligned}
\end{equation*}
Replacing $x$ by $s+2$ proves \eqref{ar:spline-formula}.
The convolution definition gives $h\ge0$ and $h(s)=0$ for $|s|>2$.
Integrating \eqref{ar:spline-formula} over $[-2,2]$ gives
\begin{equation*}
 \int_\R h(s)\,ds
 =\frac{4^4-4\cdot3^4+6\cdot2^4-4}{24}=1.
\end{equation*}
Moreover, for every real $x$,
\begin{equation*}
 \frac{d}{dx}x_+^3=3x_+^2,\qquad
 \frac{d^2}{dx^2}x_+^3=6x_+.
\end{equation*}
Both derivatives are continuous, including at zero. \eqref{ar:spline-formula} and the compact support therefore
give $h\in C_c^2(\R)$.
In fact, $\operatorname{supp}h=[-2,2]$.

For $t>1$, put $r(t)=(t^2-1)/(2t)$ and
$\phi(t)=\log r(t)$. Direct differentiation and the solution
of $r(t)=e^s$ give
\begin{equation*}
 \phi'(t)=\frac{r'(t)}{r(t)}
         =\frac{t^2+1}{t(t^2-1)}>0,\qquad
 \phi^{-1}(s)=e^s+\sqrt{1+e^{2s}}\quad(s\in\R).
\end{equation*}
Thus $\phi:(1,\infty)\to\R$ is an increasing smooth bijection.
Fix a nonnegative $C_c^2((1,2))$ probability density $\psi_0$.
Define
\begin{equation}
 \rho(t)=h(\log r(t))\frac{t^2+1}{t(t^2-1)}\quad(t>1),\qquad
 w(u,v)=\frac{\psi_0(v)}v\rho(u/v)\quad(u>v>0),
 \label{ar:real-profile}
\end{equation}
and set $\rho(t)=0$ for $t\le1$ and $w=0$ outside $u>v>0$.
Since $\phi'>0$ and $\operatorname{supp}h=[-2,2]$, the support
of $\rho$ is exactly
\begin{equation*}
 \operatorname{supp}\rho=[t_-,t_+],\qquad
 t_-=e^{-2}+\sqrt{1+e^{-4}},\qquad
 t_+=e^2+\sqrt{1+e^4}.
\end{equation*}
In particular, $1<t_-<t_+<\infty$, and
$\rho\in C_c^2(\R)$. The support of $w$ is contained in
\begin{equation*}
 \{(u,v):v\in\operatorname{supp}\psi_0,
                   \ t_-v\le u\le t_+v\},
\end{equation*}
a compact subset of $u>v>0$. The formula in
\eqref{ar:real-profile} is $C^2$ on this open set, and it
vanishes near its boundary. Thus its zero extension satisfies
$w\in C_c^2(\R^2)$.

The factor multiplying $h(\log r(t))$ in $\rho(t)$ is
$ds/dt$ for $s=\log r(t)$. Consequently this substitution gives
\begin{equation*}
 \rho(t)\,dt=h(s)\,ds,\qquad
 \int_1^\infty\rho(t)\,dt=\int_\R h(s)\,ds=1.
\end{equation*}
For fixed $v>0$, substitute $u=vt$, so that $du=v\,dt$.
It follows that
\begin{equation*}
 \begin{aligned}
 \int_{\R^2}w(u,v)\,du\,dv
 &=\int_0^\infty\psi_0(v)
          \left(\int_1^\infty\rho(t)\,dt\right)dv=1.
 \end{aligned}
\end{equation*}

For $\tau\in\R$, use $R(vt,v)=r(t)$ and the same two
substitutions to obtain
\begin{equation}
 \int_{u>v>0}w(u,v)R(u,v)^{i\tau}\,du\,dv=\int_0^\infty\psi_0(v)\,dv\int_1^\infty\rho(t)r(t)^{i\tau}\,dt=\int_\R h(s)e^{i\tau s}\,ds.
 \label{ar:Mellin-substitutions}
\end{equation}
These integrals are absolutely convergent, since the complex
factors have modulus one and both densities have integral one.
Expanding the fourfold convolution defining $h$ gives
\begin{equation*}
 \begin{aligned}
 \int_\R h(s)e^{i\tau s}\,ds
 =\int_{[-1/2,1/2]^4}
       e^{i\tau(a_1+a_2+a_3+a_4)}\,da_1\,da_2\,da_3\,da_4
 =\left(\int_{-1/2}^{1/2}e^{i\tau a}\,da\right)^4.
 \end{aligned}
\end{equation*}
Here Fubini's theorem applies to the real and imaginary parts;
see \cite[Theorem~5.32, p.~132]{AxlerMeasure}.
The absolute value of the four-variable integrand has integral
one. Direct integration gives
\begin{equation}
 \int_{-1/2}^{1/2}e^{i\tau a}\,da
 =\begin{cases}
 \dfrac{e^{i\tau/2}-e^{-i\tau/2}}{i\tau}
       =\dfrac{\sin(\tau/2)}{\tau/2},&\tau\ne0,\\[6pt]
 1,&\tau=0.
 \end{cases}
 \label{ar:interval-Fourier-integral}
\end{equation}
Combining \eqref{ar:Mellin-substitutions}--\eqref{ar:interval-Fourier-integral},
we obtain
\begin{equation}
 \int_{u>v>0}w(u,v)R(u,v)^{i\tau}\,du\,dv
 =\left(\frac{\sin(\tau/2)}{\tau/2}\right)^4\in[0,1],
 \label{ar:Mellin}
\end{equation}
where the value at $\tau=0$ is one. The interval bound follows
from $|\sin x|\le|x|$.

Choose an integer $C_0\ge2$ large enough that
\begin{equation}
 u+v<C_0-1\quad\text{on }\operatorname{supp}w.
 \label{ar:profile-bound}
\end{equation}
Every axis-parallel rectangle with both side lengths at most
$1/2$ that meets $\operatorname{supp}w$ is contained in
$\{u+v<C_0\}$. Indeed, if $(u_0,v_0)$ is a point of that
intersection, then every $(u,v)$ in the rectangle satisfies
\begin{equation*}
 u+v\le u_0+v_0+|u-u_0|+|v-v_0|
      \le u_0+v_0+1<C_0.
\end{equation*}
This is the purpose of the margin in \eqref{ar:profile-bound}.

Let $g:\Q_{>0}\to\mathbb S^1$ be completely multiplicative:
$g(q_1q_2)=g(q_1)g(q_2)$ for all $q_1,q_2\in\Q_{>0}$.
To form the finite kernel average, fix an integer $X\ge1$ and sum over
the integer pairs $(m,n)$ with $m>n>0$.  The continuous weight is sampled
at $(m/X,n/X)$, while $W_0(m,n)$ imposes the prescribed congruence weight.
For this reason we first define the unnormalized weighted sum
$I_X(g)$, and then normalize it by its value at the constant function
$1$:
\begin{equation}
 \begin{aligned}
 I_X(g)&=X^{-2}\sum_{\substack{m,n\in\Z\\m>n>0}}
    w(m/X,n/X)W_0(m,n)g(R(m,n)),\\
 Z_X&=I_X(1),\qquad K_X(g)=I_X(g)/Z_X
       \quad\text{when }Z_X>0.
 \end{aligned}
 \label{ar:kernel}
\end{equation}
The sum is finite because $w$ has compact support.
Every contributing pair satisfies $m>n>0$, so
$R(m,n)=(m^2-n^2)/(2mn)\in\Q_{>0}$ and $g(R(m,n))$ is defined.

The positivity of $Z_X$ follows from the quadrature estimate
in \Cref{ar:quadrature} below. By \Cref{ar:weight-cost},
$W_0$ is nonnegative, periodic modulo $D_0$, and has average
one on residue pairs. Applying \eqref{ar:normalization} with
$D=D_0$ and $V=W_0$ therefore gives, for a constant $A_w>0$
depending only on the fixed density $w$,
\begin{equation*}
 |Z_X-1|\le A_w\left(\frac{D_0}{X}+\frac{D_0^2}{X^2}\right).
\end{equation*}
Put $C_w=\max\{1,4A_w\}$. If $X\ge C_wD_0$, then
$D_0/X\le1/C_w\le1$, and hence
\begin{equation*}
 |Z_X-1|\le\frac{2A_w}{C_w}\le\frac12,
 \qquad \frac12\le Z_X\le\frac32.
\end{equation*}
Thus $K_X(g)$ is defined in this range. Its coefficients are
nonnegative and sum to one, so it is an average of the
unit-modulus values $g(R(m,n))$; in particular,
$K_X(1)=1$ and $|K_X(g)|\le1$.

%

\subsection{Quadrature}

%

In the quadrature estimates, $\|\cdot\|_1$ denotes the Lebesgue
integral of the absolute value on $\R^2$.

\begin{lemma}
\label{ar:quadrature}
Let $F\in C_c^2(\R^2)$. For every integer $d\ge1$ and every
coset $\mathbf a+d\Z^2$,
\begin{equation}
 \begin{split}
 \left|\sum_{n\in\mathbf a+d\Z^2}F(n)-d^{-2}\int_{\R^2}F\right|
 \le{}&d^{-1}\bigl(\|\partial_1F\|_1+\|\partial_2F\|_1\bigr)\\
 &+\|\partial_1\partial_2F\|_1.
 \end{split}
 \label{ar:residue-quadrature}
\end{equation}
In particular, let $D\ge1$ be an integer, $X>0$, and
$V:\Z^2\to[0,\infty)$ be periodic modulo $D$ in each variable,
with $\displaystyle D^{-2}\sum_{a,b\bmod D}V(a,b)=1$. Then
\begin{equation}
 \left|X^{-2}\sum_{m,n\in\Z}w(m/X,n/X)V(m,n)-1\right|
 \le C_w(D/X+D^2/X^2),
 \label{ar:normalization}
\end{equation}
where $w$ is the fixed density in \eqref{ar:real-profile} and
$C_w$ depends only on $w$.
\end{lemma}
\begin{proof}
We divide the proof into two steps.
\par
\medskip
\noindent\textbf{Step 1. Prove the residue-class quadrature estimate.}

For a compactly supported absolutely continuous function
$f:\R\to\C$, we first prove
\begin{equation}
 \left|\sum_{n\in a+d\Z}f(n)-d^{-1}\int_\R f\right|
 \le\int_\R|f'|.
 \label{ar:one-dimensional-quadrature}
\end{equation}
Put $t_i=a+di$ and $J_i=[t_i,t_i+d)$. The fundamental theorem
of calculus gives $$f(t)-f(t_i)=\int_{t_i}^t f'(s)\,ds$$ for
$t\in J_i$. Hence
\begin{equation}
 \begin{aligned}
 \left|f(t_i)-d^{-1}\int_{J_i}f(t)\,dt\right|
 \le d^{-1}\int_{J_i}\int_{t_i}^t|f'(s)|\,ds\,dt
 \le\int_{J_i}|f'(s)|\,ds.
 \end{aligned}
\end{equation}
Summing over the intervals, which partition $\R$, proves
\eqref{ar:one-dimensional-quadrature}.
Write $x_i=a_1+di$, $y_j=a_2+dj$, and define
\begin{equation}
 G(x)=\int_\R F(x,y)\,dy,\qquad
 A(x)=\int_\R|\partial_2F(x,y)|\,dy.
\end{equation}
The function $F$ and its derivatives vanish outside a common
compact rectangle. Differentiation under the integral gives
\begin{equation}
 G'(x)=\int_\R\partial_1F(x,y)\,dy,\qquad
 \int_\R|G'|\le\|\partial_1F\|_1.
\end{equation}
To control $A$, put
$B(s)=\int_\R|\partial_1\partial_2F(s,y)|\,dy$.
For $x<x'$, the reverse triangle inequality and the
fundamental theorem of calculus give
\begin{equation}
 \begin{aligned}
 |A(x')-A(x)|
 &\le\int_\R
       |\partial_2F(x',y)-\partial_2F(x,y)|\,dy\\
 &\le\int_\R\int_x^{x'}
       |\partial_1\partial_2F(s,y)|\,ds\,dy
   =\int_x^{x'}B(s)\,ds.
 \end{aligned}
\end{equation}
Since $B$ is continuous, bounded, and compactly supported,
this makes $A$ Lipschitz, hence absolutely continuous.
Dividing by $|x'-x|$ and taking the limit at points of
differentiability yields $|A'(x)|\le B(x)$ almost everywhere.
Consequently,
\begin{equation}
 \int_\R A=\|\partial_2F\|_1,\qquad
 \int_\R|A'|\le\int_\R B=\|\partial_1\partial_2F\|_1.
\end{equation}
Both $G$ and $A$ are compactly supported. Applying
\eqref{ar:one-dimensional-quadrature} to
$y\mapsto F(x_i,y)$, to $G$, and to $A$, respectively, gives
\begin{equation}
 \begin{aligned}
 \left|\sum_jF(x_i,y_j)-d^{-1}G(x_i)\right|&\le A(x_i),\\
 \left|\sum_iG(x_i)-d^{-1}\int_\R G\right|
   &\le\|\partial_1F\|_1,\\
 \sum_iA(x_i)&\le d^{-1}\|\partial_2F\|_1
                     +\|\partial_1\partial_2F\|_1.
 \end{aligned}
\end{equation}
Since $\int_\R G=\int_{\R^2}F$, the triangle inequality now yields
\begin{equation}
 \begin{aligned}
 \left|\sum_{i,j}F(x_i,y_j)-d^{-2}\int_{\R^2}F\right|
 &\le\sum_i A(x_i)
   +d^{-1}\left|\sum_iG(x_i)-d^{-1}\int_\R G\right|\\
 &\le d^{-1}\bigl(\|\partial_1F\|_1+\|\partial_2F\|_1\bigr)
      +\|\partial_1\partial_2F\|_1.
 \end{aligned}
\end{equation}
All sums are finite, and all integrals are absolutely
convergent, by compact support. This proves
\eqref{ar:residue-quadrature}.
\par
\medskip
\noindent\textbf{Step 2. Apply the estimate to the periodic normalization.}
For the second assertion, recall that $w\in C_c^2(\R^2)$
and $\int_{\R^2}w=1$. Set $F_X(x,y)=w(x/X,y/X)$.
The substitution $(x,y)=(Xu,Xv)$ gives
\begin{equation}
 \begin{aligned}
 \int_{\R^2}F_X&=X^2,\\
 \|\partial_iF_X\|_1&=X\|\partial_iw\|_1\quad(i=1,2),\\
 \|\partial_1\partial_2F_X\|_1&=\|\partial_1\partial_2w\|_1.
 \end{aligned}
\end{equation}
Put $L_w=\|\partial_1w\|_1+\|\partial_2w\|_1$ and
$M_w=\|\partial_1\partial_2w\|_1$.
For integers $0\le a,b<D$, let $S_{a,b}$ be the sum of $F_X$
over $(a,b)+D\Z^2$.
Applying \eqref{ar:residue-quadrature} to this coset gives
\begin{equation}
 \left|S_{a,b}-\frac{X^2}{D^2}\right|
 \le L_w\frac XD+M_w.
\end{equation}
Periodicity and $\sum_{a,b\bmod D}V(a,b)=D^2$ give
\begin{equation}
 \sum_{m,n\in\Z}F_X(m,n)V(m,n)-X^2
 =\sum_{a,b\bmod D}V(a,b)
       \left(S_{a,b}-\frac{X^2}{D^2}\right).
\end{equation}
Since $V\ge0$, it follows that
\begin{equation}
 \begin{aligned}
 \left|X^{-2}\sum_{m,n\in\Z}F_X(m,n)V(m,n)-1\right|
 &\le X^{-2}\sum_{a,b\bmod D}V(a,b)
       \left|S_{a,b}-\frac{X^2}{D^2}\right|\\
 &\le L_w\frac DX+M_w\frac{D^2}{X^2}.
 \end{aligned}
\end{equation}
This proves \eqref{ar:normalization} with a constant depending
only on $w$. The bound uses the total residue
mass $D^2$ of $V$, so it has no factor $\|V\|_\infty$.
\end{proof}
For the next lemma, fix $X>0$ and an integer $K\ge1$.
Partition $\R^2$ into a grid of axis-parallel half-open squares
of side $1/K$. Let $Q_\gamma$ be the finitely many squares
that meet $\operatorname{supp}w$. For each $\gamma$, define
\begin{equation*}
 C_\gamma=\{(m,n)\in\R^2:(m/X,n/X)\in Q_\gamma\}.
\end{equation*}
Thus $C_\gamma$ is the dilation of $Q_\gamma$ by the factor $X$.
Define $F:\R^2\to[0,\infty)$ by $F(m,n)=w(m/X,n/X)$, and set
\begin{equation*}
 I=\int_{\R^2}F=X^2,\qquad
 I_\gamma=\int_{C_\gamma}F,\qquad \pi_\gamma=I_\gamma/I.
\end{equation*}

For each integer $d\ge1$ and each $\mathbf a\in\Z^2$,
let $N_{\mathbf a}$ be the sum of $F$ over $\mathbf a+d\Z^2$.
When $N_{\mathbf a}>0$, define a random vector $\mathbf Z$
taking values in this coset by
\begin{equation*}
 \Pp(\mathbf Z=\mathbf z)=\frac{F(\mathbf z)}{N_{\mathbf a}}
 \qquad(\mathbf z\in\mathbf a+d\Z^2).
\end{equation*}
Almost surely, $\mathbf Z$ lies in a unique cell $C_\gamma$.
Let $\Gamma$ be the index of this cell.

\begin{lemma}
\label{ar:spatial-quadrature}
For every integer $d\ge1$ and every $\mathbf a\in\Z^2$,
\begin{equation}
 \sum_\gamma\left|
 \sum_{\mathbf n\in(\mathbf a+d\Z^2)\cap C_\gamma}F(\mathbf n)
             -\frac{I_\gamma}{d^2}\right|
 \le C_w I\left(\frac K{Xd}+\frac{K^2}{X^2}\right).
 \label{ar:partition-quadrature}
\end{equation}
If $d\le c_wX$, then $N_{\mathbf a}\ge I/(2d^2)>0$.
Thus $\mathbf Z$ and $\Gamma$ are defined.
In this case, for every integer $Q\ge1$ coprime to $d$
and every $\mathbf b\in(\Z/Q\Z)^2$,
\begin{equation}
 \sum_\gamma\left|
 \Pp(\Gamma=\gamma,\mathbf Z\equiv\mathbf b\pmod Q)
                  -\frac{\pi_\gamma}{Q^2}\right|
 \le C_w\left(\frac{dK}{XQ}+\frac{d^2K^2}{X^2}\right).
 \label{ar:spatial-residue-discrepancy}
\end{equation}
Here $C_w>0$ and the sufficiently small constant $c_w>0$
depend only on $w$.
\end{lemma}
\begin{proof}
We divide the proof into four steps.
\par
\medskip
\noindent\textbf{Step 1. Establish bounded-variation quadrature on one rectangle.}
We first extend \eqref{ar:one-dimensional-quadrature} to
compactly supported functions $h:\R\to\R$ of bounded
pointwise total variation, denoted by $\operatorname{Var}(h)$.
The values at jump points are included in this variation.
For $t_j=u+jd$, comparison on $[t_j,t_j+d]$ gives

\begin{equation}
 \begin{aligned}
 \left|h(t_j)-\frac1d\int_{t_j}^{t_j+d}h(t)\,dt\right|
 &\le \operatorname{Var}_{[t_j,t_j+d]}(h),\\
 \left|\sum_{j\in\Z}h(u+jd)-\frac1d\int_\R h\right|
 &\le \operatorname{Var}(h).
 \end{aligned}
 \label{ar:spatial-bv-quadrature}
\end{equation}
The second inequality follows by summing the first; variation
is additive on adjacent intervals. If $h$ is absolutely
continuous on $[\alpha,\beta]$, its half-open restriction satisfies
\begin{equation}
 \operatorname{Var}(\mathbf1_{[\alpha,\beta)}h)
 \le |h(\alpha)|+\int_\alpha^\beta|h'(x)|\,dx+|h(\beta)|.
 \label{ar:spatial-cutoff-variation}
\end{equation}
The jumps at $\alpha$ and $\beta$ have absolute values
$|h(\alpha)|$ and $|h(\beta)|$, respectively.
For a rectangle $C=[\alpha,\beta)\times[\eta,\zeta)$, define
\begin{equation}
 \begin{aligned}
 V_1(C)&=\int_C|\partial_1F|
   +\int_\eta^\zeta\bigl(|F(\alpha,y)|+|F(\beta,y)|\bigr)\,dy,\\
 V_2(C)&=\int_C|\partial_2F|
   +\int_\alpha^\beta\bigl(|F(x,\eta)|+|F(x,\zeta)|\bigr)\,dx,\\
 V_{12}(C)&=\int_C|\partial_1\partial_2F|
   +\int_\eta^\zeta\bigl(|\partial_2F(\alpha,y)|
                         +|\partial_2F(\beta,y)|\bigr)\,dy\\
 &\quad+\int_\alpha^\beta\bigl(|\partial_1F(x,\eta)|
                         +|\partial_1F(x,\zeta)|\bigr)\,dx
   +\sum_{\substack{\xi\in\{\alpha,\beta\}\\
                    \nu\in\{\eta,\zeta\}}}|F(\xi,\nu)|.
 \end{aligned}
 \label{ar:spatial-boundary-quantities}
\end{equation}
To apply \eqref{ar:spatial-bv-quadrature} in both coordinates, set
\begin{equation}
 \begin{aligned}
 g(x)&=\int_\eta^\zeta F(x,y)\,dy,\qquad
 G_C(x)=\mathbf1_{[\alpha,\beta)}(x)g(x),\\
 k(x)&=\int_\eta^\zeta|\partial_2F(x,y)|\,dy
                      +|F(x,\eta)|+|F(x,\zeta)|,\\
 A_C(x)&=\mathbf1_{[\alpha,\beta)}(x)k(x)
 \qquad(x\in\R).
 \end{aligned}
 \label{ar:spatial-auxiliary-functions}
\end{equation}
The function $g$ is continuously differentiable.
For $\alpha\le x<x'\le\beta$, the fundamental theorem of
calculus and $\bigl||r|-|s|\bigr|\le|r-s|$ give
\begin{equation}
 \begin{split}
 |k(x')-k(x)|\le\int_x^{x'}\biggl\{
   \int_\eta^\zeta|\partial_1\partial_2F(t,y)|\,dy
   +|\partial_1F(t,\eta)|+|\partial_1F(t,\zeta)|
   \biggr\}\,dt.
 \end{split}
\end{equation}
Thus $k$ is absolutely continuous, and $|k'|$ is bounded
almost everywhere by the integrand in braces.
Applying \eqref{ar:spatial-cutoff-variation} to $g$ and $k$
therefore yields
\begin{equation}
 \operatorname{Var}(G_C)\le V_1(C),\qquad
 \int_\R A_C=V_2(C),\qquad
 \operatorname{Var}(A_C)\le V_{12}(C).
\end{equation}
The jumps of $A_C$ at $\alpha$ and $\beta$ have absolute values
$k(\alpha)$ and $k(\beta)$, respectively.
Their sum is included in $V_{12}(C)$, by
\eqref{ar:spatial-boundary-quantities} and
\eqref{ar:spatial-auxiliary-functions}.
Write $x_i=a_1+id$ and $y_j=a_2+jd$.
By \eqref{ar:spatial-bv-quadrature} and
\eqref{ar:spatial-cutoff-variation},
\begin{equation}
 \left|\sum_{j\in\Z}\mathbf1_C(x_i,y_j)F(x_i,y_j)
                -\frac{G_C(x_i)}d\right|
 \le A_C(x_i)\qquad(i\in\Z).
\end{equation}
Since $\int_\R G_C=\int_C F$, a second application in $x$ gives
\begin{equation}
 \begin{aligned}
 \left|\sum_{\mathbf n\in(\mathbf a+d\Z^2)\cap C}F(\mathbf n)
                    -\frac1{d^2}\int_CF\right|
 &\le\sum_{i\in\Z}A_C(x_i)
   +\frac1d\left|\sum_{i\in\Z}G_C(x_i)-\frac1d\int_\R G_C\right|\\
 &\le\frac{V_1(C)+V_2(C)}d+V_{12}(C).
 \end{aligned}
 \label{ar:rectangle-variation}
\end{equation}
Each point of $\operatorname{supp}F$ belongs to exactly one $C_\gamma$.
The variation bounds use all four edges of $C$.

\par
\medskip
\noindent\textbf{Step 2. Sum the cell errors.}
Choose $R_0\ge1$ such that
$\operatorname{supp}w\subset[-R_0,R_0]^2$.
Let $M_w$ bound the supremum norms of $w$, its first
derivatives, and its mixed second derivative.
The scaling $F(x,y)=w(x/X,y/X)$ gives, on all of $\R^2$,

\begin{equation}
 |F|\le M_w,\qquad |\partial_iF|\le M_w/X\quad(i=1,2),
 \qquad |\partial_1\partial_2F|\le M_w/X^2.
 \label{ar:scaled-derivatives}
\end{equation}
Each cell has side length $h=X/K$, and at most
$(2R_0K+3)^2\le(2R_0+3)^2K^2$ cells meet the support.
Each cell has area $h^2$ and perimeter $4h$.
At each of the four vertices, $|F|\le M_w$.
Substituting \eqref{ar:scaled-derivatives} into
\eqref{ar:spatial-boundary-quantities} gives
\begin{equation}
 V_1(C_\gamma)+V_2(C_\gamma)\le M_w(2h^2/X+4h),\qquad
 V_{12}(C_\gamma)\le M_w(h^2/X^2+4h/X+4).
\end{equation}
Summing these inequalities and using $K\ge1$ and $I=X^2$, we obtain
\begin{equation}
 \begin{aligned}
 \sum_\gamma\bigl(V_1(C_\gamma)+V_2(C_\gamma)\bigr)
   &\le C_wIK/X,\\
 \sum_\gamma V_{12}(C_\gamma)
   &\le C_wIK^2/X^2.
 \end{aligned}
\end{equation}
Together with \eqref{ar:rectangle-variation}, this proves
\eqref{ar:partition-quadrature} for every $d\ge1$.
\par
\medskip
\noindent\textbf{Step 3. Normalize the original coset.}
To prove the second assertion, apply
\eqref{ar:residue-quadrature} to $F$ on the original coset.
The derivative integrals are $O_w(X)$ in first order and
$O_w(1)$ in mixed second order. Hence
\begin{equation}
 N_{\mathbf a}=\frac I{d^2}(1+\varepsilon_{\mathbf a}),
 \qquad |\varepsilon_{\mathbf a}|
       \le C_w\left(\frac dX+\frac{d^2}{X^2}\right).
\end{equation}
Choose $c_w>0$ so small that the last bound is at most $1/2$
whenever $d\le c_wX$.
Assume $d\le c_wX$ from now on. Then
$|\varepsilon_{\mathbf a}|\le1/2$ and
$N_{\mathbf a}\ge I/(2d^2)>0$, as asserted.
\par
\medskip
\noindent\textbf{Step 4. Impose the extra congruence.}
Let $Q\ge1$ be coprime to $d$, and let
$\mathbf b\in(\Z/Q\Z)^2$.
The Chinese remainder theorem, applied in each coordinate,
gives $\mathbf c\in\Z^2$ such that
\begin{equation}
 \{\mathbf n\in\mathbf a+d\Z^2:
                  \mathbf n\equiv\mathbf b\pmod Q\}
       =\mathbf c+dQ\Z^2.
\end{equation}
Define
$N_{\gamma,\mathbf b}
 =\sum_{\mathbf n\in(\mathbf c+dQ\Z^2)\cap C_\gamma}F(\mathbf n)$.
Applying \eqref{ar:partition-quadrature} with modulus $dQ$ yields
\begin{equation}
 \sum_\gamma\left|N_{\gamma,\mathbf b}
                         -\frac{I_\gamma}{d^2Q^2}\right|
 \le C_w I\left(\frac K{XdQ}+\frac{K^2}{X^2}\right).
\end{equation}
This step places no size restriction on $dQ$.
By the definition of $\mathbf Z$, the joint probability in
\eqref{ar:spatial-residue-discrepancy} is
$N_{\gamma,\mathbf b}/N_{\mathbf a}$.
Since $I_\gamma\ge0$ and $\displaystyle \sum_\gamma I_\gamma=I$,
the triangle inequality gives
\begin{equation}
 \begin{aligned}
 \sum_\gamma\left|
   \frac{N_{\gamma,\mathbf b}}{N_{\mathbf a}}
                       -\frac{\pi_\gamma}{Q^2}\right|
 &\le\frac1{N_{\mathbf a}}\sum_\gamma
   \left|N_{\gamma,\mathbf b}-\frac{I_\gamma}{d^2Q^2}\right|
   +\frac1{Q^2}\left|\frac I{d^2N_{\mathbf a}}-1\right|\\
 &\le C_w\left\{\frac{dK}{XQ}+\frac{d^2K^2}{X^2}
                +\frac{d/X+d^2/X^2}{Q^2}\right\}.
 \end{aligned}
\end{equation}
Here we used $N_{\mathbf a}^{-1}\le2d^2/I$ and
$|I/(d^2N_{\mathbf a})-1|
 =|\varepsilon_{\mathbf a}|/|1+\varepsilon_{\mathbf a}|
 \le2|\varepsilon_{\mathbf a}|$.
Finally, $K,Q\ge1$ imply
\begin{equation}
 \frac d{XQ^2}\le\frac{dK}{XQ},\qquad
 \frac{d^2}{X^2Q^2}\le\frac{d^2K^2}{X^2}.
\end{equation}
Absorbing these two terms proves
\eqref{ar:spatial-residue-discrepancy}.

\end{proof}
\subsection{Joint valuation laws}

Write $\mathbf Z=(M,N)$ for the random lattice point defined
before \Cref{ar:spatial-quadrature}.
Recall that $\Gamma=\gamma$ when $\mathbf Z\in C_\gamma$.
For a finite set $\mathcal P$ of primes, our application concerns
the joint distribution of
\begin{equation*}
 \left(\Gamma,
 \bigl(v_p(M-N),v_p(M+N),v_p(M),v_p(N)\bigr)_{p\in\mathcal P}
 \right).
 \end{equation*}

Let $\mu,\nu$ be probability measures on a countable set $S$.
Define
\begin{equation}
 \TV(\mu,\nu)=\frac12\sum_{x\in S}|\mu(x)-\nu(x)|.
 \label{ar:countable-total-variation}
\end{equation}
For $f:S\to\mathbb C$ with $|f(x)|\le1$ for every $x\in S$,
the triangle inequality gives
\begin{equation}
 |\mathbb E_\mu f-\mathbb E_\nu f|=\left|\sum_{x\in S}f(x)\bigl(\mu(x)-\nu(x)\bigr)\right|\le\sum_{x\in S}|\mu(x)-\nu(x)|=2\TV(\mu,\nu).
 \label{ar:total-variation-expectation}
\end{equation}

Fix $\kappa,B_1>0$.
Let $z,D$ be integers with $2\le z\le D^{1/2}$, and put
$s=\log D/\log z$. Assume $s\ge100$.
We use divisor sums to bound from below the indicator that
an integer has no prime divisor at most $z$.
The fundamental lemma of sieve theory, in its lower-bound form
\cite[Theorem~3.6(a)]{FordSieve2023}, provides a function
$\lambda^-:\N\to[-1,1]$ depending only on $z,D,\kappa,B_1$.
Its values $\lambda^-(e)$ are the sieve coefficients.
They satisfy $\lambda^-(1)=1$ and vanish unless $e\le D$
is squarefree and every prime divisor of $e$ is at most $z$.
Their divisor sums satisfy

\begin{equation}
 \sum_{e\mid n}\lambda^-(e)\le\mathbf1_{n=1}
 \quad\text{for every }n\ge1\text{ with prime divisors at most }z.
 \label{ar:lower-sieve-pointwise}
\end{equation}
For $n=1$, the sum equals $\lambda^-(1)=1$.
For $n>1$ whose prime divisors are all at most $z$,
\cite[(3.1)]{FordSieve2023} gives
$\sum_{e\mid n}\lambda^-(e)\le0$.
More generally, put $P_z=\prod_{\substack{p\le z\\p\ \mathrm{prime}}}p$.
Since a nonzero coefficient requires $e$ to be a squarefree
divisor of $P_z$, for every $n\ge1$ we have
\begin{equation}
 \sum_{e\mid n}\lambda^-(e)
 =\sum_{e\mid\gcd(n,P_z)}\lambda^-(e)
 \le\mathbf1_{\gcd(n,P_z)=1}.
\end{equation}
The inequality follows by applying
\eqref{ar:lower-sieve-pointwise} to $\gcd(n,P_z)$.
Let $\mathcal S_{\mathrm{sf}}$ be the set of squarefree positive integers.
Choose a number $h(p)\in[0,1)$ for each prime $p$.
Define $h:\mathcal S_{\mathrm{sf}}\to[0,1]$ by
\begin{equation}
 h(1)=1,\qquad h(e)=\prod_{\substack{p\mid e\\p\ \mathrm{prime}}}h(p)
 \quad(e\in\mathcal S_{\mathrm{sf}},\ e>1).
\end{equation}
Thus $h(ab)=h(a)h(b)$ whenever
$a,b\in\mathcal S_{\mathrm{sf}}$ are coprime.

Assume
\begin{equation}
 \prod_{\substack{a\le p\le b\\p\ \mathrm{prime}}}(1-h(p))^{-1}
 \le\left(\frac{\log b}{\log a}\right)^\kappa
          \exp(B_1/\log a)\quad(2\le a\le b\le z).
 \label{ar:sieve-dimension}
\end{equation}
Then

\begin{equation}
 \begin{split}
 \sum_e\lambda^-(e)h(e)
 =\prod_{\substack{p\le z\\p\ \mathrm{prime}}}(1-h(p))
 \bigl\{1+O_{\kappa,B_1}\bigl(
 e^{-s\log s+s\log\log\log s+C_{\kappa,B_1}s}\bigr)\bigr\}.
 \end{split}
 \label{ar:fundamental-lemma}
\end{equation}
Here the sum is over the support of $\lambda^-$, so only
squarefree values of $h$ occur.
The $O$-term is a relative error; its uniform bound is given below.
To verify \eqref{ar:fundamental-lemma}, extend $h$ to a
multiplicative function $\widetilde h:\N\to[0,1]$ by setting
\begin{equation}
 \widetilde h(1)=1,\qquad \widetilde h(p)=h(p),\qquad
 \widetilde h(p^j)=0\quad(p\text{ prime},\ j\ge2).
\end{equation}
For squarefree $e$, multiplicativity gives
$\widetilde h(e)=\prod_{\substack{p\mid e\\p\ \mathrm{prime}}}h(p)=h(e)$.
Consequently,
\begin{equation}
 \sum_{e\ge1}\lambda^-(e)\widetilde h(e)=\sum_e\lambda^-(e)h(e),\qquad
 \prod_{\substack{p\le z\\p\ \mathrm{prime}}}(1-\widetilde h(p))=\prod_{\substack{p\le z\\p\ \mathrm{prime}}}(1-h(p)).
 \label{ar:sieve-extension-identities}
\end{equation}
We check the remaining hypotheses of
\cite[Theorem~3.6(a)]{FordSieve2023}.
The function $\widetilde h$ satisfies $0\le\widetilde h(p)<1$.
Condition \eqref{ar:sieve-dimension} is precisely the condition
$(\Omega)$ in that theorem, with parameters $\kappa,B_1$.
The integers $z,D$ satisfy $2\le z\le D^{1/2}$.
The theorem uses $\max\{100,\log D/\log z\}$, which equals
$s$ under our assumption $s\ge100$.
Hence \cite[Theorem~3.6(a), p.~38]{FordSieve2023} gives
\begin{equation}
 \begin{aligned}
 \sum_{e\ge1}\lambda^-(e)\widetilde h(e)
 =\prod_{\substack{p\le z\\p\ \mathrm{prime}}}(1-\widetilde h(p))
 \left\{1+O_{\kappa,B_1}\!\left(
 \exp\!\left[-s\log s+s\log\log\log s+O_{\kappa,B_1}(s)\right]
 \right)\right\}.
 \end{aligned}
 \label{ar:sieve-source-estimate}
\end{equation}
Choose $A_{\kappa,B_1}>0$ to bound the implied constant in
the outer $O$-term of \eqref{ar:sieve-source-estimate}.
Choose $C_{\kappa,B_1}>0$ so that the inner $O_{\kappa,B_1}(s)$-term
in \eqref{ar:sieve-source-estimate} has absolute value at most
$C_{\kappa,B_1}s$.
Both constants depend only on $\kappa,B_1$.
The identities \eqref{ar:sieve-extension-identities} then give
\begin{equation}
 \left|
 \frac{\sum_e\lambda^-(e)h(e)}{\prod_{\substack{p\le z\\p\ \mathrm{prime}}}(1-h(p))}-1
 \right|
 \le A_{\kappa,B_1}
 \exp\!\left(-s\log s+s\log\log\log s+C_{\kappa,B_1}s\right).
 \label{ar:sieve-relative-error}
\end{equation}
The denominator is positive because $h(p)<1$ for every prime $p$.
This bound holds uniformly for all $z,D,h$ satisfying the stated
assumptions. It gives the precise meaning of the $O$-term in
\eqref{ar:fundamental-lemma} and proves that estimate.
See also \cite[Proposition~3.1]{Gorodetsky}.
We record the weighted lower bound in the construction of
\cite[Theorem~3.6(a), pp.~38--39]{FordSieve2023}.
Let $\mathcal P_1,\ldots,\mathcal P_t$ be its partition of all primes
$p\le z$, with products $P_j=\prod_{p\in\mathcal P_j}p$ and even cutoffs
$k_j\ge0$. For $1\le j\le t$, put
\begin{equation}
 \begin{aligned}
 U_j&=\sum_{\substack{e\mid P_j\\\omega(e)\le k_j}}\mu(e)h(e),
 &V_j&=\prod_{\substack{p\mid P_j\\p\ \mathrm{prime}}}(1-h(p)),\\
 W_j&=\sum_{\substack{e\mid P_j\\\omega(e)=k_j+1}}h(e),
 &L_j&=\log(V_j^{-1}).
 \end{aligned}
\end{equation}
Here $\mu$ is the M\"obius function and $\omega(e)$ counts distinct
prime divisors. By \cite[p.~17]{FordSieve2023}, $U_j\ge V_j>0$.
Define
\begin{equation}
 \begin{gathered}
 U=\prod_{j=1}^tU_j,\qquad V=\prod_{j=1}^tV_j,
 \qquad \eta=\sum_{j=1}^t\frac{W_j}{U_j},\\
 \mathfrak E=\sum_{j=1}^t
    \exp(L_j\mathbf1_{k_j>0})\frac{L_j^{k_j+1}}{(k_j+1)!},
 \qquad \Sigma=\sum_e\lambda^-(e)h(e)=U(1-\eta).
 \end{gathered}
\end{equation}
The identity for $\Sigma$ follows from
\cite[(3.2), p.~30]{FordSieve2023}.
The bounds in \cite[(2.5)--(2.6), p.~17, and p.~31]{FordSieve2023}
give $V\le U\le Ve^{\mathfrak E}$ and $0\le\eta\le\mathfrak E$.
To bound $\Sigma$ from above, take independent Bernoulli variables
$B_p$ with $\Pp(B_p=1)=h(p)$ for $p\le z$, and apply
\eqref{ar:lower-sieve-pointwise} to $\displaystyle N=\prod_{\substack{p\le z\\p\ \mathrm{prime}}}p^{B_p}$.
Taking expectations gives $\Sigma\le V$. Hence
\begin{equation}
 \begin{aligned}
 0\le V-\Sigma&=V-U+U\eta\le U\mathfrak E
                      \le V\mathfrak E e^{\mathfrak E},\\
 \Sigma&\ge U(1-\mathfrak E)\ge V(1-\mathfrak E)
                    &&(\mathfrak E\le1).
 \end{aligned}
 \label{ar:sieve-corrected-lower-bound}
\end{equation}
Thus the lower bound in \cite[(3.4), p.~31]{FordSieve2023}
is used with $\mathfrak E\le1$.
The parameter choices in \cite[(3.11), pp.~38--39]{FordSieve2023}
give
\begin{equation}
 \mathfrak E\ll_{\kappa,B_1}
 \exp\!\left(-s\log s+s\log\log\log s+C_{\kappa,B_1}s\right)
 \qquad(s\ge s_0(\kappa,B_1)).
 \label{ar:sieve-construction-decay}
\end{equation}
In particular, $\mathfrak E\le1/2$ for sufficiently large $s$.
For $100\le s\le s_0(\kappa,B_1)$, those choices give
$\mathfrak E=O_{\kappa,B_1}(1)$.
The first bound in \eqref{ar:sieve-corrected-lower-bound}
then supplies this remaining range after enlarging the constants
in \eqref{ar:sieve-relative-error}.

Fix from now on the four primitive forms
\begin{equation}
 L_1(m,n)=m-n,\quad L_2(m,n)=m+n,\quad
 L_3(m,n)=m,\quad L_4(m,n)=n.
 \label{ar:four-forms}
\end{equation}
For distinct $i,j\in\{1,2,3,4\}$, the coefficient matrix of
$(L_i,L_j)$ has determinant $\pm1$ or $\pm2$.
For $p\ge5$, this determinant is a unit in $\Z_p$.
Thus $(m,n)\mapsto(L_i(m,n),L_j(m,n))$ is a bijection of $\Z_p^2$.

\Needspace{24\baselineskip}
\begin{samepage}
\begin{lemma}
\label{ar:joint-valuation-data}
Let $p\ge5$ be prime, and let $\mathbf Z_p=(X,Y)$ be a
Haar-distributed random vector on $\Z_p^2$.
Put $\nu_r=v_p(L_r(\mathbf Z_p))$ for $1\le r\le4$, and write
$\boldsymbol\nu=(\nu_1,\ldots,\nu_4)$.
Almost surely, $\boldsymbol\nu$ has one of the forms
\begin{equation}
 (j,j,j,j),\qquad
 (j,\ldots,j,\underset{i}{j+k},j,\ldots,j),
 \quad j\ge0,\ k\ge1,\ 1\le i\le4,
 \label{ar:state-types}
\end{equation}
where $j,k,i$ are integers.
For each vector $\mathbf a=(a_1,\ldots,a_4)$ in
\eqref{ar:state-types}, define
\begin{equation*}
 \begin{aligned}
 E_p(\mathbf a)&=\{\nu_r\ge a_r\text{ for every }1\le r\le4\},\\
 g_p(\mathbf a)&=\Pp(E_p(\mathbf a)),\\
 h_{\mathbf a}(p)&=\Pp(\boldsymbol\nu\ne\mathbf a\mid E_p(\mathbf a)).
 \end{aligned}
 \end{equation*}
Then
\begin{equation}
 \begin{array}{c|c|c}
 \mathbf a&g_p(\mathbf a)&h_{\mathbf a}(p)\\\hline
 (j,j,j,j)&p^{-2j}&4/p-3/p^2\\
 a_i=j+k,\quad a_r=j\ (r\ne i)&p^{-2j-k}&2/p-1/p^2.
 \end{array}
 \label{ar:joint-valuation-table}
\end{equation}
\end{lemma}
\end{samepage}
\begin{proof}
We divide the proof into four steps.

\medskip
\noindent\textbf{Step 1. Establish the independent Haar coordinates.}
For $r\ne s$, the coefficient matrix of $(L_r,L_s)$ belongs
to $\operatorname{GL}_2(\Z_p)$ by \eqref{ar:four-forms}.
It permutes the residue pairs modulo every power of $p$,
so it preserves normalized Haar measure.
Thus $L_r(\mathbf Z_p)$ and $L_s(\mathbf Z_p)$ are independent
Haar-distributed random variables on $\Z_p$.
In particular,
\begin{equation*}
 \Pp(L_r(\mathbf Z_p)=0)
 \le\Pp(p^t\mid L_r(\mathbf Z_p))=p^{-t}
 \quad(t\ge1).
 \end{equation*}
Thus $\Pp(L_r(\mathbf Z_p)=0)=0$ for $1\le r\le4$.
All four valuations are therefore finite with probability one.

\medskip
\noindent\textbf{Step 2. Classify the possible valuation vectors.}
At a point where all four valuations are finite, set
$j=\min\{v_p(X),v_p(Y)\}$.
The inequality $v_p(X\pm Y)\ge j$ gives $\min_r\nu_r=j$.
Writing $\mathbf Z_p=p^j\mathbf W$, we obtain
\begin{equation*}
 \mathbf W\in\Z_p^2\setminus p\Z_p^2,\qquad
 \nu_r=j+v_p(L_r(\mathbf W))\quad(1\le r\le4).
 \end{equation*}
If two distinct forms vanished modulo $p$ at $\mathbf W$,
their invertible coefficient matrix would imply
$\mathbf W\in p\Z_p^2$.
Hence at most one of the four values $L_r(\mathbf W)$ is divisible
by $p$. If none is divisible by $p$, then $\nu_r=j$ for every $r$.
Otherwise, let $i$ be the unique index for which $p\mid L_i(\mathbf W)$,
and put $k=v_p(L_i(\mathbf W))\ge1$.
Then $\nu_i=j+k$ and $\nu_r=j$ for $r\ne i$.
These are the two possibilities in \eqref{ar:state-types}.

\medskip
\noindent\textbf{Step 3. Compute the equal-valuation row.}
For the first row of \eqref{ar:joint-valuation-table}, fix $j\ge0$ and
let $\mathbf a=(j,j,j,j)$.
The conditions on $L_3,L_4$ require both $X,Y$ to be divisible
by $p^j$, and these conditions also suffice for $L_1,L_2$.
Therefore
\begin{equation*}
 E_p(\mathbf a)=\{X,Y\in p^j\Z_p\},\qquad
 g_p(\mathbf a)=p^{-j}p^{-j}=p^{-2j}.
 \end{equation*}
Conditional on this event, $\mathbf W=p^{-j}\mathbf Z_p$ has
normalized Haar distribution on all of $\Z_p^2$.
Consequently, each of the $p^2$ residue pairs modulo $p$ has
conditional probability $p^{-2}$.
The event $\boldsymbol\nu\ne\mathbf a$ occurs exactly when
at least one $L_r(\mathbf W)$ vanishes modulo $p$.
Writing $\mathbf W=(x,y)$, the four zero sets are the lines
\begin{equation*}
 x-y=0,\qquad x+y=0,\qquad x=0,\qquad y=0
 \quad\text{in }\mathbb F_p^2.
 \end{equation*}
Each line contains $p$ points, and two distinct lines meet
only at the origin. Their union has $1+4(p-1)$ points, so
\begin{equation*}
 h_{\mathbf a}(p)=\frac{1+4(p-1)}{p^2}
 =\frac4p-\frac3{p^2}.
 \end{equation*}

\medskip
\noindent\textbf{Step 4. Compute the exceptional row and conclude.}
For the second row, let $a_i=j+k$ and $a_r=j$ for $r\ne i$,
where $j\ge0$ and $k\ge1$.
Choose $\ell\ne i$ and put
$(S,T)=(L_i(\mathbf Z_p),L_\ell(\mathbf Z_p))$.
These are independent Haar-distributed random variables.
The inverse coordinate map shows that $S,T\in p^j\Z_p$
is equivalent to $\mathbf Z_p\in p^j\Z_p^2$.
Consequently,
\begin{equation*}
 E_p(\mathbf a)=\{S\in p^{j+k}\Z_p,\ T\in p^j\Z_p\},\qquad
 g_p(\mathbf a)=p^{-(j+k)}p^{-j}=p^{-2j-k}.
 \end{equation*}
The event restricts the independent coordinates $S,T$ separately.
Thus, conditional on $E_p(\mathbf a)$, we can write
\begin{equation}
 S=p^{j+k}U,\qquad T=p^jV,
 \label{ar:state-conditional-coordinates}
\end{equation}
where $U,V$ are independent and have normalized Haar distribution
on $\Z_p$.

Since $L_i,L_\ell$ form a coordinate system, for every $r\ne i$
we can write
\begin{equation}
 L_r=\alpha_rL_i+\beta_rL_\ell,\qquad
 \alpha_r\in\Z_p,\quad\beta_r\in\Z_p^\times.
 \label{ar:state-other-forms}
\end{equation}
To verify that $\beta_r$ is a unit, take determinants of
coefficient matrices. The determinant for $(L_i,L_r)$ is
$\beta_r$ times the determinant for $(L_i,L_\ell)$,
and both determinants are units.
The identities \eqref{ar:state-conditional-coordinates} and
\eqref{ar:state-other-forms} give
\begin{equation*}
 L_r(\mathbf Z_p)=p^j(\alpha_rp^kU+\beta_rV),\qquad \alpha_rp^kU+\beta_rV\equiv\beta_rV\pmod p\quad(r\ne i).
\end{equation*}
where the congruence uses $k\ge1$.
Since $\beta_r$ is a unit, under this conditional law
\begin{equation*}
 \nu_i>a_i\iff p\mid U,\qquad
 \nu_r>a_r\iff p\mid V\quad(r\ne i).
 \end{equation*}
Thus $\boldsymbol\nu\ne\mathbf a$ exactly when $p\mid U$ or
$p\mid V$. Independence and inclusion--exclusion yield
\begin{equation*}
 h_{\mathbf a}(p)=\frac1p+\frac1p-\frac1{p^2}
 =\frac2p-\frac1{p^2}.
 \end{equation*}
This proves the second row of \eqref{ar:joint-valuation-table}.
In particular,
\begin{equation*}
 \begin{aligned}
 \Pp(\boldsymbol\nu=\mathbf a)
 &=g_p(\mathbf a)(1-h_{\mathbf a}(p))\\
 &=\begin{cases}
 p^{-2j}(p-1)(p-3)/p^2,&\mathbf a=(j,j,j,j),\\
 p^{-2j-k}(1-1/p)^2,&a_i=j+k,\ a_r=j\ (r\ne i).
 \end{cases}
 \end{aligned}
 \end{equation*}
Both expressions are positive for $p\ge5$, so every vector
in \eqref{ar:state-types} occurs with positive probability.
\end{proof}
For the next theorem, let $Y,y$ be real numbers with
$Y\ge3$ and $5\le y\le Y$, and let $\mathcal P$ be a set
of primes in $[5,y]$.
Let $\mathcal C$ be a nonempty finite set.
Let $\mathbf Z$ be a $\Z^2$-valued random vector and
$C$ a $\mathcal C$-valued random variable, defined on the same
probability space.
Assume that $\mathbf Z$ has finite support and that
$\Pp(L_i(\mathbf Z)=0)=0$ for $1\le i\le4$.

Let $\pi=(\pi_c)_{c\in\mathcal C}$ be a prescribed probability
vector: $\pi_c\ge0$ for every $c\in\mathcal C$ and
$\displaystyle \sum_{c\in\mathcal C}\pi_c=1$.
For a real constant $A_0\ge0$, consider the joint
residue-discrepancy condition\footnote{We use this phrase for
\eqref{ar:joint-discrepancy}. Its antecedents are estimates for
divisibility probabilities in the Kubilius model and Elliott's
generalization; see \cite[Sections~1.1--1.2, pp.~1--2]{Gorodetsky}.}
\begin{equation}
 \begin{split}
 \sum_{c\in\mathcal C}\left|
 \Pp(C=c,\mathbf Z\equiv\mathbf b\pmod Q)-\frac{\pi_c}{Q^2}
 \right|
 \le A_0\left(\frac1{YQ}+\frac1{Y^2}\right),
 \end{split}
 \label{ar:joint-discrepancy}
\end{equation}
for every positive integer $Q$ supported on $\mathcal P$ and every
$\mathbf b\in(\Z/Q\Z)^2$.

For each $p\in\mathcal P$, let $\mathbf Z_p$ be a
Haar-distributed random vector on $\Z_p^2$.
Assume that these vectors are independent.
Let $C^*$ be a $\mathcal C$-valued random variable, independent
of $(\mathbf Z_p)_{p\in\mathcal P}$, with
$\Pp(C^*=c)=\pi_c$ for every $c\in\mathcal C$.
Write
\begin{equation*}
 \mathbf V=(v_p(L_i(\mathbf Z)))_{p,i},\qquad
 \mathbf A=(v_p(L_i(\mathbf Z_p)))_{p,i},\qquad
 u=\frac{\log Y}{\log y}.
\end{equation*}
The indices in these vectors range over $p\in\mathcal P$
and $1\le i\le4$. Their common state space is the countable
set $(\N_0^4)^{\mathcal P}$, where
$\N_0=\{0,1,2,\ldots\}$. If $\mathcal P$ is empty, this
space consists of one empty vector. All products over an
empty set below are assigned the value one.
\Needspace{7\baselineskip}
\begin{samepage}
\begin{theorem}
\label{ar:joint-valuation-approximation}
There exist absolute constants $K\ge1$ and $0<c\le1$ such that,
for every choice of the preceding data satisfying
\eqref{ar:joint-discrepancy},
\begin{equation*}
 \TV((C,\mathbf V),(C^*,\mathbf A))
 \le K\left(e^{-cu}+\frac{A_0(1+\log Y)^{16}}{\sqrt Y}\right).
\end{equation*}
\end{theorem}
\end{samepage}
\begin{proof}
Throughout the proof, $B>0$ denotes an absolute constant
that may increase from one occurrence to the next. We divide the rest proof into six steps.

\par
\medskip
\noindent\textbf{Step 1. Verify the sieve condition.}
By \eqref{sc:chebyshev}, $\vartheta(t)\le4t$. Partial summation gives
\begin{equation}
 \begin{aligned}
 \sum_{\substack{p\le t\\p\ \mathrm{prime}}}\frac{\log p}{p}&\le B\log t &&(t\ge2),\\
 \sum_{\substack{a\le p\le b\\p\ \mathrm{prime}}}\frac1p
 &\le4\log\frac{\log b}{\log a}+\frac4{\log a}
 &&(2\le a\le b).
 \end{aligned}
 \label{ar:prime-estimates}
\end{equation}
Writing $\displaystyle\vartheta(a^-)=\sum_{\substack{p<a\\p\ \mathrm{prime}}}\log p$, the identities are
\begin{equation*}
 \begin{aligned}
 \sum_{\substack{p\le t\\p\ \mathrm{prime}}}\frac{\log p}{p}
 &=\frac{\vartheta(t)}t+\int_2^t\frac{\vartheta(v)}{v^2}\,dv,\\
 \sum_{\substack{a\le p\le b\\p\ \mathrm{prime}}}\frac1p
 &=\frac{\vartheta(b)}{b\log b}
   -\frac{\vartheta(a^-)}{a\log a}
   +\int_a^b\frac{\vartheta(v)(\log v+1)}{v^2(\log v)^2}\,dv.
 \end{aligned}
 \end{equation*}
The first is at most $4+4\log(t/2)$. The second is at most
$4\log(\log b/\log a)+4/\log a$.
Suppose $h(p)=0$ for $p<5$ and $0\le h(p)\le4/p$ for $p\ge5$.
The inequality $-\log(1-x)\le x+x^2/(2(1-x))$ for $0\le x<1$ yields
\begin{equation*}
 \begin{aligned}
 -\log(1-h(p))&\le\frac4p+\frac{40}{p^2} &&(p\ge5),\\
 \sum_{n\ge a}\frac1{n^2}
 &\le\int_{\lceil a\rceil-1}^{\infty}\frac{dt}{t^2}
 \le\frac1{a-1}\le\frac1{\log a} &&(a\ge2).
 \end{aligned}
\end{equation*}
Summing the first bound and using \eqref{ar:prime-estimates} gives
\begin{equation*}
 \prod_{\substack{a\le p\le b\\p\ \mathrm{prime}}}(1-h(p))^{-1}
 \le\left(\frac{\log b}{\log a}\right)^{16}
       \exp(56/\log a)\qquad(2\le a\le b).
 \end{equation*}
Thus \eqref{ar:sieve-dimension} holds with the fixed parameters
$\kappa=16$ and $B_1=56$.

\par
\medskip
\noindent\textbf{Step 2. Express the model probabilities and total variation.}
Let $\mathscr D\subset\N^4$ be the set of tuples obtained from
the local vectors in \eqref{ar:state-types} by
\begin{equation*}
 d_i=\prod_{p\in\mathcal P}p^{a_{i,p}}\quad(1\le i\le4),
 \qquad q(\mathbf d)=\operatorname{lcm}(d_1,d_2,d_3,d_4).
 \end{equation*}
Unique factorization identifies each valuation vector with its tuple.
We use this identification in events such as $\mathbf A=\mathbf d$.
For $\mathbf d\in\mathscr D$, put
\begin{equation*}
 \begin{gathered}
 \mathbf a_p=(v_p(d_1),\ldots,v_p(d_4)),\qquad
 g_{\mathbf d}=\prod_{p\in\mathcal P}g_p(\mathbf a_p),\\
 h_{\mathbf d}(p)=
 \begin{cases}h_{\mathbf a_p}(p),&p\in\mathcal P,\\0,&p\notin\mathcal P,
 \end{cases}
 \qquad
 h_{\mathbf d}(e)=\prod_{\substack{p\mid e\\p\ \mathrm{prime}}}h_{\mathbf d}(p)
 \quad(e\in\mathcal S_{\mathrm{sf}}).
 \end{gathered}
 \end{equation*}
Here $h_{\mathbf d}(1)=1$. By \Cref{ar:joint-valuation-data}, the exact
probability at $p$ is $g_p(\mathbf a_p)(1-h_{\mathbf d}(p))$.
The vectors $\mathbf Z_p$ are independent as $p$ varies over $\mathcal P$.
Also, $C^*$ is independent of $\mathbf A$. Therefore
\begin{equation}
 \begin{gathered}
 V_{\mathbf d}=\prod_{p\in\mathcal P}(1-h_{\mathbf d}(p)),
 \qquad \Pp(\mathbf A=\mathbf d)=g_{\mathbf d}V_{\mathbf d},\\
 m_{c,\mathbf d}:=\Pp(C^*=c,\mathbf A=\mathbf d)
 =\pi_cg_{\mathbf d}V_{\mathbf d},\qquad
 \sum_{c\in\mathcal C}\sum_{\mathbf d\in\mathscr D}m_{c,\mathbf d}=1.
 \end{gathered}
 \label{ar:model-state-mass}
\end{equation}
Both joint distributions have total probability one.
Using $|t|=2t_+-t$, where $t_+=\max\{t,0\}$ for $t\in\mathbb R$,
in \eqref{ar:countable-total-variation} gives
\begin{equation*}
 \Delta:=\TV((C,\mathbf V),(C^*,\mathbf A))
 =\sum_{c\in\mathcal C}\sum_{\mathbf d\in\mathscr D}
       \bigl(m_{c,\mathbf d}-\Pp(C=c,\mathbf V=\mathbf d)\bigr)_+.
\end{equation*}
The sum is restricted to $\mathscr D$ because
$\Pp(\mathbf A=\mathbf d)=0$ for $\mathbf d\notin\mathscr D$.
Set $S=Y^{1/4}$ and $H=Y^{1/8}$.
The terms with $q(\mathbf d)>H$ sum to at most $\Pp(q(\mathbf A)>H)$.

\par
\medskip
\noindent\textbf{Step 3. Bound the probability that $q(\mathbf A)>H$.}
Put $M_p=\max_{1\le i\le4}v_p(L_i(\mathbf Z_p))$.
Each $L_i(\mathbf Z_p)$ has normalized Haar distribution on $\Z_p$.
The union bound therefore gives
\begin{equation*}
 \Pp(M_p\ge k)\le4p^{-k}\qquad(k\ge1).
\end{equation*}
Choose $\alpha=\min\{1/4,1/\log y\}$. Telescoping and Tonelli's
theorem yield
\begin{equation*}
 \begin{aligned}
 \E p^{\alpha M_p}
 =1+\sum_{k\ge1}(p^{\alpha k}-p^{\alpha(k-1)})\Pp(M_p\ge k)
 \le 1+\frac{4(p^\alpha-1)}{p-p^\alpha}
 \le1+B\alpha\frac{\log p}{p}.
 \end{aligned}
 \end{equation*}
For the last inequality, $\alpha\log p\le1$ implies
$p^\alpha-1\le e\alpha\log p$, while
$p-p^\alpha\ge p(1-5^{-3/4})$.
Since $\displaystyle q(\mathbf A)=\prod_{p\in\mathcal P}p^{M_p}$, independence
and \eqref{ar:prime-estimates} give
\begin{equation*}
 \E q(\mathbf A)^\alpha
 \le\exp\left(B\alpha\sum_{\substack{p\le y\\p\ \mathrm{prime}}}\frac{\log p}{p}\right)\le B,
 \qquad \alpha\log H\ge\frac{\log5}{32}u.
 \end{equation*}
The last inequality uses $\alpha\log y=\min\{\log y/4,1\}
\ge\log5/4$. Markov's inequality now gives
\begin{equation*}
 \Pp(q(\mathbf A)>H)\le BH^{-\alpha}\le Be^{-\beta u},
 \qquad \beta=\frac{\log5}{32}.
\end{equation*}

\par
\medskip
\noindent\textbf{Step 4. Apply the residue-discrepancy condition.}
Let $\mathbf d\in\mathscr D$ satisfy $q=q(\mathbf d)\le H$.
Write $\mathcal E$ for the squarefree positive integers supported
on $\mathcal P$, including $1$. Define
\begin{equation*}
 \begin{aligned}
 E_{\mathbf d}&=\{d_i\mid L_i(\mathbf Z)\text{ for }1\le i\le4\},\\
 E_{\mathbf d,e}&=E_{\mathbf d}\cap
 \bigcap_{p\mid e}\bigcup_{i=1}^4
 \{v_p(L_i(\mathbf Z))\ge a_{i,p}+1\}\qquad(e\in\mathcal E).
 \end{aligned}
 \end{equation*}
This event depends only on $\mathbf Z$ modulo $qe$:
at $p\mid e$ the required exponents are at most $v_p(q)+1$,
and at other primes they are at most $v_p(q)$.
Let $\mathcal R_{\mathbf d,e}\subset(\Z/qe\Z)^2$ be its residue classes.
The Chinese remainder theorem and \Cref{ar:joint-valuation-data} give
\begin{equation*}
 \frac{|\mathcal R_{\mathbf d,e}|}{(qe)^2}
 =g_{\mathbf d}h_{\mathbf d}(e).
 \end{equation*}
By the triangle inequality, summing \eqref{ar:joint-discrepancy}
over these classes yields
\begin{equation}
 \begin{aligned}
 \sum_{c\in\mathcal C}\left|\Pp(C=c,E_{\mathbf d,e})
       -\pi_cg_{\mathbf d}h_{\mathbf d}(e)\right|
 &\le |\mathcal R_{\mathbf d,e}|A_0
       \left(\frac1{Yqe}+\frac1{Y^2}\right)\\
 &\le A_0\left(\frac{qe}{Y}+\frac{(qe)^2}{Y^2}\right).
 \end{aligned}
 \label{ar:surplus-discrepancy}
\end{equation}

\par
\medskip
\noindent\textbf{Step 5. Obtain a lower bound for each exact probability.}
Assume $u\ge u_0$, where $u_0\ge1600$ is a sufficiently large
absolute constant. Set $z=\lfloor y\rfloor$ and $D=\lfloor S/q\rfloor$.
Since $S/q\ge Y^{1/8}$ and $D\ge S/(2q)$,
\begin{equation}
 s=\frac{\log D}{\log z}
 \ge\frac{\log Y/8-\log2}{\log y}
 \ge u/8-1\ge u/16\ge100.
 \label{ar:sieve-ratio}
\end{equation}
Thus $2\le z\le D^{1/2}$. By \eqref{ar:joint-valuation-table},
$0\le h_{\mathbf d}(p)\le4/p<1$ for $p\in\mathcal P$.
Step~1 verifies \eqref{ar:sieve-dimension} with the same
parameters $(\kappa,B_1)=(16,56)$ for every $\mathbf d$.
Use the coefficients constructed in
\cite[Theorem~3.6(a), pp.~38--39]{FordSieve2023}, with these parameters.
Increase $u_0$ so that \eqref{ar:sieve-construction-decay} gives
$\mathfrak E\le1/2$ whenever $s\ge u_0/16$.
This choice is absolute, and \eqref{ar:sieve-ratio} gives
$s\ge u/16\ge u_0/16$, so $\mathfrak E\le1/2$.
Thus \eqref{ar:sieve-corrected-lower-bound} holds for every
$\mathbf d\in\mathscr D$ with $q(\mathbf d)\le H$.

Define the squarefree integer
\begin{equation*}
 T=\prod_{\substack{p\in\mathcal P\\
          \exists i:\ v_p(L_i(\mathbf Z))>a_{i,p}}}p.
 \qquad
 \{\mathbf V=\mathbf d\}=E_{\mathbf d}\cap\{T=1\}.
 \end{equation*}
Every prime divisor of $T$ belongs to $\mathcal P$, so it is
at most $\lfloor y\rfloor=z$.
Thus $T$ satisfies the prime-support restriction in
\eqref{ar:lower-sieve-pointwise}; see also
\cite[(3.1), p.~30]{FordSieve2023}.
For $e\in\mathcal E$,
$E_{\mathbf d}\cap\{e\mid T\}=E_{\mathbf d,e}$.
Apply \eqref{ar:lower-sieve-pointwise} to $T$, multiply by
$\mathbf1_{C=c}\mathbf1_{E_{\mathbf d}}$, and take expectations:
\begin{equation}
 \Pp(C=c,\mathbf V=\mathbf d)
 \ge\sum_{\substack{e\in\mathcal E\\e\le D}}
       \lambda^-(e)\Pp(C=c,E_{\mathbf d,e}).
 \label{ar:pointwise-sieve-application}
\end{equation}
Because $h_{\mathbf d}(p)=0$ outside $\mathcal P$,
\eqref{ar:fundamental-lemma} takes the form
\begin{equation*}
 \sum_{\substack{e\in\mathcal E\\e\le D}}
       \lambda^-(e)h_{\mathbf d}(e)
 =V_{\mathbf d}(1+\delta_{\mathbf d}),\qquad
 |\delta_{\mathbf d}|\le Be^{-u\log u/64}.
 \end{equation*}
Indeed, after increasing $u_0$, the exponent in
\eqref{ar:sieve-relative-error} satisfies
\begin{equation*}
 -s\log s+s\log\log\log s+B s
 \le-\tfrac12s\log s\le-\tfrac1{64}u\log u.
 \end{equation*}
Set $r_{c,\mathbf d,e}=
\Pp(C=c,E_{\mathbf d,e})-\pi_cg_{\mathbf d}h_{\mathbf d}(e)$.
Substitution in \eqref{ar:pointwise-sieve-application} gives
\begin{equation}
 \begin{aligned}
 \Pp(C=c,\mathbf V=\mathbf d)
 &\ge m_{c,\mathbf d}(1+\delta_{\mathbf d})
       +\sum_{\substack{e\in\mathcal E\\e\le D}}\lambda^-(e)r_{c,\mathbf d,e},\\
 \bigl(m_{c,\mathbf d}-\Pp(C=c,\mathbf V=\mathbf d)\bigr)_+
 &\le Be^{-u\log u/64}m_{c,\mathbf d}
       +\sum_{\substack{e\in\mathcal E\\e\le D}}|r_{c,\mathbf d,e}|.
 \end{aligned}
 \label{ar:exact-state-deficit}
\end{equation}
For the second inequality, take the absolute value of the signed
remainder and use $|\lambda^-(e)|\le1$.
Summing its first term over $c\in\mathcal C$ and $q(\mathbf d)\le H$
gives at most
$Be^{-u\log u/64}$ by \eqref{ar:model-state-mass}.

\par
\medskip
\noindent\textbf{Step 6. Sum the residue errors and conclude.}
Each $d_i$ divides $q(\mathbf d)$, so at most $\tau(q)^4$ tuples
have least common multiple $q$. Let $d_{16}(q)$ count ordered
factorizations of $q$ into sixteen positive integers. Then
\begin{equation}
 \tau(q)^4\le d_{16}(q),\qquad
 \sum_{q\le H}\frac{d_{16}(q)}q
 \le\left(\sum_{n\le H}\frac1n\right)^{16}
 \le(1+\log H)^{16}.
 \label{ar:state-counting-bound}
\end{equation}
For the first inequality at $q=p^a$, assign to
$(a_1,\ldots,a_4)\in\{0,\ldots,a\}^4$ the integers
\begin{equation*}
 n_{\boldsymbol\varepsilon}
 =\#\{1\le r\le a:(\mathbf1_{r\le a_i})_{i=1}^4
                       =\boldsymbol\varepsilon\},\qquad
 \sum_{\boldsymbol\varepsilon}n_{\boldsymbol\varepsilon}=a,
 \qquad
 a_i=\sum_{\boldsymbol\varepsilon}\varepsilon_i n_{\boldsymbol\varepsilon},
 \end{equation*}
where $\boldsymbol\varepsilon\in\{0,1\}^4$.
The last identity proves injectivity into the sixteen-part weak
compositions of $a$, counted by $d_{16}(p^a)$.
Multiplication over primes proves $\tau(q)^4\le d_{16}(q)$.
Expanding $d_{16}(q)$ and allowing each factor to range independently
up to $H$ proves the remaining inequalities in
\eqref{ar:state-counting-bound}.

Sum \eqref{ar:surplus-discrepancy} over the terms in
\eqref{ar:exact-state-deficit}. In the resulting numerical bound,
enlarge the sum over $e$ to all positive integers $e\le S/q$.
Using $\sum_{e\le R}e\le R^2$ and $\sum_{e\le R}e^2\le R^3$
for $R\ge1$, we obtain
\begin{align*}
 &A_0\sum_{q\le H}d_{16}(q)
       \sum_{e\le S/q}\left(\frac{qe}{Y}+\frac{(qe)^2}{Y^2}\right)\notag\\
 &\qquad\le A_0\left(\frac{S^2}{Y}+\frac{S^3}{Y^2}\right)
                \sum_{q\le H}\frac{d_{16}(q)}q\notag\\
 &\qquad\le A_0(Y^{-1/2}+Y^{-5/4})(1+\log H)^{16}
 \le2A_0Y^{-1/2}(1+\log Y)^{16}.
\end{align*}
Combining Steps~2, 3, and 5 with this bound gives
\begin{equation}
 \Delta\le Be^{-\beta u}+Be^{-u\log u/64}
          +2A_0Y^{-1/2}(1+\log Y)^{16}\qquad(u\ge u_0).
 \label{ar:combined-approximation-error}
\end{equation}
Choose one absolute $B\ge1$ valid for both exponential terms
in \eqref{ar:combined-approximation-error}, and set
\begin{equation*}
 c=\frac1{64},\qquad K=\max\{2B,e^{u_0/64}\}.
 \end{equation*}
Since $\beta>1/64$ and $\log u\ge1$ for $u\ge u_0$,
these constants give the stated estimate in that range.
For $1\le u<u_0$, we have $\Delta\le1\le Ke^{-u/64}$.
If $\mathcal P=\varnothing$, both valuation vectors are empty.
Taking $Q=1$ in \eqref{ar:joint-discrepancy} gives directly
\begin{equation*}
 \Delta=\frac12\sum_{c\in\mathcal C}|\Pp(C=c)-\pi_c|
 \le\frac{A_0}{2}(Y^{-1}+Y^{-2})\le A_0Y^{-1/2},
 \end{equation*}
which also satisfies the claimed bound.
\end{proof}
\subsection{Local approximation}

We compare the joint distribution of the valuations and the character
values of the unit parts of the four linear forms.
Residues modulo $D$ determine these quantities at primes dividing
$D_0q$, except on an event of small probability.

Fix $0<\delta<1/10$, and let $D_0=D_0(\delta)$ be defined by
\eqref{ar:parameters}. Let $q\ge1$ be an integer and $0<\eta\le1$.
For each prime $p\mid q$, put $a_p=v_p(q)$ and let $\chi_p$ be a
character of $(\Z/p^{a_p}\Z)^\times$.
Set $a_p=0$ for $p\nmid q$.
Let $\mathcal B=\{p\text{ prime}:p\mid D_0q\}$ and $J=|\mathcal B|$.
Since $2\mid D_0$, we have $J\ge1$.
For $p\in\mathcal B$, define
\begin{equation*}
 A=8J/\eta,\qquad b_p=v_p(D_0),\qquad
 k_p=\left\lceil\frac{\log A}{\log p}\right\rceil,\qquad K_p=b_p+k_p-1.
\end{equation*}
Set
\begin{equation}
 D=\prod_{p\in\mathcal B}p^{K_p+\max(a_p,1)}.
 \label{ar:matching-modulus}
\end{equation}
Then
\begin{equation}
 D_0\mid D,\qquad
 D\le D_0qA^J
       \prod_{\substack{p\in\mathcal B\\p\nmid q}}p
 \le D_0^2qA^J.
 \label{ar:modulus-bound}
\end{equation}
Indeed, $A\ge8$ gives $k_p\ge1$, so each exponent in $D$ is
at least $b_p$. Moreover,
\begin{equation*}
 D=D_0q\left(\prod_{p\in\mathcal B}p^{k_p-1}\right)
          \prod_{\substack{p\in\mathcal B\\p\nmid q}}p,
 \qquad p^{k_p-1}<A\quad(p\in\mathcal B).
 \end{equation*}
This gives the first upper bound in \eqref{ar:modulus-bound}.
The last product divides $D_0$, which gives the second.

Let $V:(\Z/D_0\Z)^2\to[0,\infty)$ satisfy
\begin{equation*}
 D_0^{-2}\sum_{\mathbf a\in(\Z/D_0\Z)^2}V(\mathbf a)=1,
\end{equation*}
and extend $V$ periodically to $\Z^2$.
Retain the density $w$ from \eqref{ar:real-profile}.
Let $\psi:\R^2\to\C$ be Borel measurable, with
$|\psi|\le1$ on $\operatorname{supp}w$.
Assume that $T\ge0$ and that, for every integer $k\ge1$,
\begin{equation*}
 |\psi(\mathbf x)-\psi(\mathbf x')|\le 2T/k
\end{equation*}
whenever $\mathbf x,\mathbf x'\in\operatorname{supp}w$ lie
in the same square of side $1/k$.
Let $X\ge1$ be an integer, and choose
\begin{equation*}
 K=\max\{1,\lceil8T/\eta\rceil\},\qquad
 X_*=\frac X{DK},\qquad 5\le y\le X_*.
\end{equation*}
Assume $D\le c_wX$, where $c_w>0$ is sufficiently small and
depends only on $w$. Let $\mathcal P$ be any set of primes
at most $y$ disjoint from $\mathcal B$.

Define
\begin{equation*}
 N_{X,V}=\sum_{m,n\in\Z}w(m/X,n/X)V(m,n).
 \end{equation*}
By \eqref{ar:normalization} and $D_0\le D\le c_wX$,
\begin{equation}
 \left|\frac{N_{X,V}}{X^2}-1\right|
 \le C_w(D_0/X+D_0^2/X^2)\le\frac12
 \label{ar:matching-positive-normalization}
\end{equation}
after decreasing $c_w$ if necessary. Thus $N_{X,V}>0$.
Let $\mathbf Z$ be the $\Z^2$-valued random vector with distribution
\begin{equation*}
 \Pp(\mathbf Z=(m,n))
 =\frac{w(m/X,n/X)V(m,n)}{N_{X,V}}
 \qquad((m,n)\in\Z^2).
 \end{equation*}

For the comparison distribution, let $\mathbf A$ take values in
$\{0,\ldots,D_0-1\}^2$, with
\begin{equation*}
 \Pp(\mathbf A=\mathbf a)=\frac{V(\mathbf a)}{D_0^2}
 \qquad\bigl(\mathbf a\in\{0,\ldots,D_0-1\}^2\bigr).
 \end{equation*}
Let $(\mathbf U_p)_{p\in\mathcal B\cup\mathcal P}$ be
independent Haar vectors in $\Z_p^2$, independent of $\mathbf A$.
Put
\begin{equation*}
 \mathbf Z_p=
 \begin{cases}
 \mathbf A+p^{b_p}\mathbf U_p,&p\in\mathcal B,\\
 \mathbf U_p,&p\in\mathcal P.
 \end{cases}
 \end{equation*}
Conditional on $\mathbf A=\mathbf a$, these vectors are independent.
For $p\in\mathcal B$, the distribution is normalized Haar measure
on $\mathbf a+p^{b_p}\Z_p^2$; when $b_p=0$, this is all of $\Z_p^2$.
For $p\in\mathcal P$, the distribution is normalized Haar measure
on $\Z_p^2$.

For $z\in\Q_p^\times$, write $u_p(z)=p^{-v_p(z)}z\in\Z_p^\times$.
For $p\mid q$, evaluate $\chi_p(u)$, $u\in\Z_p^\times$,
by reducing $u$ modulo $p^{a_p}$.
Using the forms in \eqref{ar:four-forms}, define the two random tuples
\begin{equation}
 \begin{aligned}
 \mathcal D_{\mathrm{int}}
 &=\Bigl((v_p(L_i(\mathbf Z)))_{p\in\mathcal B\cup\mathcal P,\,1\le i\le4},
       (\chi_p(u_p(L_i(\mathbf Z))))_{p\mid q,\,1\le i\le4}\Bigr),\\
 \mathcal D_{\mathrm{loc}}
 &=\Bigl((v_p(L_i(\mathbf Z_p)))_{p\in\mathcal B\cup\mathcal P,\,1\le i\le4},
       (\chi_p(u_p(L_i(\mathbf Z_p))))_{p\mid q,\,1\le i\le4}\Bigr).
 \end{aligned}
 \label{ar:local-data}
\end{equation}
All four $L_i(\mathbf Z)$ are positive almost surely because of
the support of $w$. Each $L_i$ has a coefficient that is a unit
in $\Z_p$, so $L_i(\mathbf Z_p)$ has no atoms, even conditional
on $\mathbf A$. In particular, $L_i(\mathbf Z_p)\ne0$ almost surely.
Thus both tuples in \eqref{ar:local-data} are defined almost surely.
They take values in the countable set
\begin{equation}
 \mathscr S=(\Z_{\ge0}^4)^{\mathcal B\cup\mathcal P}
       \times\prod_{\substack{p\mid q\\p\ \mathrm{prime}}}\{\chi_p(u):u\in(\Z/p^{a_p}\Z)^\times\}^{4}.
 \label{ar:matching-data-space}
\end{equation}
Write $\E_{\mathrm{loc}}$ for expectation with respect to
$\mathbf A$ and the vectors $\mathbf U_p$.
\begin{theorem}
\label{ar:matching}
Under the preceding hypotheses, every function
$F:\mathscr S\to\C$ with $|F|\le1$ satisfies
\begin{equation}
 \begin{split}
 \left|\E\!\left[F(\mathcal D_{\mathrm{int}})\psi(\mathbf Z/X)\right]
 -\E_{\mathrm{loc}}\!\left[F(\mathcal D_{\mathrm{loc}})\right]
     \int w\psi\right|
 \le C_w\bigl\{\eta+e^{-c\log X_*/\log y}\\
 \hspace{35mm}
 +X_*^{-1/2}(1+\log X_*)^{16}+D/X\bigr\}.
 \end{split}
 \label{ar:matching-error}
\end{equation}
Here $C_w$ depends only on $w$, and $c>0$ is absolute.
\end{theorem}
\begin{proof}
Let $0<c_w\le1$ be no larger than the constant in
\Cref{ar:spatial-quadrature}.
With $C_w$ as in \eqref{ar:normalization}, decrease $c_w$ so that
$C_w(c_w+c_w^2)\le1/4$. We divide the rest proof into seven steps.

\par
\medskip
\noindent\textbf{Step 1. Condition on a residue modulo $D_0$.}
Let $\mathbf a\in\{0,\ldots,D_0-1\}^2$ satisfy $V(\mathbf a)>0$.
Write $\Pp_{\mathbf a},\E_{\mathbf a}$ for the integer distribution
conditioned on $\mathbf Z\equiv\mathbf a\pmod{D_0}$, and
$\Pp^{\mathrm{loc}}_{\mathbf a},\E^{\mathrm{loc}}_{\mathbf a}$
for the local distribution conditioned on $\mathbf A=\mathbf a$.
Since $V$ is constant on $\mathbf a+D_0\Z^2$,
\begin{equation*}
 \begin{gathered}
 N_{\mathbf a}=\sum_{\mathbf z\in\mathbf a+D_0\Z^2}w(\mathbf z/X),
 \qquad N_{\mathbf a}\ge\frac{X^2}{2D_0^2}>0,\\
 \Pp_{\mathbf a}(\mathbf Z=\mathbf z)
 =\frac{w(\mathbf z/X)}{N_{\mathbf a}}
 \qquad(\mathbf z\in\mathbf a+D_0\Z^2).
 \end{gathered}
 \end{equation*}
The lower bound follows from \Cref{ar:quadrature} and $D_0\le D\le c_wX$.
We first prove the comparison with a bound independent of $\mathbf a$.

\par
\medskip
\noindent\textbf{Step 2. Bound large valuations at primes in $\mathcal B$.}
For $p\in\mathcal B$, the conditional local vector is
$\mathbf Z_p=\mathbf a+p^{b_p}\mathbf U_p$.
Each $L_i$ has a unit coefficient in $\Z_p$, so $L_i(\mathbf U_p)$
has Haar distribution on $\Z_p$.
If $v_p(L_i(\mathbf a))<b_p$, this valuation is unchanged by
adding $p^{b_p}L_i(\mathbf U_p)$.
Otherwise $L_i(\mathbf Z_p)/p^{b_p}$ has Haar distribution on $\Z_p$.
Since $K_p+1=b_p+k_p$ and $p^{k_p}\ge A$,
\begin{equation}
 \Pp^{\mathrm{loc}}_{\mathbf a}
    \bigl(v_p(L_i(\mathbf Z_p))>K_p\bigr)
 =\begin{cases}
 0,&v_p(L_i(\mathbf a))<b_p,\\
 p^{-k_p},&v_p(L_i(\mathbf a))\ge b_p
 \end{cases}
 \le A^{-1}.
 \label{ar:matching-local-tail}
\end{equation}
The second case uses the Haar probability $p^{-k_p}$ of $p^{k_p}\Z_p$.

Let $\mathcal R_{\mathrm{exc}}$ be the following set of residues:
\begin{equation*}
 \mathcal R_{\mathrm{exc}}
 =\left\{\mathbf c\in(\Z/D\Z)^2:
    L_i(\mathbf c)\equiv0\pmod{p^{K_p+1}}
    \text{ for some }p\in\mathcal B,\ 1\le i\le4\right\}.
 \end{equation*}
This is well defined because $v_p(D)\ge K_p+1$.
For $\mathbf c\notin\mathcal R_{\mathrm{exc}}$, the residue determines
each valuation $r=v_p(L_i)$ at $p\in\mathcal B$, and $r\le K_p$.
For $p\mid q$, it also determines $p^{-r}L_i$ modulo $p^{a_p}$, since
\begin{equation*}
 v_p(D)-r=K_p+a_p-r\ge a_p.
 \end{equation*}
Thus every component of \eqref{ar:local-data} belonging to $\mathcal B$
is fixed on each residue outside $\mathcal R_{\mathrm{exc}}$.

\par
\medskip
\noindent\textbf{Step 3. Compare the residue distributions modulo $D$.}
Put
\begin{equation*}
 \mathcal R_{\mathbf a}
 =\{\mathbf c\in(\Z/D\Z)^2:
             \mathbf c\equiv\mathbf a\pmod{D_0}\},
 \qquad r_0=|\mathcal R_{\mathbf a}|=(D/D_0)^2.
 \end{equation*}
The Chinese remainder theorem defines a unique
$\mathbf C\in(\Z/D\Z)^2$ by the congruences
$\mathbf C\equiv\mathbf Z_p\pmod{p^{v_p(D)}}$ for $p\in\mathcal B$.
Under $\Pp^{\mathrm{loc}}_{\mathbf a}$, the residues
$\mathbf Z_p\bmod p^{v_p(D)}$ are independent.
Each is uniformly distributed among the residues congruent to
$\mathbf a\pmod{p^{b_p}}$.
Therefore $\mathbf C$ is uniform on $\mathcal R_{\mathbf a}$.
The union bound and \eqref{ar:matching-local-tail} give
\begin{equation}
 \Pp^{\mathrm{loc}}_{\mathbf a}
       (\mathbf C\in\mathcal R_{\mathrm{exc}})
 \le\frac{4J}{A}=\frac\eta2.
 \label{ar:matching-local-exception}
\end{equation}

For $\mathbf c\in\mathcal R_{\mathbf a}$, define
$\displaystyle N_{\mathbf c}=\sum_{\mathbf z\in\mathbf c+D\Z^2}w(\mathbf z/X)$.
\Cref{ar:quadrature} gives
\begin{equation}
 \begin{aligned}
 \left|N_{\mathbf c}-\frac{X^2}{D^2}\right|
 &\le C_w(X/D+1),&
 N_{\mathbf c}&\ge\frac{X^2}{2D^2}>0,\\
 \left|N_{\mathbf a}-\frac{X^2}{D_0^2}\right|
 &\le C_w(X/D_0+1),&
 N_{\mathbf a}&\ge\frac{X^2}{2D_0^2}>0.
 \end{aligned}
 \label{ar:matching-coset-normalization}
\end{equation}
Consequently the integer residue probabilities are
$\theta_{\mathbf c}=N_{\mathbf c}/N_{\mathbf a}$, and
$\displaystyle \sum_{\mathbf c\in\mathcal R_{\mathbf a}}\theta_{\mathbf c}=1$.
Inserting $X^2/(D^2N_{\mathbf a})$ and using $r_0=(D/D_0)^2$ yields
\begin{equation*}
 \begin{aligned}
 \sum_{\mathbf c\in\mathcal R_{\mathbf a}}
       \left|\theta_{\mathbf c}-\frac1{r_0}\right|
 &\le\frac{r_0C_w(X/D+1)}{N_{\mathbf a}}
       +\frac{|X^2/D_0^2-N_{\mathbf a}|}{N_{\mathbf a}}\\
 &\le C_w\left(\frac D X+\frac{D^2}{X^2}
                       +\frac{D_0}X+\frac{D_0^2}{X^2}\right).
 \end{aligned}
\end{equation*}
Since $D_0\le D\le c_wX$, we obtain
\begin{equation*}
 \ell_{\mathbf a}:=
 \sum_{\mathbf c\in\mathcal R_{\mathbf a}}
       \left|\theta_{\mathbf c}-\frac1{r_0}\right|
 \le C_w(D/X+D^2/X^2)\le C'_wD/X.
\end{equation*}
The total-variation distance of the two residue distributions is
$\ell_{\mathbf a}/2$. Thus \eqref{ar:matching-local-exception} implies
\begin{equation}
 \Pp_{\mathbf a}(\mathbf Z\bmod D\in\mathcal R_{\mathrm{exc}})
 \le\frac\eta2+\frac{\ell_{\mathbf a}}2.
 \label{ar:matching-integer-exception}
\end{equation}

\par
\medskip
\noindent\textbf{Step 4. Approximate $\psi$ on the spatial squares.}
Use the squares $Q_\gamma$ of side $1/K$ from
\Cref{ar:spatial-quadrature}. Choose
$\mathbf x_\gamma\in Q_\gamma\cap\operatorname{supp}w$, and put
$\psi_\gamma=\psi(\mathbf x_\gamma)$ and
$\displaystyle \psi^{(K)}=\sum_\gamma\psi_\gamma\mathbf1_{Q_\gamma}$.
The oscillation assumption and the choice of $K$ give
\begin{equation}
 |\psi(\mathbf x)-\psi^{(K)}(\mathbf x)|
 \le2T/K\le\eta/4
 \qquad(\mathbf x\in\operatorname{supp}w).
 \label{ar:phase-approximation}
\end{equation}
This includes $T=0$. Set
\begin{equation}
 \begin{gathered}
 \pi_\gamma=\int_{Q_\gamma}w=I_\gamma/X^2,
 \qquad I_\psi=\int w\psi,\qquad
 I_\psi^{(K)}=\sum_\gamma\pi_\gamma\psi_\gamma=\int w\psi^{(K)},\\
 \sum_\gamma\pi_\gamma=1,
 \qquad |I_\psi^{(K)}-I_\psi|\le\eta/4.
 \end{gathered}
 \label{ar:matching-spatial-integral}
\end{equation}
Both $I_\psi$ and $I_\psi^{(K)}$ have modulus at most one.
Replacing $\psi(\mathbf Z/X)$ by $\psi^{(K)}(\mathbf Z/X)$
in an expectation multiplied by $F$ also costs at most $\eta/4$.

\par
\medskip
\noindent\textbf{Step 5. Apply \Cref{ar:joint-valuation-approximation} at primes in $\mathcal P$.}
For $\mathbf c\in\mathcal R_{\mathbf a}$, let $\Pp_{\mathbf c}$
be the integer distribution conditioned on
$\mathbf Z\equiv\mathbf c\pmod D$.
It assigns probability $w(\mathbf z/X)/N_{\mathbf c}$ to
$\mathbf z\in\mathbf c+D\Z^2$.
This is defined by \eqref{ar:matching-coset-normalization}; it has
finite support, on which all four forms are positive.
Let $\Gamma$ be the index of the square containing $\mathbf Z/X$.
Every positive integer $Q$ supported on $\mathcal P$ is coprime to $D$.
Applying \eqref{ar:spatial-residue-discrepancy} with $d=D$ gives,
for every $\mathbf b\in(\Z/Q\Z)^2$,
\begin{equation*}
 \begin{aligned}
 \sum_\gamma\left|
 \Pp_{\mathbf c}(\Gamma=\gamma,\mathbf Z\equiv\mathbf b\pmod Q)
       -\frac{\pi_\gamma}{Q^2}\right|
 \le C_w\left(\frac{DK}{XQ}+\frac{D^2K^2}{X^2}\right)
 =C_w\left(\frac1{X_*Q}+\frac1{X_*^2}\right).
 \end{aligned}
 \end{equation*}
This is \eqref{ar:joint-discrepancy} with $Y=X_*$ and $A_0=C_w$.
The cell index set is finite and nonempty, and $(\pi_\gamma)$ is
a probability vector. Also, $2,3\mid D_0$, so
$\mathcal P\subset[5,y]$, while $5\le y\le X_*$.
Thus all hypotheses of \Cref{ar:joint-valuation-approximation} hold.

Write $\mathbf V_{\mathcal P}$ and $\mathbf V^*_{\mathcal P}$ for
the valuation vectors at primes in $\mathcal P$ under the integer
and local distributions, respectively.
Let $\mu_{\mathbf c}$ be the law of $(\Gamma,\mathbf V_{\mathcal P})$
under $\Pp_{\mathbf c}$, and let $\nu$ be the product of
$(\pi_\gamma)$ and the law of $\mathbf V^*_{\mathcal P}$.
\Cref{ar:joint-valuation-approximation} gives
\begin{equation}
 \TV(\mu_{\mathbf c},\nu)\le\varepsilon_*
 :=C_w\left\{e^{-c\log X_*/\log y}
          +X_*^{-1/2}(1+\log X_*)^{16}\right\}.
 \label{ar:conditional-joint-valuation-error}
\end{equation}
The vectors $\mathbf Z_p$, $p\in\mathcal P$, are independent of
$\mathbf A$ and $(\mathbf U_p)_{p\in\mathcal B}$.
Since $\mathbf C$ is determined by these latter variables,
conditioning on $\mathbf A=\mathbf a$ and $\mathbf C=\mathbf c$
does not change the law of $\mathbf V^*_{\mathcal P}$.
Consequently $\nu$ is the same for every $\mathbf a,\mathbf c$.

\par
\medskip
\noindent\textbf{Step 6. Complete the comparison on the initial coset.}
For $\mathbf c\in\mathcal R_{\mathbf a}\setminus\mathcal R_{\mathrm{exc}}$,
define $F_{\mathbf c}:(\Z_{\ge0}^4)^{\mathcal P}\to\C$
by inserting the fixed data at
$\mathcal B$ and the valuation vector $\mathbf v$ at $\mathcal P$
into $F$. Set $F_{\mathbf c}=0$ for
$\mathbf c\in\mathcal R_{\mathbf a}\cap\mathcal R_{\mathrm{exc}}$.
Put $\Phi_{\mathbf c}(\gamma,\mathbf v)
=F_{\mathbf c}(\mathbf v)\psi_\gamma$.
Then $|\Phi_{\mathbf c}|\le1$ and
\begin{equation}
 \begin{aligned}
 \sum_{\mathbf c\in\mathcal R_{\mathbf a}}\theta_{\mathbf c}
       \int\Phi_{\mathbf c}\,d\mu_{\mathbf c}
 &=\E_{\mathbf a}\!\left[
    \mathbf1_{\{\mathbf Z\bmod D\notin\mathcal R_{\mathrm{exc}}\}}
    F(\mathcal D_{\mathrm{int}})\psi^{(K)}(\mathbf Z/X)\right],\\
 \frac1{r_0}\sum_{\mathbf c\in\mathcal R_{\mathbf a}}
       \int\Phi_{\mathbf c}\,d\nu
 &=\E^{\mathrm{loc}}_{\mathbf a}\!\left[
    \mathbf1_{\{\mathbf C\notin\mathcal R_{\mathrm{exc}}\}}
    F(\mathcal D_{\mathrm{loc}})\right]I_\psi^{(K)}.
 \end{aligned}
 \label{ar:matching-truncated-identities}
\end{equation}
The first identity follows by conditioning on $\mathbf Z\bmod D$.
For the second identity, work under $\Pp^{\mathrm{loc}}_{\mathbf a}$.
The vector $\mathbf C$ is uniform on $\mathcal R_{\mathbf a}$ and
independent of $\mathbf V^*_{\mathcal P}$.
Since $\nu$ is a product measure,
$\int\Phi_{\mathbf c}\,d\nu
=I_\psi^{(K)}\E^{\mathrm{loc}}_{\mathbf a}
[F_{\mathbf c}(\mathbf V^*_{\mathcal P})]$.
Averaging over $\mathbf c\in\mathcal R_{\mathbf a}$ gives the second identity.
By \eqref{ar:total-variation-expectation},
\eqref{ar:conditional-joint-valuation-error}, and $\sum\theta_{\mathbf c}=1$,
\begin{equation}
 \begin{aligned}
 \left|\sum_{\mathbf c\in\mathcal R_{\mathbf a}}\theta_{\mathbf c}
       \int\Phi_{\mathbf c}\,d\mu_{\mathbf c}
 -\frac1{r_0}\sum_{\mathbf c\in\mathcal R_{\mathbf a}}
       \int\Phi_{\mathbf c}\,d\nu\right|
 \le\sum_{\mathbf c\in\mathcal R_{\mathbf a}}
             2\theta_{\mathbf c}\varepsilon_*+\ell_{\mathbf a}
 =2\varepsilon_*+\ell_{\mathbf a}.
 \end{aligned}
 \label{ar:matching-averaged-tv}
\end{equation}

Define
\begin{equation*}
 B_{\mathbf a}=\E_{\mathbf a}
       [F(\mathcal D_{\mathrm{int}})\psi(\mathbf Z/X)],
 \qquad
 M_{\mathbf a}=\E^{\mathrm{loc}}_{\mathbf a}
       [F(\mathcal D_{\mathrm{loc}})].
 \end{equation*}
By \eqref{ar:matching-integer-exception}, the contribution from
$\mathbf Z\bmod D\in\mathcal R_{\mathrm{exc}}$ to the integer expectation
has modulus at most $\eta/2+\ell_{\mathbf a}/2$.
By \eqref{ar:matching-local-exception}, the contribution from
$\mathbf C\in\mathcal R_{\mathrm{exc}}$ to the local expectation,
multiplied by $I_\psi^{(K)}$, has modulus at most $\eta/2$.
These bounds use $|F|,|\psi^{(K)}|,|I_\psi^{(K)}|\le1$.

For the integer approximation, \eqref{ar:phase-approximation} gives
$|\E_{\mathbf a}[F(\mathcal D_{\mathrm{int}})
 (\psi-\psi^{(K)})(\mathbf Z/X)]|\le\eta/4$.
For the spatial integral, \eqref{ar:matching-spatial-integral} gives
$|(I_\psi-I_\psi^{(K)})M_{\mathbf a}|\le\eta/4$.
Combining these four bounds with
\eqref{ar:matching-truncated-identities} and
\eqref{ar:matching-averaged-tv}, we obtain
\begin{equation}
 |B_{\mathbf a}-I_\psi M_{\mathbf a}|
 \le\frac{3\eta}{2}+2\varepsilon_*+\frac{3\ell_{\mathbf a}}2
 \le C_w\left(\eta+\varepsilon_*+\frac DX\right).
 \label{ar:matching-conditional-bound}
\end{equation}
This bound is independent of $\mathbf a$.

\par
\medskip
\noindent\textbf{Step 7. Average over the residues modulo $D_0$.}
For every $\mathbf a\bmod D_0$, quadrature gives
\begin{equation*}
 \begin{gathered}
 N_{\mathbf a}=\frac{X^2}{D_0^2}(1+e_{\mathbf a}),
 \qquad |e_{\mathbf a}|\le
 E_0:=C_w(D_0/X+D_0^2/X^2)\le\frac14,\\
 \sigma_{\mathbf a}=\frac{V(\mathbf a)}{D_0^2},
 \qquad \bar e=\sum_{\mathbf a}\sigma_{\mathbf a}e_{\mathbf a},
 \qquad N_{X,V}=X^2(1+\bar e).
 \end{gathered}
 \end{equation*}
Here $\sum\sigma_{\mathbf a}=1$ and $|\bar e|\le E_0$.
The integer residue probabilities therefore satisfy
\begin{equation}
 \begin{aligned}
 \rho_{\mathbf a}
 &:=\Pp(\mathbf Z\equiv\mathbf a\pmod{D_0})
   =\sigma_{\mathbf a}\frac{1+e_{\mathbf a}}{1+\bar e},\\
 \sum_{\mathbf a}|\rho_{\mathbf a}-\sigma_{\mathbf a}|
 &=\frac{\sum_{\mathbf a}\sigma_{\mathbf a}|e_{\mathbf a}-\bar e|}
         {1+\bar e}
 \le\frac{2E_0}{1-E_0}\le C_wD_0/X.
 \end{aligned}
 \label{ar:periodic-mixture-error}
\end{equation}
Set $B_{\mathbf a}=M_{\mathbf a}=0$ when $V(\mathbf a)=0$;
then $\rho_{\mathbf a}=\sigma_{\mathbf a}=0$ as well.
Since $|I_\psi|,|M_{\mathbf a}|\le1$,
\begin{equation*}
 \begin{aligned}
 &\left|\E[F(\mathcal D_{\mathrm{int}})\psi(\mathbf Z/X)]
       -I_\psi\E_{\mathrm{loc}}[F(\mathcal D_{\mathrm{loc}})]\right|\\
 &\qquad\le\sum_{\mathbf a}\rho_{\mathbf a}
          |B_{\mathbf a}-I_\psi M_{\mathbf a}|
       +\sum_{\mathbf a}|\rho_{\mathbf a}-\sigma_{\mathbf a}|\\
 &\qquad\le C_w\left(\eta+\varepsilon_*+\frac DX\right).
 \end{aligned}
 \end{equation*}
The last inequality uses \eqref{ar:matching-conditional-bound},
\eqref{ar:periodic-mixture-error}, and $D_0\le D$.
Substituting the value of $\varepsilon_*$ from
\eqref{ar:conditional-joint-valuation-error} proves \eqref{ar:matching-error}.
The bound in \eqref{ar:periodic-mixture-error} uses $V$ only through
the probability vector $(\sigma_{\mathbf a})$.
Thus the constants are independent of $\|V\|_\infty$.
\end{proof}
\subsection{Model positivity}

Retain $0<\delta<1/10$ and the parameters $T_0,B,D_0$ in
\eqref{ar:parameters}. Let $\chi$ be a primitive Dirichlet character
of conductor $q=q_\chi$.
For each prime $p\mid q$, the Chinese remainder theorem determines
a character $\chi_p:(\Z/p^{v_p(q)}\Z)^\times\to\mathbb S^1$ such that
$\displaystyle\chi(n)=\prod_{\substack{p\mid q\\p\ \mathrm{prime}}}\chi_p(n)$ whenever $\gcd(n,q)=1$.
For $u\in\Z_p^\times$, evaluate $\chi_p(u)$ at the residue of $u$
modulo $p^{v_p(q)}$. For $p\nmid q$, let $\chi_p$ be the trivial
character on $\Z_p^\times$.

Let $g:\N\to\mathbb S^1$ be completely multiplicative, and let
$t\in\R$. Extend $g$ to $\Q_{>0}$ by $g(a/b)=g(a)/g(b)$ for
$a,b\in\N$. For every prime $p$, put
\begin{equation}
 z_p=\frac{g(p)p^{-it}}
 {\displaystyle\prod_{\substack{\ell\mid q,\ \ell\ne p\\\ell\ \mathrm{prime}}}\chi_\ell(p)}.
 \label{ar:prime-phases}
\end{equation}
Each factor in the denominator has modulus one, so $z_p\in\mathbb S^1$.
For $r\in\Q_p^\times$, write
$u_p(r)=p^{-v_p(r)}r\in\Z_p^\times$.
For every $r\in\Q_{>0}$,
\begin{equation}
 g(r)=r^{it}\prod_{p\ \mathrm{prime}}z_p^{v_p(r)}
       \prod_{\substack{\ell\mid q\\\ell\ \mathrm{prime}}}\chi_\ell(u_\ell(r)).
 \label{ar:character-factorization}
\end{equation}
Both sides are homomorphisms from $\Q_{>0}$ to $\mathbb S^1$.
For $r=p$, we have $u_p(p)=1$ and $u_\ell(p)=p$ for $\ell\ne p$,
so \eqref{ar:character-factorization} follows from
\eqref{ar:prime-phases}. Unique factorization then proves the identity
for every positive rational $r$. Only finitely many prime exponents
$v_p(r)$ are nonzero.

Assume $y\ge\max(B,q)$ and $g(p)=\chi(p)p^{it}$ for every $p>y$.
Then \eqref{ar:prime-phases} gives $z_p=1$ for $p>y$.
Take $\mathcal B=\{p\text{ prime}:p\mid D_0q\}$ and
$\mathcal P=\{p\text{ prime}:p\le y,\ p\notin\mathcal B\}$ in the local distribution
preceding \Cref{ar:matching}, with $V=W_0$.
Since each prime divisor of $D_0$ is at most $B$, the disjoint union
$\mathcal B\cup\mathcal P$ is the set of all primes at most $y$.
The normalization of $W_0$ follows from \Cref{ar:weight-cost}.

We first express $g(R)$ in terms of the data in \eqref{ar:local-data}.
For $p\le y$, set $c_p=2p^{-v_p(2)}\in\Z_p^\times$.
For $\mathbf z\in\Z_p^2$ with $\displaystyle\prod_{i=1}^4L_i(\mathbf z)\ne0$,
the identity $R=L_1L_2/(2L_3L_4)$ gives
\begin{equation}
 \begin{aligned}
 v_p(R(\mathbf z))
 &=v_p(L_1(\mathbf z))+v_p(L_2(\mathbf z))
   -v_p(L_3(\mathbf z))-v_p(L_4(\mathbf z))-v_p(2),\\
 u_p(R(\mathbf z))
 &=c_p^{-1}
   \frac{u_p(L_1(\mathbf z))u_p(L_2(\mathbf z))}
        {u_p(L_3(\mathbf z))u_p(L_4(\mathbf z))}.
 \end{aligned}
 \label{ar:model-ratio-data}
\end{equation}
On the data space $\mathscr S$ in \eqref{ar:matching-data-space}, define
$F:\mathscr S\to\mathbb S^1$ by
\begin{equation}
 \begin{aligned}
 F\bigl((a_{p,i}),(\xi_{p,i})\bigr)
 ={}\prod_{\substack{p\le y\\p\ \mathrm{prime}}}
 z_p^{a_{p,1}+a_{p,2}-a_{p,3}-a_{p,4}-v_p(2)}
 \prod_{\substack{p\mid q\\p\ \mathrm{prime}}}
 \frac{\xi_{p,1}\xi_{p,2}}
      {\chi_p(c_p)\xi_{p,3}\xi_{p,4}}.
 \end{aligned}
 \label{ar:model-test-function}
\end{equation}
Here $(a_{p,i})$ are the valuation coordinates and $(\xi_{p,i})$
are the character-value coordinates. All factors have modulus one.
When $q=1$, the character product is one.
By \eqref{ar:character-factorization} and \eqref{ar:model-ratio-data},
\begin{equation}
 \begin{aligned}
 F(\mathcal D_{\mathrm{int}})
 &=g(R(\mathbf Z))R(\mathbf Z)^{-it},\\
 F(\mathcal D_{\mathrm{loc}})
 &=\prod_{\substack{p\le y\\p\ \mathrm{prime}}}z_p^{v_p(r_p)}\chi_p(u_p(r_p)),
 \qquad r_p=R(\mathbf Z_p).
 \end{aligned}
 \label{ar:model-data-identities}
\end{equation}
The local quantities are defined almost surely, as noted after
\eqref{ar:local-data}. In particular, \eqref{ar:model-test-function}
includes the factor $2$ in the denominator of $R$.

Define $\psi:\R^2\to\C$ by $\psi(u,v)=R(u,v)^{it}$ for $u>v>0$,
and by zero elsewhere. Since $R(\mathbf Z/X)=R(\mathbf Z)$,
the integer integrand in \Cref{ar:matching} is $g(R(\mathbf Z))$.
Moreover, \eqref{ar:Mellin} gives $\displaystyle \int w\psi=K_\infty(t)\in[0,1]$.
We now compute the corresponding local expectation.
The size conditions needed to apply \Cref{ar:matching} are verified
in the proof of \Cref{ar:model-positivity}.

Put $\omega_j=(T_0+1-|j|)/(T_0(T_0+1))$ for $0<|j|\le T_0$.
The distribution at the primes $p\le B$ has the following description.
Choose an integer $J_0$ with $\Pp(J_0=j)=\omega_j$.
Conditional on $J_0=j$, the vectors $\mathbf Z_p$, $p\le B$, are
independent with densities $W_{p,j}$ relative to product Haar measure
on $\Z_p^2$. This follows from \eqref{ar:local-weight},
\eqref{ar:congruence-weight}, and \eqref{ar:weight-local-average}.
Their projective points therefore have distributions
$\mu_p(\,\cdot\mid E_{p,j,3})$.
Indeed, the two factors in the denominator of $W_{p,j}$ are the
probability $1-p^{-2}$ of primitivity and the probability $b_{p,j,3}$
of $E_{p,j,3}$ under $\mu_p$.
Thus the contribution of these primes is
\begin{equation*}
 K_{\le B}
 :=\sum_{0<|j|\le T_0}\omega_j
   \prod_{\substack{p\le B\\p\ \mathrm{prime}}}\left(
       \frac1{b_{p,j,3}}
       \int_{E_{p,j,3}}z_p^{V_p}\chi_p(U_p)\,d\mu_p\right).
 \end{equation*}
Here $V_p,U_p$ are the valuation and unit part of $R$ under $\mu_p$.
Applying \Cref{ar:Fejer} with $T=T_0$ and $s=3$ gives
$K_{\le B}\in[-1/T_0,1]$.

For $B<p\le y$, the vector $\mathbf Z_p$ has Haar distribution on
$\Z_p^2$, independently of the other such vectors and of the vectors
at primes at most $B$.
For $p\mid q$, this follows because $v_p(D_0)=0$:
conditional on $\mathbf A$, the vector $\mathbf A+\mathbf U_p$ is
Haar, and its distribution does not depend on $\mathbf A$.
For $p\nmid q$, it is part of the definition of the local distribution.
The homogeneity of $R$ and \eqref{ar:projective} show that the
expectation of $z_p^{v_p(r_p)}\chi_p(u_p(r_p))$ is $K_p(z_p,\chi_p)$.
Combining these expectations with \eqref{ar:model-data-identities} gives
\begin{equation}
 \E_{\mathrm{loc}}[F(\mathcal D_{\mathrm{loc}})]\int w\psi
 =K_\infty(t)\,K_{\le B}\,
           \prod_{\substack{B<p\le y\\p\ \mathrm{prime}}}K_p(z_p,\chi_p).
 \label{ar:ideal-main}
\end{equation}

If every $\chi_p$ with $B<p\le y$ is trivial,
\Cref{ar:local-factor-bounds} gives $K_p(z_p,\chi_p)\in[0,1]$.
Together with $K_\infty(t)\in[0,1]$ and $K_{\le B}\ge-1/T_0$,
this bounds \eqref{ar:ideal-main} below by $-1/T_0$.
If some $\chi_{p_0}$ with $B<p_0\le y$ is nontrivial,
\Cref{ar:local-factor-bounds} gives
$|K_{p_0}(z_{p_0},\chi_{p_0})|\le2/\sqrt{p_0}\le2/\sqrt B$.
Every other factor has modulus at most one, so the real part of
\eqref{ar:ideal-main} is at least $-2/\sqrt B$.
The bounds from \Cref{ar:local-factor-bounds} apply here because $B>5$.
In either case,
\begin{equation}
 \Re\left(K_\infty(t)K_{\le B}
                  \prod_{\substack{B<p\le y\\p\ \mathrm{prime}}}K_p(z_p,\chi_p)\right)
 \ge-\max(T_0^{-1},2B^{-1/2})\ge-\delta/100.
 \label{ar:ideal-lower-bound}
\end{equation}
The last inequality follows from $T_0\ge100/\delta$ and
$B=(200/\delta)^2$.

In \Cref{ar:model-positivity}, $M_1>0$ is a real parameter and
$C_*>0$ is an absolute constant. The parameter bounds are
\begin{equation}
 M_1\ge\exp(C_*\delta^{-3}),\qquad
 q_\chi\le M_1,\qquad |t|\le M_1.
 \label{ar:model-parameters}
\end{equation}

\begin{theorem}
\label{ar:model-positivity}
There are absolute constants $C_*,C_U>0$ such that, for each
$0<\delta<1/10$, an integer $2\le U\le C_U\delta^{-3}$,
depending only on $\delta$, can be chosen so that, whenever
\eqref{ar:model-parameters} holds, $y\ge M_1$, and
$X\ge(M_1y)^U$ is an integer, every finite character model $g$
with $g(p)=\chi(p)p^{it}$ for $p>y$ satisfies
\begin{equation}
 \Re K_X(g)\ge-\delta/50.
 \label{ar:model-positivity-bound}
\end{equation}
\end{theorem}
\begin{proof}
All constants depending on the fixed density $w$ are absolute.
Write $q=q_\chi$ and $W=\|W_0\|_\infty$. We divide the rest proof into six steps.

\par
\medskip
\noindent\textbf{Step 1. Choose the approximation tolerance.}
Let $C_{\mathrm{match}}$ be the constant $C_w$ in
\eqref{ar:matching-error}. Fix
$C_\eta\ge\max\{1,400C_{\mathrm{match}}\}$ and put
\begin{equation*}
 \eta=\frac{\delta}{C_\eta(1+W)}\in(0,1).
 \end{equation*}
By \Cref{ar:weight-cost}, there is an absolute $A_\eta\ge1$ such that
\begin{equation}
 \log(1/\eta)
 =\log C_\eta+\log(1+W)+\log(1/\delta)
 \le A_\eta\delta^{-3}.
 \label{ar:positivity-tolerance-size}
\end{equation}
\Cref{ar:weight-cost} also gives $\log D_0\le C\delta^{-3}$.
Choose $C_*\ge1$ large enough that
\eqref{ar:model-parameters} implies $M_1\ge\max\{D_0,B,5\}$.
In particular, $y\ge\max\{B,q,5\}$.

\par
\medskip
\noindent\textbf{Step 2. Bound the modulus.}
Use $D$ from \eqref{ar:matching-modulus}, and let $J$ be the number
of distinct prime divisors of $D_0q$.
Order them as $p_1<\cdots<p_J$. Since $p_j\ge j+1$,
\begin{equation*}
 (J+1)!\le\prod_{j=1}^Jp_j\le D_0q,
 \qquad 2^J\le D_0q.
 \end{equation*}
For $J\ge4$, at least $J/2$ factors in $(J+1)!$ are at least $J/2$.
Thus $\log((J+1)!)\ge(J/2)\log(J/2)$.
Including $J<4$ by enlarging an absolute constant gives
\begin{equation*}
 J\log(2J)\le C\log(D_0q),\qquad
 J\le\frac{\log(D_0q)}{\log2}.
 \end{equation*}
By \eqref{ar:modulus-bound},
\begin{align}
 \log D
 &\le2\log D_0+\log q+J\log(8J)+J\log(1/\eta) \notag\\
 &\le C(1+\log(1/\eta))\log(D_0q) \notag\\
 &\le C_D\delta^{-3}\log M_1.
 \label{ar:modulus-size}
\end{align}
The last line uses \eqref{ar:positivity-tolerance-size} and
$D_0,q\le M_1$.

\par
\medskip
\noindent\textbf{Step 3. Bound the spatial grid size.}
Let $L_w$ bound $|\nabla\log R|$ on a fixed compact convex subset
of $u>v>0$ containing $\operatorname{supp}w$.
For $\mathbf x,\mathbf x'\in\operatorname{supp}w$ in the same square
of side $1/k$, the mean value inequality gives
\begin{equation}
 \begin{aligned}
 |\psi(\mathbf x)-\psi(\mathbf x')|
 &\le |t|\,|\log R(\mathbf x)-\log R(\mathbf x')|\\
 &\le\sqrt2L_w|t|/k\le2T/k,
 \qquad T=\max\{1,L_w\}(1+|t|).
 \end{aligned}
 \label{ar:positivity-spatial-oscillation}
\end{equation}
Choose $K=\max\{1,\lceil8T/\eta\rceil\}$ as in \Cref{ar:matching}.
Then
\begin{equation}
 \log K\le C_w+\log(1+|t|)+\log(1/\eta)
          \le C_K\log M_1.
 \label{ar:grid-size}
\end{equation}
Here we used $|t|\le M_1$, \eqref{ar:positivity-tolerance-size},
and $\log M_1\ge C_*\delta^{-3}$.

\par
\medskip
\noindent\textbf{Step 4. Choose the exponent $U$.}
Let $s_0\ge1$ be an absolute constant such that
$16\log(1+s)\le s/4$ for $s\ge s_0$.
Let $c>0$ be the constant in \eqref{ar:matching-error}.
Let $0<c_w\le1$ be the constant in the condition $D\le c_wX$
imposed before \Cref{ar:matching}.
Choose an absolute $C'\ge2$ satisfying
\begin{equation}
 C'\ge2(C_D+C_K),\qquad cC'\ge A_\eta,\qquad
 \frac32C'C_*\ge\max\{s_0,4A_\eta,\log(1/c_w)\}.
 \label{ar:positivity-constant-choices}
\end{equation}
Set $U=\lceil C'\delta^{-3}\rceil$.
The bounds \eqref{ar:modulus-size} and \eqref{ar:grid-size} give
\begin{equation*}
 \log(DK)\le(C_D+C_K)\delta^{-3}\log M_1
            \le\frac U2\log M_1.
 \end{equation*}
Put $X_*=X/(DK)$, $s=\log X_*$, and
$\sigma=(3U/2)\log M_1$.
Since $X\ge(M_1y)^U$ and $y\ge M_1$,
\begin{equation}
 \begin{aligned}
 X_*&\ge M_1^{U/2}y^U,\\
 s&\ge U\log y+(U/2)\log M_1\ge\sigma,\\
 \sigma&\ge\frac32C'C_*\delta^{-6}
       \ge\max\{s_0,4\log(1/\eta),\log(1/c_w)\}.
 \end{aligned}
 \label{ar:positivity-effective-size}
\end{equation}
The last inequality follows from
\eqref{ar:positivity-tolerance-size},
\eqref{ar:positivity-constant-choices}, and $\delta^{-3}\ge1$.
In particular, $5\le y\le X_*$.

\par
\medskip
\noindent\textbf{Step 5. Verify the error bounds and normalization.}
The choices above give
\begin{equation}
 \begin{aligned}
 e^{-c\log X_*/\log y}
 &\le e^{-cU}\le e^{-A_\eta\delta^{-3}}\le\eta,\\
 X_*^{-1/2}(1+\log X_*)^{16}
 &=e^{-s/2+16\log(1+s)}\le e^{-s/4}\le\eta,\\
 D/X&\le M_1^{-U/2}y^{-U}\le e^{-\sigma}
       \le\min\{\eta,c_w\}.
 \end{aligned}
 \label{ar:positivity-error-terms}
\end{equation}
For the second line, use $s\ge s_0$ and
$s/4\ge\log(1/\eta)$ from \eqref{ar:positivity-effective-size}.
For the third, use $D\le DK\le M_1^{U/2}$ and
$X\ge(M_1y)^U$.
Thus $D\le c_wX$, as required in \Cref{ar:matching}.
Since $D_0\mid D$, \Cref{ar:weight-cost} and
\eqref{ar:normalization} also give
\begin{equation*}
 |Z_X-1|\le C_w(D_0/X+D_0^2/X^2)\le\frac12.
 \end{equation*}
The last inequality is \eqref{ar:matching-positive-normalization}
with $V=W_0$. Hence $Z_X\ge1/2$.

\par
\medskip
\noindent\textbf{Step 6. Apply the local approximation.}
Take $V=W_0$ and $F$ as in \eqref{ar:model-test-function}.
Define $\psi:\R^2\to\C$ by $\psi(u,v)=R(u,v)^{it}$ for $u>v>0$,
and by zero elsewhere.
By \Cref{ar:weight-cost}, $V$ is nonnegative, is periodic modulo $D_0$,
and has residue average one. The definition of $F$ gives $|F|=1$.
By \eqref{ar:positivity-spatial-oscillation},
$|\psi(\mathbf x)-\psi(\mathbf x')|\le2T/k$ for every integer $k\ge1$
and all $\mathbf x,\mathbf x'\in\operatorname{supp}w$ lying in the
same square of side $1/k$.
The bounds \eqref{ar:positivity-effective-size} and
\eqref{ar:positivity-error-terms} give $5\le y\le X_*$ and $D\le c_wX$.
We may therefore apply \Cref{ar:matching}.
Moreover, $N_{X,W_0}=X^2Z_X$ and
\eqref{ar:model-data-identities} give the exact identity
\begin{equation*}
 \E[F(\mathcal D_{\mathrm{int}})\psi(\mathbf Z/X)]
 =\frac{\displaystyle\sum_{\substack{m,n\in\Z\\m>n>0}}
       w(m/X,n/X)W_0(m,n)g(R(m,n))}{X^2Z_X}
 =K_X(g).
 \end{equation*}
Using \eqref{ar:ideal-main}, \eqref{ar:matching-error}, and
\eqref{ar:positivity-error-terms}, we obtain
\begin{equation*}
 \left|K_X(g)-K_\infty(t)K_{\le B}
                    \prod_{\substack{B<p\le y\\p\ \mathrm{prime}}}K_p(z_p,\chi_p)\right|
 \le4C_{\mathrm{match}}\eta\le\delta/100.
 \end{equation*}
\eqref{ar:ideal-lower-bound} therefore yields
\begin{equation*}
 \Re K_X(g)\ge-\delta/100-\delta/100=-\delta/50,
 \end{equation*}
which is \eqref{ar:model-positivity-bound}.
Finally,
$2\le U\le C'\delta^{-3}+1\le(C'+1)\delta^{-3}$.
Thus $C_U=C'+1$ is admissible, and $U$ depends only on $\delta$.
\end{proof}

\section{Weighted concentration}
\label{con:section}

We compare multiplicative functions that agree at small primes.
Although $\|W_0\|_\infty$ depends on $\delta$, the comparison
constant depends only on the primes dividing its period.
The key property is invariance under simultaneous multiplication
by units.

We first state the additive-function estimate that we shall use.
For a real additive function $H:\N\to\R$ and a real number
$Y\ge2$, define
\begin{equation*}
 A_H(Y)=\sum_{\substack{p\ \mathrm{prime},\,j\in\N\\p^j\le Y}}\frac{H(p^j)}{p^j},
 \qquad
 B_H(Y)=\sum_{\substack{p\ \mathrm{prime},\,j\in\N\\p^j\le Y}}\frac{|H(p^j)|^2}{p^j}.
\end{equation*}
Both sums range over primes $p$ and integers $j\ge1$.
The Tur\'an--Kubilius inequality\footnote{There is an absolute constant $C>0$ such that
every additive function $f:\N\to\C$, meaning
$f(ab)=f(a)+f(b)$ whenever $(a,b)=1$, satisfies
\[
 \sum_{n\le x}\left|f(n)-\sum_{\substack{p\ \mathrm{prime},\,j\in\N\\p^j\le x}}\frac{f(p^j)}{p^j}\right|^2
 \le Cx\sum_{\substack{p\ \mathrm{prime},\,j\in\N\\p^j\le x}}\frac{|f(p^j)|^2}{p^j}
 \qquad(x\ge2).
\]
Here $x$ is real, and the prime-power sums range over primes $p$
and integers $j\ge1$. See \cite[p.~893, equation~(1)]{Elliott}.}
states that there is an absolute constant $C_{\rm TK}$ such that,
for every such $H$ and integer $M\ge2$,
\begin{equation}
 \frac1M\sum_{n\le M}|H(n)-A_H(M)|^2
 \le C_{\rm TK}B_H(M).
 \label{con:TK}
\end{equation}

\begin{lemma}
\label{con:prime-estimates}
Let $C_E\ge1$ be the absolute constant in
\eqref{sc:bsz-euler-bounds}.  For $Z\ge4$ and $M\ge\sqrt Z$,
\begin{equation}
 \sum_{\substack{p\ \mathrm{prime},\,j\in\N\\p^j\le Z,\ p^j>M}}p^{-j}
 \le\log(2C_E)+1.
 \label{con:prime-power-interval}
\end{equation}
If every prime factor of a positive integer $D$ is at most
$B\ge2$, then
\begin{equation*}
 \frac D{\varphi(D)}\le C_E\log(2B),
\end{equation*}
where $\varphi$ is Euler's totient\footnote{For each positive integer
$n$, $\varphi(n)=\#\{a\in\Z:1\le a\le n,\ \gcd(a,n)=1\}$,
the number of integers from $1$ to $n$ that are coprime to $n$.}
and $\varphi(1)=1$.
\end{lemma}

\begin{proof}
Write $E(x)$ for the product in the proof of
\Cref{sc:prime-products}; its constant is the $C_E$ fixed above.
The terms with $j=1$ in \eqref{con:prime-power-interval} are at
most
\begin{equation*}
 \sum_{\substack{\sqrt Z<p\le Z\\p\ \mathrm{prime}}}\frac1p
 \le\log\frac{E(Z)}{E(\sqrt Z)}\le\log(2C_E).
\end{equation*}
For the terms with $j\ge2$, the sum over all primes and powers is
bounded by
\begin{equation*}
 \sum_{\substack{p\ \mathrm{prime}}}\sum_{j\ge2}p^{-j}
 =\sum_{\substack{p\ \mathrm{prime}}}\frac1{p(p-1)}
 \le\sum_{n\ge2}\frac1{n(n-1)}=1.
\end{equation*}
This proves \eqref{con:prime-power-interval}. Finally,
$$D/\varphi(D)=\prod_{\substack{p\mid D\\p\ \mathrm{prime}}}(1-p^{-1})^{-1}\le E(B)$$ when
$D>1$, and the asserted bound also holds for $D=1$.
\end{proof}
For the next two lemmas, let $h:\N\to\mathbb S^1$ be completely
multiplicative, let
$Z\ge4$ be an integer, and suppose $2\le y\le Z$ and $h(p)=1$
for every prime $p\le y$. Put
\begin{equation*}
 \Delta=\sum_{\substack{y<p\le Z\\p\ \mathrm{prime}}}\frac{1-\Re h(p)}p.
\end{equation*}

\begin{lemma}
\label{con:common-scalar}
There exist an absolute constant $C\ge1$ and a number $\zeta\in\mathbb S^1$ such that, for
every integer $M$ with $\sqrt Z\le M\le Z$,
\begin{equation}
 \frac1M\sum_{n\le M}|h(n)-\zeta|^2\le C\Delta.
 \label{con:prefix-concentration}
\end{equation}
\end{lemma}

\begin{proof}
We divide the proof into three steps.

\medskip
\noindent\textbf{Step 1. Represent the multiplicative function by an additive phase.}
For each prime $p$, choose
$\theta_p=\operatorname{Arg}h(p)\in(-\pi,\pi]$ and define
\begin{equation*}
 H(n)=\sum_{\substack{p\ \mathrm{prime}}} v_p(n)\theta_p.
 \end{equation*}
Since $v_p(ab)=v_p(a)+v_p(b)$, the function $H$ is real and
completely additive, with $H(p^j)=j\theta_p$.
Unique prime factorization and complete multiplicativity give
$$h(n)=\prod_{p\ \mathrm{prime}} e^{iv_p(n)\theta_p}=e^{iH(n)}.$$
Moreover, $\theta_p=0$ for $p\le y$.

\medskip
\noindent\textbf{Step 2. Bound the phase variance.}
We first bound $B_H(Z)$ in terms of $\Delta$.
Concavity of sine gives $\sin u\ge2u/\pi$ on $[0,\pi/2]$.
Hence, for $|\theta|\le\pi$,
\begin{equation*}
 1-\cos\theta=2\sin^2(|\theta|/2)
 \ge\frac{2\theta^2}{\pi^2},
 \qquad
 \theta^2\le\frac{\pi^2}{2}(1-\cos\theta).
 \end{equation*}
Differentiating the geometric series for $0<x<1$ gives
\begin{equation*}
 \sum_{j\ge1}j^2x^j
 =x\frac{d}{dx}\left(\frac{x}{(1-x)^2}\right)
 =\frac{x(1+x)}{(1-x)^3}.
 \end{equation*}
Taking $x=1/p\le1/2$, we obtain
\begin{equation*}
 \sum_{j\ge1}\frac{j^2}{p^j}
 =\frac1p\frac{1+1/p}{(1-1/p)^3}
 \le\frac1p\frac{3/2}{(1/2)^3}
 =\frac{12}{p}.
 \end{equation*}
Since the summands are nonnegative and $\cos\theta_p=\Re h(p)$,
\begin{equation}
 \begin{aligned}
 B_H(Z)
 =\sum_{\substack{y<p\le Z\\p\ \mathrm{prime}}}\theta_p^2
       \sum_{\substack{j\ge1\\p^j\le Z}}\frac{j^2}{p^j}
 \le12\sum_{\substack{y<p\le Z\\p\ \mathrm{prime}}}\frac{\theta_p^2}{p}
 \le6\pi^2\sum_{\substack{y<p\le Z\\p\ \mathrm{prime}}}\frac{1-\Re h(p)}p
 =6\pi^2\Delta.
 \end{aligned}
 \label{con:phase-variance-bound}
\end{equation}

\medskip
\noindent\textbf{Step 3. Transfer the variance bound to the multiplicative function.}
Fix an integer $M$ with $\sqrt Z\le M\le Z$.
We compare $A_H(M)$ with the fixed center $A_H(Z)$.
The upper bound $M\le Z$ and Cauchy--Schwarz give
\begin{equation*}
 \begin{aligned}
 |A_H(Z)-A_H(M)|^2
 &=\left|\sum_{\substack{p\ \mathrm{prime},\,j\in\N\\M<p^j\le Z}}\frac{H(p^j)}{p^j}\right|^2\\
 &\le
 \left(\sum_{\substack{p\ \mathrm{prime},\,j\in\N\\M<p^j\le Z}}\frac{|H(p^j)|^2}{p^j}\right)
 \left(\sum_{\substack{p\ \mathrm{prime},\,j\in\N\\M<p^j\le Z}}\frac1{p^j}\right)\\
 &\le L B_H(Z),\qquad L:=\log(2C_E)+1.
 \end{aligned}
 \end{equation*}
The last line uses \eqref{con:prime-power-interval}, whose hypotheses
$Z\ge4$ and $M\ge\sqrt Z$ hold here.
Since $H$ is additive and $M\ge\sqrt Z\ge2$, the
Tur\'an--Kubilius inequality \eqref{con:TK} applies at $M$.
Also $B_H(M)\le B_H(Z)$ because $M\le Z$.
Writing
$H(n)-A_H(Z)=(H(n)-A_H(M))+(A_H(M)-A_H(Z))$
and using $|a+b|^2\le2|a|^2+2|b|^2$, we obtain
\begin{equation*}
 \begin{aligned}
 \E_{n\in[M]}|H(n)-A_H(Z)|^2
 &\le2\E_{n\in[M]}|H(n)-A_H(M)|^2
       +2|A_H(M)-A_H(Z)|^2\\
 &\le2C_{\rm TK}B_H(M)+2LB_H(Z)\\
 &\le2(C_{\rm TK}+L)B_H(Z)
 \le C\Delta,
 \end{aligned}
\end{equation*}
where $C=12\pi^2(C_{\rm TK}+L)$ is absolute, by
\eqref{con:phase-variance-bound}.

Set $\zeta=e^{iA_H(Z)}\in\mathbb S^1$.
For real $a,b$, the inequality
$|e^{ia}-e^{ib}|=2|\sin((a-b)/2)|\le|a-b|$ gives
\begin{equation*}
 \E_{n\in[M]}|h(n)-\zeta|^2
 \le\E_{n\in[M]}|H(n)-A_H(Z)|^2
 \le C\Delta.
 \end{equation*}
This is \eqref{con:prefix-concentration}.
The scalar $\zeta$ depends only on $h$ and $Z$, so it works for
every integer $M$ in the stated range.
\end{proof}
Let $D\ge1$ be an integer. A weight $v:\Z/D\Z\to[0,\infty)$
has mean one and is unit-invariant when
\begin{equation}
 D^{-1}\sum_{a\bmod D}v(a)=1,
 \qquad v(ua)=v(a)\quad(u\in(\Z/D\Z)^\times).
 \label{con:one-dimensional-invariance}
\end{equation}

\begin{lemma}
\label{con:invariant-weight}
There is an absolute constant $C\ge1$ such that, whenever
\eqref{con:one-dimensional-invariance} holds,
$Z\ge4D^2$, and $y\ge D$, the scalar $\zeta$ in
\Cref{con:common-scalar} satisfies
\begin{equation}
 \frac1Z\sum_{a\le Z}v(a)|h(a)-\zeta|^2
 \le C\frac D{\varphi(D)}\Delta.
 \label{con:weighted-prefix}
\end{equation}
\end{lemma}

\begin{proof}
We divide the proof into four steps.

\medskip
\noindent\textbf{Step 1. Decompose the unit-invariant weight into gcd-orbits.}
Regard $v$ as a $D$-periodic function on the integers.
For each divisor $d\mid D$, put
\begin{equation*}
 Q_d=\frac Dd,\qquad
 \mathcal O_d=\{a\bmod D:\gcd(a,D)=d\}.
 \end{equation*}
We first show that the sets $\mathcal O_d$ are precisely the
orbits of multiplication by units modulo $D$.
Multiplication by a unit preserves the gcd with $D$.
Conversely, suppose $d<D$ and write two residues in $\mathcal O_d$
as $da_1$ and $da_2$, with $a_1,a_2$ units modulo $Q_d$.
Set $u_0\equiv a_2a_1^{-1}\pmod{Q_d}$.
For each $p^e\Vert D$, put $f=v_p(Q_d)$.
If $f\ge1$, choose a lift of $u_0\bmod p^f$ to $p^e$;
it remains nonzero modulo $p$.
If $f=0$, choose the residue $1\bmod p^e$.
The Chinese remainder theorem, applied to these pairwise coprime
prime powers, gives a unit $u\bmod D$ with $u\equiv u_0\pmod{Q_d}$.
Thus $u(da_1)\equiv da_2\pmod D$.
For $d=D$, the set $\mathcal O_D$ is the singleton zero residue.

The map $t\bmod Q_d\mapsto dt\bmod D$ is a bijection from the
units modulo $Q_d$ to $\mathcal O_d$, since
\begin{equation}
 \gcd(dt,D)=d\gcd(t,Q_d),\qquad
 dt_1\equiv dt_2\pmod D
 \ \Longleftrightarrow\ t_1\equiv t_2\pmod{Q_d}.
 \label{con:gcd-orbit-bijection}
\end{equation}
Hence $|\mathcal O_d|=\varphi(Q_d)$, including $d=D$ by
$\varphi(1)=1$.
Unit invariance makes $v$ constant on each $\mathcal O_d$.
Define its normalized mass by
\begin{equation}
 M_d=\frac1D\sum_{a\in\mathcal O_d}v(a)\ge0,\qquad \sum_{d\mid D}M_d=1,\qquad v(a)=\frac{DM_d}{\varphi(Q_d)}\quad(a\bmod D\in\mathcal O_d).
 \label{con:gcd-orbit-masses}
\end{equation}
The equality $\sum_{d\mid D}M_d=1$ follows from
\eqref{con:one-dimensional-invariance}, since the orbits partition
$\Z/D\Z$. The formula for $v(a)$ uses $|\mathcal O_d|=\varphi(Q_d)$.

\medskip
\noindent\textbf{Step 2. Reduce the weighted sum to long prefixes.}
For each $d\mid D$, set $Y_d=Z/d$ and $m_d=\lfloor Y_d\rfloor$.
Since $Z\ge4D^2$ and $d\le D$, we have
\begin{equation}
 m_d\ge Z/d-1\ge Z/D-1\ge2\sqrt Z-1\ge\sqrt Z,\qquad
 m_d\le Z/d\le Z.
 \label{con:long-prefix}
\end{equation}
Every prime factor of $d$ is at most $D\le y$, so $h(d)=1$.
Complete multiplicativity gives $h(dt)=h(t)$.
Put $F(t)=|h(t)-\zeta|^2\ge0$.
By \eqref{con:gcd-orbit-bijection}, the integers $a\le Z$
with $\gcd(a,D)=d$ are exactly $a=dt$ with
$t\le Y_d$ and $(t,Q_d)=1$.
Thus \eqref{con:gcd-orbit-masses} and $D/Z=Q_d/Y_d$ give
\begin{equation}
 \begin{aligned}
 \frac1Z\sum_{a\le Z}v(a)F(a)
 =\sum_{d\mid D}\frac{DM_d}{Z\varphi(Q_d)}
       \sum_{\substack{t\le Y_d\\(t,Q_d)=1}}F(t)
 =\sum_{d\mid D}M_d\frac{Q_d}{\varphi(Q_d)}\frac1{Y_d}
       \sum_{\substack{t\le Y_d\\(t,Q_d)=1}}F(t).
 \end{aligned}
 \label{con:weighted-gcd-decomposition}
\end{equation}

\medskip
\noindent\textbf{Step 3. Apply prefix concentration.}
Since $F(t)\ge0$, dropping the restriction $(t,Q_d)=1$
can only increase the sum.
The integers $t\le Y_d$ are exactly those $t\le m_d$, and
$m_d\le Y_d$. Therefore
\begin{equation}
 \frac1{Y_d}\sum_{\substack{t\le Y_d\\(t,Q_d)=1}}F(t)
 \le\frac1{Y_d}\sum_{t\le m_d}F(t)
 \le\E_{t\in[m_d]}F(t)
 \le C\Delta.
 \label{con:coprime-prefix-bound}
\end{equation}
The last inequality follows from \Cref{con:common-scalar}, since
\eqref{con:long-prefix} places every $m_d$ in $[\sqrt Z,Z]$.
That lemma supplies the same $\zeta$ for all these interval lengths.

\medskip
\noindent\textbf{Step 4. Use the totient comparison and conclude.}
Finally, $Q_d\mid D$, so the totient product formula gives
\begin{equation}
 \frac{Q_d}{\varphi(Q_d)}
 =\prod_{\substack{p\mid Q_d\\p\ \mathrm{prime}}}(1-p^{-1})^{-1}
 \le\prod_{\substack{p\mid D\\p\ \mathrm{prime}}}(1-p^{-1})^{-1}
 =\frac D{\varphi(D)},
 \label{con:totient-divisor-comparison}
\end{equation}
with an empty product equal to $1$.
Substituting \eqref{con:coprime-prefix-bound} and
\eqref{con:totient-divisor-comparison} into
\eqref{con:weighted-gcd-decomposition}, and using
$\displaystyle\sum_{d\mid D}M_d=1$, yields
\begin{equation*}
 \frac1Z\sum_{a\le Z}v(a)|h(a)-\zeta|^2
 \le C\Delta\sum_{d\mid D}M_d\frac{Q_d}{\varphi(Q_d)}
 \le C\frac D{\varphi(D)}\Delta.
 \end{equation*}
This proves \eqref{con:weighted-prefix} with the constant from
\Cref{con:common-scalar}.
\end{proof}
For the two-variable comparison, let $D\ge1$ be an integer and
let $V:(\Z/D\Z)^2\to[0,\infty)$ have mean one:
$$D^{-2}\sum_{a,b\bmod D}V(a,b)=1.$$
Assume $V(ua,ub)=V(a,b)$ for every unit $u$ modulo $D$.
Extend $V$ periodically to $\Z^2$.
Using the fixed $w,C_0$ of \eqref{ar:real-profile} and
\eqref{ar:profile-bound}, for each integer $X\ge1$ define
$I_{X,V}$ and $Z_{X,V}=I_{X,V}(1)$ by \eqref{ar:kernel}
with $V$ in place of $W_0$. Set $K_{X,V}=I_{X,V}/Z_{X,V}$
whenever $Z_{X,V}>0$.
The quadrature estimate \eqref{ar:normalization} gives
\begin{equation*}
 |Z_{X,V}-1|\le A_w\left(\frac DX+\frac{D^2}{X^2}\right),
 \end{equation*}
where $A_w>0$ depends only on $w$, uniformly in $D$ and $V$.
Put $B_w=\max\{1,4A_w\}$. If $X\ge B_wD$, then
$D/X\le1/B_w\le1$, so
\begin{equation}
 |Z_{X,V}-1|\le\frac{2A_w}{B_w}\le\frac12,
 \qquad \frac12\le Z_{X,V}\le\frac32.
 \label{con:kernel-positive-normalization}
\end{equation}
Thus $K_{X,V}$ is defined for every integer $X\ge B_wD$.
This sufficient lower bound for $X$ depends on $V$ only through
its chosen modulus $D$.
The proposition below assumes the stronger condition $X\ge C_wD^2$,
with $C_w\ge B_w$.

\begin{proposition}
\label{con:homogeneous-comparison}
There is a constant $C_w$, depending only on $w$, such that whenever
$X$ is an integer and
\begin{equation}
 X\ge C_wD^2,\qquad Z=C_0X,\qquad \max(2,D)\le y\le Z,
 \label{con:comparison-sizes}
\end{equation}
any two unit-modulus completely multiplicative functions $f,g$
agreeing at every prime at most $y$ satisfy
\begin{equation}
 |K_{X,V}(f)-K_{X,V}(g)|
 \le C_w\sqrt{\frac D{\varphi(D)}}
       \left(\sum_{\substack{y<p\le Z\\p\ \mathrm{prime}}}\frac{1-\Re(f(p)\overline{g(p)})}p\right)^{1/2}.
 \label{con:general-comparison}
\end{equation}
\end{proposition}

\begin{proof}
We divide the proof into four steps.

\medskip
\noindent\textbf{Step 1. Fix the constants and define the probability measure.}
Let $C_*$ be the absolute constant in \Cref{con:invariant-weight}.
Set
\begin{equation}
 A(w)=2C_*\|w\|_\infty C_0(C_0+1),
 \qquad C_w\ge\max\{B_w,4/C_0,4\sqrt{A(w)}\}.
 \label{con:comparison-constants}
\end{equation}
These constants depend only on the fixed $w$ and its chosen $C_0$.
Since $D\ge1$, \eqref{con:comparison-sizes} gives
$X\ge B_wD\ge D$ and $Z=C_0X\ge4D^2$.
Also $Z_{X,V}\ge1/2$ by \eqref{con:kernel-positive-normalization}.
Set $h=f\overline g$ and let $\Delta$ denote the prime sum in
\eqref{con:general-comparison}.
Then $h$ is completely multiplicative, $|h|=1$, and $h(p)=1$
for every prime $p\le y$.
Choose the scalar $\zeta\in\mathbb S^1$ from \Cref{con:common-scalar}.

\medskip
\noindent\textbf{Step 2. Obtain marginal concentration for the four linear forms.}
For integer pairs $m>n>0$, define
\begin{equation*}
 \nu_{X,V}(m,n)=\frac{w(m/X,n/X)V(m,n)}{X^2Z_{X,V}},
 \qquad \sum_{m>n>0}\nu_{X,V}(m,n)=1.
 \end{equation*}
The sum equals one by the definition of $Z_{X,V}$.
Thus $\nu_{X,V}$ is a probability measure, depending only on $w,V,X$,
and $K_{X,V}(\psi)=\E_{\nu_{X,V}}\psi(R(m,n))$ for $\psi=f,g$.

Put $\mathcal L=\{m-n,m+n,m,n\}$.
For each $L\in\mathcal L$, use the coordinate map $T_L$ and its
inverse in the following table:
\begin{equation*}
 \begin{array}{c|c|c}
 L&T_L(m,n)=(a,b)&T_L^{-1}(a,b)=(m,n)\\ \hline
 m-n&(m-n,n)&(a+b,b)\\
 m+n&(m+n,n)&(a-b,b)\\
 m&(m,n)&(a,b)\\
 n&(n,m)&(b,a).
 \end{array}
 \end{equation*}
Each map is linear with determinant $\pm1$, so it is a bijection
of $\Z^2$ and of $(\Z/D\Z)^2$.
Define
\begin{equation*}
 V_L(a,b)=V(T_L^{-1}(a,b)),
 \qquad v_L(a)=\frac1D\sum_{b\bmod D}V_L(a,b).
 \end{equation*}
The weight $v_L$ is nonnegative and $D$-periodic.
Linearity and the unit invariance of $V$ give
$V_L(ua,ub)=V_L(a,b)$ for every unit $u\bmod D$.
The bijection modulo $D$ and the substitution $b=uc$ therefore give
\begin{equation*}
 \begin{aligned}
 \frac1D\sum_{a\bmod D}v_L(a)
 &=\frac1{D^2}\sum_{a,b\bmod D}V_L(a,b)
 =\frac1{D^2}\sum_{m,n\bmod D}V(m,n)=1,\\
 v_L(ua)
 &=\frac1D\sum_{c\bmod D}V_L(ua,uc)
 =\frac1D\sum_{c\bmod D}V_L(a,c)=v_L(a).
 \end{aligned}
 \end{equation*}
Since $Z\ge4D^2$ and $y\ge D$, \Cref{con:invariant-weight}
applies to every $v_L$ and yields
\begin{equation}
 \frac1Z\sum_{a\in[Z]}v_L(a)|h(a)-\zeta|^2
 \le C_*\frac D{\varphi(D)}\Delta
 \qquad(L\in\mathcal L).
 \label{con:comparison-marginal-concentration}
\end{equation}
The same $\zeta$ works for all four forms because it depends only
on $h$ and $Z$.

\medskip
\noindent\textbf{Step 3. Convert marginal bounds into second moments.}
On the support of $w(m/X,n/X)$, we have $m>n>0$ and
$m+n<Z$ by \eqref{ar:profile-bound}.
Hence both $a=L(m,n)$ and its second coordinate $b$ lie in $[Z]$.
For fixed $a$, each $b$ determines a unique integer pair, and each
residue modulo $D$ occurs at most $Z/D+1$ times in $[Z]$.
Thus, for $a\in[Z]$, nonnegativity and periodicity give
\begin{equation*}
 \begin{aligned}
 S_L(a)&:=\sum_{\substack{m,n\in\Z\\m>n>0,\ L(m,n)=a}}
                 w(m/X,n/X)V(m,n)\\
 &\le\|w\|_\infty\sum_{b\in[Z]}V_L(a,b)\\
 &\le\|w\|_\infty(Z/D+1)\sum_{b\bmod D}V_L(a,b)
 =\|w\|_\infty(Z+D)v_L(a).
 \end{aligned}
 \end{equation*}
Grouping by $a=L(m,n)$ and applying
\eqref{con:comparison-marginal-concentration}, we obtain
\begin{equation}
 \begin{aligned}
 \E_{\nu_{X,V}}|h(L)-\zeta|^2
 &=\frac1{X^2Z_{X,V}}\sum_{a\in[Z]}S_L(a)|h(a)-\zeta|^2\\
 &\le\frac{\|w\|_\infty(Z+D)}{X^2Z_{X,V}}
                \sum_{a\in[Z]}v_L(a)|h(a)-\zeta|^2\\
 &\le C_*\frac{\|w\|_\infty(Z+D)Z}{X^2Z_{X,V}}
                \frac D{\varphi(D)}\Delta
 \le A(w)\frac D{\varphi(D)}\Delta.
 \end{aligned}
 \label{con:comparison-second-moments}
\end{equation}
Here $h(L)$ denotes $h(L(m,n))$.
The last inequality uses $Z=C_0X$, $D\le X$, and $Z_{X,V}\ge1/2$:
\begin{equation*}
 \frac{(Z+D)Z}{X^2Z_{X,V}}
 \le\frac{2(C_0X+X)C_0X}{X^2}=2C_0(C_0+1).
 \end{equation*}

\medskip
\noindent\textbf{Step 4. Compare the rational phases and conclude.}
Since $y\ge2$, we have $h(2)=1$.
Complete multiplicativity and the rational extension
\eqref{eq:rational-extension} give
\begin{equation}
 h(R(m,n))=\frac{h(m-n)h(m+n)}{h(2)h(m)h(n)}
 =h(m-n)h(m+n)\overline{h(m)}\,\overline{h(n)}.
 \label{con:comparison-rational-identity}
\end{equation}
For unit complex numbers $\alpha_j,\beta_j$, we have
\begin{equation*}
 \prod_{j=1}^4\alpha_j-\prod_{j=1}^4\beta_j
 =\sum_{j=1}^4
   \left(\prod_{i<j}\beta_i\right)(\alpha_j-\beta_j)
   \left(\prod_{i>j}\alpha_i\right)
 \end{equation*}
Taking absolute values bounds the left-hand side by
$\sum_{j=1}^4|\alpha_j-\beta_j|$.
Apply it to the four factors in \eqref{con:comparison-rational-identity}
and to $(\beta_1,\beta_2,\beta_3,\beta_4)
=(\zeta,\zeta,\overline\zeta,\overline\zeta)$.
The latter four numbers have product $\zeta^2\overline\zeta^2=1$.
Since $f(R)=g(R)h(R)$, $|g(R)|=1$, and conjugation preserves
absolute differences, this gives
\begin{equation*}
 |f(R)-g(R)|=|h(R)-1|
 \le\sum_{L\in\mathcal L}|h(L)-\zeta|.
 \end{equation*}
Taking expectations under the common measure $\nu_{X,V}$,
then using Cauchy--Schwarz and \eqref{con:comparison-second-moments},
we conclude that
\begin{equation*}
 \begin{aligned}
 |K_{X,V}(f)-K_{X,V}(g)|
 &\le\E_{\nu_{X,V}}|f(R)-g(R)|\\
 &\le\sum_{L\in\mathcal L}\E_{\nu_{X,V}}|h(L)-\zeta|\\
 &\le\sum_{L\in\mathcal L}
        \bigl(\E_{\nu_{X,V}}|h(L)-\zeta|^2\bigr)^{1/2}\\
 &\le4\sqrt{A(w)}\sqrt{\frac D{\varphi(D)}}\sqrt\Delta.
 \end{aligned}
 \end{equation*}
The choice \eqref{con:comparison-constants} proves
\eqref{con:general-comparison}.
\end{proof}
We now return to $D_0,W_0,K_X$ from \eqref{ar:parameters}
and \eqref{ar:kernel}. Let $\chi$ be a primitive character of
conductor $q$, let $t\in\R$, and let
$g:\N\to\mathbb S^1$ be completely multiplicative.
For $y\ge q$, define its finite model
$\widetilde g:\N\to\mathbb S^1$ by complete multiplicativity and
\begin{equation*}
 \widetilde g(p)=
 \begin{cases}g(p),&p\le y,\\ \chi(p)p^{it},&p>y.
 \end{cases}
\end{equation*}
Both functions extend uniquely to $\Q_{>0}$ by
\eqref{eq:rational-extension}. We use these extensions, denoted
by the same symbols, in $K_X(g)$ and $K_X(\widetilde g)$.

\Needspace{10\baselineskip}
\begin{corollary}
\label{con:comparison}
There is a constant $C_w\ge1$, depending only on $w$, with the
following property. For each $0<\delta<1/10$, let $D_0,W_0,K_X$
be defined by \eqref{ar:parameters}, \eqref{ar:congruence-weight},
and \eqref{ar:kernel}, respectively. Here $C_0\ge2$ is the fixed
integer chosen in \eqref{ar:profile-bound}.
One can choose $b_\delta>0$, depending only on $\delta$ and $w$,
so that for all $g,\chi,t$ as above, every integer $X$, and every
real $y$ satisfying
\begin{equation*}
 X\ge C_wD_0^2,\qquad Z=C_0X,\qquad \max(D_0,q,2)\le y\le Z,
\end{equation*}
we have
\begin{equation}
 \begin{gathered}
 |K_X(g)-K_X(\widetilde g)|
 \le b_\delta
 \left(\sum_{\substack{y<p\le Z\\p\ \mathrm{prime}}}
       \frac{1-\Re(g(p)\overline{\chi(p)}p^{-it})}{p}\right)^{1/2},\\
 b_\delta^2\le C_w\log(2/\delta).
 \end{gathered}
 \label{con:model-comparison-bound}
\end{equation}
\end{corollary}

\begin{proof}
We divide the proof into three steps.

\medskip
\noindent\textbf{Step 1. Choose the model constants.}
Let $P_w\ge1$ be a fixed constant valid in
\Cref{con:homogeneous-comparison}, and let $C_E$ be the constant
in \Cref{con:prime-estimates}. Choose
\begin{equation}
 C_w=\max\{1,P_w,6C_EP_w^2\},
 \qquad b_\delta=P_w\sqrt{\frac{D_0}{\varphi(D_0)}}.
 \label{con:model-comparison-constants}
\end{equation}
Since $D_0$ depends only on $\delta$, these choices have the
stated dependence.

\medskip
\noindent\textbf{Step 2. Verify the model and apply the comparison proposition.}
For $p>y\ge q$, we have $p\nmid q$, so $|\chi(p)p^{it}|=1$.
For $p\le y$, the prescribed value is $\widetilde g(p)=g(p)$.
Thus complete multiplicativity defines $\widetilde g$ with modulus
one on $\N$, and $g,\widetilde g$ agree at every prime at most $y$.
This includes primes dividing $q$, where the definition uses $g(p)$.

By \Cref{ar:weight-cost}, $W_0$ is nonnegative and periodic modulo
$D_0$ in each variable, with
\begin{equation*}
 \frac1{D_0^2}\sum_{a,b\bmod D_0}W_0(a,b)=1,\qquad
 W_0(ua,ub)=W_0(a,b)\quad\bigl(u\in(\mathbb Z/D_0\mathbb Z)^\times\bigr).
\end{equation*}
Since $C_w\ge P_w$, we have $X\ge C_wD_0^2\ge P_wD_0^2$.
Together with $\max(D_0,2)\le y\le Z$, this verifies the size
assumptions of \Cref{con:homogeneous-comparison}.
We apply that proposition with $D=D_0$ and $V=W_0$ to
$g,\widetilde g$.
For the prime sum, we have
\begin{equation*}
 \begin{gathered}
 g(p)\overline{\widetilde g(p)}
 =g(p)\overline{\chi(p)p^{it}}
 =g(p)\overline{\chi(p)}p^{-it}\qquad(p>y),\\
 \Delta:=\sum_{\substack{y<p\le Z\\p\ \mathrm{prime}}}
       \frac{1-\Re(g(p)\overline{\chi(p)}p^{-it})}{p}.
 \end{gathered}
\end{equation*}
The proposition therefore yields
\begin{equation}
 |K_X(g)-K_X(\widetilde g)|
 \le P_w\sqrt{\frac{D_0}{\varphi(D_0)}}\sqrt\Delta
 =b_\delta\sqrt\Delta.
 \label{con:model-proposition-application}
\end{equation}

\medskip
\noindent\textbf{Step 3. Estimate the totient factor.}
It remains to bound $b_\delta$.
By \eqref{ar:parameters},
$\displaystyle D_0=\prod_{\substack{p\le B\\p\ \mathrm{prime}}}p^{2T_0+5}$ with $B=40000\delta^{-2}$.
Euler's totient formula and \Cref{con:prime-estimates} give
\begin{equation}
 \frac{D_0}{\varphi(D_0)}
 =\prod_{\substack{p\le B\\p\ \mathrm{prime}}}\frac{p^{2T_0+5}}{p^{2T_0+4}(p-1)}
 =\prod_{\substack{p\le B\\p\ \mathrm{prime}}}(1-p^{-1})^{-1}
 \le C_E\log(2B).
 \label{con:model-totient-bound}
\end{equation}
Since $\delta<1/10$ and $20000<20^4$, we have
\begin{equation}
 \log(2B)=2\log(2/\delta)+\log20000
 \le2\log(2/\delta)+4\log20\le6\log(2/\delta).
 \label{con:model-logarithmic-bound}
\end{equation}
Combining \eqref{con:model-comparison-constants},
\eqref{con:model-totient-bound}, and \eqref{con:model-logarithmic-bound},
we obtain
\begin{equation*}
 b_\delta^2
 =P_w^2\frac{D_0}{\varphi(D_0)}
 \le6C_EP_w^2\log(2/\delta)
 \le C_w\log(2/\delta).
\end{equation*}
Together with \eqref{con:model-proposition-application}, this proves
\eqref{con:model-comparison-bound}.
\end{proof}

\section{Quadratic correlations}
\label{sec:scalar}

The main result is \Cref{thm:scalar_inverse}: a unit-modulus
completely multiplicative function with a large quadratic correlation
on an arithmetic progression has small prime distance from a
Dirichlet character times $n^{it}$. A prime-block decomposition
produces correlated prime pairs, whose difference phases yield
simultaneous rational approximations to the phase coefficients.
Freezing the residual phase gives a large mean on a residue class,
and a multiplicative mean-value theorem completes the argument.

\subsection{Prime-block decomposition}

The following decomposition is based on the argument of
Bourgain, Sarnak and Ziegler \cite[Section~2]{BSZ}. We give its
finite estimates, since their dependence on the correlation is used
later. For $0<\alpha<1/10$, define
\begin{equation}
 \begin{aligned}
 j_0&=\left\lceil\frac{3\log(1/\alpha)}{\log(1+\alpha)}\right\rceil,
 &j_1&=\lceil j_0/\alpha\rceil,\\
 x_j&=(1+\alpha)^j,
 &K&=\lceil x_{j_1+1}\rceil.
 \end{aligned}
 \label{sc:bsz-parameters}
\end{equation}
Let $P_j$ be the primes in $[x_j,x_{j+1})$ for
$j_0\le j\le j_1$.
We use the abbreviations
\begin{equation}
 s_j=|P_j|/x_j,\qquad S=\sum_{j=j_0}^{j_1}s_j,
 \qquad J=j_1-j_0+1.
 \label{sc:bsz-block-data}
\end{equation}
For a positive integer $N$, set
\begin{equation}
 Q_j=\left\{m\in\N:m\le N/x_{j+1},\ p\nmid m
             \text{ for every }p\in\bigcup_{i=j_0}^jP_i\right\}.
 \label{sc:bsz-Qj}
\end{equation}

\begin{lemma}\label{sc:decomposition}
Let $0<\alpha<1/10$ and $N\in\N$, and use
\eqref{sc:bsz-parameters}--\eqref{sc:bsz-Qj}.
The sets $P_jQ_j=\{pm:p\in P_j,m\in Q_j\}$ are pairwise disjoint
subsets of $[N]$. Every element of $P_jQ_j$ has a unique such
representation. There is an absolute constant $C\ge1$, independent
of $\alpha$, $N$, and the resulting parameters $j_0,j_1,x_j,K$, such that
\begin{align}
 \sum_j|P_j||Q_j|&\le N,\label{sc:bsz-disjoint-mass}\\
 \left|[N]\setminus\bigcup_jP_jQ_j\right|
 &\le C\alpha(1+\log(1/\alpha))N+2K,
 \label{sc:bsz-coverage}\\
 s_j&\le2\alpha,\qquad S\le C(1+\log(1/\alpha)), \notag
\\
 J&\le C\alpha^{-2}\log(1/\alpha),\qquad
 \log K\le C\alpha^{-1}\log(1/\alpha).
 \label{sc:bsz-cutoff-bound}
\end{align}
\end{lemma}

\begin{proof}
We use \Cref{sc:prime-products,tr:rough-number}.
All sums and unions indexed by $j$ below range over
$j_0\le j\le j_1$. Put
\begin{equation*}
 L=\log(1/\alpha),\qquad A=x_{j_0},\qquad B=x_{j_1+1},
 \qquad K=\lceil B\rceil.
 \end{equation*}
Then $L>1$ and $2<A<B$. We divide the rest proof into six steps.

\par
\medskip
\noindent\textbf{Step~1: Bounds for $s_j$ and $S$.}
The definition of $j_0$ gives
$\log A=j_0\log(1+\alpha)\ge3L$, hence $A\ge\alpha^{-3}$.
An interval of length $h$ contains at most $h+1$ integers.
Since $x_{j+1}-x_j=\alpha x_j$ and $x_j\ge A$, we obtain
\begin{equation}
 |P_j|\le\alpha x_j+1,\qquad
 s_j\le\alpha+x_j^{-1}\le\alpha+\alpha^3\le2\alpha.
 \label{sc:decomposition-block-size}
\end{equation}

Write $\displaystyle E(t)=\prod_{\substack{p\le t\\p\ \mathrm{prime}}}(1-p^{-1})^{-1}$ and
$\displaystyle E(t^-)=\prod_{\substack{p<t\\p\ \mathrm{prime}}}(1-p^{-1})^{-1}$.
By \Cref{sc:prime-products}, there is an absolute $C_E\ge1$
such that $\log t\le E(t)\le C_E\log t$ for $t\ge2$.
If $t$ is prime, then $E(t^-)=(1-t^{-1})E(t)$;
otherwise $E(t^-)=E(t)$. Therefore
\begin{equation*}
 \tfrac12\log t\le\tfrac12 E(t)\le E(t^-)\le C_E\log t,
 \qquad
 \prod_{\substack{A\le p<B\\p\ \mathrm{prime}}}(1-p^{-1})^{-1}
 =\frac{E(B^-)}{E(A^-)}\le2C_E\frac{\log B}{\log A}.
 \end{equation*}
For $0<u<1$, the identity
$\displaystyle -\log(1-u)=\int_0^u(1-v)^{-1}\,dv$ gives
$u\le-\log(1-u)$. Taking logarithms of the product above yields
\begin{equation*}
 \sum_{\substack{A\le p<B\\p\ \mathrm{prime}}}\frac1p
 \le\log\prod_{\substack{A\le p<B\\p\ \mathrm{prime}}}(1-p^{-1})^{-1}
 \le\log(2C_E)+\log\frac{\log B}{\log A}.
 \end{equation*}
Since $j_0/\alpha\le j_1<j_0/\alpha+1$,
\begin{equation}
 \frac{\log A}{\log B}=\frac{j_0}{j_1+1}\le\alpha,
 \qquad
 \frac{\log B}{\log A}=\frac{j_1+1}{j_0}
 \le\frac1\alpha+\frac2{j_0}\le\frac2\alpha.
 \label{sc:decomposition-log-ratios}
\end{equation}
The last inequality uses $j_0\ge1$ and $2\alpha<1$.

To estimate $S$, write $|P_j|/x_j$ as a sum of $1/x_j$
over $p\in P_j$. Each such prime satisfies
$p<x_{j+1}=(1+\alpha)x_j$, so $1/x_j<(1+\alpha)/p$.
Moreover, the intervals $[x_j,x_{j+1})$ partition $[A,B)$.
Thus the sets $P_j$ are disjoint and contain each prime in
$[A,B)$ exactly once. Summing the preceding comparison gives
\begin{align}
 S&=\sum_j\frac{|P_j|}{x_j}
   =\sum_j\sum_{p\in P_j}\frac1{x_j}\notag\\
 &\le(1+\alpha)\sum_j\sum_{p\in P_j}\frac1p
   =(1+\alpha)\sum_{\substack{A\le p<B\\p\ \mathrm{prime}}}\frac1p\notag\\
 &\le2\log(2C_E)+2\log\frac{\log B}{\log A}
   \le C_S(1+L),\qquad C_S:=2\log(4C_E)+2.
 \label{sc:bsz-block-total}
\end{align}
The non-strict inequality includes empty sums. The last line uses
$1+\alpha<2$ and \eqref{sc:decomposition-log-ratios}.

\par
\medskip
\noindent\textbf{Step~2: Disjoint product sets.}
Let $p\in P_j$, $m\in Q_j$, and $n=pm$.
No prime in $\bigcup_{i=j_0}^jP_i$ divides $m$.
Hence $v_p(n)=1$, no other prime in $P_j$ divides $n$,
and no prime in $P_i$ divides $n$ when $j_0\le i<j$.
Also
\begin{equation*}
 n=pm<x_{j+1}m\le x_{j+1}\frac{N}{x_{j+1}}=N.
 \end{equation*}
Thus $n\in[N]$. The integer $n$ determines $j$ as the least
index $i$ with $j_0\le i\le j_1$ for which a prime in $P_i$
divides $n$. It then determines $p$ as its unique prime divisor
in $P_j$, and $m=n/p$. Hence the representations are unique,
the product sets are pairwise disjoint, and
\begin{equation*}
 \left|\bigcup_jP_jQ_j\right|
 =\sum_j|P_jQ_j|=\sum_j|P_j||Q_j|\le N.
 \end{equation*}
This proves \eqref{sc:bsz-disjoint-mass}.

\par
\medskip
\noindent\textbf{Step~3: The exceptional sets.}
For $n\in[N]$, define
\begin{equation*}
 T_j(n)=\sum_{p\in P_j}v_p(n).
 \end{equation*}
The sum counts prime factors with multiplicity; for example,
$p^2\mid n$ with $p\in P_j$ implies $T_j(n)\ge2$.
Let $\mathcal E_0$ consist of the $n\in[N]$ with $T_j(n)=0$
for every $j$. For $n\in[N]\setminus\mathcal E_0$, let $j(n)$
be the least index with $T_{j(n)}(n)>0$, and let
$\mathcal E_1$ consist of those $n$ with $T_{j(n)}(n)\ge2$.

For $n\in[N]\setminus(\mathcal E_0\cup\mathcal E_1)$,
there is a unique prime $p\in P_{j(n)}$ dividing $n$,
and $v_p(n)=1$. Set $m=n/p$.
No prime in $\bigcup_{i=j_0}^{j(n)}P_i$ divides $m$, so
$m\in Q_{j(n)}$ if and only if $m\le N/x_{j(n)+1}$.
Step~2 shows that any representation of $n$ in a product set
must use this same index and prime.
Let $\mathcal E_2$ consist of those
$n\in[N]\setminus(\mathcal E_0\cup\mathcal E_1)$ for which
$m=n/p>N/x_{j(n)+1}$. It follows that
$\mathcal E_0,\mathcal E_1,\mathcal E_2$ are disjoint and their
union is $[N]\setminus\bigcup_jP_jQ_j$.
In particular, if $N\notin\mathcal E_0$ and $T_{j(N)}(N)=1$, then
$m=N/p>N/x_{j(N)+1}$, so $N\in\mathcal E_2$.

\par
\medskip
\noindent\textbf{Step~4: The first two exceptional sets.}
The integers in $\mathcal E_0$ have no prime divisor in $[A,B)$.
Apply \Cref{tr:rough-number}: its hypotheses $2\le A<B$ and
$N\ge1$ hold, and its integer parameter is $\lceil B\rceil=K$.
Writing $C_R$ for its absolute constant and using
\eqref{sc:decomposition-log-ratios}, we obtain
\begin{equation*}
 |\mathcal E_0|\le C_RN\frac{\log A}{\log B}+K
 \le C_R\alpha N+K.
 \end{equation*}

For $n\in\mathcal E_1$, the inequality $T_{j(n)}(n)\ge2$
gives primes $p,q\in P_{j(n)}$ with $pq\mid n$.
Here $p=q$ is allowed and means $p^2\mid n$.
For fixed $p,q$, there are $\lfloor N/(pq)\rfloor$ such multiples.
Summing over all $j$ and all ordered pairs $p,q\in P_j$ therefore gives
\begin{equation*}
 \begin{aligned}
 |\mathcal E_1|
 \le\sum_j\sum_{p,q\in P_j}\left\lfloor\frac{N}{pq}\right\rfloor
 \le N\sum_j\left(\sum_{p\in P_j}\frac1p\right)^2
 \le N\sum_js_j^2\le2\alpha N\sum_js_j=2\alpha NS.
 \end{aligned}
 \end{equation*}
Repeated counting only increases this upper bound.
The last two inequalities use $p\ge x_j$ for $p\in P_j$
and $s_j\le2\alpha$ from \eqref{sc:decomposition-block-size}.

\par
\medskip
\noindent\textbf{Step~5: The remaining exceptional set.}
For $n\in\mathcal E_2$, put $j=j(n)$ and write $n=pm$
as in Step~3. Then $N/x_{j+1}<m\le N/p\le N/x_j$.
For each fixed $p\in P_j$, the number of possible integers $m$
is at most the number of all integers in $(N/x_{j+1},N/x_j]$.
The latter is exactly the floor difference below:
\begin{equation*}
 \begin{aligned}
 \left\lfloor\frac{N}{x_j}\right\rfloor
 -\left\lfloor\frac{N}{x_{j+1}}\right\rfloor
 \le\frac{N}{x_j}-\frac{N}{x_{j+1}}+1
 =\frac{\alpha N}{(1+\alpha)x_j}+1
 \le\frac{\alpha N}{x_j}+1.
 \end{aligned}
 \end{equation*}
We have dropped the divisibility conditions on $m$ and enlarged
the upper endpoint from $N/p$ to $N/x_j$.
The sets $P_j$ are disjoint and all their elements are less than
$B\le K$, so
\begin{equation*}
 \sum_j|P_j|=\#\{p\text{ prime}:A\le p<B\}\le K.
 \end{equation*}
Summing over $p$ and $j$ gives
\begin{equation*}
 |\mathcal E_2|
 \le\sum_j|P_j|\left(\frac{\alpha N}{x_j}+1\right)
 \le\alpha NS+K.
 \end{equation*}
Thus, by Step~3 and \eqref{sc:bsz-block-total},
\begin{equation*}
 \begin{aligned}
 \left|[N]\setminus\bigcup_jP_jQ_j\right|
=|\mathcal E_0|+|\mathcal E_1|+|\mathcal E_2|
 \le C_R\alpha N+3\alpha NS+2K
 \le(C_R+3C_S)\alpha(1+L)N+2K.
 \end{aligned}
 \end{equation*}
This proves \eqref{sc:bsz-coverage}.

\par
\medskip
\noindent\textbf{Step~6: The parameter bounds.}
The identity $\displaystyle \log(1+\alpha)=\int_0^\alpha(1+t)^{-1}\,dt$
gives $\alpha/2\le\log(1+\alpha)\le\alpha$.
Using $L>1$, $\alpha<1$, and the ceiling bounds, we obtain
\begin{equation*}
 j_0\le\frac{6L}{\alpha}+1\le\frac{7L}{\alpha},
 \qquad
 J\le j_1+1\le\frac{j_0}{\alpha}+2
 \le\frac{9L}{\alpha^2}.
 \end{equation*}
Finally, $K=\lceil B\rceil\le2B$ since $B>1$. Hence
\begin{equation*}
 \log K\le\log2+\log B
 =\log2+(j_1+1)\log(1+\alpha)
 \le\log2+\frac{9L}{\alpha}\le\frac{10L}{\alpha}.
 \end{equation*}
Taking $C\ge\max\{C_R+3C_S,C_S,10\}$ proves all the estimates.
\end{proof}

\subsection{Correlated prime pairs}
For $0<\rho<1/2$, set $\alpha=(\rho/C_*)^4$, where $C_*\ge10$
is an absolute constant fixed in the next proof. The symbols
$P_j,Q_j,x_j,J,S,K$ now have the values specified in
\eqref{sc:bsz-parameters}--\eqref{sc:bsz-Qj} for this choice of $\alpha$.

For $N\ge K$, a one-bounded sequence $a:[N]\to\mathbb C$,
and distinct primes $p,q\le K$, define
\begin{equation*}
 n_{p,q}=\lfloor N/\max(p,q)\rfloor,\qquad
 R_{p,q}=\left|n_{p,q}^{-1}\sum_{m\le n_{p,q}}
                 a(pm)\overline{a(qm)}\right|.
\end{equation*}
\begin{proposition}\label{prop:bsz_dense}
There is an absolute constant $C\ge1$ such that the following holds.
Let $0<\rho<1/2$, let $N\ge1$ be an integer,
let $a:[N]\to\mathbb C$ be one-bounded, and let
$\nu:\mathbb N\to\mathbb C$ be multiplicative with
$|\nu(n)|\le1$ for every $n\in\mathbb N$. If
\begin{equation}
 \left|N^{-1}\sum_{n\le N}\nu(n)a(n)\right|\ge\rho,
 \qquad N\ge\exp(C\rho^{-5}),
 \label{sc:dense-bsz-assumptions}
\end{equation}
then $S>0$ and some block $P_j$, with $x=x_j$, satisfies
\begin{equation}
 |P_j|\ge\frac{\rho^2x}{16J}
 \label{sc:dense-bsz-block-size}
\end{equation}
and the set of ordered distinct pairs
\begin{equation*}
 \mathcal E=\{(p,q)\in P_j^2:p\ne q,
                   R_{p,q}\ge\rho^2/(32S)\}
\end{equation*}
satisfies
\begin{equation}
 |\mathcal E|\ge\frac{\rho^2}{32S}|P_j|^2
 \ge\frac{\rho^6}{8192J^2S}x^2.
 \label{sc:dense-bsz-edge-count}
\end{equation}
The absolute constant $C$ is chosen sufficiently large so that the
lower bound in \eqref{sc:dense-bsz-assumptions} implies $N\ge16K/\rho$; consequently,
every $n_{p,q}$ defined above is positive.
In particular, there is an absolute constant $c>0$ such that
\begin{equation}
 \begin{aligned}
 J&\le C\rho^{-8}\log(2/\rho),&
 S&\le C\log(2/\rho),\\
 |\mathcal E|&\ge c\rho^{25}x^2,&
 R_{p,q}&\ge c\rho^3\quad((p,q)\in\mathcal E).
 \end{aligned}
 \label{sc:dense-bsz-polynomial}
\end{equation}
The number $x$ also satisfies $x\le K\le\exp(C\rho^{-5})$.
\end{proposition}

\begin{proof}
All sums over $j$ below range from $j_0$ to $j_1$. We divide the rest proof into six steps.

\par
\medskip
\noindent\textbf{Step~1: Choice of constants.}
Let $C_0\ge1$ be an absolute constant valid in
\Cref{sc:decomposition}, and write $L=\log(1/\alpha)$.
We first fix $C_*\ge10$ so that $\alpha=(\rho/C_*)^4$ satisfies
\begin{equation*}
 C_0\alpha(1+L)\le\rho/8,
 \qquad \sqrt2\alpha\le\rho/4
 \qquad(0<\rho<1/2).
\end{equation*}
This choice is uniform in $\rho$. Indeed,
$\rho^3\le1/8$ and $\log(1/\rho)\le1/\rho$ give
\begin{equation*}
 \frac{C_0\alpha(1+L)}{\rho}
 \le\frac{C_0}{C_*^4}
       \left(\frac{1+4\log C_*}{8}+1\right),
 \qquad
 \frac{\sqrt2\alpha}{\rho}\le\frac{\sqrt2}{8C_*^4},
 \end{equation*}
and both right-hand sides tend to zero as $C_*\to\infty$.
With $C_*$ now fixed, \eqref{sc:bsz-cutoff-bound} yields
\begin{equation}
 \log\frac{16K}{\rho}
 \le4C_0C_*^4\rho^{-4}\log\frac{C_*}{\rho}
       +\log\frac{16}{\rho}
 \le C_1\rho^{-5}
 \label{sc:dense-bsz-size-threshold}
\end{equation}
for an absolute $C_1$, again using $\log(1/\rho)\le1/\rho$.
Choose the constant $C$ in \eqref{sc:dense-bsz-assumptions}
at least $C_1$. Then $N\ge16K/\rho$, so every $n_{p,q}$ is
positive. For the omitted set
$D=[N]\setminus\bigcup_jP_jQ_j$, \Cref{sc:decomposition} gives
\begin{equation}
 |D|\le C_0\alpha(1+L)N+2K
       \le\rho N/8+\rho N/8=\rho N/4.
 \label{sc:dense-bsz-coverage-error}
\end{equation}

\par
\medskip
\noindent\textbf{Step~2: Decomposition and Cauchy--Schwarz.}
Extend $a$ by zero outside $[N]$, and define
\begin{equation*}
 \begin{gathered}
 A_j=|P_j||Q_j|,\qquad
 B_j=\sum_{p\in P_j}\sum_{m\in Q_j}\nu(pm)a(pm),\\
 r_j=\begin{cases}
 |P_j|^{-2}\displaystyle\sum_{\substack{p,q\in P_j\\p\ne q}}R_{p,q},
       &P_j\ne\varnothing,\\
 0,    &P_j=\varnothing.
 \end{cases}
 \end{gathered}
 \end{equation*}
In particular $0\le r_j\le1$.
Uniqueness of the representations $pm$ and disjointness of the
sets $P_jQ_j$, together with $|\nu(n)a(n)|\le1$, imply
\begin{equation}
 \sum_jA_j\le N,
 \qquad
 \rho N\le\left|\sum_{n\le N}\nu(n)a(n)\right|
          \le |D|+\sum_j|B_j|.
 \label{sc:dense-bsz-decomposition-sum}
\end{equation}
For $p\in P_j$ and $m\in Q_j$, the definition of $Q_j$ gives
$p\nmid m$. Thus multiplicativity gives $\nu(pm)=\nu(p)\nu(m)$.
Put $M_j=\lfloor N/x_j\rfloor$; then $Q_j\subseteq[M_j]$.
Cauchy--Schwarz and $|\nu(m)|\le1$ give
\begin{equation}
 \begin{aligned}
 |B_j|^2
 \le |Q_j|\sum_{m\in Q_j}
       \left|\sum_{p\in P_j}\nu(p)a(pm)\right|^2
 \le |Q_j|\sum_{m\le M_j}
       \left|\sum_{p\in P_j}\nu(p)a(pm)\right|^2.
 \end{aligned}
 \label{sc:dense-bsz-local-cauchy}
\end{equation}
The last inequality only enlarges a sum of nonnegative terms.

\par
\medskip
\noindent\textbf{Step~3: Diagonal and off-diagonal terms.}
On expanding the last square, the diagonal sum is at most
$M_j|P_j|\le (N/x_j)|P_j|$.
For $p\ne q$ in $P_j$, we have $n_{p,q}\le M_j$ because
$\max(p,q)\ge x_j$. The zero extension makes the terms with
$m>n_{p,q}$ vanish. Since
$|\nu(p)\overline{\nu(q)}|\le1$, the off-diagonal sum is therefore
bounded in absolute value by
\begin{equation*}
 \sum_{\substack{p,q\in P_j\\p\ne q}}
 \left|\sum_{m\le M_j}a(pm)\overline{a(qm)}\right|
 =\sum_{\substack{p,q\in P_j\\p\ne q}}n_{p,q}R_{p,q}
 \le\frac{N}{x_j}|P_j|^2r_j.
 \end{equation*}
Using $s_j=|P_j|/x_j$ in
\eqref{sc:dense-bsz-local-cauchy}, and then
$\sqrt{u+v}\le\sqrt u+\sqrt v$ for $u,v\ge0$, we obtain
\begin{equation}
 |B_j|\le\sqrt{NA_j/x_j}+\sqrt{NA_js_jr_j}.
 \label{sc:dense-bsz-local-bound}
\end{equation}

\par
\medskip
\noindent\textbf{Step~4: Total correlation.}
The definition of $j_0$ gives $x_{j_0}\ge\alpha^{-3}$. Hence
\begin{equation*}
 \sum_jx_j^{-1}
 \le\sum_{j\ge j_0}(1+\alpha)^{-j}
 =\frac{1+\alpha}{\alpha x_{j_0}}
 \le2\alpha^2.
 \end{equation*}
Applying Cauchy--Schwarz and $\sum_jA_j\le N$ to the first
term in \eqref{sc:dense-bsz-local-bound} gives
\begin{equation*}
 \sum_j\sqrt{A_jN/x_j}
 \le\left(\sum_jA_j\right)^{1/2}
      \left(N\sum_jx_j^{-1}\right)^{1/2}
 \le\sqrt2\alpha N\le\rho N/4.
\end{equation*}
The second term satisfies
\begin{equation*}
 \sum_j\sqrt{NA_js_jr_j}
 \le\sqrt N\left(\sum_jA_j\right)^{1/2}
              \left(\sum_js_jr_j\right)^{1/2}
 \le N\left(\sum_js_jr_j\right)^{1/2}.
 \end{equation*}
Combining these bounds with
\eqref{sc:dense-bsz-coverage-error} and
\eqref{sc:dense-bsz-decomposition-sum}, we obtain
\begin{equation*}
 \rho N\le\rho N/4+\rho N/4
             +N\left(\sum_js_jr_j\right)^{1/2}.
 \end{equation*}
Consequently,
\begin{equation}
 \sum_js_jr_j\ge\rho^2/4.
 \label{sc:dense-bsz-energy}
\end{equation}
Since $r_j\le1$, this also gives $\displaystyle S=\sum_js_j\ge\rho^2/4>0$.

\par
\medskip
\noindent\textbf{Step~5: Selection of an interval and prime pairs.}
The indices with small $s_j$ or small $r_j$ satisfy
\begin{equation*}
 \sum_{s_j<\rho^2/(16J)}s_jr_j\le\frac{\rho^2}{16},
 \qquad
 \sum_{r_j<\rho^2/(16S)}s_jr_j
 \le\frac{\rho^2}{16S}\sum_js_j=\frac{\rho^2}{16}.
 \end{equation*}
Here the first estimate uses $r_j\le1$ and the fact that there
are $J$ indices. Removing both sets of indices from
\eqref{sc:dense-bsz-energy} leaves a sum at least $\rho^2/8>0$.
Thus some $j$ satisfies
\begin{equation*}
 s_j\ge\frac{\rho^2}{16J},\qquad
 r_j\ge\frac{\rho^2}{16S}.
 \end{equation*}
Set $x=x_j$ and $M=|P_j|$. The first inequality gives
\eqref{sc:dense-bsz-block-size}. Put $t=\rho^2/(32S)$.
Let $\mathcal E$ be the set of ordered distinct pairs in $P_j^2$
with $R_{p,q}\ge t$, and put $u=|\mathcal E|/M^2$.
Since $r_j\ge2t>0$, we have $M\ge2$.
Each summand defining $r_j$ is at most $1$, and it is less than
$t$ for ordered distinct pairs outside $\mathcal E$. Therefore
\begin{equation*}
 2t\le r_j
 \le\frac{|\mathcal E|+t\bigl(M(M-1)-|\mathcal E|\bigr)}{M^2}
 \le u+t.
 \end{equation*}
It follows that $u\ge t$, and hence
\begin{equation*}
 |\mathcal E|\ge\frac{\rho^2}{32S}M^2
 \ge\frac{\rho^2}{32S}
       \left(\frac{\rho^2x}{16J}\right)^2
 =\frac{\rho^6}{8192J^2S}x^2.
 \end{equation*}
This proves \eqref{sc:dense-bsz-edge-count}.

\par
\medskip
\noindent\textbf{Step~6: Bounds in terms of $\rho$.}
Write $\ell=\log(2/\rho)$. Since $C_*$ is fixed,
$L=4\log(C_*/\rho)\le C_2\ell$ for an absolute $C_2$.
The bounds for $J$ and $S$ in \Cref{sc:decomposition} therefore give
\begin{equation*}
 J\le C_0C_*^8\rho^{-8}L\le C_3\rho^{-8}\ell,
 \qquad S\le C_0(1+L)\le C_4\ell,
 \end{equation*}
where $C_3,C_4$ are absolute. Using $\ell\le2/\rho$ in
\eqref{sc:dense-bsz-edge-count} and the definition of
$\mathcal E$, we obtain
\begin{equation*}
 \begin{aligned}
 |\mathcal E|
 &\ge\frac{\rho^{22}}{8192C_3^2C_4\ell^3}x^2
  \ge\frac{\rho^{25}}{65536C_3^2C_4}x^2,\\
 R_{p,q}
 &\ge\frac{\rho^2}{32C_4\ell}
  \ge\frac{\rho^3}{64C_4}
  \qquad((p,q)\in\mathcal E).
 \end{aligned}
 \end{equation*}
Finally, $x=x_j\le K$ by definition, and
\eqref{sc:dense-bsz-size-threshold} gives $K\le\exp(C_1\rho^{-5})$.
Enlarging $C$ to at least $\max\{1,C_1,C_3,C_4\}$ and taking
$c=\min\{(65536C_3^2C_4)^{-1},(64C_4)^{-1}\}$ proves
\eqref{sc:dense-bsz-polynomial} and the asserted bound for $K$.
Both constants are absolute.
\end{proof}

For the progression case, let $D,N$ be positive integers, let
$b\in\mathbb Z$ and $\xi,\beta\in\mathbb R$, and let
$I\subseteq\mathbb R$ be an interval. Suppose
$\mathcal P=I\cap\{n\in\mathbb Z:n\equiv b\pmod D\}\subseteq[N]$
is nonempty, and set
$a(n)=\mathbf1_{\mathcal P}(n)e(\xi n+\beta n^2)$.
The difference phase is
\begin{equation*}
 Q_{p,q}(m)=(p-q)\xi m+(p^2-q^2)\beta m^2.
\end{equation*}

\begin{samepage}
\begin{lemma}\label{sc:cross-progression}
Let $0<\rho<1/2$, and let $\nu:\mathbb N\to\mathbb C$ be
a one-bounded multiplicative function. Suppose
\eqref{sc:dense-bsz-assumptions} holds for the sequence $a$ above.
Set $\alpha=(\rho/C_*)^4$ and
$\displaystyle S=\sum_{j=j_0}^{j_1}|P_j|/x_j$ as in
\eqref{sc:bsz-block-data}. Choose $x$ and $\mathcal E$ by
\Cref{prop:bsz_dense}. Then $D\le2/\rho$.
For each $(p,q)\in\mathcal E$, there exist $a_0\in\mathbb Z$
and positive integers $d,M$ such that
\begin{equation}
 \{m\in[n_{p,q}]:pm,qm\in\mathcal P\}
 =\{a_0+d\ell:1\le\ell\le M\},
 \label{sc:dense-bsz-cross-AP}
\end{equation}
with
\begin{equation*}
 d\le D^2,\qquad |a_0|\le N/x+D^2,\qquad
 M\ge\frac{\rho^2N}{128Sx}.
\end{equation*}
Moreover,
\begin{equation}
 \begin{gathered}
 \sum_{m\le n_{p,q}}a(pm)\overline{a(qm)}
 =\sum_{\ell=1}^{M}e\bigl(Q_{p,q}(a_0+d\ell)\bigr),\\
 \left|\frac1M\sum_{\ell=1}^{M}e\bigl(Q_{p,q}(a_0+d\ell)\bigr)\right|
 \ge\frac{\rho^2}{32S}.
 \end{gathered}
 \label{sc:dense-bsz-cross-average}
\end{equation}
\end{lemma}
\end{samepage}

\begin{proof}
We divide the proof into four steps.	
	
\par
\medskip
\noindent\textbf{Step~1: Bound for the modulus.}
The set $\mathcal P$ contains at most $N/D+1$ integers, so the
correlation hypothesis in \eqref{sc:dense-bsz-assumptions} gives
\begin{equation*}
 \rho\le\frac1N\left|\sum_{n\le N}\nu(n)a(n)\right|
 \le\frac{|\mathcal P|}{N}\le\frac1D+\frac1N.
 \end{equation*}
Here $|\nu(n)a(n)|\le1$ and $a(n)=0$ outside $\mathcal P$,
which justifies the middle inequality. The lower bound for $N$
implies $N\ge2/\rho$, and hence
$1/N\le\rho/2$. Therefore $1/D\ge\rho/2$, which proves
$D\le2/\rho$.

\par
\medskip
\noindent\textbf{Step~2: The arithmetic progression.}
Fix $(p,q)\in\mathcal E$. Put $g_p=\gcd(p,D)$ and
$g_q=\gcd(q,D)$, and define
\begin{equation*}
 d=\operatorname{lcm}\left(\frac D{g_p},\frac D{g_q}\right)
 \le D^2.
 \end{equation*}
By \Cref{prop:bsz_dense}, $S>0$, $n_{p,q}\ge1$, and
\begin{equation}
 \left|\sum_{m\le n_{p,q}}a(pm)\overline{a(qm)}\right|
 =n_{p,q}R_{p,q}
 \ge\frac{\rho^2}{32S}n_{p,q}>0.
 \label{sc:cross-progression-nonempty}
\end{equation}
Hence some $m\in[n_{p,q}]$ satisfies $pm,qm\in\mathcal P$.
Let $m_1$ be the least such integer.

Since $pm_1,qm_1\equiv b\pmod D$, the congruences
$pm\equiv b\pmod D$ and $qm\equiv b\pmod D$ are equivalent to
$D\mid p(m-m_1)$ and $D\mid q(m-m_1)$.
For each $r\in\{p,q\}$, put $g_r=\gcd(r,D)$.
Then $\gcd(r/g_r,D/g_r)=1$, so dividing by $g_r$ and applying
Euclid's lemma gives
\begin{equation*}
 D\mid r(m-m_1)
 \quad\Longleftrightarrow\quad
 \frac D{g_r}\mid\frac r{g_r}(m-m_1)
 \quad\Longleftrightarrow\quad
 \frac D{g_r}\mid m-m_1.
 \end{equation*}
Thus the two congruences hold if and only if both $D/g_p$ and
$D/g_q$ divide $m-m_1$. By the definition of their least common
multiple $d$, this is equivalent to $d\mid m-m_1$, that is,
$m\equiv m_1\pmod d$.

Because $p,q>0$, the restrictions $pm,qm\in I$ and
$1\le m\le n_{p,q}$ define an interval.
The integers in this interval congruent to $m_1$ modulo $d$
are therefore $m_1,m_1+d,\ldots,m_1+(M-1)d$ for some $M\ge1$.
Setting $a_0=m_1-d$ proves \eqref{sc:dense-bsz-cross-AP}.

\par
\medskip
\noindent\textbf{Step~3: Bounds for $a_0$ and $M$.}
Since $1\le m_1\le n_{p,q}$,
\begin{equation*}
 |a_0|=|m_1-d|\le n_{p,q}+d
 \le \frac Nx+D^2,
 \end{equation*}
where $n_{p,q}\le N/x$ because $p,q\ge x$.
Recall that $K=\lceil x_{j_1+1}\rceil$ in
\eqref{sc:bsz-parameters}. Since $x=x_j\le K$,
\Cref{prop:bsz_dense} gives $N\ge16K/\rho\ge4x$.
Also $p,q\in P_j\subset[x,2x)$, so
$N/\max(p,q)>N/(2x)\ge2$.
Applying $\lfloor y\rfloor\ge y/2$ for $y\ge2$, we obtain
\begin{equation*}
 n_{p,q}=\left\lfloor\frac N{\max(p,q)}\right\rfloor
 \ge\frac N{2\max(p,q)}\ge\frac N{4x}.
 \end{equation*}
The product $a(pm)\overline{a(qm)}$ has absolute value $1$
for the $M$ integers in \eqref{sc:dense-bsz-cross-AP} and is
zero for all other $m\le n_{p,q}$. The triangle inequality and
\eqref{sc:cross-progression-nonempty} therefore give
\begin{equation*}
 M\ge\left|\sum_{m\le n_{p,q}}a(pm)\overline{a(qm)}\right|
 =n_{p,q}R_{p,q}
 \ge\frac{\rho^2}{32S}n_{p,q}
 \ge\frac{\rho^2N}{128Sx}.
 \end{equation*}

\par
\medskip
\noindent\textbf{Step~4: Phase identity and average.}
For every $m$ in the progression, $pm,qm\in\mathcal P$, and therefore
\begin{equation*}
 a(pm)\overline{a(qm)}=e\left(\xi(pm-qm)+\beta\bigl((pm)^2-(qm)^2\bigr)\right)
 =e\bigl(Q_{p,q}(m)\bigr).
 \end{equation*}
All other terms with $m\le n_{p,q}$ vanish. Consequently,
\eqref{sc:dense-bsz-cross-average} gives the exact sum identity.
Finally, the progression is contained in $[n_{p,q}]$, so $M\le n_{p,q}$.
Using \eqref{sc:cross-progression-nonempty} once more,
\begin{equation*}
 \left|\frac1M\sum_{\ell=1}^{M}
 e\bigl(Q_{p,q}(a_0+d\ell)\bigr)\right|
 =\frac{n_{p,q}R_{p,q}}{M}
 \ge R_{p,q}
 \ge\frac{\rho^2}{32S}.
 \end{equation*}
This proves the final assertion in \eqref{sc:dense-bsz-cross-average}.
\end{proof}

\subsection{Simultaneous rational approximation}

For a positive integer $R$, let $\tau(R)$ be the number of its
positive divisors. In this subsection, the letters $p$ and $q$
denote arbitrary integer variables; no primality assumption is made.
We abbreviate $\|x\|_{\R/\Z}$ to $\|x\|$ within this subsection.

\begin{lemma}\label{sc:congruence-counts}
For every real $X\ge1$ and integer $R\ge1$,
\begin{equation}
 \#\{(p,q)\in([X,2X]\cap\mathbb Z)^2:p\ne q,
                  \ R\mid p^2-q^2\}
 \le\frac{8X^2\tau(R)}R\le\frac{16X^2}{\sqrt R}.
 \label{sc:dense-square-count}
\end{equation}
Similarly,
\begin{equation}
 \#\{(p,q)\in([X,2X]\cap\mathbb Z)^2:p\ne q,
                  \ R\mid p-q\}\le\frac{4X^2}{R}.
 \label{sc:dense-linear-count}
\end{equation}
\end{lemma}

\begin{proof}
We divide the proof into two steps.

\medskip
\noindent\textbf{Step 1. Count pairs with $R\mid p^2-q^2$.}
Put $I=[X,2X]\cap\mathbb Z$. We first count the pairs in
\eqref{sc:dense-square-count}. For $(p,q)\in I^2$ with $p\ne q$,
set
\[
 A=p-q,\qquad B=p+q.
\]
Then
\[
 0<|A|\le X,\qquad 2X\le B\le4X,
 \qquad p^2-q^2=AB.
\]
The identities
\[
 p=\frac{A+B}{2},\qquad q=\frac{B-A}{2}
\]
show that $(p,q)\mapsto(A,B)$ is injective.

Suppose that $R\mid AB$. Define
\[
 g=\gcd(|A|,R),\qquad A_0=\frac Ag,\qquad R_0=\frac Rg.
\]
Then $g\mid R$ and $\gcd(|A_0|,R_0)=1$. The divisibility
$R\mid AB$ implies $R_0\mid A_0B$. Euclid's lemma therefore gives
\[
 \frac Rg=R_0\mid B.
\]
Thus each admissible pair is assigned its unique exact-gcd value
$g=\gcd(|A|,R)\mid R$; for that value, $g\mid A$ and
$(R/g)\mid B$.

Fix $g\mid R$. If we discard the exact condition
$\gcd(|A|,R)=g$, the number of possible nonzero $A$ is at most
\[
 2\left\lfloor\frac Xg\right\rfloor\le\frac{2X}{g}.
\]
If we also discard the lower bound $B\ge2X$, then the number of
positive $B\le4X$ divisible by $R/g$ is at most
\[
 \left\lfloor\frac{4X}{R/g}\right\rfloor
 =\left\lfloor\frac{4Xg}{R}\right\rfloor
 \le\frac{4Xg}{R}.
\]
Dropping these restrictions, as well as the parity and interval
conditions, only enlarges the set being counted. Hence
\begin{align*}
 \#\{(p,q)\in I^2:p\ne q,\ R\mid p^2-q^2\}
 \le\sum_{g\mid R}
 2\left\lfloor\frac Xg\right\rfloor
 \left\lfloor\frac{4Xg}{R}\right\rfloor
 \le\sum_{g\mid R}\frac{2X}{g}\frac{4Xg}{R}
 =\frac{8X^2\tau(R)}{R}.
\end{align*}

Every positive divisor $d$ of $R$ pairs with $R/d$. At most one member
of each pair is smaller than $\sqrt R$; if $R$ is a square, the divisor
$\sqrt R$ is counted only once. Therefore
\[
 \tau(R)\le2\sqrt R,
\]
and so
\[
 \frac{8X^2\tau(R)}{R}\le\frac{16X^2}{\sqrt R}.
\]
This proves \eqref{sc:dense-square-count}.

\medskip
\noindent\textbf{Step 2. Count pairs with $R\mid p-q$.}
For \eqref{sc:dense-linear-count}, put $h=p-q$. The conditions
$p,q\in I$, $p\ne q$, and $R\mid h$ imply that $h$ is a nonzero
multiple of $R$ with $|h|\le X$. There are at most
\[
 2\left\lfloor\frac XR\right\rfloor
\]
such values of $h$. Once $h$ is fixed, choosing $q$ determines
$p=q+h$, so there are at most $|I|$ corresponding pairs. Since
\[
 |I|\le\lfloor X\rfloor+1\le2X\qquad(X\ge1),
\]
we obtain
\[
 \#\{(p,q)\in I^2:p\ne q,\ R\mid p-q\}
 \le2\left\lfloor\frac XR\right\rfloor |I|
 \le\frac{4X^2}{R}.
\]
This proves \eqref{sc:dense-linear-count}.
\end{proof}

For the next proposition, fix an integer $H\ge2$, a real $X\ge1$,
a positive integer $N$, and coefficients $\alpha,\beta\in\mathbb R$.
Consider a set of pairs $\mathcal E$\footnote{In applying \Cref{thm:scalar_inverse}, set
$X=x$, $\alpha=\xi$, and $a=a_0$; use the particular
correlation-selected set $\mathcal E$ supplied by
\Cref{prop:bsz_dense}, with $H$ chosen as in
\eqref{sc:scalar-H-choice}. \Cref{sc:cross-progression,sc:circle-lemma}
then verify the remaining hypotheses.}
satisfying
\begin{equation}
 \mathcal E\subseteq
 \{(p,q)\in([X,2X]\cap\mathbb Z)^2:p\ne q\},\qquad
 |\mathcal E|\ge X^2/H,
 \label{sc:dense-pair-hypothesis}
\end{equation}
with the following property. For each $(p,q)\in\mathcal E$, there are
integers $1\le r,d\le H$, an integer $a$, and a real $M>0$ such that
\begin{equation}
 |a|\le N/X+H,\qquad M\ge N/(HX),
 \label{sc:dense-progression-hypothesis}
\end{equation}
and
\begin{align}
 \|r d^2(p^2-q^2)\beta\|&\le H/M^2,
 \label{sc:dense-quadratic-hypothesis}\\
 \|r\{d(p-q)\alpha+2ad(p^2-q^2)\beta\}\|&\le H/M.
 \label{sc:dense-linear-hypothesis}
\end{align}
The choices of $r,d,a,M$ may depend on the pair $(p,q)$.

\begin{proposition}\label{sc:rational-approx}
Assume the pair and coefficient conditions
\eqref{sc:dense-pair-hypothesis}--\eqref{sc:dense-linear-hypothesis}.
If $X\ge8H$ and
\begin{equation*}
 N\ge10^9 H^{30}X^2,
\end{equation*}
then there are integers $u,v$ and positive integers $S,R$ such that
\begin{align*}
 R&\le256H^9,
 &\left|\beta-\frac vR\right|&\le\frac{4H^4}{N^2},
 \\
 S&\le1024H^{14},
 &\left|\alpha-\frac uS\right|&\le\frac{520H^7}{N}.
\end{align*}
Set $Q=RS\le2^{18}H^{23}$. Then both approximating fractions
can be written with denominator $Q$:
\[
 \frac uS=\frac{uR}{Q},\qquad \frac vR=\frac{vS}{Q}.
\]
\end{proposition}

\begin{proof}
For each $(p,q)\in\mathcal E$, fix parameters $r,d,a,M$ satisfying
\eqref{sc:dense-progression-hypothesis}--\eqref{sc:dense-linear-hypothesis}.
Put $I=[X,2X]\cap\mathbb Z$ and $\Delta_{p,q}=p^2-q^2$. We divide the rest proof into three steps.

\par
\medskip
\noindent\textbf{Step 1. Find a pair with large separation.}
Fix $q\in I$. The condition $|p-q|<X/(8H)$ restricts $p$ to an
interval of length $X/(4H)$. This interval contains at most
$X/(4H)+1$ integers. Since $|I|\le X+1$, the number of such pairs
in $I^2$ is at most
\[
 (X+1)\left(\frac X{4H}+1\right)
 \le \frac{X^2}{2H}+2X
 \le \frac{3X^2}{4H}.
\]
The last inequality uses $X\ge8H$. Since
$|\mathcal E|\ge X^2/H$, some $(p_0,q_0)\in\mathcal E$ satisfies
\begin{equation}
 |p_0-q_0|\ge\frac{X}{8H},
 \qquad
 |\Delta_{p_0,q_0}|
 =|p_0-q_0|(p_0+q_0)\ge\frac{X^2}{4H}.
 \label{sc:dense-far-pair}
\end{equation}
For every pair in $\mathcal E$ we also have
\[
 |p-q|\le X,
 \qquad
 |\Delta_{p,q}|\le4X^2.
\]

\par
\medskip
\noindent\textbf{Step 2. Approximate $\beta$.}
Denote the parameters chosen for $(p_0,q_0)$ by
$r_0,d_0,a_0,M_0$. By \eqref{sc:dense-quadratic-hypothesis} and
$M_0\ge N/(HX)$,
\[
 \|r_0d_0^2\Delta_{p_0,q_0}\beta\|
 \le \frac{H}{M_0^2}
 \le \frac{H^3X^2}{N^2}.
\]
Put
$\Lambda_0=r_0d_0^2\Delta_{p_0,q_0}$. Then
$|\Lambda_0|\ge X^2/(4H)$. Choosing an integer $m_0$ nearest to
$\Lambda_0\beta$ gives
\[
 \left|\beta-\frac{m_0}{\Lambda_0}\right|
 \le \frac{H^3X^2/N^2}{|\Lambda_0|}
 \le \frac{4H^4}{N^2}.
\]
Write $m_0/\Lambda_0=v/R$ in lowest terms, with $v\in\mathbb Z$,
$R\ge1$, and $\gcd(v,R)=1$. Since $r_0d_0^2\le H^3$ and
$|\Delta_{p_0,q_0}|\le4X^2$,
\begin{equation}
 R\le4H^3X^2,
 \qquad
 \beta=\frac vR+\beta_0,
 \qquad
 |\beta_0|\le\frac{4H^4}{N^2}.
 \label{sc:dense-initial-beta}
\end{equation}

For each $(p,q)\in\mathcal E$, combining
\eqref{sc:dense-quadratic-hypothesis} with
\eqref{sc:dense-initial-beta} gives
\[
 \left\|\frac{rd^2\Delta_{p,q}v}{R}\right\|
 \le \frac{H^3X^2}{N^2}
      +4H^3X^2\frac{4H^4}{N^2}
 \le \frac{17H^7X^2}{N^2}.
\]
Multiplication by the preliminary bound $R\le4H^3X^2$ yields
\[
 R\frac{17H^7X^2}{N^2}
 \le \frac{68H^{10}X^4}{N^2}
 \le \frac{68}{10^{18}H^{50}}<1,
\]
where we used $N\ge10^9H^{30}X^2$. The rational number inside the norm has
denominator dividing $R$. Any nonintegral rational with denominator
dividing $R$ has distance at least $1/R$ from the nearest integer.
Hence the strict inequality forces it to be an integer. Thus
$R\mid rd^2\Delta_{p,q}v$; since $(v,R)=1$, Euclid's lemma gives
\begin{equation*}
 R\mid rd^2\Delta_{p,q}
 \qquad ((p,q)\in\mathcal E).
\end{equation*}

Partition $\mathcal E$ according to the chosen values of $(r,d)$.
There are at most $H^2$ classes. Hence one class $\mathcal E_1$ satisfies
\[
 |\mathcal E_1|\ge\frac{|\mathcal E|}{H^2}\ge\frac{X^2}{H^3}.
\]
Let $(r_1,d_1)$ be its associated pair and put
\[
 R_1=\frac{R}{\gcd(R,r_1d_1^2)}.
\]
The divisibility just proved implies
$R_1\mid\Delta_{p,q}$ for every $(p,q)\in\mathcal E_1$. \Cref{sc:congruence-counts}
therefore gives
\[
 \frac{X^2}{H^3}
 \le |\mathcal E_1|
 \le \frac{16X^2}{\sqrt{R_1}},
\]
so $R_1\le256H^6$. Consequently,
\[
 R\le r_1d_1^2R_1\le256H^9.
\]
Together with \eqref{sc:dense-initial-beta}, this proves the two
assertions concerning $\beta$.

\par
\medskip
\noindent\textbf{Step 3. Approximate $\alpha$.}
By \eqref{sc:dense-linear-hypothesis} and $M\ge N/(HX)$,
for each $(p,q)\in\mathcal E$ we have
\[
 \left\|r\{d(p-q)\alpha+2ad\Delta_{p,q}\beta\}\right\|
 \le \frac{H^2X}{N}.
\]
Using $\|Rz\|\le R\|z\|$ and substituting $\beta=v/R+\beta_0$
gives
\[
 \|Rrd(p-q)\alpha\|
 \le \frac{RH^2X}{N}
 +2Rrd|a|\,|\Delta_{p,q}|\,|\beta_0|.
\]
The size assumption on $N$ implies $N\ge HX$, so
$|a|\le N/X+H\le2N/X$. Using
$rd\le H^2$, $|\Delta_{p,q}|\le4X^2$, and
\eqref{sc:dense-initial-beta}, we obtain
\begin{equation*}
 \|Rrd(p-q)\alpha\|
 \le \frac{RH^2X}{N}+\frac{64RH^6X}{N}
 \le \frac{65RH^6X}{N}.
\end{equation*}

Apply this inequality to $(p_0,q_0)$ and put
$\Lambda_1=Rr_0d_0(p_0-q_0)$. By \eqref{sc:dense-far-pair},
\[
 \frac{RX}{8H}\le|\Lambda_1|\le RH^2X.
\]
Choose an integer $m_1$ nearest to $\Lambda_1\alpha$. Then
\[
 \left|\alpha-\frac{m_1}{\Lambda_1}\right|
 =\frac{\|\Lambda_1\alpha\|}{|\Lambda_1|}
 \le\frac{65RH^6X/N}{RX/(8H)}
 =\frac{520H^7}{N}.
\]
Write $m_1/\Lambda_1=u/S$ in lowest terms, with $u\in\mathbb Z$,
$S\ge1$, and $\gcd(u,S)=1$. Since $S\le|\Lambda_1|$,
\begin{equation}
 S\le RH^2X\le256H^{11}X,
 \qquad
 \alpha=\frac uS+\alpha_0,
 \qquad
 |\alpha_0|\le\frac{520H^7}{N}.
 \label{sc:dense-initial-alpha}
\end{equation}

For each $(p,q)\in\mathcal E$, substitute
$\alpha=u/S+\alpha_0$ into the bound for $\|Rrd(p-q)\alpha\|$.
Using \eqref{sc:dense-initial-alpha} and $rd|p-q|\le H^2X$ gives
\[
 \left\|\frac{Rrd(p-q)u}{S}\right\|
 \le \frac{65RH^6X}{N}+\frac{520RH^9X}{N}
 \le \frac{585RH^9X}{N}.
\]
Using $R\le256H^9$ and $S\le256H^{11}X$,
\[
 S\frac{585RH^9X}{N}
 \le \frac{585\cdot256^2H^{29}X^2}{N}<1,
\]
by $N\ge10^9H^{30}X^2$. The rational number inside the norm has
denominator dividing $S$. Any nonintegral rational with denominator
dividing $S$ has distance at least $1/S$ from the nearest integer.
Hence the strict inequality forces it to be an integer. Thus
$S\mid Rrd(p-q)u$; since $(u,S)=1$, Euclid's lemma gives
\begin{equation*}
 S\mid Rrd(p-q)
 \qquad ((p,q)\in\mathcal E).
\end{equation*}

Again partition $\mathcal E$ according to the chosen values of $(r,d)$.
Choose a class $\mathcal E_2$ satisfying
$$|\mathcal E_2|\ge\frac{|\mathcal E|}{H^2}\ge\frac{X^2}{H^3}.$$
Let $(r_2,d_2)$ be its associated pair and set
\[
 S_1=\frac{S}{\gcd(S,Rr_2d_2)}.
\]
The divisibility just proved gives
$S_1\mid p-q$ throughout $\mathcal E_2$. \Cref{sc:congruence-counts} now yields
\[
 \frac{X^2}{H^3}
 \le |\mathcal E_2|
 \le \frac{4X^2}{S_1},
\]
so $S_1\le4H^3$. Therefore
\[
 S\le Rr_2d_2S_1
 \le (256H^9)(H^2)(4H^3)
 =1024H^{14}.
\]
Together with \eqref{sc:dense-initial-alpha}, this proves the two
assertions concerning $\alpha$.

Finally, with $Q=RS$,
\[
 Q\le(256H^9)(1024H^{14})=2^{18}H^{23},
\]
and
\[
 \frac uS=\frac{uR}{Q},
 \qquad
 \frac vR=\frac{vS}{Q}.
\]
Thus both rational approximations have the common denominator $Q$.
\end{proof}

\begin{lemma}\label{sc:primitive-congruence}
Let $R\ge1$ be an integer and let $a,b\in\mathbb Z$ satisfy
$\gcd(a,b,R)=1$. For real $X\ge1$,
\begin{equation*}
 \#\{(p,q)\in([X,2X]\cap\mathbb Z)^2:R\mid ap+bq\}
 \le\frac{X^2}{R}+2X+1.
\end{equation*}
\end{lemma}

\begin{proof}
Put
\[
 I=[X,2X]\cap\mathbb Z,\qquad
 \mathcal S=\{(p,q)\in I^2:R\mid ap+bq\}.
\]
Set
\[
 g=\gcd(a,R).
\]
Then $g$ is a positive integer with $1\le g\le R$. Every common
divisor of $b$ and $g$ also divides $a$, $b$, and $R$. Hence
$\gcd(b,g)=1$.

Take $(p,q)\in\mathcal S$. Since $g\mid R$, the divisibility
$R\mid ap+bq$ implies $g\mid ap+bq$. Since $g\mid a$, we also have
$g\mid ap$, and therefore $g\mid bq$. As $\gcd(b,g)=1$, Euclid's
lemma gives $g\mid q$.

Thus only multiples of $g$ in $I$ can occur as the second coordinate.
The interval $[X,2X]$ has length $X$, so it contains at most
$X/g+1$ multiples of $g$.

Fix one such $q$ and write $q=gq'$. Dividing
$R\mid ap+bq$ by $g$ gives
\[
 \frac ag\,p+bq'\equiv0\pmod{R/g}.
\]
Since
\[
 \gcd\left(\frac ag,\frac Rg\right)=1,
\]
the congruence places $p$ in one residue class modulo $R/g$.

An interval of length $X$ contains at most $X/m+1$ integers in one
residue class modulo a positive integer $m$. With $m=R/g$, at most
$Xg/R+1$ values of $p$ are possible.
Consequently,
\begin{align*}
 |\mathcal S|
 \le\left(\frac Xg+1\right)
       \left(\frac{Xg}{R}+1\right)
 =\frac{X^2}{R}+\frac Xg+\frac{Xg}{R}+1
 \le\frac{X^2}{R}+2X+1,
\end{align*}
because $1\le g\le R$, so $X/g\le X$ and $Xg/R\le X$.

The argument includes $a=0$ and $R=1$. If $a=0$, then
$g=\gcd(0,R)=R$ and the reduced modulus $R/g$ is $1$. If $R=1$,
then $g=1$ and again the reduced modulus is $1$.
\end{proof}

\Needspace{13\baselineskip}
\begin{samepage}
\begin{lemma}\label{lem:linear_dense}
Let $H\ge2$ be an integer, let $X\ge32H$ be a real number, and
let $N$ be a positive integer satisfying
\begin{equation*}
 N\ge10^6H^3X^3.
\end{equation*}
Set $I=[X,2X]\cap\mathbb Z$.  Let $A,B\in\mathbb R$ and suppose that
$\mathcal E\subseteq I^2$ satisfies
\begin{equation*}
 |\mathcal E|\ge X^2/H,\qquad
 \|pA+qB\|_{\mathbb R/\mathbb Z}\le HX/N
       \quad((p,q)\in\mathcal E).
\end{equation*}
Then there exist integers $a,b$ and a positive integer $R$
satisfying $\gcd(a,b,R)=1$ and $R\le2H$ such that
\begin{equation}
 \left|A-\frac aR\right|\le\frac{64H^2}N,
 \qquad
 \left|B-\frac bR\right|\le\frac{64H^2}N.
 \label{sc:dense-linear-phase-output}
\end{equation}
\end{lemma}
\end{samepage}

\begin{proof}
We divide the proof into four steps.

\medskip
\noindent\textbf{Step 1. Select two pairs with a nonzero determinant.}
The lower bound on $|\mathcal E|$ makes $\mathcal E$ nonempty.
Choose $(p_1,q_1)\in\mathcal E$.  For an integer $p\in I$, the real numbers $q$ satisfying
\[
 |p_1q-q_1p|<\frac{X^2}{16H}
\]
form an interval of length
\[
 \frac{2X^2}{16Hp_1}=\frac{X^2}{8Hp_1}
 \le\frac{X}{8H},
\]
because $p_1\ge X$.  An interval of real length $L$ contains at most
$L+1$ integers, and $|I|\le X+1$.  Therefore the strip
\[
 \mathcal S=\left\{(p,q)\in I^2:
 |p_1q-q_1p|<\frac{X^2}{16H}\right\}
\]
satisfies
\begin{equation*}
 |\mathcal S|
 \le (X+1)\left(\frac X{8H}+1\right)
 \le\frac{X^2}{4H}+2X
 <\frac{X^2}{H}.
\end{equation*}
The first $\le$ follows from
$(X-1)(X/(8H)+1)\ge0$, and the final strict inequality from $X\ge32H$.
Since $|\mathcal E|\ge X^2/H$, choose
$(p_2,q_2)\in\mathcal E\setminus\mathcal S$.  Its determinant
\[
 \Delta=p_1q_2-p_2q_1
\]
then satisfies
\[
 \frac{X^2}{16H}\le|\Delta|\le4X^2.
\]
The upper bound follows from $p_i,q_i\in[X,2X]$ and
$|p_1q_2-p_2q_1|\le\max(p_1q_2,p_2q_1)\le4X^2$.

\medskip
\noindent\textbf{Step 2. Solve the two approximate linear equations.}
For each $i=1,2$, the hypothesis supplies an integer $k_i$ with
\[
 |p_iA+q_iB-k_i|\le\frac{HX}{N}.
\]
Put
\[
 \varepsilon_i=p_iA+q_iB-k_i.
\]
Let $(A_0,B_0)$ be the exact solution of
\[
 p_iA_0+q_iB_0=k_i\qquad(i=1,2).
\]
Since $\Delta\ne0$, Cramer's rule gives
\[
 A_0=\frac{q_2k_1-q_1k_2}{\Delta},
 \qquad
 B_0=\frac{-p_2k_1+p_1k_2}{\Delta}.
\]
It also gives
\[
 A-A_0=\frac{q_2\varepsilon_1-q_1\varepsilon_2}{\Delta},
 \qquad
 B-B_0=\frac{-p_2\varepsilon_1+p_1\varepsilon_2}{\Delta}.
\]
Since $p_i,q_i\le2X$ and $|\Delta|\ge X^2/(16H)$,
\begin{equation}
 |A-A_0|,\ |B-B_0|
 \le\frac{4X(HX/N)}{|\Delta|}
 \le\frac{64H^2}{N}.
 \label{sc:dense-linear-phase-initial-errors}
\end{equation}

\medskip
\noindent\textbf{Step 3. Put the exact solution in reduced rational form.}
To make the common denominator explicit, set
\[
 D=|\Delta|,\qquad
 \alpha=q_2k_1-q_1k_2,\qquad
 \beta=p_1k_2-p_2k_1,
 \qquad
 g=\gcd(|\alpha|,|\beta|,D).
\]
Let
\[
 R=\frac Dg,\qquad
 a=\operatorname{sgn}(\Delta)\frac{\alpha}{g},\qquad
 b=\operatorname{sgn}(\Delta)\frac{\beta}{g}.
\]
Then $R$ is a positive integer, $\gcd(a,b,R)=1$, and the Cramer formulas
give exactly
\[
 A_0=\frac aR,\qquad B_0=\frac bR.
\]
Moreover,
\[
 R\le D=|\Delta|\le4X^2.
\]

\medskip
\noindent\textbf{Step 4. Force divisibility and bound the denominator.}
For $(p,q)\in\mathcal E$, the triangle inequality on
$\mathbb R/\mathbb Z$ and \eqref{sc:dense-linear-phase-initial-errors}
give
\[
 \left\|\frac{pa+qb}{R}\right\|
 =\|pA_0+qB_0\|
 \le\frac{HX}{N}+(p+q)\frac{64H^2}{N}
 \le\frac{257H^2X}{N},
\]
because $p+q\le4X$ and $H\ge2$.  Put
$C=257H^2X/N$.  Using $R\le4X^2$ and the hypothesis on $N$,
\[
 RC\le\frac{1028H^2X^3}{N}
 \le\frac{1028}{10^6H}<1.
\]
Hence $C<1/R$.  For an integer $m$, the rational number $m/R$ has
distance either $0$ or at least $1/R$ from $\mathbb Z$: if $R\nmid m$,
then $m-Rj$ is a nonzero integer for every $j\in\mathbb Z$, so
$\|m/R\|=\min_j|m-Rj|/R\ge1/R$.  Applying this with $m=pa+qb$ forces
\[
 R\mid ap+bq\qquad((p,q)\in\mathcal E).
\]

\Cref{sc:primitive-congruence}, applied to the primitive triple $(a,b,R)$, now gives
\[
 \frac{X^2}{H}
 \le|\mathcal E|
 \le\frac{X^2}{R}+2X+1
 \le\frac{X^2}{R}+3X.
\]
Since $X\ge32H$,
$3X\le3X^2/(32H)$, and therefore
\[
 \frac1R\ge\frac1H-\frac3{32H}
 =\frac{29}{32H}.
\]
Thus
\[
 R\le\frac{32H}{29}<2H.
\]
Together with \eqref{sc:dense-linear-phase-initial-errors}, this proves
\eqref{sc:dense-linear-phase-output}.
\end{proof}

\subsection{Quadratic exponential sums}

We record a quantitative equidistribution theorem. For the related notions,
see \cite{LSS}.

\newcommand{\LSSTheoremCitation}{\cite[Theorem~5.2]{LSS}}
\begin{theorem}[Leng--Sah--Sawhney, \LSSTheoremCitation]\label{sc:LSS-5.2}
Fix integers $\ell,k,s\ge1$, $\delta\in(0,1/10)$, and positive integers
$M$ and $d$, together with $F:G/\Gamma\to\mathbb C$. Suppose that $G$ is a dimension-$d$, at most
$s$-step connected, simply connected nilpotent Lie group with a given
degree-$k$ filtration, and the nilmanifold $G/\Gamma$ is complexity at
most $M$ with respect to this filtration. Let $g$ be a polynomial
sequence on $G$ with respect to this filtration. Furthermore suppose
that $\|F\|_{\rm Lip}\le1$ and $F$ has $G^{(s)}$-vertical frequency $\xi$\footnote{Here $G^{(s)}$ denotes the $s$-th term of the lower central series: $G^{(1)}=G$ and $G^{(j+1)}=[G,G^{(j)}]$.}
such that the height of $\xi$ is bounded by $M/\delta$. Suppose that
\[
 N\ge (M/\delta)^{O_{k,\ell}(d^{O_{k,\ell}(1)})},
 \qquad
 \left|\mathbb E_{\widetilde n\in[N]^\ell}
 F(g(\widetilde n)\Gamma)\right|\ge\delta.
\]
There exists an integer $0\le r\le\dim(G/[G,G])$ such that:
\begin{itemize}
\item there are horizontal characters $\eta_1,\ldots,\eta_r:G\to\mathbb R$
whose heights are bounded by
$(M/\delta)^{O_{k,\ell}(d^{O_{k,\ell}(1)})}$;
\item for every $1\le i\le r$,
\[
 \|\eta_i\circ g\|_{C^\infty[N]}
 \le (M/\delta)^{O_{k,\ell}(d^{O_{k,\ell}(1)})};
\]
\item for any $w_1,\ldots,w_s\in G/[G,G]$ such that the $w_i$ are
orthogonal to all of $\eta_1,\ldots,\eta_r$, one has
\[
 \xi([[[w_1,w_2],w_3],\ldots,w_s])=0.
\]
\end{itemize}
\end{theorem}

Now, we use the above theorem to imply the following lemma.

\begin{lemma}\label{sc:circle-lemma}
There is an absolute constant $C_{\mathrm W}\ge1$ with the following
property. Let $0<\tau<1/10$, let $M\in\N$ satisfy
$M\ge\tau^{-C_{\mathrm W}}$, and let $A,B\in\R$.
If $Q(n)=An+Bn^2$ satisfies
\begin{equation*}
 \left|M^{-1}\sum_{n=1}^M e(Q(n))\right|\ge\tau,
\end{equation*}
then there is an integer $1\le r\le\tau^{-C_{\mathrm W}}$ with
\begin{equation}
 M\|r(A+B)\|_{\R/\Z}\le\tau^{-C_{\rm W}},\qquad
 M^2\|2rB\|_{\R/\Z}\le\tau^{-C_{\rm W}}.
\label{sc:circle-output}
\end{equation}
\end{lemma}

\begin{proof}
We divide the proof into five steps.

\medskip
\noindent\textbf{Step 1. Realize the scalar phase on the circle nilmanifold.}
Set
\[
 G=\mathbb R,\qquad \Gamma=\mathbb Z,\qquad
 G_0=G_1=G_2=\mathbb R,\qquad G_3=\{0\}.
\]
Choose an absolute positive integer $\mathfrak M_0$ that bounds the
complexity of this filtered $\R/\Z$. We apply \Cref{sc:LSS-5.2} with
$\ell=1$, $k=2$, $s=1$, $d=1$, complexity bound $\mathfrak M_0$,
and $N=M$. The sequence $g(n)=Q(n)$ is polynomial for the displayed
degree-two filtration.

\medskip
\noindent\textbf{Step 2. Verify the vertical-frequency and correlation hypotheses.}
Choose $L_0\ge1$ so that
$\|x\mapsto e(x)\|_{\rm Lip}\le L_0$ on $\mathbb R/\mathbb Z$, and set
\[
 F(x+\Gamma)=\frac{e(x)}{L_0},\qquad
 \delta=\frac{\tau}{L_0},\qquad 0<\delta<1/10.
\]
Then $\|F\|_{\rm Lip}\le1$, and
\[
 \left|\mathbb E_{n\in[M]}F(g(n)+\Gamma)\right|
 =\frac1{L_0}\left|\mathbb E_{n\in[M]}e(Q(n))\right|
 \ge\delta.
\]
In this specialization, $G^{(1)}=G_2=\mathbb R$. Define
$\xi(z)=z$. It is integer-valued on $G\cap\Gamma=\mathbb Z$, has height
one, and is nonzero. Moreover,
\[
 F(z+x+\Gamma)=e(z)F(x+\Gamma)=e(\xi(z))F(x+\Gamma).
\]
Thus the vertical-frequency hypothesis of \Cref{sc:LSS-5.2} is satisfied,
and $1\le\mathfrak M_0/\delta$.

\medskip
\noindent\textbf{Step 3. Apply the quantitative inverse theorem.}
For the fixed values $(k,\ell,d)=(2,1,1)$, let $C_N$ be an absolute
exponent for the lower bound on $N$ in \Cref{sc:LSS-5.2}. Let
$C_\eta$ be an absolute exponent for the bounds on the heights and the
$C^\infty[M]$ norms of the horizontal characters in that theorem.
Choose $C_{\mathrm W}$ so large that, for every $0<\tau<1/10$,
\[
 \left(\frac{\mathfrak M_0L_0}{\tau}\right)^{\max\{C_N,C_\eta\}}
 \le \tau^{-C_{\mathrm W}}.
\]
Since $M\ge\tau^{-C_{\mathrm W}}$, the lower bound for $N$ required by
\Cref{sc:LSS-5.2} holds. \Cref{sc:LSS-5.2} then yields an integer $0\le q\le1$
and horizontal characters $\eta_1,\ldots,\eta_q$ whose heights and
the norms $\|\eta_j\circ g\|_{C^\infty[M]}$ are at most
$\tau^{-C_{\mathrm W}}$.
Because $s=1$, the final conclusion of \Cref{sc:LSS-5.2} says that $\xi$
vanishes on
\[
 W:=\bigcap_{j=1}^q\ker_{\mathbb R}\eta_j,
\]
with $W=G$ when $q=0$.

\medskip
\noindent\textbf{Step 4. Extract a nonzero integer frequency.}
If all of these horizontal characters were zero, including the case $q=0$,
then $W=G=\mathbb R$. This contradicts $\xi(1)=1$ and the preceding
vanishing conclusion. Therefore $q=1$ and $\eta_1$ is nonzero. Every
horizontal character on $G=\mathbb R$ is linear, and integer-valuedness on
$\Gamma=\mathbb Z$ gives $\eta_1(x)=rx$ for some
$r\in\mathbb Z\setminus\{0\}$. Replacing $r$ by $|r|$ preserves its
height and all norms, so
\[
 1\le r\le\tau^{-C_{\mathrm W}}.
\]

\medskip
\noindent\textbf{Step 5. Read off the coefficient bounds.}
Finally,
\[
 n^2=n+2\binom n2,
 \qquad
 Q(n)=(A+B)\binom n1+2B\binom n2,
\]
and hence
\[
 rQ(n)=r(A+B)\binom n1+2rB\binom n2.
\]
For the scalar polynomial
$P(n)=\alpha_0+\alpha_1\binom n1+\alpha_2\binom n2$, use the norm
\[
 \|P\|_{C^\infty[M]}
 =\max\left\{M\|\alpha_1\|_{\mathbb R/\mathbb Z},
                 M^2\|\alpha_2\|_{\mathbb R/\mathbb Z}\right\}.
\]
Applying this norm to $rQ$ and using the bound supplied by \Cref{sc:LSS-5.2}
gives
\[
 M\|r(A+B)\|_{\mathbb R/\mathbb Z}\le\tau^{-C_{\mathrm W}},\qquad
 M^2\|2rB\|_{\mathbb R/\mathbb Z}\le\tau^{-C_{\mathrm W}},
\]
which is exactly \eqref{sc:circle-output}.
\end{proof}

We shall also need a bound for the linear coefficient $A$ of
$Q(n)=An+Bn^2$. If the two quantities in \eqref{sc:circle-output}
are at most $L$, the triangle inequality on $\R/\Z$ gives
\begin{equation*}
 \|2rA\|_{\R/\Z}
 \le2\|r(A+B)\|_{\R/\Z}+\|2rB\|_{\R/\Z}
 \le\frac{2L}{M}+\frac{L}{M^2}
 \le\frac{3L}{M}.
\end{equation*}
Hence the same multiplier $2r$ controls both $A$ and $B$:
\begin{equation}
 M\|2rA\|_{\R/\Z}\le3L,\qquad
 M^2\|2rB\|_{\R/\Z}\le L.
 \label{sc:dense-bsz-coefficient-bounds}
\end{equation}

\enlargethispage{2pt}
\subsection{Means on residue classes}

For a multiplicative $u:\N\to\C$ with $|u|\le1$ and real
$V\ge3$, $T\ge1$, write
\begin{equation*}
 m(u,V,T)=\min_{|t|\le T}\sum_{\substack{p\le V\\p\ \mathrm{prime}}}
                       \frac{1-\Re(u(p)p^{-it})}{p}.
\end{equation*}
\newcommand{\HMTTheoremCitation}{\cite[p.~1192]{GranvilleSoundararajan}}
\begin{theorem}[Hal\'asz--Montgomery--Tenenbaum, \HMTTheoremCitation]\label{sc:halasz-external}
There is an absolute constant $C_H\ge1$ such that, for every $u,V,T$
just specified,\footnote{The cited theorem is unnumbered. With $f=u$ and
$x=V$, its quantity $M(V,S)$ is defined in
\cite[(1.2), p.~1192]{GranvilleSoundararajan} by minimizing over $|t|\le2S$.
For $T\ge2$, take $S=T/2$. Then $M(V,S)=m(u,V,T)$ and
$S^{-1/2}=\sqrt2\,T^{-1/2}$. For $1\le T<2$, the trivial estimate
$|V^{-1}\sum_{n\le V}u(n)|\le1\le\sqrt2\,T^{-1/2}$ suffices.
Both cases therefore satisfy the displayed bound for a suitable absolute $C_H$.}
\begin{equation*}
 \left|V^{-1}\sum_{n\le V}u(n)\right|
 \le C_H\bigl((1+m(u,V,T))e^{-m(u,V,T)}+T^{-1/2}\bigr).
\end{equation*}
\end{theorem}

\begin{lemma}\label{sc:residue-mean}
There is an absolute constant $C\ge1$ with the following property.
Assume $0<\kappa<1/2$, $Q_0,N\in\mathbb N$, and
$N\ge(2Q_0/\kappa)^C$. Let
\[
 f:\mathbb N\longrightarrow\mathbb S^1
\]
be a completely multiplicative function. Suppose that $I\subseteq[N]$
is an interval and that $b_0$ is a residue class modulo $Q_0$ satisfying
\begin{equation*}
 \left|\sum_{n\in I,\ n\equiv b_0\pmod{Q_0}}f(n)\right|\ge\kappa N.
\end{equation*} Then there are a primitive character $\chi$ of conductor dividing
$Q_0$ and $t\in\mathbb R$ such that
\begin{equation*}
 |t|\le(Q_0/\kappa)^C,
 \qquad \sum_{\substack{p\le N\\p\ \mathrm{prime}}}\frac{1-\Re(f(p)\overline{\chi(p)}p^{-it})}{p}
 \le C\log(2Q_0/\kappa).
\end{equation*}
\end{lemma}

\begin{proof}
We divide the proof into seven steps.	
	
\par
\medskip
\noindent\textbf{Step 1. Reduce to a large prefix sum.}
For an integer $0\le Y\le N$, write
\[
 S(Y):=\sum_{\substack{n\le Y\\ n\equiv b_0\pmod{Q_0}}}f(n).
\]
Write the interval as $I=\{L+1,\ldots,U_1\}$, where
$0\le L<U_1\le N$. The assumed sum is $S(U_1)-S(L)$.
Hence one of $S(U_1)$ and $S(L)$ has modulus at least
$\kappa N/2$. Denote its endpoint by $U$. Since $|f(n)|=1$,
$|S(U)|\le U$, and therefore
\[
 U\ge\frac{\kappa N}{2}.
\]
Thus $U\le N$ and
\begin{equation*}
 |S(U)|\ge\frac{\kappa N}{2}\ge\frac{\kappa U}{2}.
\end{equation*}

\par
\medskip
\noindent\textbf{Step 2. Select an additive Fourier frequency.}
By additive-character orthogonality modulo $Q_0$,
\begin{equation*}
 S(U)=\frac1{Q_0}\sum_{b=0}^{Q_0-1}e(-bb_0/Q_0)
       \sum_{n\le U}f(n)e(bn/Q_0).
\end{equation*}
The modulus of this average is at most the largest modulus of its
inner sums. We may therefore fix $b\pmod{Q_0}$ such that
\begin{equation}
 \left|\sum_{n\le U}f(n)e(bn/Q_0)\right|
 \ge\frac{\kappa N}{2}\ge\frac{\kappa U}{2}.
 \label{sc:inverse-rational-Fourier}
\end{equation}

\par
\medskip
\noindent\textbf{Step 3. Remove the contribution of large $Q_0$-smooth factors.}
Every positive integer $n$ has a unique factorization $n=dm$ in
which every prime divisor of $d$ divides $Q_0$ and $(m,Q_0)=1$.
Using complete multiplicativity of $f$, the sum in
\eqref{sc:inverse-rational-Fourier} is
\begin{equation*}
 \sum_{\substack{d\le U\\p\ \mathrm{prime},\ p\mid d\Rightarrow p\mid Q_0}}
     f(d)\sum_{\substack{m\le U/d\\(m,Q_0)=1}}
                      f(m)e(bdm/Q_0).
\end{equation*}
Put
\[
 C_{Q_0}:=\prod_{\substack{p\mid Q_0\\p\ \mathrm{prime}}}(1-p^{-1/2})^{-1},
 \qquad
 D:=\left\lceil\left(\frac{8C_{Q_0}}{\kappa}\right)^2\right\rceil.
\]
For $d>D$ the inner sum has modulus at most $U/d$. Since
$d^{-1}\le D^{-1/2}d^{-1/2}$,
\begin{equation*}
 U\sum_{\substack{d>D\\p\ \mathrm{prime},\ p\mid d\Rightarrow p\mid Q_0}}\frac1d
 \le UD^{-1/2}
 \sum_{\substack{d\ge1\\p\ \mathrm{prime},\ p\mid d\Rightarrow p\mid Q_0}}d^{-1/2}
 =UD^{-1/2}C_{Q_0}
 \le\frac{\kappa U}{8}.
\end{equation*}
Subtracting the contribution of $d>D$ from
\eqref{sc:inverse-rational-Fourier}, the terms with $d\le D$ have
combined modulus at least $3\kappa U/8$.
There are at most $D$ such values of $d$, so for one of them,
which we now fix,
\begin{equation*}
 \left|\sum_{\substack{m\le U/d\\(m,Q_0)=1}}
                      f(m)e(bdm/Q_0)\right|
 \ge\frac{\kappa U}{4D}.
\end{equation*}
Set $V=\lfloor U/d\rfloor$. Since $V\le U$, division by $V$ gives
\begin{equation*}
 \left|V^{-1}\sum_{\substack{m\le V\\(m,Q_0)=1}}
                      f(m)e(bdm/Q_0)\right|
 \ge\frac{\kappa}{4D}.
 \end{equation*}

\par
\medskip
\noindent\textbf{Step 4. Choose a long coprime average.}
The Euler product satisfies $C_{Q_0}\le Q_0^{C_0}$ for an absolute
constant $C_0$, and hence
$D\le(2Q_0/\kappa)^{C_1}$ for an absolute $C_1$.
Choose the absolute exponent $C$ in the lemma large enough that
$N\ge12D/\kappa$ follows from
$N\ge(2Q_0/\kappa)^C$. Since $U\ge\kappa N/2$ and $d\le D$,
\begin{equation*}
 \frac Ud\ge\frac{\kappa N}{2D}\ge6,
 \qquad
 V=\left\lfloor\frac Ud\right\rfloor
 \ge\frac{U}{2d}
 \ge\frac{\kappa N}{4D}\ge3.
\end{equation*}
Thus $V$ lies in the range required by \Cref{sc:halasz-external}.

\par
\medskip
\noindent\textbf{Step 5. Expand the coprime phase in characters.}
Let $\mathcal X(Q_0)$ be the set of Dirichlet characters modulo $Q_0$.
Then $|\mathcal X(Q_0)|=\varphi(Q_0)$, where $\varphi$ denotes
Euler's totient function (and $\varphi(1)=1$). For $\psi\in\mathcal X(Q_0)$,
define
\[
 c_\psi=\frac1{\varphi(Q_0)}
    \sum_{a\in(\Z/Q_0\Z)^\times}e(bda/Q_0)\overline{\psi(a)}.
\]
Character orthogonality gives
\begin{equation*}
 \mathbf1_{(m,Q_0)=1}e(bdm/Q_0)
   =\sum_{\psi\in\mathcal X(Q_0)}c_\psi\psi(m),
 \qquad |c_\psi|\le1.
\end{equation*}
Since $|\mathcal X(Q_0)|=\varphi(Q_0)$, the triangle inequality
applied to the preceding average gives one $\psi\in\mathcal X(Q_0)$ for which
\begin{equation}
 \left|V^{-1}\sum_{m\le V}f(m)\psi(m)\right|
 \ge r:=\frac{\kappa}{4D\varphi(Q_0)}
 \ge(2Q_0/\kappa)^{-C_2}
 \label{sc:inverse-character-mean}
\end{equation}
for an absolute $C_2$.

\par
\medskip
\noindent\textbf{Step 6. Apply the Hal\'asz--Montgomery--Tenenbaum estimate.}
The function $u=f\psi$ is multiplicative and satisfies $|u(n)|\le1$;
its possible zero values occur only when $(n,Q_0)>1$. Apply
\Cref{sc:halasz-external} with
\[
 T:=\left(\frac{4C_H}{r}\right)^2.
\]
Let
\[
 m_*:=\min_{|t|\le T}\sum_{\substack{p\le V\\p\ \mathrm{prime}}}
 \frac{1-\Re(f(p)\psi(p)p^{-it})}{p}.
\]
The minimum exists because the sum is continuous in $t$ and
$[-T,T]$ is compact. Also $m_*\ge0$. By
\eqref{sc:inverse-character-mean} and \Cref{sc:halasz-external},
\[
 r\le C_H\bigl((1+m_*)e^{-m_*}+T^{-1/2}\bigr)
 =C_H(1+m_*)e^{-m_*}+\frac r4.
\]
Therefore
\[
 \frac{3r}{4}\le C_H(1+m_*)e^{-m_*}
 \le2C_He^{-m_*/2},
\]
where $(1+x)e^{-x}\le2e^{-x/2}$ for $x\ge0$.
Taking logarithms gives
\[
 m_*\le2\log\left(\frac{8C_H}{3r}\right)
 \le2\log\left(\frac{4C_H}{r}\right).
\]
Let $t$ be a minimizing point. Then $|t|\le T$, and
\eqref{sc:inverse-character-mean} implies, after enlarging the
absolute exponent if necessary,
\begin{equation*}
 |t|\le\left(\frac{Q_0}{\kappa}\right)^C.
\end{equation*}
Moreover,
\begin{equation*}
 \sum_{\substack{p\le V\\p\ \mathrm{prime}}}\frac{1-\Re(f(p)\psi(p)p^{-it})}{p}
 \le2\log\left(\frac{4C_H}{r}\right).
\end{equation*}

\par
\medskip
\noindent\textbf{Step 7. Pass to a primitive character and extend the prime range.}
Let $\chi$ be the unique primitive character of conductor $q_\chi\mid Q_0$
that induces $\overline\psi$; thus
\[
 \psi(n)=\overline{\chi(n)}\qquad\text{whenever }(n,Q_0)=1.
\]
In particular, for $p\nmid Q_0$ the prime summands above agree.
For $p\mid Q_0$, the summand involving $\psi(p)=0$ equals $1/p$,
whereas the summand involving $\overline{\chi(p)}$ lies between
$0$ and $2/p$. Replacing $\psi$ by $\overline\chi$ therefore increases
the sum by at most
\[
 \sum_{\substack{p\mid Q_0\\p\ \mathrm{prime}}}\frac1p\le\log(2Q_0).
\]
It remains to extend the sum from $p\le V$ to $p\le N$. Each added
summand is at most $2/p$, so
\[
 2\sum_{\substack{V<p\le N\\p\ \mathrm{prime}}}\frac1p
 \le2\sum_{V<n\le N}\frac1n
 \le2\log(N/V)+\frac2V.
\]
Since $V\ge\kappa N/(4D)$ and $V\ge3$,
\[
 \log(N/V)\le\log(4D/\kappa)
 \ll\log(2Q_0/\kappa).
\]
The lower bound for $r$ in \eqref{sc:inverse-character-mean} also gives
\[
 \log(4C_H/r)\ll\log(2Q_0/\kappa).
\]
Adding this estimate, the prime-divisor correction, and the contribution
of primes $V<p\le N$ gives
\[
 \sum_{\substack{p\le N\\p\ \mathrm{prime}}}\frac{1-\Re(f(p)\overline{\chi(p)}p^{-it})}{p}
 \le C\log(2Q_0/\kappa).
\]
This is the required conclusion.
\end{proof}

\subsection{Scalar inverse theorem}

We now state the following theorem. Fix $0<\rho<1/2$ and $N\in\N$.
Let $f:\N\to\mathbb S^1$ be completely multiplicative,
let $\xi,\beta\in\mathbb R$, and let
$\mathcal P\subseteq[N]$ be an arithmetic progression.
Throughout this subsection, $\|z\|=\|z\|_{\mathbb R/\mathbb Z}$
for real $z$.

\begin{samepage}
\begin{theorem}\label{thm:scalar_inverse}
There is an absolute constant $C\ge1$, independent of
$\rho,N,f,\xi,\beta$, and $\mathcal P$, such that, if
\begin{equation*}
 N\ge\exp(C\rho^{-5}),\qquad
 \left|N^{-1}\sum_{n\in\mathcal P}
          f(n)e(\xi n+\beta n^2)\right|\ge\rho,
\end{equation*}
then there are positive integers $Q,Q_0\le\rho^{-C}$,
integers $u,v$, an interval $I\subseteq[N]$, and an integer $b_0$ satisfying
\begin{gather}
 |\xi-u/Q|\le\rho^{-C}/N,\qquad
 |\beta-v/Q|\le\rho^{-C}/N^2,
 \label{sc:dense-bsz-scalar-major-arcs}\\
 \left|\sum_{\substack{n\in I\\n\equiv b_0\pmod{Q_0}}}f(n)\right|
 \ge\rho^C N.
 \label{sc:dense-bsz-scalar-periodic-mean}
\end{gather}
There are also a primitive Dirichlet character $\chi$ and a real $t$ with
\begin{equation}
 q_\chi,\ |t|\le\rho^{-C},\qquad
 \sum_{\substack{p\le N\\p\ \mathrm{prime}}}\frac{1-\Re(f(p)\overline{\chi(p)}p^{-it})}{p}
 \le C\log(2/\rho).
 \label{sc:dense-bsz-scalar-pretentious}
\end{equation}
\end{theorem}
\end{samepage}

\begin{proof}
The constant $C$ in the statement is fixed at the end of the proof.
We record below the finitely many absolute lower bounds it must satisfy. We divide the rest proof into six steps.

\par
\medskip
\noindent\textbf{Step 1. Find many correlated prime pairs.}
Let $D$ be the positive common difference of $\mathcal P$,
with $D=1$ if $\mathcal P$ is a singleton. Let $b_*=\min\mathcal P$,
and let $I_*=[\min\mathcal P,\max\mathcal P]\cap\mathbb Z$
be the smallest integer interval containing $\mathcal P$. Then
\begin{equation*}
 \mathcal P=\{n\in I_*:n\equiv b_*\pmod D\},\qquad
 a(n)=\mathbf1_{\mathcal P}(n)e(\xi n+\beta n^2)
 \quad(n\in[N]).
\end{equation*}
The hypothesis ensures that $\mathcal P$ is nonempty.
Also $|a(n)|\le1$, and $f$ is multiplicative with $|f(n)|=1$.

We recall the parameters of \Cref{prop:bsz_dense}. Let $C_*\ge10$
be the fixed absolute constant used there, and define
\begin{equation*}
 \begin{gathered}
 \alpha=(\rho/C_*)^4,\qquad
 j_0=\left\lceil\frac{3\log(1/\alpha)}{\log(1+\alpha)}\right\rceil,
 \qquad j_1=\lceil j_0/\alpha\rceil,\\
 x_i=(1+\alpha)^i,\qquad
 P_i=\{p\text{ prime}:x_i\le p<x_{i+1}\}
 \quad(j_0\le i\le j_1),\\
 J=j_1-j_0+1,\qquad
 S=\sum_{i=j_0}^{j_1}\frac{|P_i|}{x_i},\qquad
 K=\lceil x_{j_1+1}\rceil.
 \end{gathered}
\end{equation*}
Here $j_0,j_1,J,K$ are positive integers. The integer $J$ counts
the prime blocks, $S$ is the displayed sum of their sizes divided
by their left endpoints, and $K$ bounds every prime in these blocks.
All these quantities depend only on $\rho$.

Let $C_{\mathrm{BSZ}}\ge1$ be an absolute constant valid for both
the size hypothesis and the parameter bounds in \Cref{prop:bsz_dense}.
Require $C\ge C_{\mathrm{BSZ}}$. Applying that proposition with
$\nu=f$ gives $S>0$ and an index $j$ with $j_0\le j\le j_1$.
Put $x=x_j$; then $P_j\subseteq[x,2x]$ and $x\le K$.
For distinct $p,q\in P_j$, define
\begin{equation*}
 \begin{gathered}
 n_{p,q}=\left\lfloor\frac{N}{\max(p,q)}\right\rfloor,\qquad
 R_{p,q}=\frac1{n_{p,q}}
       \left|\sum_{m=1}^{n_{p,q}}a(pm)\overline{a(qm)}\right|,\\
 \mathcal E=\{(p,q)\in P_j^2:p\ne q,\ R_{p,q}\ge\rho^2/(32S)\}.
 \end{gathered}
\end{equation*}
The integer $n_{p,q}\ge1$ is the largest $m$ for which both
$pm$ and $qm$ lie in $[N]$; $R_{p,q}$ measures the correlation
of the two sequences on this interval. The proposition gives
\begin{equation}
 |\mathcal E|\ge\frac{\rho^6x^2}{8192J^2S},
 \qquad R_{p,q}\ge\frac{\rho^2}{32S}
 \quad((p,q)\in\mathcal E).
 \label{sc:scalar-pair-data}
\end{equation}
In particular, $\mathcal E$ is nonempty.

By \Cref{sc:cross-progression}, $D\le2/\rho$.
Let $C_{\mathrm W}$ be the absolute constant in \Cref{sc:circle-lemma}.
Choose a real $\tau$ and a positive integer $H$ by
\begin{equation}
 \begin{aligned}
 \tau&=\min\{\rho^2/(32S),1/100\},\\
 H&=\left\lceil\max\left\{
 2,\frac4{\rho^2},\frac{8192J^2S}{\rho^6},
 \frac{128S}{\rho^2},3\tau^{-C_{\mathrm W}}
 \right\}\right\rceil.
 \end{aligned}
 \label{sc:scalar-H-choice}
\end{equation}
The number $\tau$ is a positive lower bound for the correlations
in \eqref{sc:scalar-pair-data}, with $\tau<1/10$ as required by
\Cref{sc:circle-lemma}. The integer $H$ combines the bounds for
$D^2$, $x^2/|\mathcal E|$, $128S/\rho^2$, and
$3\tau^{-C_{\mathrm W}}$ that will be used below.
The estimates
\[
 J\le C_{\mathrm{BSZ}}\rho^{-8}\log(2/\rho),\qquad
 S\le C_{\mathrm{BSZ}}\log(2/\rho),\qquad
 \tau^{-1}\le100+32S/\rho^2
\]
give fixed absolute constants $A_*,C_2\ge1$ such that
\begin{equation}
 H\le\rho^{-A_*},\qquad K\le\exp(C_2\rho^{-5}).
 \label{sc:scalar-H-growth}
\end{equation}
In particular, $D^2\le H$ and $|\mathcal E|\ge x^2/H$.

Fix $(p,q)\in\mathcal E$. By \Cref{sc:cross-progression},
the integers $m\in[n_{p,q}]$ with $pm,qm\in\mathcal P$
form an arithmetic progression. Write $M$ for its number of
terms, $d$ for its positive common difference, and $a_0+d$
for its first term. Here $a_0$ is an integer and $d,M$
are positive integers; they may depend on $(p,q)$. We have
\begin{equation}
 \begin{gathered}
 \{m\in[n_{p,q}]:pm,qm\in\mathcal P\}
       =\{a_0+d\ell:1\le\ell\le M\},\\
 d\le D^2\le H,\qquad |a_0|\le N/x+H,\qquad
 M\ge\frac{\rho^2N}{128Sx}\ge\frac{N}{Hx}.
 \end{gathered}
 \label{sc:scalar-cross-progression-data}
\end{equation}
Define the real quadratic polynomial
$Q_{p,q}(y)=(p-q)\xi y+(p^2-q^2)\beta y^2$.
The sum identity in \Cref{sc:cross-progression} gives
\begin{equation}
 \left|\frac1M\sum_{\ell=1}^M e(Q_{p,q}(a_0+d\ell))\right|
 =\frac{n_{p,q}}{M}R_{p,q}
 \ge R_{p,q}\ge\frac{\rho^2}{32S}\ge\tau,
 \label{sc:scalar-quadratic-average}
\end{equation}
where $M\le n_{p,q}$ because the progression lies in $[n_{p,q}]$.

\par
\medskip
\noindent\textbf{Step 2. Bound the polynomial coefficients.}
Require $C\ge C_3$, where
$C_3=\log(10^9)+30A_*+2C_2$. Then
\begin{equation}
 N\ge10^9H^{30}x^2.
 \label{sc:scalar-size-for-coefficients}
\end{equation}
Indeed, \eqref{sc:scalar-H-growth} and $x\le K$ give
\begin{equation*}
 \log(10^9H^{30}x^2)
 \le\log(10^9)+30A_*\log(1/\rho)+2C_2\rho^{-5}
 \le C_3\rho^{-5}
\end{equation*}
using $\log(1/\rho)\le\rho^{-5}$. Consequently
$M\ge N/(Hx)\ge H\ge\tau^{-C_{\mathrm W}}$.

Fix any $(p,q)\in\mathcal E$, and set $A_1=p-q$ and $A_2=p^2-q^2$.
Expanding $Q_{p,q}$ gives
\begin{equation*}
 Q_{p,q}(a_0+d\ell)=Q_{p,q}(a_0)+A\ell+B\ell^2,
 \quad A=dA_1\xi+2a_0dA_2\beta,\quad B=d^2A_2\beta.
\end{equation*}
The constant term does not change the modulus in
\eqref{sc:scalar-quadratic-average}. Since $0<\tau\le1/100$ and
$M\ge\tau^{-C_{\mathrm W}}$, \Cref{sc:circle-lemma} applies.
It gives an integer $1\le r'\le\tau^{-C_{\mathrm W}}$.
Set $r=2r'$. By \eqref{sc:dense-bsz-coefficient-bounds} and
$H\ge3\tau^{-C_{\mathrm W}}$, we have $r\le H$ and
\begin{align}
 \|rd^2(p^2-q^2)\beta\|&\le H/M^2,
 \\
 \|r\{d(p-q)\xi+2a_0d(p^2-q^2)\beta\}\|&\le H/M.
 \label{sc:dense-bsz-scalar-linear-input}
\end{align}
These inequalities hold for every pair in $\mathcal E$; the integers
$r,d,a_0,M$ may depend on the pair.

\par
\medskip
\noindent\textbf{Step 3. Obtain a common rational denominator.}
First suppose $x\ge8H$. The set $\mathcal E$, the bounds in
\eqref{sc:scalar-cross-progression-data}, and
\eqref{sc:scalar-size-for-coefficients} verify the hypotheses of
\Cref{sc:rational-approx}, with $X=x$ and $\alpha=\xi$.
Hence there are integers $u_1,v_1$ and positive integers $S_1,R_1$
such that
\begin{equation*}
 R_1\le256H^9,\qquad \left|\beta-\frac{v_1}{R_1}\right|\le\frac{4H^4}{N^2},\qquad S_1\le1024H^{14},\qquad \left|\xi-\frac{u_1}{S_1}\right|\le\frac{520H^7}{N}.
\end{equation*}
Both fractions therefore have a common denominator
$R_1S_1\le2^{18}H^{23}$.

Now suppose $x<8H$, and choose one pair in $\mathcal E$.
Since $p\ne q$ are positive, $A_1,A_2\ne0$. Put
$R_0=rd^2|A_2|$. We have $|A_1|\le x<8H$,
$|A_2|\le4x^2<256H^2$, and
\begin{equation*}
 \|R_0\beta\|\le\frac{H}{M^2}
 \le\frac{H^3x^2}{N^2}\le\frac{64H^5}{N^2}.
\end{equation*}
Choose $v_0\in\mathbb Z$ so that $|R_0\beta-v_0|=\|R_0\beta\|$. Then
\begin{equation}
 R_0\le256H^5,\qquad
 \beta=v_0/R_0+\beta_0,\qquad
 |\beta_0|=\frac{\|R_0\beta\|}{R_0}\le\frac{64H^5}{N^2}.
 \label{sc:dense-bsz-small-block-beta}
\end{equation}
The size bound \eqref{sc:scalar-size-for-coefficients} implies
$N\ge Hx$, so $|a_0|\le2N/x$.
Multiply \eqref{sc:dense-bsz-scalar-linear-input} by $R_0$ and
substitute \eqref{sc:dense-bsz-small-block-beta}.
The term $2ra_0dA_2v_0$ is an integer. Thus
\begin{align*}
 \|R_0rdA_1\xi\|
 &\le \frac{R_0H}{M}+2R_0rd|a_0|\,|A_2|\,|\beta_0|\notag\\
 &\le \frac{R_0H^2x}{N}
       +2R_0H^2\frac{2N}{x}(4x^2)\frac{64H^5}{N^2}\notag\\
 &\le\frac{1025R_0H^7x}{N}
 \le\frac{8200R_0H^8}{N}.
\end{align*}
Put $S_0=R_0rd|A_1|$. Then
$R_0\le S_0\le2048H^8$. Choosing an integer $u_0$ nearest to
$S_0\xi$ gives
\begin{equation*}
 \left|\xi-\frac{u_0}{S_0}\right|
 =\frac{\|S_0\xi\|}{S_0}
 \le\frac{8200H^8}{N},\qquad
 R_0S_0\le2^{19}H^{13}.
\end{equation*}

In either case, set $B_\rho=2^{20}H^{23}$ and use the product of
the two denominators. After rescaling the numerators, we obtain
integers $u,v$ and a positive integer $Q\le B_\rho$ such that
\begin{equation*}
 \xi=u/Q+\xi_0,\qquad \beta=v/Q+\beta_0,\qquad
 |\xi_0|\le B_\rho/N,\quad |\beta_0|\le B_\rho/N^2.
\end{equation*}
Set $C_4=20+23A_*$. Since $2^{20}\le\rho^{-20}$,
\eqref{sc:scalar-H-growth} gives $B_\rho\le\rho^{-C_4}$.
Thus \eqref{sc:dense-bsz-scalar-major-arcs} holds whenever $C\ge C_4$.

\par
\medskip
\noindent\textbf{Step 4. Control the remaining exponential factor.}
For real $y\in[1,N]$, put $P_0(y)=\xi_0y+\beta_0y^2$. Then
\begin{equation*}
 |P_0'(y)|\le|\xi_0|+2N|\beta_0|\le3B_\rho/N.
\end{equation*}
Require $C\ge\log2000+C_4+1$. Then
\[
 \log(2000B_\rho/\rho)
 \le\log2000+(C_4+1)\log(1/\rho)\le C\rho^{-5},
\]
so $N\ge2000B_\rho/\rho$. Set the positive integer
$L=\lfloor\rho N/(1000B_\rho)\rfloor$.
Divide $[N]$ into consecutive intervals $I_1,\ldots,I_{m_0}$
of length $L$, except possibly the last. Then
\begin{equation*}
 L\ge\frac{\rho N}{2000B_\rho}\ge1,\qquad m_0=\lceil N/L\rceil\le N/L+1\le\frac{2000B_\rho}{\rho}+1\le\frac{2001B_\rho}{\rho}.
\end{equation*}
The ceiling accounts for the last interval when its length is less than $L$.
Let $n_j$ be the first point of $I_j$. The mean value theorem and
$|e(s)-e(t)|\le2\pi|s-t|$ give
\begin{equation}
 |e(P_0(n))-e(P_0(n_j))|
 \le6\pi B_\rho L/N<\rho/10
 \qquad(n\in I_j).
 \label{sc:scalar-phase-freezing-error}
\end{equation}

\par
\medskip
\noindent\textbf{Step 5. Find a large sum on one residue class.}
The function $n\mapsto e((un+vn^2)/Q)$ is $Q$-periodic. Put
\begin{equation*}
 Q_0=\operatorname{lcm}(D,Q)\le DQ\le2B_\rho/\rho.
\end{equation*}
Split each $\mathcal P\cap I_j$ into residue classes modulo $Q_0$.
For $1\le j\le m_0$ and $0\le b<Q_0$, write
\begin{equation*}
 \mathcal C_{j,b}=\{n\in\mathcal P\cap I_j:n\equiv b\pmod{Q_0}\}.
\end{equation*}
Discard the empty sets. The remaining sets partition $\mathcal P$,
and there are at most $m_0Q_0\le4002B_\rho^2/\rho^2$ of them.
All sums over $(j,b)$ below run over these nonempty sets.
Since $Q\mid Q_0$, the function $e((un+vn^2)/Q)$ is constant
on each $\mathcal C_{j,b}$. Define
$z_{j,b}=e((ub+vb^2)/Q)e(P_0(n_j))$, so $|z_{j,b}|=1$.
Using \eqref{sc:scalar-phase-freezing-error} and $|f(n)|=1$, we get
\begin{equation*}
 \left|\sum_{n\in\mathcal P}f(n)e(\xi n+\beta n^2)
       -\sum_{j,b}z_{j,b}\sum_{n\in\mathcal C_{j,b}}f(n)\right|
 \le\frac{\rho|\mathcal P|}{10}\le\frac{\rho N}{10}.
\end{equation*}
The assumed correlation and the triangle inequality therefore give
\begin{equation*}
 \sum_{j,b}\left|\sum_{n\in\mathcal C_{j,b}}f(n)\right|
 \ge\frac{9\rho N}{10}.
\end{equation*}
At least one pair $(j,b)$ satisfies
\begin{equation}
 \left|\sum_{n\in\mathcal C_{j,b}}f(n)\right|
 \ge\frac{9\rho^3N}{40020B_\rho^2}
 \ge\kappa_\rho N,\qquad
 \kappa_\rho:=\frac{\rho^3}{5000B_\rho^2}.
 \label{sc:scalar-residue-sum}
\end{equation}
For this pair, take $I=I_j\cap I_*$ and $b_0=b$.
The set $\mathcal C_{j,b}$ is nonempty, so $b\equiv b_*\pmod D$.
As $D\mid Q_0$, it follows that
\[
 \mathcal C_{j,b}=\{n\in I:n\equiv b_0\pmod{Q_0}\}.
\]
The estimate $B_\rho\le\rho^{-C_4}$ gives
\begin{equation}
 Q_0\le2\rho^{-(C_4+1)}\le\rho^{-(C_4+2)},\qquad
 \kappa_\rho^{-1}\le5000\rho^{-(2C_4+3)}
       \le\rho^{-(2C_4+16)}.
 \label{sc:scalar-residue-power-bounds}
\end{equation}
Here $2\le\rho^{-1}$ and $5000\le\rho^{-13}$.
Thus $C\ge2C_4+16$ ensures both $Q_0\le\rho^{-C}$ and
$\kappa_\rho\ge\rho^C$, proving
\eqref{sc:dense-bsz-scalar-periodic-mean}.

\par
\medskip
\noindent\textbf{Step 6. Apply the residue-class lemma.}
We have $0<\kappa_\rho<1/2$. Let $C_{\mathrm{res}}\ge1$ be the
absolute constant in \Cref{sc:residue-mean}.
Put $E=3C_4+19$. By \eqref{sc:scalar-residue-power-bounds},
\begin{equation}
 \frac{2Q_0}{\kappa_\rho}\le\rho^{-E},\qquad
 \log\left((2Q_0/\kappa_\rho)^{C_{\mathrm{res}}}\right)
 \le C_{\mathrm{res}}E\log(1/\rho)
 \le C_{\mathrm{res}}E\rho^{-5}.
 \label{sc:scalar-residue-threshold}
\end{equation}
Requiring $C\ge C_{\mathrm{res}}E$ therefore gives
$N\ge(2Q_0/\kappa_\rho)^{C_{\mathrm{res}}}$, as needed.

Applying \Cref{sc:residue-mean} to \eqref{sc:scalar-residue-sum}
gives a primitive Dirichlet character $\chi$ and a real $t$ with
\begin{equation*}
 \begin{gathered}
 q_\chi\mid Q_0,\qquad
 |t|\le(Q_0/\kappa_\rho)^{C_{\mathrm{res}}},\\
 \sum_{\substack{p\le N\\p\ \mathrm{prime}}}\frac{1-\Re(f(p)\overline{\chi(p)}p^{-it})}{p}
 \le C_{\mathrm{res}}\log(2Q_0/\kappa_\rho).
 \end{gathered}
\end{equation*}
The estimates \eqref{sc:scalar-residue-power-bounds} and
\eqref{sc:scalar-residue-threshold} now give
\begin{equation*}
 \begin{gathered}
 q_\chi\le Q_0\le\rho^{-(C_4+2)},\qquad
 |t|\le\rho^{-C_{\mathrm{res}}E},\\
 \sum_{\substack{p\le N\\p\ \mathrm{prime}}}\frac{1-\Re(f(p)\overline{\chi(p)}p^{-it})}{p}
 \le C_{\mathrm{res}}E\log(1/\rho).
 \end{gathered}
\end{equation*}
We may therefore choose the single constant in the theorem to be
\begin{equation*}
 C=\max\{C_{\mathrm{BSZ}},\,C_3,\,\log2000+C_4+1,\,
            2C_4+16,\,C_{\mathrm{res}}(3C_4+19)\}.
\end{equation*}
Every entry in this maximum is an absolute constant.
This choice satisfies all size requirements above and gives
\eqref{sc:dense-bsz-scalar-major-arcs}--\eqref{sc:dense-bsz-scalar-pretentious}.
\end{proof}

\section{A multiplicative inverse theorem}
\label{inv:section}

This section proves the multiplicative inverse theorem, \Cref{thm:mult_inverse}.
The proof first extracts a structured quadratic sequence from a large
correlation and then converts that structure into a character description
of the multiplicative function. Throughout this section, $C,c>0$ denote
absolute constants, which may change from line to line. For a real or
integer vector $v$, $|v|$ denotes its maximum norm, and $\|v\|_2$ its
Euclidean norm. The height of an integer matrix is its maximum entry norm.

\subsection{Bohr sets and polynomial sequences} In this subsection, we collect some notation used later.

\medskip
\noindent\textbf{Finite-group notation.}
Let $A$ be a finite abelian group and write
$\mathbb T=\mathbb R/\mathbb Z$. Its dual group $\widehat A$ consists
of the homomorphisms $A\to\mathbb T$. We use normalized averages and
\begin{equation*}
 \widehat u(\xi)=\mathbb E_{x\in A}u(x)e(-\xi(x)),\qquad
 (u*v)(x)=\mathbb E_{y\in A}u(y)v(x-y),\qquad
 e(t)=e^{2\pi i t}.
\end{equation*}

For a nonempty finite set $S\subseteq\widehat A$ and $0<\rho<1/2$,
define
\begin{equation*}
 B(S,\rho)=
 \{x\in A:\max_{\xi\in S}\|\xi(x)\|_{\mathbb R/\mathbb Z}<\rho\}.
\end{equation*}
Its rank is $|S|$. We call it regular when
\begin{equation}
 \left|
 \frac{|B(S,(1+\kappa)\rho)|}{|B(S,\rho)|}-1
 \right|
 \le100|S||\kappa|
 \qquad (|\kappa|\le(100|S|)^{-1}).
 \label{inv:regular-definition}
\end{equation}
We use the standard estimates
\begin{equation*}
 |B(S,\rho)|\ge(c\rho)^{|S|}|A|,
 \qquad
 |B(S,2\rho)|\le4^{|S|}|B(S,\rho)|.
\end{equation*}
There is also a $\rho'\in[\rho/2,\rho]$ for which $B(S,\rho')$
is regular in the sense of \eqref{inv:regular-definition}.
\footnote{The size estimates and the regular-radius assertion are given
in \cite[Section 2, p.~7, immediately before Lemma 2.1]{LengQuadratic}.
Our regularity convention agrees with \cite[Definition 1.4(ii)]{JamneshanTao}.}

\medskip
\noindent\textbf{Local quadratic functions and progressions.}
Let $D\subseteq A$. A function $\phi:D\to\mathbb T$ is locally
quadratic if
\begin{equation*}
 \sum_{\omega\in\{0,1\}^3}(-1)^{|\omega|}
 \phi(x+\omega\cdot h)=0
\end{equation*}
whenever all eight arguments belong to $D$, where
$h=(h_1,h_2,h_3)$ and
$\omega\cdot h=\omega_1h_1+\omega_2h_2+\omega_3h_3$.
Local linearity is defined by the analogous vanishing of second
differences.

A symmetric coset progression of rank $r$ is
\begin{equation*}
 P=H+\left\{\sum_{j=1}^r n_jv_j:|n_j|\le L_j\right\},
 \qquad H\le A,\quad v_j\in A,
\end{equation*}
where $L_1,\ldots,L_r$ are nonnegative integers. Let
$B_L=\{(n_1,\ldots,n_r)\in\mathbb Z^r:|n_j|\le L_j\}$. We say that
$P$ is proper when the map
$$B_L\to A/H, (n_1,\ldots,n_r)\mapsto H+\sum_j n_jv_j,$$
is injective. Its density is $|P|/|A|$. For a set
$E\subseteq A$ and an integer $k\ge1$, write
\[
 kE=\{e_1+\cdots+e_k:e_1,\ldots,e_k\in E\}.
\]

\medskip
\noindent\textbf{Degree-two filtrations and polynomial sequences.}
Let $G$ be a connected, simply connected nilpotent Lie group. A
degree-two filtration on $G$ is a sequence
\[
 G=G_0=G_1\supseteq G_2\supseteq G_3=\{1\},
\]
where $G_2$ is closed, connected, and central and
$[G,G]\subseteq G_2$.  We call $G_2$ the central level of the
filtration: every element of $G_2$ commutes with every element of $G$.
The quotient map
$\pi:G\to G/G_2$ records the horizontal coordinate, while a central
coordinate is a coordinate along $G_2$.  In the Heisenberg coordinates
introduced below, $x=(\theta,\lambda,c)$ has horizontal coordinate
$(\theta,\lambda)$ and central coordinate $c$; in a product of factors,
these coordinates are read component by component.  Thus the phrase
``degree-two central filtration'' refers precisely to the chain
$G=G_1\supseteq G_2\supseteq G_3=\{1\}$ with $G_2\subseteq Z(G)$.
Fix a lattice $\Gamma$ in $G$; that is,
$\Gamma$ is a discrete subgroup and $G/\Gamma$ is compact.
Assume that $G_2/(\Gamma\cap G_2)$ is compact. The image of $\Gamma$
in $G/G_2$ is required to be discrete and cocompact.
Write
$G^{(1)}=G$ and $G^{(j+1)}=[G,G^{(j)}]$ for the lower central series.

A polynomial sequence adapted to this filtration has the form
\begin{equation}
 g(n)=g_0u^nv^{\binom n2},
 \qquad g_0,u\in G,\quad v\in G_2.
 \label{inv:polynomial-definition}
\end{equation}
With $\partial_hg(n)=g(n+h)g(n)^{-1}$, one computes
\begin{equation*}
 \partial_hg(n)=g_0u^hg_0^{-1}v^{hn+\binom h2},\qquad
 \partial_k\partial_hg(n)=v^{hk},\qquad
 \partial_l\partial_k\partial_hg(n)=1.
\end{equation*}
Thus the successive differences take values in $G_1,G_2,G_3$.
Conversely, if the second differences take values in $G_2$ and the
third differences are trivial, then
$\partial_1\partial_1g(n)=v\in G_2$ is independent of $n$.
Hence $\partial_1g(n)=av^n$ for some $a\in G$, and induction gives
$g(n)=a^nv^{\binom n2}g(0)$. Writing
$a=g(0)ug(0)^{-1}$ recovers \eqref{inv:polynomial-definition}.

\medskip
\noindent\textbf{The explicit Heisenberg models.}
We use the two-step groups
\begin{equation}
 \mathcal H(r)=\mathbb R^r\times\mathbb R^r\times\mathbb R,
 \quad
 (\theta,\lambda,c)(\theta',\lambda',c')
 =(\theta+\theta',\lambda+\lambda',c+c'+\lambda\cdot\theta'),
 \label{inv:group-definition}
\end{equation}
with lattice
$\Gamma(r)=\mathbb Z^r\times\mathbb Z^r\times\mathbb Z$ and
filtration subgroup
$\mathcal H(r)_2=\{(0,0,c):c\in\mathbb R\}$. We also use products
of at most two such groups, with the product lattice and filtration.
The dimension of $\mathcal H(r)$ is $2r+1$. The Lie bracket\footnote{The Lie bracket is the commutator bracket of the left-invariant
vector fields generated by the corresponding tangent vectors. In the
coordinates above, it is the displayed vector in the $c$-coordinate
direction.} is
\begin{equation*}
 [(\theta,\lambda,c),(\theta',\lambda',c')]_{\mathrm{Lie}}
 =(0,0,\lambda\cdot\theta'-\lambda'\cdot\theta),
 \qquad
 \exp(\theta,\lambda,c)
 =(\theta,\lambda,c+\tfrac12\lambda\cdot\theta).
\end{equation*}
In these coordinates, every coefficient in the group law, the Lie
bracket, and the exponential map has absolute value at most one.

\medskip
\noindent\textbf{Metrics, smooth sequences, and Lipschitz norms.}
For the rest of this subsection, let $G$ be a Heisenberg group
$\mathcal H(r)$ or a product of at most two such groups, and let
$\Gamma$ be its integer lattice. Use the Euclidean inner product
in the displayed Lie algebra coordinates and extend it by right
translations. Denote the resulting Riemannian distance by $d_G$.

Let $Q\ge1$ be real and $N\ge1$ an integer. A sequence
$\varepsilon:\{0,\ldots,N\}\to G$ is $(Q,N)$-smooth if
\[
 d_G(\varepsilon(n),1)\le Q\quad(0\le n\le N),\qquad
 d_G(\varepsilon(n+1),\varepsilon(n))\le Q/N
 \quad(0\le n<N).
\]
Thus $Q$ bounds the distance from the identity, and $Q/N$ bounds
the distance between consecutive terms.

The quotient metric on $G/\Gamma$ is
\[
 d_{G/\Gamma}(x\Gamma,y\Gamma)
   =\inf_{\gamma\in\Gamma}d_G(x,y\gamma).
\]
For $F:G/\Gamma\to\mathbb C$, define
\[
 \|F\|_{\mathrm{Lip}}
 =\|F\|_\infty+
 \sup_{x\Gamma\ne y\Gamma}
 \frac{|F(x\Gamma)-F(y\Gamma)|}
      {d_{G/\Gamma}(x\Gamma,y\Gamma)}.
\]

\medskip
\noindent\textbf{Characters and vertical frequencies.}
For the same group $G$ and lattice $\Gamma$, a horizontal character
is a continuous homomorphism
$\eta:G\to\mathbb R$ with $\eta(\Gamma)\subseteq\mathbb Z$.
It annihilates $[G,G]$ and therefore descends to a homomorphism on
$G/[G,G]$. For $\mathcal H(r)$,
\[
 \eta(\theta,\lambda,c)=a\cdot\theta+b\cdot\lambda,
 \qquad a,b\in\mathbb Z^r,
\]
and its height is $\max(|a|,|b|)$; on a product, use the maximum over
the factors. If $w\in G/[G,G]$, we say that $w$ is orthogonal to
$\eta$ when the descended homomorphism satisfies $\eta(w)=0$.
For $G=\mathbb R$ and $\Gamma=\mathbb Z$, horizontal characters
are $x\mapsto mx$ with $m\in\mathbb Z$.

For a positive integer $s$, a vertical character of level $s$ is a
continuous homomorphism $\xi:G^{(s)}\to\mathbb R$ with
$\xi(G^{(s)}\cap\Gamma)\subseteq\mathbb Z$. A function
$F:G/\Gamma\to\mathbb C$ has vertical frequency $\xi$ at level $s$
when
\begin{equation*}
 F(zx\Gamma)=e(\xi(z))F(x\Gamma)
 \qquad(z\in G^{(s)},\ x\in G).
\end{equation*}

\medskip
\noindent\textbf{Polynomial norms.}
Let $N\ge1$ be an integer. For an $\mathbb R/\mathbb Z$-valued
polynomial in binomial coordinates
\[
 P(n)=a_0+a_1\binom n1+a_2\binom n2,
 \qquad a_0,a_1,a_2\in\mathbb R/\mathbb Z,
\]
define
\[
 \|P\|_{C^\infty[N]}
 =\max\{N\|a_1\|_{\mathbb R/\mathbb Z},
        N^2\|a_2\|_{\mathbb R/\mathbb Z}\}.
\]
If $g$ is a degree-two polynomial sequence and $\eta$ is horizontal,
then $\eta\circ g$ is reduced modulo $\mathbb Z$ and is measured by
this norm.

\subsection{Initial quadratic correlation}

We first prove the progression-to-Bohr estimate needed for
\Cref{inv:bohr-rank}.

\begin{lemma}[Progression-to-Bohr estimate]
\label{inv:progression-bohr}
Let $A$ be a finite abelian group. Let $P\subseteq A$ be a symmetric
coset progression of rank $r\ge1$ and density
\[
 \alpha:=\frac{|P|}{|A|}>0.
\]
Set
\[
 S:=r+\log(2/\alpha).
\]
Then there is a nonempty finite set $\Psi\subseteq\widehat A$ such that
\begin{equation*}
 B\!\left(\Psi,\frac{1}{4|\Psi|}\right)\subseteq P,
 \qquad
 |\Psi|\ll S\log(2S).
\end{equation*}
\end{lemma}

\begin{proof}
Write the progression in the form
\[
 P=H+\left\{\sum_{j=1}^r n_jv_j: |n_j|\le L_j\right\},
 \qquad H\le A.
\]
We divide the rest proof into five steps.

\par
\medskip
\noindent\textbf{Step 1. A smaller progression.}
Choose an even integer $k\ge4$, put
\[
 q_j:=\left\lfloor\frac{L_j}{k}\right\rfloor,
 \qquad
 Q:=H+\left\{\sum_{j=1}^r n_jv_j: |n_j|\le q_j\right\}.
\]
Since $kq_j\le L_j$ and the sum of $k$ elements of the subgroup $H$
again lies in $H$, we have
\begin{equation*}
 kQ\subseteq P.
\end{equation*}
Since $L_j<k(q_j+1)$, we have
$2L_j+1\le2k(2q_j+1)$. Divide $[-L_j,L_j]\cap\mathbb Z$ into
at most $2k$ consecutive blocks, each with at most $2q_j+1$ terms.
Each block lies in a translate of $[-q_j,q_j]\cap\mathbb Z$.
Thus there is a set $T_j\subseteq\mathbb Z$ with $|T_j|\le2k$ and
\begin{equation*}
 [-L_j,L_j]\cap\mathbb Z
 \subseteq\bigcup_{t\in T_j}\bigl(t+([-q_j,q_j]\cap\mathbb Z)\bigr).
\end{equation*}
Put $\displaystyle E=\left\{\sum_{j=1}^r t_jv_j:t_j\in T_j\right\}\subseteq A$. Then
\begin{equation*}
 P\subseteq\bigcup_{t\in E}(t+Q),\qquad
 |P|\le |E|\,|Q|\le(2k)^r|Q|.
\end{equation*}
Consequently, with
\[
 \delta:=\frac{|Q|}{|A|},
\]
we have
\begin{equation}
 \delta\ge \alpha(2k)^{-r}.
 \label{inv:eq:progression-density}
\end{equation}

\par
\medskip
\noindent\textbf{Step 2. The large spectrum of $Q$.}
Let
\begin{equation*}
 a:=\frac12\,\delta^{\,1+1/(k-2)},
 \qquad
 \mathcal L:=\{\gamma\in\widehat A:
          |\widehat{\mathbf 1_Q}(\gamma)|\ge a\}.
\end{equation*}
Parseval gives
\[
 \sum_{\gamma\in\widehat A}
       |\widehat{\mathbf 1_Q}(\gamma)|^2
   =\mathbb E_{x\in A}\mathbf 1_Q(x)^2
   =\delta.
\]
Therefore
\begin{equation}
 \sum_{\gamma\notin\mathcal L}
       |\widehat{\mathbf 1_Q}(\gamma)|^k
 \le a^{\,k-2}
       \sum_{\gamma\in\widehat A}
       |\widehat{\mathbf 1_Q}(\gamma)|^2
 =2^{-(k-2)}\delta^k.
 \label{inv:eq:progression-spectrum-tail}
\end{equation}
Because $Q=-Q$, every Fourier coefficient
$\widehat{\mathbf 1_Q}(\gamma)$ is real. Since $k$ is even,
$\widehat{\mathbf 1_Q}(\gamma)^k\ge0$.

\par
\medskip
\noindent\textbf{Step 3. Positivity on a Bohr set.}
Write $\mathbf1_Q^{*k}$ for the normalized $k$-fold convolution of
$\mathbf1_Q$. Explicitly,
\begin{equation*}
 \mathbf1_Q^{*k}(x)
 =\frac{1}{|A|^{k-1}}
   \#\{(y_1,\ldots,y_k)\in Q^k:y_1+\cdots+y_k=x\}.
\end{equation*}
In particular, $\mathbf1_Q^{*k}(x)>0$ if and only if $x\in kQ$.
Fourier inversion gives
\begin{equation*}
 \mathbf 1_Q^{*k}(x)
 =\sum_{\gamma\in\widehat A}
       \widehat{\mathbf 1_Q}(\gamma)^k e(\gamma(x)).
\end{equation*}
If $x\in B(\mathcal L,1/4)$, then
$\Re e(\gamma(x))\ge0$ for every $\gamma\in\mathcal L$. The zero
frequency contributes $\delta^k$. The retained nonzero frequencies have
nonnegative real part, while the total absolute value of the discarded
frequencies is bounded by
\eqref{inv:eq:progression-spectrum-tail}. Hence
\begin{equation*}
 \Re\bigl(\mathbf 1_Q^{*k}(x)\bigr)
 \ge \delta^k-2^{-(k-2)}\delta^k>0.
\end{equation*}
Thus $\mathbf 1_Q^{*k}(x)>0$, so $x$ is a sum of $k$ elements of $Q$.
Therefore
\begin{equation*}
 B(\mathcal L,1/4)\subseteq kQ\subseteq P.
\end{equation*}

\par
\medskip
\noindent\textbf{Step 4. Compressing the large spectrum.}
We use the spectral form of Chang's lemma appearing in the proof of
\cite[Proposition~28]{Milicevic}: for $0<b\le1$, the set
\[
 \{\gamma\in\widehat A:
     |\widehat{\mathbf 1_Q}(\gamma)|\ge b\delta\}
\]
is contained in the $\{-1,0,1\}$-span of a set of at most
\begin{equation*}
 C b^{-2}\log(2/\delta)
\end{equation*}
characters, where $C$ is absolute. Apply this with
\begin{equation*}
 b:=\frac a\delta=\frac12\,\delta^{1/(k-2)}.
\end{equation*}
By Chang's lemma, after including the trivial character (which changes
the bound by at most one, absorbed because
$b^{-2}\log(2/\delta)\ge4\log2>1$), choose a set
$\Psi\subseteq\widehat A$ such that every element of
$\mathcal L$ is a sum
$\displaystyle\sum_{\psi\in\Psi}\varepsilon_\psi\psi$ with
$\varepsilon_\psi\in\{-1,0,1\}$ and
\begin{equation}
 |\Psi|\ll
 \delta^{-2/(k-2)}\log(2/\delta).
 \label{inv:eq:progression-chang-bound}
\end{equation}
If $x\in B(\Psi,1/(4|\Psi|))$, then
\begin{equation*}
 \|\gamma(x)\|_{\mathbb R/\mathbb Z}
 \le\sum_{\psi\in\Psi}\|\psi(x)\|_{\mathbb R/\mathbb Z}
 <\frac14
 \qquad(\gamma\in\mathcal L).
\end{equation*}
Thus
\begin{equation}
 B\!\left(\Psi,\frac1{4|\Psi|}\right)
 \subseteq B(\mathcal L,1/4)\subseteq P.
 \label{inv:eq:progression-bohr-containment}
\end{equation}

\par
\medskip
\noindent\textbf{Step 5. Choosing $k$.}
Set $T=\log(2S)$. Since $S\ge1$, we have $T\ge\log2$.
Choose
\begin{equation}
 k=2\lceil8ST\rceil,\qquad
 16ST\le k\le20ST,\qquad k-2\ge8ST.
 \label{inv:eq:progression-k-choice}
\end{equation}
These bounds follow from $ST\ge\log2$ and
$16ST\le k<16ST+2$. Also,
$\log S\le T$, $\log T\le T$, and $\log40\le6T$, so
$\log(2k)\le\log40+\log S+\log T\le8T$.
Since $r\le S$ and $\log(2/\alpha)\le S$,
\eqref{inv:eq:progression-density} gives
\begin{equation}
 \log(2/\delta)
 \le\log(2/\alpha)+r\log(2k)
 \le S+8ST\le10ST.
 \label{inv:eq:progression-log-delta}
\end{equation}
Together with \eqref{inv:eq:progression-k-choice}, this yields
\begin{equation*}
 \frac{\log(2/\delta)}{k-2}\le\frac{10ST}{8ST}=\frac54.
\end{equation*}
Consequently,
\begin{equation}
 \delta^{-2/(k-2)}
 =\exp\!\left(\frac{2\log(1/\delta)}{k-2}\right)
 \le e^{5/2}.
 \label{inv:eq:progression-exponent}
\end{equation}
All constants in these three estimates are numerical; in particular,
they are independent of $A$, $r$, and $\alpha$.
Substituting \eqref{inv:eq:progression-log-delta} and
\eqref{inv:eq:progression-exponent} into
\eqref{inv:eq:progression-chang-bound} yields
\begin{equation*}
 |\Psi|\ll S\log(2S).
\end{equation*}
Together with \eqref{inv:eq:progression-bohr-containment}, this proves
the lemma.
\end{proof}

We now use \Cref{inv:progression-bohr} to establish the local
quadratic inverse estimate.

\begin{proposition}
\label{inv:bohr-rank}
Let $L$ be an odd prime and put
\[
 G:=\mathbb Z/L\mathbb Z.
\]
Let $0<\epsilon<1/2$ and put
$\ell:=2+\log(2/\epsilon)$. Suppose that $u:G\to\mathbb C$ satisfies
\[
 |u(x)|\le1\quad(x\in G),
 \qquad
 \|u\|_{U^3(G)}\ge\epsilon.
\]
Then there exist an absolute constant $C\ge1$, a nonempty set
$S\subseteq\widehat G$ with $r:=|S|\ge1$, a regular Bohr set
\[
 B_G(S,\rho):=
 \{x\in G:\max_{\xi\in S}\|\xi(x)\|_{\mathbb R/\mathbb Z}<\rho\},
 \qquad 0<\rho<1/2000,
\]
a point $x_0\in G$, a locally quadratic function
$\phi:B_G(S,1000\rho)\to\mathbb T$, and a real number $\tau>0$
such that
\begin{gather*}
 \left|\mathbb E_{h\in B_G(S,\rho)}
             u(x_0+h)e(-\phi(h))\right|\ge\tau,
 \\
 r\le C\ell^{43},
 \\
 \log(2/\rho),\ \log(2/\tau)\le C\ell.
\end{gather*}
\end{proposition}

\begin{proof}
We divide the proof into seven steps.	
	
\par
\medskip
\noindent\textbf{Step 1. Extract a graph.}
The pointwise bound on $u$ gives
\[
 \|u\|_{L^{1024}(G)}\le1.
 \footnote{Here $\|v\|_{L^{1024}(G)}
 :=(\mathbb E_{x\in G}|v(x)|^{1024})^{1/1024}$ is the normalized
 counting norm.}
\]
By $\|u\|_{U^3(G)}\ge\epsilon$ and the above $L^{1024}$ bound,
the Cauchy--Schwarz, Balog--Szemer\'edi--Gowers, and graph-selection
steps in the finite-abelian-group form of
\cite[Lemmas~A.6--A.7]{LengQuadratic} supply a set
$H'\subseteq G$ and a map $\omega:H'\to\widehat G$ with
\[
 |H'|\ge\epsilon^{C_0}|G|
\]
for an absolute $C_0$. Put
\[
 \Gamma':=\{(h,\omega(h)):h\in H'\}.
\]
For each $h\in G$, the set
$\{\eta\in\widehat G:(h,\eta)\in4\Gamma'-4\Gamma'\}$ has at most one element.
Put $\delta_1:=|H'|/|G|$. Then, with $C_1:=C_0$,
\begin{equation*}
 \delta_1\ge\epsilon^{C_1}.
\end{equation*}
\par
\medskip
\noindent\textbf{Step 2. Obtain the first progression.}
We use \cite[Lemma~A.1]{LengQuadratic}. Its parameter $K\ge1$
measures the additive enlargement of the smaller summand: if
$|A_1+A_2|\le K\min(|A_1|,|A_2|)$, then
$(A_1-A_1)+(A_2-A_2)$ contains a proper symmetric coset progression
of rank at most $d(K)$ and size at least
$\exp(-h(K))|A_1+A_2|$, with
\begin{equation*}
 d(K)\ll\log^6(2K),\qquad
 h(K)\ll\log^6(2K)\log\log(4K).
\end{equation*}
Apply this with $A_1=A_2=H'$. Since
\[
 |H'+H'|\le|G|=\delta_1^{-1}|H'|,
\]
the admissible expansion parameter is
\[
 K_1:=\delta_1^{-1}.
\]
The cited sumset theorem gives a proper symmetric coset progression
$P_1\subseteq2H'-2H'$ with
\[
 |P_1|\ge\exp(-h(K_1))|H'+H'|.
\]
Moreover,
\[
 \log(2K_1)=\log(2/\delta_1)
 \le C_1\log(2/\epsilon)\ll\ell.
\]
Substitution into the bounds for $d$ and $h$, together with
$\alpha_1:=|P_1|/|G|$, gives
\begin{equation}
 r_1:=\operatorname{rank}(P_1)\ll\ell^6,\qquad
 \log(2/\alpha_1)\ll\ell^6\log(2\ell).
 \label{inv:eq:first-progression-data}
\end{equation}

\par
\medskip
\noindent\textbf{Step 3. Apply \Cref{inv:progression-bohr} once.}
If $r_1=0$, add one zero generator of length zero; this leaves $P_1$
and $\alpha_1$ unchanged. Apply \Cref{inv:progression-bohr} in
the ambient group $G$. Put
\[
 T_1:=r_1+\log(2/\alpha_1).
\]
The lemma gives a Bohr set
\[
 B_1=B_G(S_1,\rho_1)\subseteq P_1
\]
with a frequency set $S_1$ of rank
\begin{equation*}
 d_1:=|S_1|
 \ll T_1\log(2T_1)
 \ll\ell^7.
\end{equation*}
The last inequality follows from \eqref{inv:eq:first-progression-data};
the same lemma gives $\rho_1^{-1}\ll\ell^7$.

\par
\medskip
\noindent\textbf{Step 4. Linearize on a smaller Bohr set.}
The graph property from Step~1 is the hypothesis for the linearization
argument in \cite[Appendix~A.2]{LengQuadratic}. It produces a regular Bohr
set
\[
 B_2=B_G(S_2,\rho_2)\subseteq B_1,\qquad |S_2|\le d_1,
\]
a locally linear map
\[
 M:B_G(S_2,2000\rho_2)\to\widehat G,
\]
and a translate $x_1\in G$. Put
\[
 E_0:=\{h\in B_2:x_1+h\in H'\},\qquad
 \delta_0:=\frac{|E_0|}{|B_2|}.
\]
The quantitative part of the linearization argument gives
\[
 \rho_2\ge\epsilon^{C_2},\qquad \delta_0\ge\epsilon^{C_2}.
\]
The linearization estimate \cite[Lemma~A.5]{LengQuadratic} gives
functions $a:B_2\to\mathbb C$ and $b:G\to\mathbb C$, with
$|a|,|b|\le1$, such that
\begin{equation}
 \left|\mathbb E_{h\in B_2}\mathbb E_{x\in G}
 a(h)b(x)u(x+h)e\!\left(-2\langle M(h),x\rangle\right)\right|
 \ge\epsilon^{C_2}.
 \label{inv:eq:linearized-correlation}
\end{equation}
Applying the Cauchy--Schwarz and pigeonhole steps in
\cite[Appendix~A.3]{LengQuadratic} to this correlation gives a
set $E_1\subseteq B_2$ and an absolute $C_3$ such that
\begin{equation}
 \delta_3:=\frac{|E_1|}{|B_2|}\ge\epsilon^{C_3},\qquad
 \|2\{x,z\}\|_{\mathbb R/\mathbb Z}
 \le\epsilon^{-C_3}\max_{\xi\in S_2}\|\xi(x)\|_{\mathbb R/\mathbb Z}
 \quad(x\in B_2, z\in2E_1-2E_1),
 \label{inv:eq:alternating-input}
\end{equation}
where
\[
 \{x,z\}:=\langle M(x),z\rangle-\langle M(z),x\rangle.
\]
This is the only property of $E_1$ used in the next step.

\par
\medskip
\noindent\textbf{Step 5. Obtain the second progression.}
Since
$E_1+E_1\subseteq B_G(S_2,2\rho_2)$, the Bohr-volume estimate gives
\[
 |E_1+E_1|\le4^{d_1}|B_2|=4^{d_1}\delta_3^{-1}|E_1|.
\]
Set
\[
 K_2:=4^{d_1}\delta_3^{-1}.
\]
Thus $|E_1+E_1|\le K_2|E_1|$, so
\cite[Lemma~A.1]{LengQuadratic} supplies a proper symmetric coset progression
$P_3\subseteq2E_1-2E_1$ with
\[
 |P_3|\ge e^{-h(K_2)}|E_1+E_1|
 \ge e^{-h(K_2)}\delta_3|B_2|.
\]
Moreover,
\begin{equation*}
 \log(2K_2)
 \le\log2+d_1\log4+\log(1/\delta_3)
 \ll\ell^7,
\end{equation*}
\begin{equation*}
 \operatorname{rank}(P_3)\ll(\log(2K_2))^6\ll\ell^{42}.
\end{equation*}
The Bohr-volume estimate also gives
\begin{equation*}
 \log\frac{|G|}{|B_2|}\ll d_1\log(2/\rho_2)\ll\ell^8.
\end{equation*}
Consequently, with $\alpha_3:=|P_3|/|G|$,
\begin{equation*}
 \log(2/\alpha_3)
 \le h(K_2)+\log(2/\delta_3)+\log(|G|/|B_2|)+O(1)
 \ll\ell^{42}\log(2\ell),
\end{equation*}
so $\alpha_3\ge\exp(-C\ell^{42}\log(2\ell))$.

\par
\medskip
\noindent\textbf{Step 6. Convert the second progression into a symmetric
frequency form.}
Apply \Cref{inv:progression-bohr} to $P_3$ in $G$. If
$r_3:=\operatorname{rank}(P_3)$, we obtain a regular Bohr set
\[
 B_3=B_G(S_3,\rho_3)\subseteq P_3
\]
with
\begin{equation*}
 |S_3|\ll(r_3+\log(2/\alpha_3))
 \log(2r_3+2\log(2/\alpha_3))\ll\ell^{43},
 \qquad \rho_3^{-1}\ll\ell^{43}.
\end{equation*}
Since $\log\ell\ll\log(2/\epsilon)$, the second bound also gives
$\rho_3\ge\epsilon^C$ after changing the absolute constant.
For $\xi\in\widehat G$, let $\xi/2$ denote the unique character
whose double is $\xi$; it exists because $L$ is odd. Set
\[
 S_4:=S_2\cup S_3\cup\{\xi/2:\xi\in S_2\}.
\]
Thus $|S_4|\le2d_1+|S_3|\ll\ell^{43}$. Choose a regular radius
$\overline\rho_4\in[\min(\rho_2,\rho_3)/2,\min(\rho_2,\rho_3)]$
and put $\overline B_4:=B_G(S_4,\overline\rho_4)$.
For $x,z\in\overline B_4$, we have $x/2\in B_2$ and
$z\in B_3\subseteq2E_1-2E_1$. Local linearity gives
$\{x,z\}=2\{x/2,z\}$. Hence \eqref{inv:eq:alternating-input} gives
\begin{equation}
 \begin{gathered}
 \overline\rho_4\ge\epsilon^{a_4},\\
 \|\{x,z\}\|_{\mathbb R/\mathbb Z}
 \le\epsilon^{-b_4}\max_{\xi\in S_2}
       \|(\xi/2)(x)\|_{\mathbb R/\mathbb Z}
 \le\epsilon^{-b_4}\max_{\eta\in S_4}
       \|\eta(x)\|_{\mathbb R/\mathbb Z},
 \end{gathered}
 \label{inv:eq:alternating-form-control}
\end{equation}
where $a_4>0$ is absolute and $b_4=C_3$. This is the division by $2$
used in \cite[Appendix~A.3]{LengQuadratic}.
Write $c_*:=\epsilon^{C_*}$, where $C_*>0$ is chosen so
that $2\pi c_*\le\tfrac12\epsilon^{C_2}$. Put
\[
 t:=\min\{\overline\rho_4,c_*\epsilon^{b_4}\}.
\]
Choose a regular radius $\rho_4\in[t/2,t]$ and set
$B_4:=B_G(S_4,\rho_4)$. If
$A_4:=\max\{a_4,C_*+b_4\}$, then
$t\ge\epsilon^{A_4}$ and
$\rho_4\ge\tfrac12\epsilon^{A_4}\ge\epsilon^{A_4+1}$ because
$0<\epsilon<1/2$. Thus, with $C_4:=A_4+1$,
\begin{equation}
 \epsilon^{C_4}\le\rho_4\le c_*\epsilon^{b_4}.
 \label{inv:eq:alternating-form-radius}
\end{equation}
The defect estimate in \eqref{inv:eq:alternating-form-control} now gives
\[
 \|\{x,z\}\|_{\mathbb R/\mathbb Z}\le c_*
 \qquad(x,z\in B_4).
\]
Define
\[
 \mathcal B(x,z):=\langle M(x),z\rangle+\langle M(z),x\rangle,
 \qquad \Phi(x):=\langle M(x),x\rangle.
\]
The local linearity of $M$ gives, for $h,x\in B_4$,
\begin{equation}
 \mathcal B(h,x)=\Phi(x+h)-\Phi(x)-\Phi(h).
 \label{inv:eq:symmetric-phase}
\end{equation}

\par
\medskip
\noindent\textbf{Step 7. Localize the correlation and select one phase.}
We spell out the two operations that are implicit in the localization
argument of \cite[Appendix~A.4]{LengQuadratic}.

\smallskip
\noindent\emph{7a. Symmetrization of the bilinear phase.}
For $h,x\in B_4$, the definitions in Step~6 give
\[
 2\langle M(h),x\rangle
   =\mathcal B(h,x)+\{h,x\}.
\]
Consequently, by \eqref{inv:eq:alternating-form-radius},
\begin{equation}
 \left|e\!\left(-2\langle M(h),x\rangle\right)
       -e\!\left(-\mathcal B(h,x)\right)\right|
 \le 2\pi\|\{h,x\}\|_{\mathbb R/\mathbb Z}
 \le 2\pi c_*.
 \label{inv:eq:step7-phase-error}
\end{equation}
Choose $C_*$ in Step~6 so that
$2\pi c_*\le\tfrac12\epsilon^{C_2}$. Thus replacing the original
phase in \eqref{inv:eq:linearized-correlation} by the symmetric phase
loses at most one half of its lower bound. Moreover,
\eqref{inv:eq:symmetric-phase} gives
\begin{equation}
 e\!\left(-\mathcal B(h,x)\right)
 =e\!\left(-\Phi(x+h)\right)e\!\left(\Phi(x)\right)
  e\!\left(\Phi(h)\right).
 \label{inv:eq:step7-factorization}
\end{equation}
The first and third factors in \eqref{inv:eq:step7-factorization}
depend on $x+h$ and $h$, respectively, and can therefore be absorbed
into the two one-bounded coefficients produced by the localization
argument. The only phase that remains attached to the variable on
which $u$ is tested is $e(\Phi(x))$.

\smallskip
\noindent\emph{7b. Quantitative localization.}
Apply the quantitative conclusion of
\cite[Appendix~A.4]{LengQuadratic} to the correlation after the
replacement in \eqref{inv:eq:step7-phase-error}. It gives a regular Bohr
set
\[
 B_5=B_G(S,\rho),\qquad
 B_G(S,1000\rho)\subseteq B_4,\qquad
 |S|\ll\ell^{43},\qquad \epsilon^{C_5}\le\rho<1/2000.
\]
and, for each $y\in G$, functions $a_y,b_y:G\to\mathbb C$ with
$|a_y|,|b_y|\le1$, such that
\begin{equation}
 \mathbb E_{y\in G}\left|
 \mathbb E_{h\in B_4}\mathbb E_{x\in B_5}
 a_y(h)b_y(h+x)u(y+x)e\!\left(\Phi(x)\right)
 \right|
 \ge\epsilon^{C_6}.
 \label{inv:eq:localized-bilinear-correlation}
\end{equation}
The same localization chooses $B_5$ so that
\begin{equation}
 |B_4+B_5|\le2|B_4|.
 \label{inv:eq:step7-doubling}
\end{equation}
It also chooses $\rho$ small enough that
$B_4+B_5\subseteq B_G(S_2,2000\rho_2)$, the domain of $M$ from
Step~4. Thus $\Phi(h+x)$ and the coefficient $b_y(h+x)$ below are
defined wherever they are used. The displayed inclusion ensures that all phases
in \eqref{inv:eq:step7-factorization} are evaluated inside the domain
where $M$ is locally linear.

\smallskip
\noindent\emph{7c. Fourier extraction of a single linear character.}
Fix $y\in G$ and set
\[
 v_y(x):=u(y+x)e(\Phi(x)),\quad
 f_y:=a_y\mathbf1_{B_4},\quad
 q_y:=b_y\mathbf1_{B_4+B_5},\quad
 w_y:=v_y\mathbf1_{B_5}.
\]
All three functions are regarded as functions on $G$, extended by zero
outside their displayed supports. With the normalized Fourier transform,
Fourier inversion gives
\begin{align}
 &\mathbb E_{h\in B_4}\mathbb E_{x\in B_5}
       a_y(h)b_y(h+x)v_y(x) \notag\\
 &=\frac{|G|^2}{|B_4||B_5|}
   \sum_{\xi\in\widehat G}
      \widehat f_y(-\xi)\widehat q_y(\xi)\widehat w_y(-\xi).
 \label{inv:eq:step7-fourier-identity}
\end{align}
By Cauchy--Schwarz and Parseval,
\begin{align}
 \left|\mathbb E_{h\in B_4}\mathbb E_{x\in B_5}
       a_y(h)b_y(h+x)v_y(x)\right|
 &\le
 \left(\frac{|B_4+B_5|}{|B_4|}\right)^{1/2}
 \sup_{\xi\in\widehat G}
 \left|\mathbb E_{x\in B_5}v_y(x)e(\xi(x))\right|.
 \label{inv:eq:step7-fourier-bound}
\end{align}
Indeed,
$\|f_y\|_{L^2(G)}^2\le|B_4|/|G|$,
$\|q_y\|_{L^2(G)}^2\le|B_4+B_5|/|G|$, and
\[
 \widehat w_y(\xi)=\frac{|B_5|}{|G|}
       \mathbb E_{x\in B_5}v_y(x)e(-\xi(x));
\]
substitution into \eqref{inv:eq:step7-fourier-identity} yields
\eqref{inv:eq:step7-fourier-bound}. By
\eqref{inv:eq:step7-doubling}, the prefactor is at most $\sqrt2$.
Averaging over $y$ in
\eqref{inv:eq:localized-bilinear-correlation} therefore gives, after
increasing the absolute exponent,
\begin{equation}
 \mathbb E_{y\in G}\sup_{\xi\in\widehat G}
 \left|\mathbb E_{h\in B_5}u(y+h)
 e\!\left(\Phi(h)+\xi(h)\right)\right|
 \ge\epsilon^{C_7}.
 \label{inv:eq:localized-quadratic-correlation}
\end{equation}

\smallskip
\noindent\emph{7d. Pigeonhole and conclusion.}
Choose $y_0\in G$ and $\xi_0\in\widehat G$ for which the quantity in
\eqref{inv:eq:localized-quadratic-correlation} is at least
$\epsilon^{C_7}$. Put
\[
 x_0:=y_0,\qquad
 \phi(h):=-\Phi(h)-\xi_0(h),\qquad
 \tau:=\epsilon^{C_7}.
\]
The function $\xi_0$ is linear. The identity
\eqref{inv:eq:symmetric-phase}, together with the local linearity of $M$,
shows that $\Phi$ has vanishing third differences on
$B_G(S,1000\rho)$; hence $\phi$ is locally quadratic there. Therefore
\[
 \left|\mathbb E_{h\in B_G(S,\rho)}
       u(x_0+h)e(-\phi(h))\right|\ge\tau.
\]
Finally, the displayed Bohr-set bounds give
$r=|S|\ll\ell^{43}$ and
$\log(2/\rho)\ll\ell$, while
$\tau=\epsilon^{C_7}$ gives $\log(2/\tau)\ll\ell$.
Increasing the absolute constant $C$ once proves all the asserted
bounds and completes the proof.
\end{proof}

\Cref{inv:bohr-rank} gives a correlation with a locally
quadratic phase on a Bohr set in
\[
 A:=\mathbb Z/L\mathbb Z,
\]
where $L$ is an odd prime. For $S\subseteq\widehat A$ write
\[
 B_A(S,t):=\{a\in A:\max_{\xi\in S}\|\xi(a)\|_{\mathbb R/\mathbb Z}<t\}.
\]
We write $[n]_L$ for the residue class of $n$ in $A$.
Let $u:A\to\mathbb C$ satisfy $|u(a)|\le1$ for all $a\in A$.
The next lemma converts this local correlation into a correlation with
a degree-two nilsequence.

Under these assumptions, decrease $\tau$ if necessary so that
$0<\tau<1/2$. We have $0<\rho<1/2000$, $S\subseteq\widehat A$ is
nonempty with $r:=|S|$, $B_A(S,\rho)$ is regular, $x_0\in A$, and
$\phi:B_A(S,1000\rho)\to\mathbb T$ is locally quadratic with
\[
 \left|\mathbb E_{h\in B_A(S,\rho)}
 u(x_0+h)e(-\phi(h))\right|\ge\tau.
\]
Put
\begin{equation*}
 \Lambda:=r\log(2r)+\log(2/\rho)+\log(2/\tau).
\end{equation*}
If $S$ contains a nontrivial character, let
\[
 (G,\Gamma):=(\mathcal H(r),\Gamma(r)),
\]
with the standard metric and degree-two filtration from
\eqref{inv:group-definition}. If every member of $S$ is trivial, let
$(G,\Gamma)=(\mathbb R,\mathbb Z)$ with its degree-two filtration.
In both cases $\dim G\le2r+1$.

\begin{lemma}
\label{inv:realization}
Under the above assumptions, there exist a degree-two polynomial sequence $g:\mathbb Z\to G$,
a function $F:G/\Gamma\to\mathbb C$, and a real number $M\ge2$ such that
\[
 |F|\le1,\qquad \|F\|_{\rm Lip}\le M,
\]
and $F$ has vertical frequency one at level $2$, namely
\[
 F((0,0,t)x\Gamma)=e(t)F(x\Gamma)
 \quad(t\in\mathbb R,\ x\in G)
\]
in the Heisenberg case (with the analogous formula
$F(t+x+\mathbb Z)=e(t)F(x+\mathbb Z)$ in the circle case), and
\begin{gather}
 \left|\frac1L\sum_{n=1}^L
 u([n]_L)F(g(n)+\Gamma)\right|
 \ge\exp(-Cr\Lambda),
 \label{inv:eq:realization-correlation}\\
 \log(2M)\le C\Lambda.
 \label{inv:eq:realization-complexity}
\end{gather}
\end{lemma}

\begin{proof}
Write
\[
 B_t:=B_A(S,t)=\{x\in A:\max_{\xi\in S}\|\xi(x)\|_{\mathbb R/\mathbb Z}<t\}.
\]
We divide the rest proof into eight steps.

\par
\medskip
\noindent\textbf{Step 1. Dispose of the all-trivial-frequency case.}
If every character in $S$ is trivial, then $B_\rho=A$.  The vanishing
third differences of the locally quadratic function $\phi$ on $A$ give, for $n=0,\ldots,L-1$,
\begin{equation*}
 \phi([n]_L)=c+an+b\binom n2\pmod1.
\end{equation*}
Let $n_0$ represent $x_0$ and set
\[
 G=\mathbb R,\qquad \Gamma=\mathbb Z,\qquad
 F(t+\Gamma)=e(t),
\]
\[
 g(n):=-c-a(n-n_0)-b\binom{n-n_0}{2}.
\]
Then $g$ is a degree-two polynomial sequence and
$F(g(n)+\Gamma)=e(-\phi([n-n_0]_L))$.  Consequently the assumed local
correlation becomes
\[
 \left|\frac1L\sum_{n=1}^L u([n]_L)F(g(n)+\Gamma)\right|\ge\tau.
\]
The function $F$ has vertical frequency one, $|F|=1$, and its Lipschitz
norm is absolute.  Thus the lemma follows in this case.

From now on assume that $S$ contains a non-trivial
character.  First we prove the lattice extension fact needed in Step~3.

\par
\medskip
\noindent\textbf{Step 2. Lift the Bohr set to a lattice.}
Assume that $S$ contains a non-trivial character.  Identify
\[
 \mathcal L:=\{\theta\in\mathbb R^r:\theta_\xi=\xi(x)\pmod 1
       \text{ for some }x\in A\}\subset\mathbb R^r,
 \qquad r=|S|.
\]
The projection from $\mathcal L$ to $A$ is well-defined and injective modulo
$\mathbb Z^r$: indeed a non-trivial character of $A$ is injective because
$L$ is prime.  Hence $\mathcal L/\mathbb Z^r\simeq A$, so $\mathcal L$ is a
rank-$r$ lattice containing $\mathbb Z^r$ and
\begin{equation}
 [\mathcal L:\mathbb Z^r]=L,\qquad \det\mathcal L=L^{-1}.                                      \label{inv:eq:lem73-lattice}
\end{equation}
Let $\pi:\mathcal L\to A$ denote the induced projection. Since
$\rho<1/2000$, every element of $B_t$ with $t\le2\rho$ has a unique
representative $\theta\in\mathcal L$ with $|\theta|_\infty<t$; conversely,
such a representative projects to an element of $B_t$.
For each $t>0$, the discrete John theorem of Tao and Vu \cite[Theorem 1.6]{TaoVuJohn} gives fixed real numbers $N_i>0$ and the lattice basis 
$w_1,\ldots,w_r$ such that,
with $\kappa=(Cr)^{3r/2}$,
\begin{equation}
 [-t\kappa^{-1},t\kappa^{-1}]^r\cap\mathcal L
 \subseteq\Big\{\sum_i n_iw_i:n_i\in \Z,|n_i|<tN_i\Big\}
 \subseteq[-t,t]^r\cap\mathcal L .                                  \label{inv:eq:lem73-john}
\end{equation}
For later use, put
\[
 \mathcal C_t:=\Big\{\sum_{i} n_iw_i:
                    n_i\in\mathbb Z,\ |n_i|<tN_i\Big\}.
\]
The lattice-point asymptotic for a full-rank lattice gives
\[
 \frac{|\mathcal L\cap[-t\kappa^{-1},t\kappa^{-1}]^r|}{(2t)^r}
 \longrightarrow \frac{\kappa^{-r}}{\det\mathcal L}
 =L\kappa^{-r},
 \qquad
 \frac{|\mathcal C_t|}{(2t)^r}\longrightarrow\prod_iN_i
 \quad (t\to\infty).
\]
Taking cardinalities in the first inclusion of \eqref{inv:eq:lem73-john}
and passing to the limit, and using \eqref{inv:eq:lem73-lattice}, yields
\begin{equation}
 \prod_iN_i\ge L\kappa^{-r}=L(Cr)^{-3r^2/2}.                       \label{inv:eq:lem73-volume}
\end{equation}
We will use only the schematic consequence
$$\prod_iN_i\ge L\exp(-Cr^2\log(2r)).$$ 

\par
\medskip
\noindent\textbf{Step 3. Extend the local quadratic phase.}
Let $y\in\mathcal L$ and let
\[
 \psi:\mathcal L\cap\bigl(y+[-2\rho,2\rho]^r\bigr)\longrightarrow\mathbb T
\]
be locally quadratic.  Put
\[
 \sigma_*:=\rho\kappa^{-1}/1000=\rho(Cr)^{-3r/2}/1000.
\]
Define
\[
 D:=\Bigl\{n\in\mathbb Z^r:|n_i|<\rho N_i/8\ \text{for every }i\Bigr\},
 \qquad
 \Psi:D\to \mathbb T,n\mapsto\psi\!\left(y+\sum_{i=1}^rn_iw_i\right).
\]
The second inclusion in \eqref{inv:eq:lem73-john}, used with $t=\rho/8$,
shows that all points used in $\Psi$ lie in
$y+[-\rho/8,\rho/8]^r$ and hence in the domain of $\psi$.  Write
$\Delta_i\Psi(n):=\Psi(n+e_i)-\Psi(n)$ whenever both terms are
defined.  Local quadraticity says that every admissible third difference
$\Delta_i\Delta_j\Delta_k\Psi$ vanishes.  Therefore each second
difference $\Delta_i\Delta_j\Psi$ is constant on the coordinate-connected
interior of $D$.

Let $J:=\{i:\rho N_i>32\}$.  For $i\in J$, the points $0,e_i,2e_i$
lie in $D$; for distinct $i,j\in J$, the points
$0,e_i,e_j,e_i+e_j$ lie in $D$.  Define, in $\mathbb T$,
\[
 a_i:=\Delta_i\Psi(0),\qquad b_i:=\Delta_i^2\Psi(0),\qquad
 c_{ij}:=\Delta_i\Delta_j\Psi(0)\quad(i<j, i,j\in J).
\]
On the subbox
$D_J:=\{n\in D:n_i=0\text{ for }i\notin J\}$, the Newton identity is
\begin{equation}
 \Psi(n)=\psi(y)+\sum_{i\in J}a_in_i
       +\sum_{i\in J}b_i\binom{n_i}{2}
       +\sum_{i<j,\ i,j\in J}c_{ij}n_in_j.                         \label{inv:eq:lem73-newton}
\end{equation}

We now extend this formula to the whole lattice.  Every
$\theta\in\mathcal L$ has a unique representation
$\theta=y+\sum_i n_iw_i$.  Define
\[
 \widetilde\psi\!\left(y+\sum_i n_iw_i\right):=
 \psi(y)+\sum_{i\in J}a_in_i
       +\sum_{i\in J}b_i\binom{n_i}{2}
       +\sum_{i<j,\ i,j\in J}c_{ij}n_in_j
       \quad\pmod 1.
\]
This is a global quadratic map on $\mathcal L$.  If
$\theta-y\in[-\sigma_*,\sigma_*]^r\cap\mathcal L$, the first inclusion
in \eqref{inv:eq:lem73-john}, with $t=\rho/1000$, gives unique integers
$n_i$ satisfying $|n_i|<\rho N_i/1000$.  For $i\notin J$ this is
$\le 32/1000<1$, so $n_i=0$; for $i\in J$ it is $<\rho N_i/8$, so
$n\in D_J$.  Hence \eqref{inv:eq:lem73-newton} gives
\[
 \widetilde\psi(\theta)=\psi(\theta)\qquad
 \bigl(\theta\in\mathcal L\cap(y+[-\sigma_*,\sigma_*]^r)\bigr).
\]
The construction is translation-invariant in $y$.

\par
\medskip
\noindent\textbf{Step 4. Smooth the correlation and choose a good translate.}
We now treat the non-trivial-frequency case.  Let
\[
 C_0:=\left|\mathbb E_{h\in B_\rho}u(x_0+h)e(-\phi(h))\right|\ge\tau .
\]
Choose a regular radius
\begin{equation}
 0<\sigma\le\min\{\sigma_*,\tau\rho/(1000r)\},\qquad
 \sigma\ge\tfrac12\min\{\sigma_*,\tau\rho/(1000r)\}.              \label{inv:eq:lem73-sigma}
\end{equation}
For $k\in B_\sigma$, regularity gives
\[
 B_{\rho-\sigma}\subseteq B_\rho\cap(B_\rho-k),\qquad
 B_\rho\cup(B_\rho-k)\subseteq B_{\rho+\sigma},
\]
and consequently
\begin{equation*}
 \left|\mathbb E_{h\in B_\rho}f(h)-
 \mathbb E_{h\in B_\rho}\mathbb E_{k\in B_\sigma}f(h+k)\right|
 \le C r\sigma/\rho\le\tau/4,                                    
\end{equation*}
where $f(h)=u(x_0+h)e(-\phi(h))$ on $B_{1000\rho}$ (and $f=0$ outside
that set).  Thus there is $h_0\in B_\rho$ with
\begin{equation}
 \left|\mathbb E_{k\in B_\sigma}u(x_0+h_0+k)e(-\phi(h_0+k))\right|
 \ge\tau/2.                                                        \label{inv:eq:lem73-good-translate}
\end{equation}
\par
\medskip
\noindent\textbf{Step 5. Recenter the phase and build the Heisenberg orbit.}
The translated phase $k\mapsto\phi(h_0+k)$ is locally quadratic on
$B_{2\rho}$. Choose a lift $y\in\mathcal L$ with $|y|_\infty<\rho$ whose
image in $A$ is $h_0$. Apply the construction in Step 3 to
$\psi(\vartheta):=\phi(\pi(\vartheta))$ on the lattice points near $y$.
It supplies a global quadratic
$\widetilde q:\mathcal L\to\mathbb R/\mathbb Z$ agreeing with
$\phi(\pi(\vartheta))$ for $\vartheta\in y+[-\sigma,\sigma]^r$.
Define the recentered phase
$q_0(\theta):=\widetilde q(y+\theta)$; then
$q_0(\theta)=\phi(h_0+\pi(\theta))$ for every $\theta$ projecting to
$B_\sigma$. The extension scale in \eqref{inv:eq:lem73-sigma} ensures this agreement.

Choose a real quadratic polynomial $Q$ on $\mathbb R^r$ with
\[
 Q(\theta)\equiv-q_0(\theta)\pmod1\qquad(\theta\in\mathcal L).
\]
In the integer-valued basis
$1,z_i,\binom{z_i}{2},z_iz_j$, subtract an integer-valued polynomial
$Q_{\mathbb Z}$ so that
\begin{equation}
 Q=Q_0+Q_{\mathbb Z},
 \quad\text{all coefficients of $Q_0$ belonging to }[-1,1].       \label{inv:eq:lem73-quadratic-split}
\end{equation}
Choose $a\in\mathcal L$ whose image in $A$ is $[1]_L$; then $na\in\mathcal L$ has image $[n]_L$ for every $n\in\mathbb Z$.  Write
\[
 Q_{\mathbb Z}(z+na)-Q_{\mathbb Z}(z)=\lambda(n)\cdot z+c(n).
\]
Because $Q_{\mathbb Z}$ is quadratic, $\lambda(n)=n\lambda_1$ is linear
and $c(n)=n c_1+\binom n2c_2$ is quadratic.  Define, in the coordinates
of $\mathcal H(r)$,
\begin{equation*}
 h(n):=(na,\lambda(n),c(n)).                                    
\end{equation*}
If $u_1=(a,\lambda_1,c_1)$ and
$v_1=(0,0,c_2-\lambda_1\cdot a)$, then the Heisenberg product gives
\begin{equation*}
 h(n)=u_1^n v_1^{\binom n2};                                      
\end{equation*}
thus $h$ is a degree-two polynomial sequence.  

\par
\medskip
\noindent\textbf{Step 6. Construct the nilsequence function.}
We now turn the real quadratic polynomial into a Lipschitz function on the Heisenberg nilmanifold.
Let $w=c\tau/r$, with $c>0$ an absolute sufficiently small constant.  Let
$\varphi:\mathbb R^r\to[0,1]$ be $1$ on
$[-(1-w)\sigma,(1-w)\sigma]^r$, supported on $[-\sigma,\sigma]^r$, and
linear in the maximum norm on the intervening annulus.  Then
\begin{equation*}
 \|\varphi\|_{\rm Lip}\le (w\sigma)^{-1}.                         
\end{equation*}
Regularity at $\sigma$ gives
\begin{equation}
 \frac{|B_\sigma\setminus B_{(1-w)\sigma}|}{|B_\sigma|}\le C r w\le\tau/4. 
                                                                            \label{inv:eq:lem73-cutoff-boundary}
\end{equation}
Define on $\mathcal H(r)$
\begin{equation}
 F_0(\theta,\lambda,c')
 :=\sum_{m\in\mathbb Z^r}\varphi(\theta+m)
 e\big(Q_0(\theta+m)+c'+\lambda\cdot m\big).                     \label{inv:eq:lem73-periodization}
\end{equation}
The support has radius $<1/4$, so at most one summand is non-zero and
$|F_0|\le1$. Let $\gamma=(m_0,b,k)\in\Gamma(r)$, where
$m_0,b\in\mathbb Z^r$ and $k\in\mathbb Z$. The group law gives
\[
\begin{aligned}
 F_0((\theta,\lambda,c')\gamma)
 &=\sum_m \varphi(\theta+m_0+m)
 e\bigl(Q_0(\theta+m_0+m)+c'+k+\lambda\cdot m_0
                 +(\lambda+b)\cdot m\bigr)\\
 &=\sum_j \varphi(\theta+j)
 e\bigl(Q_0(\theta+j)+c'+\lambda\cdot j
                 +k+b\cdot(j-m_0)\bigr)
 =F_0(\theta,\lambda,c'),
\end{aligned}
\]
because $k+b\cdot(j-m_0)\in\mathbb Z$. Hence $F_0$ descends to
$\mathcal H(r)/\Gamma(r)$.  Denote the descended function by $F$.
Moreover,
\begin{equation*}
 F_0((0,0,t)x\Gamma)=e(t)F_0(x\Gamma).                            
\end{equation*}
\par
\medskip
\noindent\textbf{Step 7. Transfer the Bohr correlation to an average on $A$.}
Let $k\in\mathbb Z$ satisfy $[k]_L\in B_\sigma$.
Since $\pi(a)=[1]_L$ and $\ker\pi=\mathbb Z^r$, every lift of
$[k]_L$ to $\mathcal L$ has the form $ka+m$ with $m\in\mathbb Z^r$.
The condition $[k]_L\in B_\sigma$ guarantees the existence of
$m=m(k)\in\mathbb Z^r$ such that
\[
\theta_k:=ka+m\in(-\sigma,\sigma)^r.
\]
This lift is unique in $[-\sigma,\sigma]^r$: two such lifts
would differ by an integer vector of maximum norm at most
$2\sigma<1$, so their difference must be zero.
Since $\varphi$ is supported in $[-\sigma,\sigma]^r$, only this
term can contribute to the sum defining $F_0(h(k))$. Consequently,
\begin{equation}
 F_0(h(k))=\varphi(ka+m)e\big(Q(ka+m)\big).                        \label{inv:eq:lem73-orbit-evaluation}
\end{equation}
The equality is modulo one: the exponent in \eqref{inv:eq:lem73-periodization} is
\[
 Q_0(ka+m)+Q_{\mathbb Z}(ka+m)-Q_{\mathbb Z}(m)
 \equiv Q(ka+m)\pmod1,
\]
because $Q_{\mathbb Z}(m)\in\mathbb Z$. On $B_{(1-w)\sigma}$, \eqref{inv:eq:lem73-orbit-evaluation} equals
$e(-q(k))$.  Combining \eqref{inv:eq:lem73-good-translate}, \eqref{inv:eq:lem73-cutoff-boundary}, and $|u|\le1$ gives
\begin{equation}
 \left|\mathbb E_{k\in B_\sigma}u(x_0+h_0+k)F_0(h(k))\right|\ge\tau/4. 
                                                                            \label{inv:eq:lem73-bohr-correlation}
\end{equation}

\par
\medskip
\noindent\textbf{Step 8. Normalize the correlation and bound the Lipschitz norm.}
For completeness, the lattice estimate needed to normalize \eqref{inv:eq:lem73-bohr-correlation} is
\begin{equation*}
 \frac{|B_\sigma|}{L}
 \ge \exp(-Cr^2\log(2r))\sigma^r.                                 
\end{equation*}
Indeed, the second inclusion in \eqref{inv:eq:lem73-john} gives at least
$\prod_i(2\lceil\sigma N_i\rceil-1)$ points in $B_\sigma$.  Since
$2\lceil x\rceil-1\ge x$ for $x\ge0$, this is bounded below by
$\sigma^r\prod_iN_i$.  Using \eqref{inv:eq:lem73-volume} therefore gives
the displayed bound, after enlarging $C$.
Therefore \eqref{inv:eq:lem73-bohr-correlation} implies
\begin{equation}
 \left|\mathbb E_{k\in A}u(x_0+h_0+k)F_0(h(k))\right|
 \ge \exp(-Cr\Lambda).                                            \label{inv:eq:lem73-cyclic-correlation}
\end{equation}
Choose $n_0\in\{0,\ldots,L-1\}$ representing $x_0+h_0$ and put
$g(n):=h(n-n_0)$.  Translation of a degree-two polynomial sequence is
again degree two, and the change of variable $k=[n-n_0]_L$ converts
\eqref{inv:eq:lem73-cyclic-correlation} into the required correlation with $u([n]_L)$.

Finally, use the right-invariant coordinate fields
$\partial_{\theta_i}$,
$\partial_{\lambda_i}+\theta_i\partial_{c'}$, and $\partial_{c'}$.
On every non-zero summand of \eqref{inv:eq:lem73-periodization}, $|\theta_i+m_i|\le\sigma$; hence the
middle field differentiates its phase by $m_i+\theta_i$ of size at most
$\sigma$.  The first fields differentiate $\varphi$ and $Q_0$; the
coefficients of $Q_0$ in \eqref{inv:eq:lem73-quadratic-split} give $\|\nabla Q_0\|\le Cr^2$ on the support.
Consequently
\begin{equation*}
 \|F_0\|_{\rm Lip}\le Cr^2\bigl(1+(w\sigma)^{-1}\bigr).            
\end{equation*}
By \eqref{inv:eq:lem73-sigma}, $w=c\tau/r$, and the definition of $\sigma_*$,
\begin{equation*}
 \log(2\sigma^{-1})\le C\Lambda,\qquad
 \log(2\|F_0\|_{\rm Lip})\le C\Lambda.                            
\end{equation*}
Set $M=\max\{2,\|F_0\|_{\rm Lip}\}$.  With the descended function
$F$ and this value of $M$, the bounds in the lemma are exactly
\eqref{inv:eq:realization-correlation} and \eqref{inv:eq:realization-complexity}.

This completes the non-trivial-frequency case.  Together with Step~1,
this proves the lemma.
\end{proof}

\begin{corollary}
\label{inv:initial-inverse}
Let $N\ge 1$ be an integer.  Let
$f:[N]\to\mathbb C$ satisfy $|f(n)|\le1$ for $n\in[N]$ and
\[
 \|f\|_{U^3[N]}\ge\epsilon,
 \qquad 0<\epsilon<1/2.
\]
Put $\ell:=2+\log(2/\epsilon)$.  Then there exist an integer
$d\ge1$, real $M\ge2$ and $0<\eta<1/2$, a filtered nilmanifold
$G/\Gamma$ of dimension $d$, a degree-two polynomial sequence
$g:\mathbb Z\to G$, and a function $F:G/\Gamma\to\mathbb C$ such that
\[
 |F(x\Gamma)|\le1,\qquad \|F\|_{\rm Lip}\le M,
\]
where $(G,\Gamma)$ is either $(\mathbb R,\mathbb Z)$ or
$(\mathcal H(r),\Gamma(r))$ for an integer $r\ge1$.  The function $F$ has
central frequency one in the following sense: in the Heisenberg
case,
\[
 F((0,0,t)x\Gamma)=e(t)F(x\Gamma)
 \qquad(t\in\mathbb R,\ x\in G),
\]
while in the circle case,
\[
 F(t+x+\mathbb Z)=e(t)F(x+\mathbb Z)
 \qquad(t,x\in\mathbb R).
\]
Moreover,
\begin{equation*}
 \left|\frac1N\sum_{n=1}^N f(n)F(g(n)\Gamma)\right|\ge\eta,
\end{equation*}
and, moreover,
\begin{equation*}
 d\ll\ell^{43},
 \qquad \log(2M/\eta)\ll\ell^{87}.
\end{equation*}
\end{corollary}

\begin{proof}
We divide the proof into three steps.

\medskip
\noindent\textbf{Step 1. Embed the interval into a cyclic group.}
By Bertrand's postulate choose a prime $L$ with $8N<L<16N$.  Let
$A:=\mathbb Z/L\mathbb Z$ and define the zero extension $f_0:A\to\mathbb C$ by
\[
 f_0([n]_L):=
 \begin{cases}
 f(n),&n\in[N],\\
 0,&n\in\{0,\ldots,L-1\}\setminus[N].
 \end{cases}
\]
The residue classes $[1]_L,\ldots,[N]_L$ are distinct.  By the
interval-embedding identity \eqref{tr:embedding-identity},
\[
 \|f_0\|_{U^3(A)}
 =\|f\|_{U^3[N]}
  \|\mathbf1_{[N]}\|_{U^3(A)}.
\]
There is an absolute $c_0>0$ with
$\|\mathbf1_{[N]}\|_{U^3(A)}\ge c_0$.  Indeed, if $N\ge64$, take
\[
 x\in[3N/8,5N/8]\cap\mathbb Z,
 \qquad h_1,h_2,h_3\in[-\lfloor N/32\rfloor,\lfloor N/32\rfloor]\cap\mathbb Z.
\]
Every one of the eight vertices $x+\omega\cdot h$ lies in $[N]$, so there
are $\gg N^4$ supported cubes.  Since $L<16N$, their proportion among
all $L^4$ cyclic cubes is bounded below by an absolute constant.  If
$N<64$, the $N$ constant cubes already give the same conclusion after
reducing $c_0$.  Hence
\[
 \|f_0\|_{U^3(A)}\ge c_0\epsilon.
\]
Write $\epsilon_0:=c_0\epsilon$ and
$\ell_0:=2+\log(2/\epsilon_0)\ll\ell$.

\medskip
\noindent\textbf{Step 2. Extract a locally quadratic phase.}
Apply \Cref{inv:bohr-rank} to $u=f_0$ with parameter $\epsilon_0$.
It gives a nonempty $S\subseteq\widehat A$, with
$r:=|S|\ll\ell_0^{43}$, a regular radius $0<\rho<1/2000$, a point
$x_0\in A$, a locally quadratic function
$\phi:B_A(S,1000\rho)\to\mathbb T$, and $\tau_0>0$ such that
\[
 \left|\mathbb E_{h\in B_A(S,\rho)}
 f_0(x_0+h)e(-\phi(h))\right|\ge\tau_0,
\]
with $\log(2/\rho),\log(2/\tau_0)\ll\ell_0$.  Replace
$\tau_0$ by $\tau:=\min\{\tau_0,1/4\}$; the correlation remains true and
$\log(2/\tau)\ll\ell$.

\medskip
\noindent\textbf{Step 3. Realize the phase and transfer the parameters.}
Apply \Cref{inv:realization} to $u=f_0$ and the above data.  It gives
$g:\mathbb Z\to G$, $F:G/\Gamma\to\mathbb C$, and $M\ge2$ with the
properties in the statement, together with
\[
 \left|\frac1L\sum_{n=1}^L
 f_0([n]_L)F(g(n)\Gamma)\right|\ge\eta_0,
 \qquad \eta_0\ge\exp(-Cr\Lambda),
\]
where
\[
 \Lambda:=r\log(2r)+\log(2/\rho)+\log(2/\tau),
 \qquad \log(2M)\le C\Lambda.
\]
In the circle case $d=1$, and in the Heisenberg case $d=2r+1$; hence
$d\ll\ell^{43}$.  Because $f_0([n]_L)=0$ for $N<n<L$, the cyclic average is
exactly
\[
 \left|\frac1L\sum_{n=1}^N f(n)F(g(n)\Gamma)\right|\ge\eta_0.
\]
Since $L/N>1$, replacing $L^{-1}$ by $N^{-1}$ only increases its magnitude.
Set $\eta:=\min\{\eta_0,1/4\}$ to ensure $0<\eta<1/2$.
Finally,
\[
 r\log(2r)\ll\ell^{44},\qquad \Lambda\ll\ell^{44},
 \qquad r\Lambda\ll\ell^{87},
\]
so
\[
 \log(2M/\eta)\ll\Lambda+r\Lambda\ll\ell^{87}.
\]
This proves the corollary.
\end{proof}
\subsection{Fourier and bracket estimates}

We first approximate Lipschitz functions on a torus by finite Fourier sums,
with explicit dependence on the dimension.

\begin{lemma}
\label{inv:fejer}
Let $m\ge 1$ be an integer, let $L_{\rm Lip}\in[1,\infty)$ be a real
number, and let $\tau\in\mathbb R$ satisfy $0<\tau<1/2$.  Give
$\mathbb T^m:=\mathbb R^m/\mathbb Z^m$ its Euclidean quotient metric.
If $F:\mathbb T^m\to\mathbb C$ satisfies $|F(x)|\le1$ and
$\|F\|_{\rm Lip}\le L_{\rm Lip}$, then there are an integer $K\ge2$ and
coefficients $c_k$ such that
\begin{gather*}
 F_K(x):=\sum_{k\in[-K,K]^m\cap\mathbb Z^m}c_ke(k\cdot x),
 \qquad |c_k|\le1,\\
 \|F-F_K\|_\infty\le\tau,
 \qquad K\le C\left(\frac{mL_{\rm Lip}}{\tau}\right)^4.
\end{gather*}
In particular,
\[
 \sum_{k\in[-K,K]^m\cap\mathbb Z^m}|c_k|\le(2K+1)^m.
\]
\end{lemma}

\begin{proof}
We divide the proof into four steps.

\medskip
\noindent\textbf{Step 1. Establish the one-dimensional Fej'er-kernel bounds.}
For an integer $K\ge2$, define
\[
 \mathcal F_K(t):=\frac1{K+1}\left|\sum_{j=0}^K e(jt)\right|^2,
 \qquad t\in\mathbb T.
\]
Orthogonality gives $\int_{\mathbb T}\mathcal F_K(t)\,dt=1$.  The
geometric-sum formula and $|\sin(\pi t)|\gg\|t\|_{\mathbb R/\mathbb Z}$ give
\[
 0\le\mathcal F_K(t)\le K+1,
 \qquad
 \mathcal F_K(t)\le
 \frac{C}{(K+1)\|t\|_{\mathbb R/\mathbb Z}^2}.
\]
Writing $d(t):=\|t\|_{\mathbb R/\mathbb Z}$ and splitting at
$d(t)=1/(K+1)$ gives
\begin{align*}
 \mu_K:=\int_{\mathbb T}d(t)\mathcal F_K(t)\,dt
 \le2(K+1)\int_0^{1/(K+1)}t\,dt
   +\frac{C}{K+1}\int_{1/(K+1)}^{1/2}\frac{dt}{t}
 \le C\frac{\log(2K)}{K+1}.
\end{align*}

\medskip
\noindent\textbf{Step 2. Smooth on the torus and estimate the approximation error.}
Let
\[
 \mathcal K_K(y):=\prod_{j=1}^m\mathcal F_K(y_j),
 \qquad y=(y_1,\ldots,y_m)\in\mathbb T^m,
\]
and define
\[
 F_K(x):=\int_{\mathbb T^m}F(x-y)\mathcal K_K(y)\,dy.
\]
The quotient metric is
\[
 d_{\mathbb T^m}(y,0)
 =\left(\sum_{j=1}^m\|y_j\|_{\mathbb R/\mathbb Z}^2\right)^{1/2}
 \le\sum_{j=1}^m\|y_j\|_{\mathbb R/\mathbb Z}.
\]
Since $\mathcal K_K$ is a product probability kernel,
\begin{align*}
 |F_K(x)-F(x)|
 \le L_{\rm Lip}\int d_{\mathbb T^m}(y,0)\mathcal K_K(y)\,dy
 \le L_{\rm Lip}\,m\mu_K
 \le C\frac{mL_{\rm Lip}\log(2K)}{K+1}.
\end{align*}

\medskip
\noindent\textbf{Step 3. Choose the truncation scale.}
Put
\[
 A:=\frac{mL_{\rm Lip}}{\tau}\ge2,
 \qquad K:=\left\lceil C_0A^2\log(2A)\right\rceil,
\]
where $C_0$ is a sufficiently large absolute constant.  Then
$K\ge2$, $\log(2K)\ll\log(2A)$, and the preceding estimate gives
$\|F-F_K\|_\infty\le\tau$.  Since $A\ge2$ and $\log(2A)\le A$,
\[
 K\le C A^4=C\left(\frac{mL_{\rm Lip}}{\tau}\right)^4.
\]
\medskip
\noindent\textbf{Step 4. Identify the Fourier coefficients and finish.}
The one-dimensional Fourier expansion
\[
 \mathcal F_K(t)=\sum_{|j|\le K}
 \left(1-\frac{|j|}{K+1}\right)e(jt)
\]
shows that
\[
 c_k=\widehat F(k)\prod_{j=1}^m
 \left(1-\frac{|k_j|}{K+1}\right),
 \qquad
 \widehat F(k):=\int_{\mathbb T^m}F(x)e(-k\cdot x)\,dx.
\]
Thus $|c_k|\le|\widehat F(k)|\le1$, and there are $(2K+1)^m$
possible frequencies.  Therefore
\[
 \sum_k|c_k|\le(2K+1)^m.
\]
This proves the lemma.
\end{proof}
The next lemma is a one-parameter bracket estimate.  It converts the
assumption that a bracket polynomial is close to an integer on a positive
density of $h$ into two complementary systems of integer relations.  The
notation must distinguish centered representatives from distance to the
nearest integer: for $t\in\mathbb R$ put
\[
 \|t\|_{\mathbb R/\mathbb Z}:=\min_{q\in\mathbb Z}|t-q|,
 \qquad \{t\}\in[-1/2,1/2)\ \text{with }t-\{t\}\in\mathbb Z,
\]
and for $z=(z_1,\ldots,z_m)\in\mathbb R^m$ put
$\{z\}:=(\{z_1\},\ldots,\{z_m\})$.  Thus
$a\cdot\{h\alpha\}$ is an ordinary real number.

\begin{lemma}
\label{inv:bracket}
Let $m,N\in\mathbb Z_{\ge1}$ with $N\ge2$, let
$0<\rho<1/10$ and $M\ge2$ be real numbers, and let $K\ge1$ be real.
Let $a,\alpha\in\mathbb R^m$ and $\beta,\gamma\in\mathbb R$, with
$|a|\le M$.  Suppose that
\begin{equation*}
 \left\|\beta+\gamma h+a\cdot\{h\alpha\}\right\|_{\mathbb R/\mathbb Z}
       \le K/N
\end{equation*}
for at least $\rho N$ integers $h\in[N]$.  If
\begin{equation*}
 N\ge\left(\frac{mM}{\rho}\right)^{C(m+1)^2}K,
\end{equation*}
then there are an integer $t$ with $0\le t\le m$ and two families
\[
 y_1,\ldots,y_t\in\mathbb Z^m,\qquad
 x_1,\ldots,x_{m-t}\in\mathbb Z^m.
\]
The vectors $y_1,\ldots,y_t$ are linearly independent over $\mathbb R$.
The vectors $x_1,\ldots,x_{m-t}$ are also linearly independent over
$\mathbb R$. They satisfy
\begin{align}
 \operatorname{span}_{\mathbb R}(y_1,\ldots,y_t)
   &=\bigcap_{j=1}^{m-t}x_j^\perp,
 \\
 |y_i|,|x_j|&\le\left(\frac{mM}{\rho}\right)^{C(m+1)^2},
 \label{inv:eq:bracket-heights}\\
 \|y_i\cdot a\|_{\mathbb R/\mathbb Z},\ \|x_j\cdot\alpha\|_{\mathbb R/\mathbb Z}
 &\le\left(\frac{mM}{\rho}\right)^{C(m+1)^2}K/N.
 \label{inv:eq:bracket-output}
\end{align}
Here the span of an empty list is $\{0\}$ and the intersection of an
empty family of hyperplanes is $\mathbb R^m$.
\end{lemma}

\begin{proof}
We divide the proof into four steps.	
	
\par
\medskip
\noindent\textbf{Step 1. Apply the external lemma with one fixed bound.}
In this proof, the constant $C$ in the statement is chosen once and
then kept fixed. We use Leng's one-parameter refined
bracket-polynomial lemma~\cite[Lemma~3.1]{LengEquid}.
Its quantitative form is as follows. Suppose that $D,N\ge1$ are
integers with $N\ge2$, that $0<\delta<1/10$, $M'\ge2$, $K'\ge2$,
and that $b,\theta\in\mathbb R^D$, $\beta_0\in\mathbb R$, with
$|b|\le M'$. If
\begin{equation}
 \|\beta_0+b\cdot\{h\theta\}\|_{\mathbb R/\mathbb Z}<K'/N
 \quad\text{for at least $\delta N$ integers $h\in[N]$},
 \label{inv:eq:lem76-external-input}
\end{equation}
then either
\[
 N\ll(D^DM'/\delta)^{O(1)}
 \quad\text{or}\quad K'/N>1/100,
\]
or there are linearly independent integer families
$w_1,\ldots,w_s$ and $\eta_1,\ldots,\eta_{D-s}$ such that
\begin{align*}
 \operatorname{span}_{\mathbb R}(w_1,\ldots,w_s)
   &=\bigcap_{j=1}^{D-s}\eta_j^\perp,
 \\
 |w_i|,\ |\eta_j|&\le(D^DM'/\delta)^{O(1)},
 \\
 |w_i\cdot b|,\ \|\eta_j\cdot\theta\|_{\mathbb R/\mathbb Z}
   &\le (M'/\delta)^{O(1)}D^{O(D)}K'/N.
\end{align*}

Set
\begin{equation*}
 D:=m+1,\qquad \widetilde a:=(a,1),\qquad
 \widetilde\alpha:=(\alpha,\gamma),\qquad K_*:=2K.
\end{equation*}
Then $|\widetilde a|\le2M$, and
\begin{equation*}
 \widetilde a\cdot\{h\widetilde\alpha\}
 =a\cdot\{h\alpha\}+\{h\gamma\}
 \equiv a\cdot\{h\alpha\}+\gamma h\pmod1.
\end{equation*}
The hypothesis therefore gives \eqref{inv:eq:lem76-external-input}
with $b=\widetilde a$, $\theta=\widetilde\alpha$,
$\beta_0=\beta$, $\delta=\rho$, $M'=2M$, and $K'=K_*$.

Define the single auxiliary bound
\begin{equation*}
 R:=\frac{mM}{\rho},\qquad T:=R^{C(m+1)^2/4}.
\end{equation*}
Choose $C\ge1$ large enough that $T$ dominates the exceptional
threshold and every height or error factor in the Leng's lemma.
This choice is uniform in all the parameters. Indeed, $R>20m$,
$D=m+1\le2m<R/10$, and
\begin{equation*}
 \frac{2D^DM}{\rho}
 =\frac{2R}{m}D^D
 \le 2R(R/10)^D\le R^{D+1}.
\end{equation*}
Consequently all these upper bounds are $R^{O(D)}$ and are
at most $R^{CD^2/4}=T$ for one sufficiently large absolute $C$.
Since $T\ge R>20$, our hypothesis now reads
\begin{equation*}
 N\ge T^4K>T,\qquad \frac{K_*}{N}\le\frac2{T^4}<\frac1{100}.
\end{equation*}
Neither exceptional alternative can occur. We obtain the above
orthogonal families in $\mathbb Z^D$ with
\begin{equation}
 |w_i|,\ |\eta_j|\le T,\qquad
 |w_i\cdot\widetilde a|,\
 \|\eta_j\cdot\widetilde\alpha\|_{\mathbb R/\mathbb Z}
 \le TK_*/N.
 \label{inv:eq:lem76-leng}
\end{equation}

\par
\medskip
\noindent\textbf{Step 2. Project the $w$-relations.}
Write $w_i=(w_i',q_i)$ with $w_i'\in\mathbb Z^m$ and $q_i\in\mathbb Z$.
Let $Y:=\operatorname{span}_{\mathbb R}\{w_1',\ldots,w_s'\}$ and
$t:=\dim Y$.  Choose a maximal linearly independent sublist of the
$w_i'$ and call it $y_1,\ldots,y_t$.  Then $|y_i|\le T$ and the $y_i$
span $Y$.  If $y_i=w_{\nu(i)}'$, then the last coordinate $q_{\nu(i)}$
is an integer, so
\begin{equation*}
 \|y_i\cdot a\|_{\mathbb R/\mathbb Z}
 =\|w_{\nu(i)}\cdot\widetilde a-q_{\nu(i)}\|_{\mathbb R/\mathbb Z}
 \le |w_{\nu(i)}\cdot\widetilde a|\le TK_*/N.
\end{equation*}

\par
\medskip
\noindent\textbf{Step 3. Construct an integer basis of $Y^\perp$.}
Put $p:=D-s$ and
\begin{equation*}
 W:=\operatorname{span}_{\mathbb R}(w_1,\ldots,w_s),\qquad
 E:=\operatorname{span}_{\mathbb R}(\eta_1,\ldots,\eta_p)=W^\perp,
 \qquad U:=Y^\perp.
\end{equation*}
Here a basis of $U$ means a basis over $\mathbb R$ whose vectors have
integer coordinates. We construct it directly from the $\eta_j$.

\smallskip
\noindent\emph{3a. Identify the vectors whose last coordinate is zero.}
For $u\in\mathbb R^m$, the identity
$(u,0)\cdot w_i=u\cdot w_i'$ gives
\begin{equation}
 (u,0)\in E\iff u\cdot w_i'=0\quad(1\le i\le s)
 \iff u\in Y^\perp=U.
 \label{inv:eq:lem76-horizontal-intersection}
\end{equation}
Thus deleting the last coordinate identifies
$E\cap(\mathbb R^m\times\{0\})$ with $U$.
Write
\begin{equation*}
 \eta_j=(v_j,b_j),\qquad v_j\in\mathbb Z^m,\quad b_j\in\mathbb Z
 \quad(1\le j\le p).
\end{equation*}

\smallskip
\noindent\emph{3b. Choose a basis of that intersection.}
If $b_j=0$ for every $j$, all the $\eta_j$ already lie in
$\mathbb R^m\times\{0\}$. Therefore $v_1,\ldots,v_p$ form a basis of
$U$. In this case $p=\dim U=m-t$, and we set
\begin{equation}
 x_j:=v_j,\qquad (x_j,0)=\eta_j\quad(1\le j\le p).
 \label{inv:eq:lem76-zero-last-coordinate}
\end{equation}
This also covers $p=0$, when $U=\{0\}$ and the $x$-family is empty.

Otherwise, reorder the $\eta_j$ so that $b_1\ne0$. For $2\le j\le p$,
set
\begin{equation}
 x_{j-1}:=b_1v_j-b_jv_1\in\mathbb Z^m,\qquad
 (x_{j-1},0)=b_1\eta_j-b_j\eta_1\in E.
 \label{inv:eq:lem76-coordinate-elimination}
\end{equation}
By \eqref{inv:eq:lem76-horizontal-intersection}, every $x_{j-1}$ lies
in $U$. To prove independence, suppose that
$\sum_{j=2}^p d_jx_{j-1}=0$ with $d_j\in\mathbb R$. Then
\begin{equation*}
 \sum_{j=2}^p b_1d_j\eta_j
 -\left(\sum_{j=2}^p b_jd_j\right)\eta_1=0.
\end{equation*}
The $\eta_j$ are linearly independent and $b_1\ne0$, so every $d_j=0$.

To prove spanning, take $u\in U$. By
\eqref{inv:eq:lem76-horizontal-intersection}, there are real numbers
$d_1,\ldots,d_p$ such that
\begin{equation*}
 (u,0)=\sum_{j=1}^p d_j\eta_j,\qquad
 \sum_{j=1}^p b_jd_j=0,\qquad
 d_1=-\frac1{b_1}\sum_{j=2}^p b_jd_j.
\end{equation*}
Substitution yields
\begin{equation*}
 (u,0)=\sum_{j=2}^p\frac{d_j}{b_1}
                  (b_1\eta_j-b_j\eta_1),\qquad
 u=\sum_{j=2}^p\frac{d_j}{b_1}x_{j-1}.
\end{equation*}
Hence these $p-1$ vectors form a basis of $U$, and $p-1=m-t$.
If $p=1$, the sum is empty and the same argument gives $U=\{0\}$.

\par
\medskip
\noindent\textbf{Step 4. Verify orthogonality and the quantitative bounds.}
In both cases, $x_1,\ldots,x_{m-t}$ form a basis of $U=Y^\perp$.
Consequently,
\begin{equation*}
 \bigcap_{j=1}^{m-t}x_j^\perp
 =U^\perp=Y=\operatorname{span}_{\mathbb R}(y_1,\ldots,y_t).
\end{equation*}
This identity also holds when $t=0$ or $t=m$, with the stated empty-family
conventions.

By \eqref{inv:eq:lem76-leng}, $|v_j|,|b_j|\le T$ and $T\ge1$.
In the first case of Step~3b, \eqref{inv:eq:lem76-zero-last-coordinate}
gives
\begin{equation*}
 |x_j|\le T,\qquad
 \|x_j\cdot\alpha\|_{\mathbb R/\mathbb Z}
 =\|\eta_j\cdot\widetilde\alpha\|_{\mathbb R/\mathbb Z}
 \le TK_*/N.
\end{equation*}
In the second case, \eqref{inv:eq:lem76-coordinate-elimination} gives
\begin{equation*}
 |x_{j-1}|\le |b_1||v_j|+|b_j||v_1|\le2T^2.
\end{equation*}
Since $b_1,b_j$ are integers, the triangle inequality on
$\mathbb R/\mathbb Z$ gives
\begin{equation*}
 \begin{aligned}
 \|x_{j-1}\cdot\alpha\|_{\mathbb R/\mathbb Z}
 =\|b_1\eta_j\cdot\widetilde\alpha
       -b_j\eta_1\cdot\widetilde\alpha\|_{\mathbb R/\mathbb Z}
 \le |b_1|\|\eta_j\cdot\widetilde\alpha\|_{\mathbb R/\mathbb Z}
       +|b_j|\|\eta_1\cdot\widetilde\alpha\|_{\mathbb R/\mathbb Z}
 \le2T^2K_*/N.
 \end{aligned}
\end{equation*}
Together with Step~2 and $K_*=2K$, these estimates imply, for all
applicable indices,
\begin{equation}
 |y_i|,|x_j|\le2T^2,\qquad
 \|y_i\cdot a\|_{\mathbb R/\mathbb Z},\quad
 \|x_j\cdot\alpha\|_{\mathbb R/\mathbb Z}\le4T^2K/N.
 \label{inv:eq:lem76-combined-bounds}
\end{equation}
Since $T\ge2$, the same bound $T$ chosen in Step~1 gives
\begin{equation*}
 2T^2\le T^4,\qquad 4T^2\le T^4
   =R^{C(m+1)^2}.
\end{equation*}
Thus \eqref{inv:eq:lem76-combined-bounds} gives exactly
\eqref{inv:eq:bracket-heights} and \eqref{inv:eq:bracket-output}.
\end{proof}

\subsection{Horizontal relations}

For the next theorem, fix $q\in\{1,2\}$ and integers $r_\nu\ge1$, and let
\[
 G:=\prod_{\nu=1}^{q}\mathcal H(r_\nu)
\]
with the explicit lattice
\[
 \Gamma:=\prod_{\nu=1}^{q}
 \bigl(\mathbb Z^{r_\nu}\times\mathbb Z^{r_\nu}\times\mathbb Z\bigr).
\]
Set $G_1:=G$, $G_2:=[G,G]$, and $G_3:=\{1\}$. In the present
coordinates,
\[
 G_2=\{z(c):c\in\mathbb R^q\},
 \qquad [G,G_2]=\{1\};
\]
these are the only subgroup levels used below. Put $d:=\dim G$. Write
\begin{equation*}
 z(c):=((0,0,c_1),\ldots,(0,0,c_q)),\qquad
 c=(c_1,\ldots,c_q)\in\mathbb R^q
\end{equation*}
for a central element. Fix $M\ge2$ and an integer vector
$\ell=(\ell_1,\ldots,\ell_q)\in\mathbb Z^q$ with
$\max_\nu|\ell_\nu|\le M$, and define
\begin{equation*}
 \xi:G_2\longrightarrow\mathbb R,\qquad
 \xi(z(c)):=\ell\cdot c.
\end{equation*}
Thus $\xi$ is integer-valued on $\Gamma\cap G_2$. Let
$F:G/\Gamma\to\mathbb C$ satisfy $|F|\le1$ and
$\|F\|_{\rm Lip}\le M$, and assume that
\begin{equation}
 F(z(c)x\Gamma)=e(\ell\cdot c)F(x\Gamma)
 \qquad(c\in\mathbb R^q,\ x\in G).
\label{inv:eq:vertical-frequency-assumption}
\end{equation}
The displayed identity is the exact assumption used below: $z(c)\in G_2$
acts on $F$ by the scalar $e(\xi(z(c)))$.  The coordinate height of
$\xi$ is $\max_\nu|\ell_\nu|$.  Since $G$ is two-step nilpotent, a
polynomial sequence for $G_1\supseteq G_2\supseteq G_3$ is exactly
the representation
\[
 g(n)=g_0u^nv^{\binom n2},
 \qquad g_0,u\in G,\quad v\in G_2.
\]

\begin{theorem}
\label{inv:vertical-kernel}
With this setup, let $N\ge2$ be an integer and suppose
\begin{equation*}
 \left|N^{-1}\sum_{n=1}^N F(g(n)\Gamma)\right|\ge\delta,
 \qquad 0<\delta<1/10,\qquad Z:=10Md/\delta.
\end{equation*}
If $N\ge Z^{C(d+1)^4}$, there are an integer $s$ with
$0\le s\le2d$ and horizontal characters
$\lambda_1,\ldots,\lambda_s:G\to\mathbb R$ which are integer-valued
on $\Gamma$ and vanish on $G_2$. Put
\[
 W:=\bigcap_{i=1}^{s}\ker\lambda_i,
 \qquad \ker\lambda_i:=\{w\in G:\lambda_i(w)=0\text{ in }\mathbb R\}.
\]
Then
\begin{equation*}
 |\lambda_i|,\ \|\lambda_i\circ g\|_{C^\infty[N]}
 \le Z^{C(d+1)^4},
 \qquad
 \xi([W,W])=0.
\end{equation*}
Here $[W,W]$ denotes the subgroup generated by all commutators
$[w_1,w_2]=w_1w_2w_1^{-1}w_2^{-1}$ with $w_1,w_2\in W$.

\end{theorem}

\begin{proof}
We use the convention $[x,y]=xyx^{-1}y^{-1}$ throughout the proof.
Put $$m:=\dim(G/G_2)=2\sum_{\nu=1}^{q}r_\nu$$ and identify
$G/G_2$ with $\mathbb R^m$ by the horizontal coordinates. Let
$\pi:G\to\mathbb R^m$ be the quotient homomorphism. We divide the rest proof into seven steps.

\par
\medskip
\noindent\textbf{Step 1. Choose bounded representatives and normalize the polynomial sequence.}
We first make the bounded representatives used in the proof explicit.
For one Heisenberg factor, write $x=(\theta,\lambda,c)$ with
$\theta,\lambda\in\mathbb R^r$ and $c\in\mathbb R$. Choose
$p,q\in\mathbb Z^r$ so that
$\theta_0:=\theta-p$ and $\lambda_0:=\lambda-q$ belong to
$[-1/2,1/2)^r$, and then choose $k\in\mathbb Z$ so that
$c_0:=c-\lambda_0\cdot p-k$ belongs to $[-1/2,1/2)$. The group law
then gives
\begin{equation*}
 x=(\theta_0,\lambda_0,c_0)(p,q,k).
\end{equation*}
For the product $G$, let $\mathcal B$ be the product of these centered half-open boxes and apply this construction in every factor. Thus
every $x\in G$ can be written as $x=b\gamma$ with $\gamma\in\Gamma$
and with $b\in\mathcal B$. In particular, if $u^h=\zeta_h\gamma_h$, then
$\pi(\zeta_h)=\{h\pi(u)\}$, where $\{\,\cdot\,\}$ denotes the
representative in $[-1/2,1/2)^m$.

Write $g_0=b\gamma_0$ by this construction. Since $v$ is central,
\begin{equation}
 g_0u^nv^{\binom n2}
 =b(\gamma_0u\gamma_0^{-1})^nv^{\binom n2}\gamma_0.
\label{inv:constant-normalization}
\end{equation}
Define $\widetilde F(x\Gamma):=F(bx\Gamma)$ and replace $u$ by
$\gamma_0u\gamma_0^{-1}$. The final factor $\gamma_0\in\Gamma$ is
invisible on $G/\Gamma$, so this replacement leaves the nilsequence
unchanged. It remains to check that the translation by $b$ preserves a
controlled Lipschitz norm.

We use on $G$ the right-invariant metric associated with the Euclidean
norm in the displayed coordinates, and write $d_{G/\Gamma}$ for the
induced quotient metric. If
$x(t)=((\theta_\nu(t),\lambda_\nu(t),c_\nu(t)))_\nu$ is a differentiable
path, write $\operatorname{speed}(x(t))$ for its metric speed.\footnote{For
a differentiable path $x:I\to G$, the metric speed at $t$ is
$\displaystyle\lim_{s\to0}d_G(x(t+s),x(t))/|s|$, when this limit exists;
the length of the path is the integral of this speed.}
Differentiating $x(t+s)x(t)^{-1}$ at $s=0$ gives
\begin{equation}
 \operatorname{speed}(x(t))^2
 =\sum_{\nu=1}^{q}\left(
 \|\dot\theta_\nu\|_2^2+\|\dot\lambda_\nu\|_2^2+
 |\dot c_\nu-\dot\lambda_\nu\cdot\theta_\nu|^2\right).
\label{inv:eq:speed-formula}
\end{equation}
If $b_\nu=(p_\nu,q_\nu,r_\nu)$, applying the same formula to the
path $b x(t)$ changes the last term in factor $\nu$ to
\[
 \dot c_\nu-\dot\lambda_\nu\cdot\theta_\nu
 +q_\nu\cdot\dot\theta_\nu-p_\nu\cdot\dot\lambda_\nu.
\]
Put $E_\nu:=q_\nu\cdot\dot\theta_\nu-p_\nu\cdot\dot\lambda_\nu$. Since each coordinate of $p_\nu$ and $q_\nu$ has absolute value at most $1/2$, Cauchy--Schwarz gives
\begin{equation*}
 \left(\sum_{\nu=1}^{q}|E_\nu|^2\right)^{1/2}
 \le \sqrt d\left(\sum_{\nu=1}^{q}
 (\|\dot\theta_\nu\|_2^2+\|\dot\lambda_\nu\|_2^2)\right)^{1/2}
 \le \sqrt d\,\operatorname{speed}(x(t)).
\end{equation*}
The Euclidean triangle inequality applied to the central coordinates now gives
\begin{equation}
 \operatorname{speed}(b x(t))
 \le(1+\sqrt d)\operatorname{speed}(x(t)).
\label{inv:eq:left-translation-speed}
\end{equation}
The quotient distance is the infimum of the lengths of lifted paths,
with lattice translates allowed at the endpoints. Applying
\eqref{inv:eq:left-translation-speed} to every such path and then taking
the two infima gives
\begin{equation}
 d_{G/\Gamma}(bx\Gamma,by\Gamma)
 \le(1+\sqrt d)d_{G/\Gamma}(x\Gamma,y\Gamma),
 \qquad
 \|\widetilde F\|_{\rm Lip}\le M(1+\sqrt d).
\label{inv:eq:translated-lipschitz}
\end{equation}
We now rename $\widetilde F$ as $F$.

Next normalize the linear coefficient. Write $u=u_0\gamma$ with
$\gamma\in\Gamma$ and with $u_0\in\mathcal B$.
Since commutators are central,
\begin{equation}
 (u_0\gamma)^n
 =u_0^n\gamma^n[\gamma,u_0]^{\binom n2}.
\label{inv:eq:normalize-linear}
\end{equation}
The commutator factor is central, so it can be moved past $\gamma^n$;
the factor $\gamma^n$ then disappears on the quotient, and the remaining
central factor is absorbed into $v$. We may therefore replace $u$ by
$u_0$ and $v$ by $v[\gamma,u_0]$, obtaining
\begin{equation*}
 g(n)=u^nv^{\binom n2},\qquad
 \alpha:=\pi(u)\in[-1/2,1/2)^m,\qquad
 \omega:=\xi(v).
\end{equation*}
The real number $\omega$ is not subject to a quantitative bound.
For a horizontal character $\lambda$, let $P_{\rm old}$ and
$P_{\rm new}$ denote the scalar polynomials before and after the two
normalizations, viewed in $\mathbb R/\mathbb Z$. The calculation above
gives the exact identity
\begin{equation*}
 P_{\rm old}(n)-P_{\rm new}(n)
 =\lambda(g_0)+n\lambda(\gamma)\pmod 1,
 \qquad \lambda(\gamma)\in\mathbb Z.
\end{equation*}
Thus the second term is zero in $\mathbb R/\mathbb Z$, and only the
constant coefficient changes. Since the $C^\infty[N]$ norm ignores the
constant coefficient and measures the remaining coefficients modulo
$\mathbb Z$, $\|\lambda\circ g\|_{C^\infty[N]}$ is unchanged.

\par
\medskip
\noindent\textbf{Step 2. Find many large derivative sums.}
Set
\[
 A_n:=F(g(n)\Gamma),\qquad
 S_h:=\sum_{n=1}^{N-h}A_{n+h}\overline{A_n}
 \quad(1\le h<N).
\]
The hypothesis and $|A_n|\le1$ imply
\begin{equation*}
 \delta^2N^2
 \le\left|\sum_{n=1}^NA_n\right|^2
 \le N+2\sum_{h=1}^{N-1}|S_h|.
\end{equation*}
Since the theorem's lower bound on $N$ implies $N\ge4\delta^{-2}$,
the sum on the right gives
$$\sum_h|S_h|\ge3\delta^2N^2/8.$$ Shifts with
$|S_h|<\delta^2N/8$ contribute less than $\delta^2N^2/8$.
Consequently, for
\[
 \mathcal H:=\{h\in[N-1]: |S_h|\ge\delta^2N/8\},
\]
the remaining shifts contribute at least $\delta^2N^2/4$. Since
$|S_h|\le N-h\le N$, we obtain
\begin{equation*}
 |\mathcal H|\ge\delta^2N/4,\qquad
 |S_h|\ge\delta^2N/8,\qquad
 N-h\ge\delta^2N/8\quad(h\in\mathcal H).
\end{equation*}

\par
\medskip
\noindent\textbf{Step 3. Rewrite each derivative on a horizontal torus.}

\noindent\emph{3a. The torus function.}
For $h\in\mathcal H$, write
$u^h=\zeta_h\gamma_h$ as in Step 1 and define
\begin{equation*}
 H_h(y):=F(\zeta_hy\Gamma)\overline{F(y\Gamma)}.
\end{equation*}
Because $\Gamma$ is the quotient lattice,
$F(x\gamma\Gamma)=F(x\Gamma)$ for $\gamma\in\Gamma$; hence
$H_h$ is right-$\Gamma$-invariant. For $z\in G_2$,
\eqref{inv:eq:vertical-frequency-assumption} gives
\begin{equation}
 H_h(zy)
 =e(\xi(z))F(\zeta_hy\Gamma)\,
  e(-\xi(z))\overline{F(y\Gamma)}
 =H_h(y).
\label{inv:eq:equ-vertical-cancellation}
\end{equation}
If $\pi(y')-\pi(y)\in\mathbb Z^m$, choose
$\gamma\in\Gamma$ with this horizontal coordinate. Then
$y'=zy\gamma$ for some $z\in G_2$, so
\eqref{inv:eq:equ-vertical-cancellation} and lattice invariance
show that $H_h(y)$ depends only on $\pi(y)\bmod\mathbb Z^m$. Thus
there is a well-defined $\Phi_h:\mathbb T^m\to\mathbb C$ with
\begin{equation}
 \Phi_h(\pi(y)\bmod\mathbb Z^m)=H_h(y),\qquad |\Phi_h|\le1.
\label{inv:eq:equ-torus-function}
\end{equation}

\noindent\emph{3b. Lipschitz bound.}
Use the horizontal section
\[
 \sigma(t):=((\theta_\nu,\lambda_\nu,0))_{\nu=1}^{q},
 \qquad t=((\theta_\nu,\lambda_\nu))_\nu\in\mathbb R^m.
\]
Choose lifts $t,t'\in[-1,1]^m$ of two torus points for which
$\|t-t'\|_2$ equals their Euclidean quotient distance. Along
$t_s=t+s(t'-t)$, \eqref{inv:eq:speed-formula} gives
\begin{equation*}
 \operatorname{speed}(\sigma(t_s))^2
 \le(1+d)\|t'-t\|_2^2.
\end{equation*}
Hence \eqref{inv:eq:translated-lipschitz} and
\eqref{inv:eq:left-translation-speed} imply
\begin{equation*}
 \begin{aligned}
 d_{G/\Gamma}(\sigma(t)\Gamma,\sigma(t')\Gamma)
   &\le(1+\sqrt d)\|t-t'\|_2,\\
 d_{G/\Gamma}(\zeta_h\sigma(t)\Gamma,
              \zeta_h\sigma(t')\Gamma)
   &\le(1+\sqrt d)^2\|t-t'\|_2.
 \end{aligned}
\end{equation*}
The elementary inequality
$|z_1\overline z_2-w_1\overline w_2|
 \le|z_1-w_1|+|z_2-w_2|$ for $|z_i|,|w_i|\le1$, together with
\eqref{inv:eq:translated-lipschitz}, therefore gives
\begin{equation}
 \|\Phi_h\|_{\rm Lip}
 \le1+M(1+\sqrt d)^2(2+\sqrt d)
 \le L,\qquad L:=100M(d+1)^3.
\label{inv:eq:equ-torus-lipschitz}
\end{equation}
The periodicity in \eqref{inv:eq:equ-torus-function} permits the same segment argument across
the boundaries of the coordinate cube.

\noindent\emph{3c. The bilinear phase.}
For horizontal vectors $s,t$, define
\begin{equation*}
 B_\ell(s,t):=
 \sum_{\nu=1}^{q}\ell_\nu
 \bigl(s_{\lambda,\nu}\cdot t_{\theta,\nu}
       -t_{\lambda,\nu}\cdot s_{\theta,\nu}\bigr).
\end{equation*}
The group law gives
\begin{equation}
 \xi([x,y])=B_\ell(\pi(x),\pi(y)).
\label{inv:eq:commutator-form}
\end{equation}
Define $a\in\mathbb R^m$ by $a\cdot t=B_\ell(\alpha,t)$. Its
coordinates are
$(\ell_\nu\alpha_{\lambda,\nu},-\ell_\nu\alpha_{\theta,\nu})_\nu$,
so, in the maximum norm used for character heights,
$|a|\le M/2\le L$.

For $h\in\mathcal H$, the identity
$$\binom{n+h}{2}-\binom n2=hn+\binom h2$$ and \eqref{inv:eq:normalize-linear} give
\begin{equation*}
 \begin{aligned}
 \Delta_{h,n}
 :=\bigl(\zeta_h^{-1}g(n+h)\gamma_h^{-1}\bigr)g(n)^{-1}
 =[\gamma_h,u]^n v^{hn+\binom h2}
  =[u,\zeta_h]^n v^{hn+\binom h2}\in G_2.
 \end{aligned}
\end{equation*}
Equivalently,
\[
 g(n+h)\Gamma=\Delta_{h,n}\,\zeta_hg(n)\Gamma.
\]
Applying \eqref{inv:eq:vertical-frequency-assumption}, then multiplying by $\overline{F(g(n)\Gamma)}$,
gives
\begin{equation}
 A_{n+h}\overline{A_n}
 =e(\xi(\Delta_{h,n}))\,H_h(g(n)).
\label{inv:eq:central-displacement}
\end{equation}
By \eqref{inv:eq:commutator-form}--\eqref{inv:eq:central-displacement},
 \begin{equation*}
 \xi(\Delta_{h,n})
 =n\,a\cdot\{h\alpha\}
   +\left(hn+\binom h2\right)\omega.
\end{equation*}
Since $\pi(g(n))=n\alpha$, \eqref{inv:eq:equ-torus-function} now yields the exact formula
\begin{equation}
 A_{n+h}\overline{A_n}
 =e\left(\binom h2\omega\right)\Phi_h(n\alpha)
  e\left(n\bigl(a\cdot\{h\alpha\}+h\omega\bigr)\right).
\label{inv:eq:equ-derivative-phase}
\end{equation}

\par
\medskip
\noindent\textbf{Step 4. Obtain one bracket relation for many shifts.}
Apply \Cref{inv:fejer} with
$\tau:=\delta^2/32$ and the common Lipschitz bound $L$ in \eqref{inv:eq:equ-torus-lipschitz}.
The resulting Fourier frequency bound $K_0$ is an integer independent of $h$. For each
$h\in\mathcal H$,
\begin{equation*}
 P_h(t):=\sum_{k\in[-K_0,K_0]^m\cap\mathbb Z^m}
 c_{h,k}e(k\cdot t),\qquad
 |c_{h,k}|\le1,\qquad
 \|\Phi_h-P_h\|_\infty\le\tau,
\end{equation*}
where $K_0\ll(mL/\tau)^4$. Put $J:=(2K_0+1)^m$.
Replacing $\Phi_h$ in \eqref{inv:eq:equ-derivative-phase} changes the sum by at most
$\tau(N-h)\le\delta^2N/32$. Therefore
\begin{equation*}
 \left|\sum_k c_{h,k}\sum_{n=1}^{N-h}
 e\bigl(n\vartheta_{h,k}\bigr)\right|
 \ge\frac{3\delta^2N}{32},
 \qquad
 \vartheta_{h,k}:=k\cdot\alpha+a\cdot\{h\alpha\}+h\omega.
\end{equation*}
By the triangle inequality,
 \begin{equation*}
 \frac{3\delta^2N}{32}
 \le J\max_k\left|\sum_{n=1}^{N-h}
 e(n\vartheta_{h,k})\right|.
\end{equation*}
Thus, for some $k_h\in[-K_0,K_0]^m\cap\mathbb Z^m$,
\begin{equation*}
 \left|\sum_{n=1}^{N-h}e(n\vartheta_{h,k_h})\right|
 \ge\frac{3\delta^2N}{32J}.
\end{equation*}
The geometric-sum bound
 \begin{equation*}
 \left|\sum_{n=1}^{H}e(n\theta)\right|
 \le\frac1{2\|\theta\|_{\mathbb R/\mathbb Z}}
\end{equation*}
implies
 \begin{equation*}
 \|\vartheta_{h,k_h}\|_{\mathbb R/\mathbb Z}
 \le\frac{16J}{3\delta^2N}.
\end{equation*}
Set
\begin{equation*}
 \rho:=\frac{\delta^2}{8J},\qquad
 K_{\rm b}:=\rho^{-1}=\frac{8J}{\delta^2}.
\end{equation*}
Then
\begin{equation}
 \|k_h\cdot\alpha+a\cdot\{h\alpha\}+h\omega\|_{\mathbb R/\mathbb Z}
 \le\frac{K_{\rm b}}N.
\label{inv:eq:many-brackets}
\end{equation}
Partition $\mathcal H$ according to the value of $k_h$. There are at
most $J$ classes, so one class contains at least
$|\mathcal H|/J\ge2\rho N$ shifts. Fix its frequency $k$. Then
\eqref{inv:eq:many-brackets} holds with this same $k$ for at least $\rho N$ integers
$h\in[N]$.

\par
\medskip
\noindent\textbf{Step 5. Apply Lemma 7.6 and check its parameters.}
We apply \Cref{inv:bracket} with
\begin{equation*}
 \beta:=k\cdot\alpha,\qquad
 \gamma:=\omega,\qquad
 M_{\rm a}:=L,\qquad
 K:=K_{\rm b}=\rho^{-1}.
\end{equation*}
Let $C_0$ be a fixed absolute constant large enough for the constants
in Lemmas~\Cref{inv:fejer} and~\Cref{inv:bracket}. Since
$m\le Z$, $L\le Z^3$, and $\tau^{-1}=32\delta^{-2}\le Z^2$,
\Cref{inv:fejer} gives $K_0\ll Z^{24}$. Increasing $C_0$ once,
we may assume
\begin{equation*}
 L\le Z^{C_0},\qquad 2K_0+1\le Z^{C_0}.
\end{equation*}
Consequently,
\begin{equation*}
 J\le Z^{C_0d},\qquad
 \rho^{-1}=8J\delta^{-2}\le Z^{C_0(d+1)}.
\end{equation*}
Since $mL\le Z^4$,
\[
 \frac{mL}{\rho}\le Z^{(C_0+4)(d+1)}.
\]
Define
\begin{equation*}
 B:=\left(\frac{mL}{\rho}\right)^{C_0(m+1)^2}.
\end{equation*}
Then
\begin{equation*}
 B\le Z^{C_0(C_0+4)(d+1)^3},
 \qquad
 \frac{MB}{\rho}
 \le Z^{(C_0(C_0+4)+C_0+1)(d+1)^4}
\end{equation*}
Choose the theorem's absolute constant
$C\ge\max\{C_0(C_0+4)+C_0+1,3\}$ once and for all. Then the hypothesis
$N\ge Z^{C(d+1)^4}$ implies both $N\ge4\delta^{-2}$ and the sample-size inequality
$N\ge B/\rho$ required by \Cref{inv:bracket}.

\Cref{inv:bracket} now gives integer vectors
$y_1,\ldots,y_t$ and $x_1,\ldots,x_{m-t}$ such that
\begin{equation}
 \begin{gathered}
 Y:=\operatorname{span}_{\mathbb R}(y_1,\ldots,y_t)
   =\bigcap_{j=1}^{m-t}x_j^\perp,\qquad
 |x_j|,\ |y_i|\le B,\\
 \|x_j\cdot\alpha\|_{\mathbb R/\mathbb Z},\
 \|y_i\cdot a\|_{\mathbb R/\mathbb Z}
 \le \frac{B}{\rho N}.
 \end{gathered}
\label{inv:eq:horizontal-relations}
\end{equation}
Each family is linearly independent over $\mathbb R$.

\par
\medskip
\noindent\textbf{Step 6. Construct the horizontal characters.}
Define
\begin{equation*}
 \chi_j(z):=x_j\cdot\pi(z),\qquad
 \mu_i(z):=B_\ell(y_i,\pi(z)).
\end{equation*}
Both maps are homomorphisms because $\pi$ is a homomorphism.
They vanish on $G_2$ and are integer-valued on $\Gamma$ because
$x_j,y_i,\ell$ have integer coordinates. More explicitly,
\begin{equation*}
 \mu_i(z)
 =\sum_{\nu=1}^{q}\ell_\nu
 \bigl(y_{i,\lambda,\nu}\cdot\theta_\nu(z)
       -y_{i,\theta,\nu}\cdot\lambda_\nu(z)\bigr).
\end{equation*}
Hence
\begin{equation*}
 |\chi_j|\le B,\qquad |\mu_i|\le MB,\qquad
 \chi_j(u)=x_j\cdot\alpha,\qquad
 \mu_i(u)=-y_i\cdot a.
\end{equation*}
Since $v\in G_2$, $\lambda(g(n))=n\lambda(u)$ for every one of these
characters $\lambda$. Therefore
\begin{equation*}
 \|\lambda\circ g\|_{C^\infty[N]}
 =N\|\lambda(u)\|_{\mathbb R/\mathbb Z}
 \le \frac{B}{\rho}.
\end{equation*}
The normalization in Step 1 shows that the same estimate holds for
the original sequence. There are $m\le d\le2d$ characters in total.

\par
\medskip
\noindent\textbf{Step 7. Prove that the kernel is isotropic for $\xi$.}
Let $W$ be the common real kernel of all the characters in Step 6.
The equations $\chi_j(w)=0$ imply
$\pi(w)\in Y$ by \eqref{inv:eq:horizontal-relations}, while $\mu_i(w)=0$ for every $w\in W$.
For $w_1,w_2\in W$, write
$\pi(w_1)=\sum_{i=1}^{t}s_i y_i$ with $s_i\in\mathbb R$. Then
\begin{equation*}
 \begin{aligned}
 \xi([w_1,w_2])
 =B_\ell(\pi(w_1),\pi(w_2))
 =\sum_{i=1}^{t}s_i B_\ell(y_i,\pi(w_2))
 =\sum_{i=1}^{t}s_i\mu_i(w_2)=0.
 \end{aligned}
\end{equation*}
This is an equality in $\mathbb R$, because the kernels defining $W$
are real kernels. Since $[W,W]\subseteq G_2$ and $\xi$ is additive on
$G_2$, it vanishes on the subgroup generated by these commutators.
Thus $\xi([W,W])=0$.
\end{proof}

\subsection{Reduction to quadratic phases}

We record the quantitative forms of the two results used next.
\Cref{prop:bsz_dense} applies to $|a(n)|\le1$ and
a multiplicative function $f$ with $|f|\le1$. If
\begin{equation*}
 \left|N^{-1}\sum_{n\le N}f(n)a(n)\right|\ge\eta,
 \qquad N\ge\exp(C\eta^{-5}),\qquad 0<\eta<1/2,
\end{equation*}
then there are $2\le x\le\exp(C\eta^{-5})$
and at least $c\eta^{25}x^2$ ordered distinct prime pairs
$p,q\in[x,2x]$ for which
\begin{equation}
 \left|\frac1{N_{p,q}}\sum_{n\le N_{p,q}}
             a(pn)\overline{a(qn)}\right|\ge c\eta^3,
 \qquad N_{p,q}=\left\lfloor\frac N{\max(p,q)}\right\rfloor.
 \label{inv:dense-interface}
\end{equation}
\Cref{lem:linear_dense} says that if $H\ge2$ is an integer,
$x\ge32H$ is real, and $N$ is a positive integer satisfying
$N\ge10^6H^3x^3$, then any set of at least $x^2/H$ pairs in
$([x,2x]\cap\mathbb Z)^2$ satisfying
$\|pA+qB\|_{\mathbb R/\mathbb Z}\le Hx/N$ obeys
\begin{equation}
 \left|A-\frac aR\right|,
 \left|B-\frac bR\right|\le\frac{64H^2}{N},
 \qquad a,b\in\mathbb Z,\quad R\text{ a positive integer},\quad R\le2H.
 \label{inv:dense-linear-interface}
\end{equation}
We apply this lemma separately to each of a fixed finite list
of characters; the denominator may depend on the character.

The next two results use the following fixed model. Let $r\ge1$,
let $G=\mathcal H(r)$ with
$\Gamma=\Gamma(r)=\mathbb Z^r\times\mathbb Z^r\times\mathbb Z$, and
write
\[
 G_1:=G,\qquad G_2:=[G,G]=\{z(c):c\in\mathbb R\},\qquad
 G_3:=\{1\},\qquad d:=\dim G=2r+1,
\]
where $z(c)=(0,0,c)$. Thus $G$ is step two and $G_2$ is its one-dimensional
center. Let $\sigma\in\{-1,1\}$ and define
$\xi:G_2\to\mathbb R$ by $\xi(z(c)):=\sigma c$. Assume that
\[
 F(z(c)x\Gamma)=e(\sigma c)F(x\Gamma)
 \qquad(c\in\mathbb R,\ x\in G).
\]
This is the precise vertical-frequency condition used below. Let
$f:\mathbb N\to\mathbb S^1$ be completely multiplicative. For an integer
$N\ge1$, a degree-two polynomial sequence $g:\mathbb Z\to G$, and
$0<\eta<1/2$, assume
\begin{equation}
 \left|N^{-1}\sum_{n\le N}f(n)F(g(n)\Gamma)\right|\ge\eta,
 \qquad Z:=10dM/\eta .
\label{inv:eq:Heisenberg-input}
\end{equation}
Let
\[
 \mathcal B_r:=\{(\theta,\lambda,c):
   \theta,\lambda\in[-1/2,1/2)^r,\ c\in[-1/2,1/2)\}
\]
be the centered coordinate box, and assume below that $u\in\mathcal B_r$.

\begin{lemma}
\label{inv:dense-constraints}
Assume \eqref{inv:eq:Heisenberg-input}, with
\[
 g(n)=u^nv^{\binom n2},\qquad u\in\mathcal B_r,\qquad v\in G_2.
\]
If $N\ge\exp(Z^{C(d+1)^6})$, then there exist an integer $s$ with
$0\le s\le2r$ and homomorphisms
$\lambda_1,\ldots,\lambda_s:G\to\mathbb R$ such that each $\lambda_i$
vanishes on $G_2$, is integer-valued on $\Gamma$, and the family is
linearly independent over $\mathbb R$. Equivalently,
$\lambda_i(\theta,\lambda,c)=a_i\cdot\theta+b_i\cdot\lambda$ with
$a_i,b_i\in\mathbb Z^r$, and the vectors $(a_i,b_i)$ are linearly
independent over $\mathbb R$. Put
\[
 H:=\bigcap_{i=1}^{s}\ker_{\mathbb R}\lambda_i,\qquad
 \ker_{\mathbb R}\lambda_i:=\{w\in G:\lambda_i(w)=0\text{ in }\mathbb R\}.
\]
For $s=0$ we use the convention $H=G$.
There is an integer $J$ with
\begin{equation}
 Z\le J\le Z^{C(d+1)^6},\qquad
 |\lambda_i|\le J,\qquad
 \|\lambda_i(u)\|_{\mathbb R/\mathbb Z}\le J/N,\qquad
 H\supseteq G_2,\qquad
 \xi([H,H])=0 .
\label{inv:dense-constraints-output}
\end{equation}
Here $|\lambda_i|$ denotes the maximum absolute value of the integer
coefficients of $\lambda_i$ in the horizontal coordinates. The subgroup
$H$ is connected, normal, and abelian.
\end{lemma}

\begin{proof}
We divide the proof into six steps.	
	
\par
\medskip
\noindent\textbf{Step 1. Produce many correlated prime pairs.}
Set $a(n):=F(g(n)\Gamma)$. Applying
\eqref{inv:dense-interface} to $a$ gives a real number $x$ with
$2\le x\le\exp(Z^C)$ and a set $\mathcal P$ of at least
$c\eta^{25}x^2$ ordered distinct prime pairs $(p,q)\in[x,2x]^2$ such that
\begin{equation}
 \left|\frac1{N_{p,q}}\sum_{n\le N_{p,q}}
 F(g(pn)\Gamma)\overline{F(g(qn)\Gamma)}\right|
 \ge c\eta^3,\qquad
 N_{p,q}:=\left\lfloor\frac{N}{\max(p,q)}\right\rfloor .
\label{inv:eq:lem78-prime-pairs}
\end{equation}
Because $p,q\le2x$ and the assumed lower bound for $N$ is enormous,
$N_{p,q}\ge N/(4x)$ for all $(p,q)\in\mathcal P$.

\par
\medskip
\noindent\textbf{Step 2. Apply \Cref{inv:vertical-kernel} to each pair.}
Fix $(p,q)\in\mathcal P$. Work on
$G^\times:=G\times G$ with lattice
$\Gamma^\times:=\Gamma\times\Gamma$ and center
$G_2^\times:=G_2\times G_2$. Define
\[
 \mathcal F(x_1,x_2):=F(x_1\Gamma)\overline{F(x_2\Gamma)},
 \qquad
 g_{p,q}(n):=(g(pn),g(qn)).
\]
The identity
$$\binom{pn}{2}=p^2\binom n2+n\binom p2$$ gives
$$g(pn)=(u^pv^{\binom p2})^n(v^{p^2})^{\binom n2},$$
and the analogous formula holds for $g(qn)$. Thus $g_{p,q}$ is again
a degree-two polynomial sequence adapted to the step-two filtration.
Then the left side of \eqref{inv:eq:lem78-prime-pairs} is the average
of $\mathcal F(g_{p,q}(n))$. Moreover,
\[
 \mathcal F(z(c_1)x_1,z(c_2)x_2)
 =e\bigl(\sigma(c_1-c_2)\bigr)\mathcal F(x_1,x_2),
\]
so the vertical frequency on the product is
$\Xi(c_1,c_2)=\sigma(c_1-c_2)=(\xi,-\xi)$.
The product has dimension $2d$, and the Lipschitz norm of $\mathcal F$
is at most a fixed multiple of $M$. After increasing the absolute constant in the threshold if necessary, \Cref{inv:vertical-kernel} applies with length $N_{p,q}$.
It yields at most $4d$ linearly independent real characters
$\Lambda_{p,q,j}$ on $G^\times$ that are integer-valued on
$\Gamma^\times$, vanish on $G_2^\times$, and satisfy
\begin{equation}
 |\Lambda_{p,q,j}|,\ 
 \|\Lambda_{p,q,j}\circ g_{p,q}\|_{C^\infty[N_{p,q}]}
 \le L_0,\qquad L_0:=Z^{C(d+1)^4}.
\label{inv:eq:lem78-product-bounds}
\end{equation}
The exponent $C(d+1)^4$ is enlarged once to absorb the product
dimension $2d$.

Every such character has the form
\[
 \Lambda_{p,q,j}(x_1,x_2)
 =\alpha_{p,q,j}(x_1)+\beta_{p,q,j}(x_2),
\]
where $\alpha_{p,q,j}$ and $\beta_{p,q,j}$ are horizontal integer
characters on $G$.

\par
\medskip
\noindent\textbf{Step 3. Fix one tuple of characters.}
Write $m:=\dim(G/G_2)=2r<d$. Order each family and append zero pairs until exactly $4d$ positions are present. Each pair
$(\alpha_{p,q,j},\beta_{p,q,j})$ has at most $2m\le2d$ integer
coefficients, each of size at most $L_0$. Consequently the number of
resulting $4d$-tuples is at most
\begin{equation*}
 Q:=(2\lfloor L_0\rfloor+1)^{8d^2}.
\end{equation*}
Choose an integer $H_0$ satisfying
\[
 H_0\ge\max\{2,4L_0,Q/(c\eta^{25})\}.
\]
By increasing the absolute constant $C$, we may also assume
$H_0\le Z^{C(d+1)^6}$. Pigeonholing the pairs in $\mathcal P$, one resulting $4d$-tuple
\[
 (\alpha_1,\beta_1),\ldots,(\alpha_{4d},\beta_{4d})
\]
occurs for at least $x^2/H_0$ pairs. Fix this tuple and one corresponding
pair $(p,q)$ whenever a single pair is needed below.

Because the characters vanish on $G_2^\times$,
\[
 \Lambda_i(g_{p,q}(n))
 =n\bigl(p\alpha_i(u)+q\beta_i(u)\bigr).
\]
Thus \eqref{inv:eq:lem78-product-bounds} and
$N_{p,q}\ge N/(4x)$ imply, for every retained pair and every $i$,
\begin{equation}
 \|p\alpha_i(u)+q\beta_i(u)\|_{\mathbb R/\mathbb Z}
 \le\frac{L_0}{N_{p,q}}
 \le\frac{4L_0x}{N}
 \le\frac{H_0x}{N}.
\label{inv:eq:lem78-linear-relations}
\end{equation}

\par
\medskip
\noindent\textbf{Step 4. The case $x\ge32H_0$.}
Apply \eqref{inv:dense-linear-interface} with $H=H_0$ to each of the
$4d$ fixed linear relations in \eqref{inv:eq:lem78-linear-relations}.
The required conditions $x\ge32H_0$ and
$N\ge10^6H_0^3x^3$ follow from the present case and the assumed
threshold, after increasing $C$. We obtain integers
$1\le R_i\le2H_0$ for which
\begin{equation*}
 \|R_i\alpha_i(u)\|_{\mathbb R/\mathbb Z},\
 \|R_i\beta_i(u)\|_{\mathbb R/\mathbb Z}
 \le\frac{128H_0^3}{N}.
\end{equation*}
Take all nonzero characters $R_i\alpha_i$ and $R_i\beta_i$, and then
retain a maximal linearly independent subfamily. Its real common kernel
is unchanged. If $w_1,w_2$ lie in this kernel, then
$(w_j,1)$ lies in the real common kernel of the product characters.
Therefore \Cref{inv:vertical-kernel} gives
\[
 0=\Xi([(w_1,1),(w_2,1)])=\xi([w_1,w_2]).
\]

\par
\medskip
\noindent\textbf{Step 5. The case $x<32H_0$.}
Fix one retained pair $(p,q)$; it has $p\ne q$. Define
\[
 \lambda_i:=p\alpha_i+q\beta_i
 \qquad(1\le i\le4d),
\]
and retain a maximal linearly independent subfamily. These are
integer-valued horizontal characters. Their heights and values at $u$
satisfy
\begin{equation*}
 |\lambda_i|\le4xL_0\le128H_0L_0,\qquad
 \|\lambda_i(u)\|_{\mathbb R/\mathbb Z}
 \le\frac{H_0x}{N}\le\frac{32H_0^2}{N}.
\end{equation*}
Let $w_1,w_2$ lie in their common real kernel and set
\[
 \widetilde w_j:=(w_j^p,w_j^q)\in G\times G.
\]
Every discarded character is an $\mathbb R$-linear combination of the
retained ones, so it also vanishes on $w_j$. Hence all product characters
vanish on $\widetilde w_j$, because
$p\alpha_i(w_j)+q\beta_i(w_j)=\lambda_i(w_j)=0$.
\Cref{inv:vertical-kernel}, together with bilinearity of the step-two commutator, gives
\begin{equation*}
 0=\Xi\!\left(
 [\widetilde w_1,\widetilde w_2]\right)
 =\Xi\bigl([w_1,w_2]^{p^2},[w_1,w_2]^{q^2}\bigr)
 =(p^2-q^2)\xi([w_1,w_2]).
\end{equation*}
Since $p\ne q$, we conclude $\xi([w_1,w_2])=0$.

\par
\medskip
\noindent\textbf{Step 6. Finish the quantitative and structural conclusions.}
In both cases the resulting characters vanish on $G_2$ and are
integer-valued on $\Gamma$. After retaining a maximal linearly
independent subfamily, their number is at most $2r$, and their
heights and $N$-scaled errors are bounded by a fixed power of $H_0L_0$.
Choose $J$ to be an integer satisfying
\[
 J\ge Z,\qquad
 J\ge128H_0L_0,\qquad
 J\ge128H_0^3,\qquad
 J\le Z^{C(d+1)^6},
\]
which is possible after enlarging $C$ once more. Then the bounds in
\eqref{inv:dense-constraints-output} hold.

Let $V\subseteq\mathbb R^{2r}$ be the intersection of the real kernels
of the horizontal coefficient vectors of the retained characters.
Since every character vanishes on $G_2$, the common kernel is
$H=\pi^{-1}(V)$, where $\pi:G\to G/G_2\simeq\mathbb R^{2r}$.
Thus $H$ is connected and contains $G_2$. It is normal because
$G/G_2$ is abelian and $[G,H]\subseteq G_2\subseteq H$. The two case
arguments prove $\xi([H,H])=0$. Since $\xi(z(c))=\sigma c$ is injective
on the one-dimensional center, $[H,H]=\{1\}$; hence $H$ is abelian.
\end{proof}

\begin{theorem}
\label{inv:scalar-reduction}
Assume \eqref{inv:eq:Heisenberg-input} for an arbitrary
degree-two polynomial $g$. There is an absolute $C$ such that, if
\begin{equation*}
 N\ge\exp\bigl(Z^{C(d+1)^8}\bigr),
\end{equation*}
there exist real numbers $a,b$ and a nonempty arithmetic progression
$\mathcal P\subseteq[N]$ (with positive integer common difference) such that
\begin{equation*}
 \left|N^{-1}\sum_{n\in\mathcal P}
       f(n)e(an+bn^2)\right|\ge Z^{-C(d+1)^8}.
\end{equation*}
The common difference of $\mathcal P$ is at most $Z^{C(d+1)^7}$.
\end{theorem}

\begin{proof}
We divide the proof into six steps.	
	
\par
\medskip
\noindent\textbf{Step 1. Normalize the polynomial sequence.}
The constant and linear normalizations in
\eqref{inv:constant-normalization} and
\eqref{inv:eq:normalize-linear} replace the original data by data with
exactly the same correlation and with controlled changes in the Lipschitz
norm.  In particular, after enlarging the absolute constant in the
hypothesis, we may assume
\begin{equation}
 g(n)=u^n v^{\binom n2},\qquad
 u\in\mathcal B_r,\qquad v\in G_2,
 \label{inv:eq:thm79-normalized}
\end{equation}
where $G=\mathcal H(r)$ and $d=2r+1$.  The equivariance relation
$F(z(c)x\Gamma)=e(\sigma c)F(x\Gamma)$ is preserved by these
normalizations.  The bound $|F|\le1$ is preserved, and its Lipschitz norm
is at most a fixed power of $dM$.  The linear normalization changes the
central coefficient to $v[\gamma,u_0]$; this new central factor is retained
exactly in \eqref{inv:eq:thm79-normalized}.  Thus the parameter
$Z=10dM/\eta$ changes by at most a fixed power, which is absorbed into
$Z^{C(d+1)^8}$ and $Z^{C(d+1)^7}$ below.

\par
\medskip
\noindent\textbf{Step 2. Use \Cref{inv:dense-constraints} to obtain an
abelian subgroup.}
Apply \Cref{inv:dense-constraints} to the normalized data in
\eqref{inv:eq:thm79-normalized}.  We obtain independent horizontal
characters $\lambda_1,\ldots,\lambda_k$ and
\[
 H_1:=\bigcap_{i=1}^{k}\ker_{\mathbb R}\lambda_i,
 \qquad H_1\supseteq G_2,
\]
where $H_1$ is connected, normal, and abelian.  Let
$m:=2r=\dim(G/G_2)$, let $A$ be the $k\times m$ integer matrix whose
rows are the horizontal coefficient vectors of the $\lambda_i$, and let
$\alpha\in[-1/2,1/2)^m$ be the horizontal coordinate of $u$.  The
lemma gives an integer $J$ such that
\begin{equation*}
 Z\le J\le Z^{C(d+1)^6},
 \qquad \|A\alpha\|_{\mathbb R^k/\mathbb Z^k}\le J/N,
 \qquad \text{every entry of }A\text{ has absolute value at most }J.
\end{equation*}
Here the norm on the left is the maximum of the coordinate distances to
$\mathbb Z$.  Since $G$ is non-abelian but $H_1$ is abelian, $k\ge1$.

After permuting the horizontal coordinates, choose a nonsingular $k\times k$
column minor $B$ of $A$ and write $A=[B\ C_1]$.  Put
$\mathfrak q:=\det B\ne0$, and choose $b_0\in\mathbb Z^k$ with
$\|A\alpha-b_0\|_\infty\le J/N$.  Define vectors supported on the pivot
columns by
\begin{equation*}
 e_B:=B^{-1}(A\alpha-b_0),\qquad
 c_B:=B^{-1}b_0,\qquad e_{\rm free}=c_{\rm free}:=0,\qquad
 w:=\alpha-e-c.
\end{equation*}
Then
\begin{equation*}
\begin{aligned}
 Aw=0,\quad Ae=A\alpha-b_0,\quad Ac=b_0,
 \quad |e|\le \frac{(dJ)^{C(d+1)}}{N},\quad
 |c|\le(dJ)^{C(d+1)},\quad \mathfrak q c\in\mathbb Z^m.
\end{aligned}
\end{equation*}
For the bounds, write
$B^{-1}=\mathfrak q^{-1}\operatorname{adj}(B)$, where
$\operatorname{adj}(B)$ is the integer adjugate matrix characterized by
$B\operatorname{adj}(B)=\mathfrak q I_k$; each entry of this matrix is a
signed determinant of a $(k-1)\times(k-1)$ submatrix of $B$.  Since every
entry of $B$ is bounded by $J$, these determinants are bounded by a fixed
power of $J$.  Cramer's rule then gives the displayed bound for $e$ and
$c$ because each coordinate of $b_0$ is $O(dJ)$, while
$|\mathfrak q|\le(dJ)^{C(d+1)}$.

Embed the vectors $e,c\in\mathbb R^m$ into the Lie algebra by giving
them zero entries in the $c$-coordinate direction of every Heisenberg
factor.  Define
\begin{equation*}
 \varepsilon(n):=\exp(ne),\qquad
 \gamma(n):=\exp(nc),\qquad
 h(n):=\varepsilon(n)^{-1}u^n\gamma(n)^{-1}.
\end{equation*}
The horizontal coordinate of $h(n)$ is $nw\in\ker A$, so $h(n)\in H_1$.
The Heisenberg group law adds a bilinear commutator term; therefore the
$c$-coordinate of $h(n)$ is a polynomial of degree at most two.  Since
$v\in G_2$ is central,
\begin{equation*}
 g(n)=\varepsilon(n)h_1(n)\gamma(n),\qquad
 h_1(n):=h(n)v^{\binom n2}\in H_1.
\end{equation*}
All quadratic commutator terms are contained in $h_1(n)$; no error term is
being discarded.

\par
\medskip
\noindent\textbf{Step 3. Control the slowly varying and periodic factors.}
Choose a fixed power bound $P_0=(dJ)^{C(d+1)}$ large enough that
$P_0\ge d,M,\eta^{-1},2|\mathfrak q|$ and that all estimates below hold.
The path $t\mapsto\exp(te)$ has constant speed $\|e\|_2$, so
\begin{equation}
 d_G(\varepsilon(n),1)\le n\|e\|_2\le P_0,
 \qquad
 d_G(\varepsilon(n+1),\varepsilon(n))\le\|e\|_2\le P_0/N.
 \label{inv:eq:thm79-epsilon-bounds}
\end{equation}
Thus $\varepsilon(n)$ varies by at most $P_0/N$ per step.  Moreover,
$\mathfrak q c\in\mathbb Z^m$ and
\[
 \exp(\theta,\lambda,0)
   =(\theta,\lambda,\tfrac12\lambda\cdot\theta)
\]
show that $\exp(2\mathfrak q c)\in\Gamma$.  Define
$D_\gamma$ to be the least positive integer $D$ for which
$\gamma(n+D)\Gamma=\gamma(n)\Gamma$ for every integer $n$; equivalently,
$D_\gamma$ is the least positive $D$ with $\exp(Dc)\in\Gamma$.
The preceding calculation supplies such a period, so
\begin{equation}
 D_\gamma\le2|\mathfrak q|\le P_0,
 \qquad P_0\le Z^{C(d+1)^7}.
 \label{inv:eq:thm79-gamma-period}
\end{equation}

\par
\medskip
\noindent\textbf{Step 4. Select one arithmetic cell.}
Put
\[
 L_{\rm cell}:=\left\lfloor\frac{\eta N}{20MP_0}\right\rfloor.
\]
The theorem's lower bound on $N$ implies
$L_{\rm cell}\ge\eta N/(40MP_0)\ge1$.  Partition $[N]$ into consecutive
intervals of length $L_{\rm cell}$, and then split each interval according
to the residue class of $n$ modulo $D_\gamma$.  Each resulting nonempty
cell is an arithmetic progression, and the number of cells is at most
\begin{equation*}
 D_\gamma\left(1+\frac{N}{L_{\rm cell}}\right)
 \le\frac{41P_0^3}{\eta}.
\end{equation*}
For a cell $\mathcal C$ choose $n_\mathcal C\in\mathcal C$ and replace
$\varepsilon(n)$ by $\varepsilon(n_\mathcal C)$.  By
\eqref{inv:eq:thm79-epsilon-bounds}, the pointwise error in $F$ is at most
$MP_0L_{\rm cell}/N\le\eta/20$.  Summing over $n$ and using $|f(n)|=1$,
we obtain
\begin{equation*}
 \sum_{\mathcal C}
 \left|\frac1N\sum_{n\in\mathcal C}
 f(n)F(\varepsilon(n_\mathcal C)h_1(n)\gamma(n)\Gamma)\right|
 \ge\frac{19\eta}{20}.
\end{equation*}
Consequently one cell, denoted by $\mathcal P$, satisfies
\begin{equation*}
 \left|\frac1N\sum_{n\in\mathcal P}
 f(n)F(\varepsilon(n_0)h_1(n)\gamma(n)\Gamma)\right|
 \ge R,
 \qquad R:=\frac{\eta^2}{50P_0^3},
\end{equation*}
where $n_0\in\mathcal P$ is fixed.  The progression $\mathcal P$ has step
$D_\gamma$.

\par
\medskip
\noindent\textbf{Step 5. Put the selected cell on a torus.}
On $\mathcal P$, the coset $\gamma(n)\Gamma$ is constant.  Let
$\alpha_0\in G$ be its representative in the centered fundamental box
$\mathcal B_r$.  Define
\[
 F_{\mathcal P}(z):=F(\varepsilon(n_0)\alpha_0z).
\]
Then, for $n\in\mathcal P$,
\begin{equation*}
 F(\varepsilon(n_0)h_1(n)\gamma(n)\Gamma)
 =F_{\mathcal P}(\alpha_0^{-1}h_1(n)\alpha_0\Gamma).
\end{equation*}
Normality of $H_1$ implies that
$\widetilde h(n):=\alpha_0^{-1}h_1(n)\alpha_0$ remains in $H_1$.

Recall that $\operatorname{adj}(B)$ is the integer matrix made from the
signed $(k-1)\times(k-1)$ minors of $B$, so that
$B\operatorname{adj}(B)=\mathfrak q I_k$.  Let $U$ be the integer matrix
\begin{equation*}
 U:=\begin{pmatrix}-\operatorname{adj}(B)C_1\\ \mathfrak q I\end{pmatrix}.
\end{equation*}
Its columns span $\ker A$ over $\mathbb R$.  The vectors $2U_j$, together
with the lattice vector whose horizontal coordinates are zero and whose
$c$-coordinate is $1$, generate a lattice $\Lambda\subseteq
H_1\cap\Gamma$.  Indeed, the horizontal coordinates of $2U_j$ are even integers,
and the Heisenberg $c$-coordinate
$\frac12\lambda\cdot\theta$ is then an integer.  Thus
$H_1/\Lambda$ is a torus of dimension
$s_H=m-k+1\le d$.  The sequence $\widetilde h(n)\Lambda$ is a degree-two
polynomial sequence on this abelian torus.

Let $\widetilde F_{\mathcal P}$ denote the pullback of $F_{\mathcal P}$
to $H_1/\Lambda$.  It has supremum at most $1$.
Its Lipschitz norm is at most $P_0^C$.  The vectors $2U_j$ have size at
most $P_0$.  To account for the left translation by
$a_0:=\varepsilon(n_0)\alpha_0$, let $X_0$ be the Lie-algebra coordinate
vector with $a_0=\exp(X_0)$, and let $V_j$ be the tangent vector associated
with $2U_j$ (and likewise for the central lattice direction).  In a
step-two group, differentiating
$\exp(X_0)\exp(tV_j)\exp(-X_0)$ at $t=0$ gives
$V_j+[X_0,V_j]$.  Here $[X_0,V_j]$ is the Lie bracket, which is central
and bilinear in the horizontal coordinates.  The bounds on $X_0$ and
$V_j$ therefore imply the claimed fixed power of $P_0$.

\par
\medskip
\noindent\textbf{Step 6. Extract one scalar quadratic phase.}
Apply \Cref{inv:fejer} on the $s_H$-dimensional torus with error
parameter $R/4$.  We obtain
\[
 T(x)=\sum_{k\in[-K,K]^{s_H}\cap\mathbb Z^{s_H}}c_k e(k\cdot x),
 \qquad \|\widetilde F_{\mathcal P}-T\|_\infty\le R/4,
\]
where
\begin{equation*}
 K\le C\left(\frac{s_HP_0^C}{R}\right)^4\le P_0^C,
 \qquad
 \sum_k|c_k|\le(2K+1)^{s_H}.
\end{equation*}
The error in the normalized correlation is at most $R/4$, so
\begin{equation}
 \left|\frac1N\sum_{n\in\mathcal P}f(n)T(\widetilde h(n)\Lambda)\right|
 \ge\frac{3R}{4}.
 \label{inv:eq:thm79-trig-correlation}
\end{equation}
Write the torus orbit in additive coordinates as
\[
 \widetilde h(n)\Lambda=t_0+n t_1+\binom n2t_2
 \pmod{\mathbb Z^{s_H}},
 \qquad t_0,t_1,t_2\in\mathbb R^{s_H}.
\]
For each $k$, this gives real numbers $a_k,b_k,c_k'$ such that
\[
 e(k\cdot\widetilde h(n))=e(c_k'+a_kn+b_kn^2).
\]
The constant factor $e(c_k')$ has absolute value one and can be absorbed
into the Fourier coefficient.  By the triangle inequality applied to
\eqref{inv:eq:thm79-trig-correlation}, some $k$ satisfies
\begin{equation*}
 \left|\frac1N\sum_{n\in\mathcal P}f(n)e(a_kn+b_kn^2)\right|
 \ge\frac{3R}{4(2K+1)^{s_H}}.
\end{equation*}
Finally, $R^{-1}$ and $P_0$ are bounded by fixed powers of
$Z^{C(d+1)^7}$, while $s_H\le d$.  Therefore
\[
 \frac{3R}{4(2K+1)^{s_H}}
 \ge Z^{-C(d+1)^8}.
\]
By \eqref{inv:eq:thm79-gamma-period}, the step of $\mathcal P$ is at most
$Z^{C(d+1)^7}$.  This proves the scalar correlation and the progression
bound in the theorem.
\end{proof}

\subsection{Proof of the inverse theorem}

We now combine the preceding reductions with the scalar inverse theorem.

\begin{theorem}
\label{thm:mult_inverse}
There is an absolute constant $C\ge1$, independent of $\epsilon$, $N$, and $f$,
such that, for $0<\epsilon<1/2$, on writing
$\ell=2+\log(2/\epsilon)$, every completely multiplicative
function $f:\mathbb N\to\mathbb C$ with $|f(n)|=1$ satisfying
\begin{equation*}
 N\ge\exp\exp(C\ell^{431}),\qquad
 \|f\|_{U^3[N]}\ge\epsilon
\end{equation*}
has a primitive Dirichlet character $\chi$ and a real $t$
such that
\begin{equation}
 q_\chi,\ |t|\le\exp(C\ell^{431}),\qquad
 \sum_{\substack{p\le N\\p\ \mathrm{prime}}}\frac{1-\Re(f(p)\overline{\chi(p)}p^{-it})}{p}
       \le C\ell^{431}.
 \label{inv:eq:numerical-inverse-output}
\end{equation}
\end{theorem}

\begin{proof}
We divide the proof into four steps.	
	
\par
\medskip
\noindent\textbf{Step 1. Obtain the structured nilsequence.}
Apply \Cref{inv:initial-inverse} to the restriction of the given
completely multiplicative function $f$ to $[N]$.  Its hypotheses hold because
$|f(n)|=1$ and $\|f\|_{U^3[N]}\ge\epsilon$.  We obtain an integer
$d\ge1$, a real number $M\ge2$, a number $0<\eta<1/2$, and either the
circle nilmanifold $\mathbb R/\mathbb Z$ or a Heisenberg nilmanifold
$G/\Gamma$, together with a degree-two polynomial sequence $g$ and a
function $F$ satisfying
\[
 |F|\le1,\qquad \|F\|_{\rm Lip}\le M,\qquad
 \left|\frac1N\sum_{n=1}^{N}f(n)F(g(n)\Gamma)\right|\ge\eta.
\]
The same corollary gives
\begin{equation*}
 d\ll\ell^{43},\qquad \log(2M/\eta)\ll\ell^{87},
 \qquad \ell:=2+\log(2/\epsilon).
\end{equation*}
Set
\[
 \eta_0:=\min\{\eta,1/4\},
 \qquad Z:=\frac{10dM}{\eta_0}.
\]
Because $0<\eta<1/2$, we have $\eta_0\ge\eta/2$.  The correlation is
still at least $\eta_0$.  Moreover,
\begin{equation}
 \log Z\le \log(20d)+\log(2M/\eta)\ll\ell^{87},
 \qquad \log(\eta_0^{-1})\ll\ell^{87}.
 \label{inv:eq:thm710-parameter-bounds}
\end{equation}
Here and below all implicit constants are absolute and may be enlarged.

\par
\medskip
\noindent\textbf{Step 2. Dispose of the circle case.}
Suppose first that $(G,\Gamma)=(\mathbb R,\mathbb Z)$.  In additive
notation write
\[
 g(n)=g_0+nu+\binom n2v,
 \qquad g_0,u,v\in\mathbb R.
\]
The frequency-one identity from \Cref{inv:initial-inverse} gives
$F(x+\mathbb Z)=F(0+\mathbb Z)e(x)$ for every real $x$.  Hence, with
$F_0:=F(0+\mathbb Z)$,
\[
 F(g(n)+\mathbb Z)=F_0e(g_0)e\!\left((u-v/2)n+(v/2)n^2\right).
\]
The correlation bound forces $F_0\ne0$.  Taking absolute values and
using $|F_0|\le1$, it implies
\begin{equation}
 \left|\frac1N\sum_{n=1}^{N}f(n)e(\xi n+\beta n^2)\right|
 \ge\eta_0,
 \qquad \xi:=u-v/2,\quad \beta:=v/2.
 \label{inv:eq:thm710-circle-correlation}
\end{equation}
Apply \Cref{thm:scalar_inverse} with
$\rho:=\eta_0$ and $\mathcal P=[N]$.  Its correlation hypothesis is
\eqref{inv:eq:thm710-circle-correlation}.  Its length hypothesis also holds:
by \eqref{inv:eq:thm710-parameter-bounds},
$\eta_0^{-5}\le\exp(C\ell^{87})$, so the assumed bound
$N\ge\exp\exp(C\ell^{431})$ dominates
$N\ge\exp(C_s\eta_0^{-5})$, where $C_s$ is the absolute constant in
\Cref{thm:scalar_inverse}.  That theorem therefore supplies a
primitive Dirichlet character $\chi$ and a real number $t$ with
\[
 q_\chi,\ |t|\le\eta_0^{-C_s},
 \qquad
 \sum_{\substack{p\le N\\p\ \mathrm{prime}}}\frac{1-\Re(f(p)\overline{\chi(p)}p^{-it})}{p}
 \le C_s\log(2/\eta_0).
\]
The bounds in \eqref{inv:eq:thm710-parameter-bounds} imply that these
quantities satisfy the conclusion of \Cref{thm:mult_inverse} after
increasing its absolute constant.

\par
\medskip
\noindent\textbf{Step 3. Reduce the Heisenberg case to a scalar phase.}
Now suppose that $(G,\Gamma)=(\mathcal H(r),\Gamma(r))$.  The data obtained
in Step~1 satisfy the input condition
\eqref{inv:eq:Heisenberg-input} with the present value of $\eta_0$ and with
$Z=10dM/\eta_0$.  The dimension and parameter bounds give
\begin{equation}
 (d+1)^8\log Z\ll\ell^{8\cdot43+87}=\ell^{431}.
 \label{inv:eq:thm710-exponent}
\end{equation}
Consequently, after enlarging the absolute constant in the hypothesis on
$N$, we have
\[
 N\ge\exp\!\left(Z^{C_9(d+1)^8}\right),
\]
where $C_9$ is the constant in \Cref{inv:scalar-reduction}.
We may therefore apply that theorem.  It gives real numbers $a,b$ and a
nonempty arithmetic progression $\mathcal P\subseteq[N]$ such that
\begin{equation}
 \left|\frac1N\sum_{n\in\mathcal P}
 f(n)e(an+bn^2)\right|\ge\rho_0,
 \qquad \rho_0:=Z^{-C_9(d+1)^8}.
 \label{inv:eq:thm710-scalar-correlation}
\end{equation}
The progression has positive integer common difference, but no further
property of that difference is needed here.  Put
\[
 \rho:=\min\{\rho_0,1/4\}.
\]
Then $0<\rho<1/2$ and the left-hand side of
\eqref{inv:eq:thm710-scalar-correlation} is at least $\rho$.
Furthermore, by \eqref{inv:eq:thm710-exponent},
\begin{equation}
 \rho^{-1}\le 4+Z^{C_9(d+1)^8}
 \le\exp(C'\ell^{431})
 \label{inv:eq:thm710-rho-bound}
\end{equation}
for an absolute $C'$.  Hence
$N\ge\exp(C_s\rho^{-5})$ after another enlargement of the absolute
constant in the assumed bound for $N$.

\par
\medskip
\noindent\textbf{Step 4. Apply the scalar inverse theorem and finish.}
Apply \Cref{thm:scalar_inverse} to the original completely
multiplicative function $f$, the same ambient interval $[N]$, the
progression $\mathcal P$, and the phase $an+bn^2$.  This is a direct
application of its stated hypothesis: the theorem normalizes the sum by
$N$, so no reparametrization of $\mathcal P$ is required.  By
\eqref{inv:eq:thm710-scalar-correlation} and the definition of $\rho$,
the correlation hypothesis holds, and the preceding paragraph verifies its
length hypothesis.  We obtain a primitive Dirichlet character $\chi$ and a
real number $t$ satisfying
\begin{equation}
 q_\chi,\ |t|\le\rho^{-C_s},
 \qquad
 \sum_{\substack{p\le N\\p\ \mathrm{prime}}}\frac{1-\Re(f(p)\overline{\chi(p)}p^{-it})}{p}
 \le C_s\log(2/\rho).
 \label{inv:eq:thm710-final-scalar-output}
\end{equation}
By \eqref{inv:eq:thm710-rho-bound},
$\rho^{-C_s}\le\exp(C\ell^{431})$ and
$\log(2/\rho)\ll\ell^{431}$.  Thus \eqref{inv:eq:thm710-final-scalar-output}
is exactly the conclusion
\eqref{inv:eq:numerical-inverse-output}, after choosing the absolute
constant in \Cref{thm:mult_inverse} large enough.
\end{proof}

\section{Proof of the main theorem}
\label{main:section}

We combine the preceding results with the numerical exponent $431$
of \Cref{thm:mult_inverse}. All parameters in this section
are chosen independently of the multiplicative function.

\subsection{The uniformity estimate}

Throughout this subsection, let
$w:\mathbb R^2\to[0,\infty)$ be the fixed $C_c^2$ probability density
from \eqref{ar:real-profile}, so that $\displaystyle \int_{\mathbb R^2}w=1$, and let
$C_0\in\mathbb N$ be the support parameter from
\eqref{ar:profile-bound}.  Once a real $\delta\in(0,1/10)$ is fixed, let
$D_0=D_0(\delta)$ and let
$W_0=W_{0,\delta}:\mathbb Z^2\to[0,\infty)$ be the period and periodic
weight constructed in
\eqref{ar:parameters}--\eqref{ar:congruence-weight}.  Thus
\[
 \begin{aligned}
 \operatorname{supp}w&\subseteq\{(s,t):s>t>0,\ s+t<C_0-1\},\\
 W_0&\text{ is $D_0$-periodic in each variable and }
 D_0^{-2}\sum_{a,b\bmod D_0}W_0(a,b)=1.
 \end{aligned}
\]
All constants denoted by $C_w$ below depend only on the fixed weight $w$
and the fixed integer $C_0$; in particular, they are independent of
$\delta$, $D_0$, $W_0$, $X$, and $g$.
For each integer $X\ge1$ and each $g\in\mathcal M$, $Z_X$ and $K_X(g)$
are the normalization and normalized kernel defined in \eqref{ar:kernel}
whenever $Z_X>0$; explicitly, the kernel averages $g(R(m,n))$ with
$R(m,n)=(m^2-n^2)/(2mn)$.

\Needspace{14\baselineskip}
\begin{samepage}
\begin{proposition}\label{tr:uniform-estimate}
For every real $\delta$ with $0<\delta<1/10$, with
$D_0=D_0(\delta)$ and $W_0=W_{0,\delta}$ chosen as above, there are
real numbers $L_\delta\ge1$ and $X_\delta\ge1$, depending only on
$\delta$, the fixed data $(w,C_0)$, and the prescribed construction of
$(D_0,W_0)$, such that
\begin{equation}
 \log(2L_\delta)+\log(2X_\delta)\le C_w\delta^{-3},
 \label{tr:uniform-constants}
\end{equation}
and, for every integer $X\ge X_\delta$ and every $g\in\mathcal M$,
\begin{equation}
 |K_X(g)|\le L_\delta\|g\|_{U^3[C_0X]}+\delta/100.
 \label{tr:uniform-bound}
\end{equation}
Consequently, with
\begin{equation*}
 \epsilon_\delta=\frac{\delta}{100L_\delta},
 \qquad \mathcal L_\delta=2+\log(2/\epsilon_\delta),
\end{equation*}
we have $\mathcal L_\delta\le C_w\delta^{-3}$, and
$\|g\|_{U^3[C_0X]}<\epsilon_\delta$ implies
$|K_X(g)|\le\delta/50$.
\end{proposition}

\end{samepage}
\begin{proof}
We divide the proof into six steps.	
	
\par
\medskip
\noindent\textbf{Step 1. Approximate the fixed spatial weight.}
Fix $\delta$ and put
$W:=\|W_0\|_\infty$.  Choose an integer
\begin{equation}
 J:=\left\lceil C_w(1+W)/\delta\right\rceil,
 \label{tr:mesh-choice}
\end{equation}
where $C_w$ depends only on the fixed weight $w$ and the fixed integer
$C_0$, and is large enough for
all estimates below.  Partition $\mathbb R^2$ into half-open squares of
side length $1/J$, and retain the squares meeting $\operatorname{supp}w$.
Since the support is compact, the number of retained squares is at most
$C_wJ^2$.  The support lies a positive distance from the boundary of
$s>t>0$ and from the line $s+t=C_0-1$; after increasing $C_w$ if needed,
every retained square is contained in
$s>t>0$ and $s+t<C_0$.

Let $w_J$ equal the value of $w$ at the centre of each retained square and
be zero outside their union.  The mean-value theorem gives
\begin{equation*}
 \|w-w_J\|_\infty\le C_w/J,
 \qquad
 w_J=\sum_{\mathcal R}c_{\mathcal R}\mathbf 1_{\mathcal R},
 \quad 0\le c_{\mathcal R}\le\|w\|_\infty,
 \quad \#\{\mathcal R\}\le C_wJ^2.
\end{equation*}
Here and below $\mathcal R$ ranges over the retained squares.

\par
\medskip
\noindent\textbf{Step 2. Bound the error caused by replacing $w$ by $w_J$.}
Both $w$ and $w_J$ vanish outside a fixed bounded set.  Hence, for every
integer $X\ge1$, that set contains at most $C_wX^2$ lattice points after
scaling by $X$.  By \eqref{ar:kernel}, $I_X(g)$ is the unnormalized sum
and $Z_X=I_X(1)$ is its value for the constant function $1$.  Since
$|W_0(m,n)|\le W$, replacing $w$ by $w_J$ in $I_X(g)$ changes it by at most
\[
 X^{-2}\,(C_wX^2)\,(C_w/J)\,W\le C_wW/J.
\]
Let $\widetilde I_X(g)$ denote this modified unnormalized sum, and define
$\widetilde K_X(g):=\widetilde I_X(g)/Z_X$ using the original normalization
$Z_X=I_X(1)$.  By the normalization estimate \eqref{ar:normalization},
$Z_X\ge1/2$ whenever $X\ge C_wD_0$.  Consequently,
\[
 |K_X(g)-\widetilde K_X(g)|\le 2C_wW/J\le\delta/100
\]
by \eqref{tr:mesh-choice}.

\par
\medskip
\noindent\textbf{Step 3. Expand the periodic weight.}
The $D_0$-periodicity of $W_0$ gives the finite Fourier expansion
\begin{align*}
 W_0(m,n)&=\sum_{\alpha,\beta\bmod D_0}
       \widehat W_0(\alpha,\beta)
       e((\alpha m+\beta n)/D_0),\notag\\
 \widehat W_0(\alpha,\beta)&:=D_0^{-2}\sum_{a,b\bmod D_0}
       W_0(a,b)e(-\alpha a/D_0-\beta b/D_0),\notag\\
 \sum_{\alpha,\beta\bmod D_0}|\widehat W_0(\alpha,\beta)|
       &\le D_0^2W.
\end{align*}
The first identity is Fourier inversion on $(\mathbb Z/D_0\mathbb Z)^2$;
it follows from the orthogonality relation
$$D_0^{-1}\sum_{j\bmod D_0}e(jr/D_0)=\mathbf 1_{D_0\mid r}.$$
The last inequality follows from the triangle inequality applied to the
finite definition of each coefficient.

\par
\medskip
\noindent\textbf{Step 4. Rewrite each summand as a four-linear form.}
For $m>n>0$, complete multiplicativity and $|g(q)|=1$ give
\begin{equation*}
 g\!\left(\frac{m^2-n^2}{2mn}\right)
 =\overline{g(2)}\,g(m-n)g(m+n)\overline{g(m)}\,\overline{g(n)}.
\end{equation*}
Indeed, $(m^2-n^2)/(2mn)=(m-n)(m+n)/(2mn)$, and
$g(1/q)=\overline{g(q)}$ for $q>0$.
Set $Z:=C_0X$ and $P:=32Z+1$.  For a retained rectangle
$\mathcal R=I\times J'$ and a Fourier pair $(\alpha,\beta)$, define
functions on $\mathbb Z/P\mathbb Z$ by zero extension from $[Z]$:
Here $I$ and $J'$ are intervals in the scaled variables $m/X$ and $n/X$.
We use the notation
\[
 XI:=\{m\in\mathbb Z:m/X\in I\},
 \qquad
 XJ':=\{n\in\mathbb Z:n/X\in J'\};
\]
thus $XI$ and $XJ'$ are sets of integers, not real intervals.
\begin{align*}
 a_Z(u)&:=g(u)\mathbf 1_{[Z]}(u),
 &b_Z(u)&:=g(u)\mathbf 1_{[Z]}(u),\\
 c_Z(u)&:=\overline{g(u)}\mathbf 1_{XI}(u)e(\alpha u/D_0),
 &d_Z(u)&:=\overline{g(u)}\mathbf 1_{XJ'}(u)e(\beta u/D_0).
\end{align*}
The indicators in the last two definitions select integer points.  All four
functions have modulus at most one.  If $c_Z(m)d_Z(n)\ne0$, then
$m>n>0$ and $m+n<Z$ because every retained square is contained in
$s>t>0$ and $s+t<C_0$.  Therefore
$m-n,m+n,m,n$ all lie in $[Z]$, and the four-linear expression has no
wrap-around in $\mathbb Z/P\mathbb Z$.  The phases in $c_Z$ and $d_Z$ are
evaluated only at these selected integer representatives; no periodic
extension of the phases is required, because the functions are defined by
zero extension.

\par
\medskip
\noindent\textbf{Step 5. Apply the $U^3$ four-forms estimate.}
Let $G=\mathbb Z/P\mathbb Z$ and write
\[
 \Lambda:=\mathbb E_{m,n\in G}
       a_Z(m-n)b_Z(m+n)c_Z(m)d_Z(n).
\]
Because $P=32Z+1$ is odd and $P>8Z$, the hypotheses of
\Cref{tr:four-forms} and \Cref{tr:interval-embedding} hold.  The
four-forms estimate gives
\[
 |\Lambda|\le\|a_Z\|_{U^3(G)}.
\]
The support observation in Step~4 shows that the cyclic sum defining
$\Lambda$ is exactly the corresponding integer sum: every nonzero term
has $m>n>0$ and $m+n<Z$, so reduction modulo $P$ introduces no wrap-around.
The factor $\overline{g(2)}$ in the multiplicative identity has modulus one
and therefore disappears after taking absolute values.  Since a normalized
average on $G^2$ is $P^{-2}$ times an unnormalized sum, the contribution to
$I_X(g)$ from this rectangle and Fourier mode is bounded by
\begin{equation}
 \frac{P^2}{X^2}|\Lambda|
 \le (P/X)^2\|a_Z\|_{U^3(\mathbb Z/P\mathbb Z)}
 \le(32C_0+1)^2\|g\|_{U^3[Z]}.
 \label{tr:mesh-contribution}
\end{equation}
Here the last inequality is the interval-embedding estimate applied to the
zero extension $a_Z$ of $g|_{[Z]}$, together with
$P/X=32C_0+1/X\le32C_0+1$.

\par
\medskip
\noindent\textbf{Step 6. Sum the finitely many contributions.}
There are at most $C_wJ^2$ retained rectangles, and the Fourier coefficients
have total absolute mass at most $D_0^2W$.  Summing
\eqref{tr:mesh-contribution}, dividing by $Z_X\ge1/2$, and absorbing the
fixed factors into $C_w$ gives
\begin{equation}
 1\le L_\delta\le C_wJ^2D_0^2(1+W),
 \qquad X_\delta:=\lceil C_wD_0\rceil.
 \label{tr:explicit-uniform-constant}
\end{equation}
Thus $|\widetilde K_X(g)|\le L_\delta\|g\|_{U^3[Z]}$.  Combining this
with the estimate for $|K_X(g)-\widetilde K_X(g)|$ from Step~2 proves
\eqref{tr:uniform-bound} for every integer $X\ge X_\delta$.

Finally, the weight-construction estimate \Cref{ar:weight-cost} gives
$\log(2D_0),\log(2W)\le C_w\delta^{-3}$.  Taking logarithms in
\eqref{tr:mesh-choice} and \eqref{tr:explicit-uniform-constant} yields
\eqref{tr:uniform-constants}.  From
$\epsilon_\delta=\delta/(100L_\delta)$ we obtain
$\mathcal L_\delta\le C_w\delta^{-3}$.  If
$\|g\|_{U^3[C_0X]}<\epsilon_\delta$, then
\eqref{tr:uniform-bound} gives
\[
 |K_X(g)|<L_\delta\epsilon_\delta+\delta/100
 =\delta/50.
\]
This proves the proposition.
\end{proof}
\subsection{Monotone prime sums}

To compare $g$ with twisted Dirichlet characters at different prime
cutoffs, we use one fixed set of admissible parameters.  At each cutoff
we take the smallest prime-distance among those parameters.
\Cref{main:minimized-energy} records the two facts needed later: this
minimum is nondecreasing with the cutoff, and a minimizing pair controls
the contribution of the primes between two cutoffs.

Fix an integer $Q\ge1$ and a real $T_1\ge0$.
For $g\in\mathcal M$ and real $z\ge1$, set
\begin{equation}
 E_g(z)=\min_{\substack{\chi\ \mathrm{primitive},\ q_\chi\le Q\\
                             |t|\le T_1}}
       \sum_{\substack{p\le z\\p\ \mathrm{prime}}}
       \frac{1-\Re(g(p)\overline{\chi(p)}p^{-it})}{p}.
 \label{main:energy-definition}
\end{equation}
For a pair $(\chi,t)$ in these ranges and a cutoff $y\ge1$, let
$\widetilde g$ be the completely multiplicative function specified by
\begin{equation}
 \widetilde g(p)=
 \begin{cases}
 g(p),&p\le y,\\
 \chi(p)p^{it},&p>y
 \end{cases}.
 \label{main:truncated-model}
\end{equation}

\begin{lemma}\label{main:minimized-energy}
For every $g\in\mathcal M$, the minimum in
\eqref{main:energy-definition} exists and is nonnegative and
nondecreasing in $z$. A minimizing pair at $z$ satisfies
\begin{equation}
 \sum_{\substack{y<p\le z\\p\ \mathrm{prime}}}
       \frac{1-\Re(g(p)\overline{\chi(p)}p^{-it})}{p}
 \le E_g(z)-E_g(y)\qquad(1\le y\le z).
 \label{main:tail-minimum}
\end{equation}
If $y\ge Q$, its truncation \eqref{main:truncated-model} belongs to
$\mathcal M$.
\end{lemma}

\begin{proof}
We divide the proof into four steps.	
	
\par
\medskip
\noindent\textbf{Step 1. Existence and nonnegativity of the minimum.}
There are finitely many primitive Dirichlet characters with conductor at
most $Q$; the set is nonempty because it contains the principal character
of conductor one.  Fix one such character $\chi$.  For real $t$ in
$[-T_1,T_1]$, write
\[
 D_{\chi,t}(z):=\sum_{\substack{p\le z\\p\ \mathrm{prime}}}
 \frac{1-\Re(g(p)\overline{\chi(p)}p^{-it})}{p}.
\]
This is a finite sum of continuous functions of $t$, since
$p^{-it}=e^{-it\log p}$.  It therefore attains its minimum on the compact
interval $[-T_1,T_1]$.  Taking the minimum over the finite set of
characters also attains a minimum.  Every summand is nonnegative because
$|g(p)|=1$ and $|\chi(p)|\le1$, so $E_g(z)\ge0$.

\par
\medskip
\noindent\textbf{Step 2. Monotonicity in the cutoff.}
If $z_2\ge z_1$, then for every admissible pair $(\chi,t)$,
$D_{\chi,t}(z_2)\ge D_{\chi,t}(z_1)$, because the second sum contains all
terms of the first and the added terms are nonnegative.  Taking the minimum
over the same admissible set gives
\[
 E_g(z_2)\ge E_g(z_1).
\]

\par
\medskip
\noindent\textbf{Step 3. A minimizing pair controls the tail.}
Let $(\chi,t)$ minimize $E_g(z)$, and let $D(u):=D_{\chi,t}(u)$.  Then
$D(z)=E_g(z)$.  Since $E_g(y)$ is the minimum over all admissible pairs,
$D(y)\ge E_g(y)$.  Subtracting gives
\[
 \sum_{\substack{y<p\le z\\p\ \mathrm{prime}}}
 \frac{1-\Re(g(p)\overline{\chi(p)}p^{-it})}{p}
 =D(z)-D(y)
 \le E_g(z)-E_g(y),
\]
which is \eqref{main:tail-minimum}.

\par
\medskip
\noindent\textbf{Step 4. The truncation remains completely multiplicative.}
Assume $y\ge Q$.  Every prime divisor of the conductor $q_\chi$ is at most
$Q$, so $p>y$ implies $p\nmid q_\chi$ and hence $|\chi(p)|=1$.  Therefore
both values prescribed in \eqref{main:truncated-model} have modulus one:
$|g(p)|=1$ for $p\le y$, and $|\chi(p)p^{it}|=1$ for $p>y$.  Extending
these prime values multiplicatively defines an element of $\mathcal M$.
\end{proof}
The next lemma explains how to average the errors from different
cutoffs. Truncation at $C_1$ permits a uniform estimate even when
the original sequence exceeds $C_1$.

\begin{lemma}\label{main:capped-average}
Let $k\ge1$ be an integer, and let $C_1,a,b\in\mathbb R$ satisfy
$C_1>0$ and $a,b\ge0$. Suppose
$E_0\le E_1\le\cdots\le E_k$ are nonnegative real numbers and
$z_1,\ldots,z_k\in\mathbb C$ satisfy
\begin{equation}
 \begin{cases}
 \Re z_j\ge-a,&E_j>C_1,\\
 \Re z_j\ge-a-b\sqrt{E_j-E_{j-1}},&E_j\le C_1.
 \end{cases}
 \label{main:capped-hypotheses}
\end{equation}
Then
\begin{equation*}
 \Re\left(\frac1k\sum_{j=1}^kz_j\right)
 \ge-a-b\sqrt{C_1/k}.
\end{equation*}
\end{lemma}

\begin{proof}
Put $F_j=\min(C_1,E_j)$. If $E_j\le C_1$, monotonicity gives
$F_j-F_{j-1}=E_j-E_{j-1}$. If $E_j>C_1$, then
$F_j-F_{j-1}\ge0$, so the first alternative in
\eqref{main:capped-hypotheses} also implies
\begin{equation}
 \Re z_j\ge-a-b\sqrt{F_j-F_{j-1}}.
 \label{main:capped-pointwise}
\end{equation}
By Cauchy--Schwarz and telescoping,
\begin{align*}
 \frac1k\sum_{j=1}^k\sqrt{F_j-F_{j-1}}
 &\le\left(\frac1k\sum_{j=1}^k(F_j-F_{j-1})\right)^{1/2}\\
 &=\sqrt{(F_k-F_0)/k}\le\sqrt{C_1/k}.
\end{align*}
Averaging \eqref{main:capped-pointwise} gives the assertion.
\end{proof}
\subsection{Construction of the measure}

The kernels chosen below depend only on $\delta$.  In the proof, the
minimizing pair $(\chi,t)$ may depend on $g$ and on the index $j$, but it
is used only to estimate the already-fixed kernel $K_{X_j}$; it does not
enter the definition of $\mu_\delta$.

\begin{proposition}\label{main:finite-measure}
For every $0<\delta<1/10$, there is a finitely supported probability
measure $\mu_\delta$ on Pythagorean ratios of height at most $H_\delta$
such that
\begin{equation}
 \Re\sum_r\mu_\delta(r)g(r)\ge-\delta
                  \qquad(g\in\mathcal M).
 \label{main:universal-lower-bound}
\end{equation}
and
\begin{equation*}
 H_\delta\le\operatorname{Exp}_2
          \left(C\delta^{-1295}\log^2(2/\delta)\right).
\end{equation*}
Here $C\ge1$ is absolute.
\end{proposition}

\begin{proof}
We divide the proof into six steps.	
	
\par
\medskip
\noindent\textbf{Step 1. Fix the inverse-theorem parameters.}
Fix a real $\delta\in(0,1/10)$ and put $\ell=\log(2/\delta)$.  The
fixed density $w$, fixed integer $C_0$, and prescribed
$\delta$-dependent parameters $D_0(\delta),W_{0,\delta}$ determine the
kernels $K_X$.  These objects are fixed once $\delta$ is fixed and do not
depend on $g$.  Choose $L_\delta,X_\delta,\epsilon_\delta$, and
$\mathcal L_\delta$ as in \Cref{tr:uniform-estimate}.  In particular,
\begin{equation}
 \mathcal L_\delta\le C\delta^{-3},\qquad
 \log(2L_\delta)+\log(2X_\delta)\le C\delta^{-3}.
 \label{main:uniform-parameters}
\end{equation}
Set the integer $Q$, the real parameter $T_1$, the real threshold $C_1$,
and the integer threshold $N_1$ by
\begin{align*}
 Q&:=\left\lceil\exp(A\delta^{-1293})\right\rceil,
 &T_1&:=Q,
 &C_1&:=A\delta^{-1293},
 \notag\\
 N_1&:=\left\lceil\operatorname{Exp}_2(A\delta^{-1293})\right\rceil.
\end{align*}
Here $A$ is an absolute constant chosen sufficiently large.  Since
$1293=3\cdot431$, \Cref{thm:mult_inverse}, together with
\eqref{main:uniform-parameters}, gives the following contrapositive for
 every completely multiplicative $g\in\mathcal M$ and every integer
$z\ge N_1$:
\begin{equation}
 E_g(z)>C_1\quad\Longrightarrow\quad
 \|g\|_{U^3[z]}<\epsilon_\delta.
 \label{main:inverse-contrapositive}
\end{equation}
Indeed, if the norm were at least $\epsilon_\delta$,
\Cref{thm:mult_inverse} would produce an admissible primitive
character and real parameter in the
ranges $q_\chi\le Q$ and $|t|\le T_1$ whose prime distance is at most
$C_1$, contradicting $E_g(z)>C_1$.

\par
\medskip
\noindent\textbf{Step 2. Choose all arithmetic scales independently of $g$.}
Choose an absolute constant $A'$ sufficiently large and put
\begin{equation}
 M_1:=\left\lceil\operatorname{Exp}_2(A'\delta^{-1293})\right\rceil.
 \label{main:initial-parameter}
\end{equation}
Increase $A'$ so that the integer $M_1$ satisfies
\begin{equation*}
 M_1\ge\max\{2,Q,T_1,N_1,X_\delta,D_0,
 C_wD_0^2,\exp(C_*\delta^{-3})\}.
\end{equation*}
Here $C_*$ is the lower-bound constant in
\Cref{ar:model-positivity}, and $C_wD_0^2$ is the lower bound required
by \Cref{con:comparison}.  Such an $A'$ exists because
\eqref{ar:weight-cost} gives
$\log(2D_0)+\log(2\|W_0\|_\infty)\ll\delta^{-3}$, whereas
$\log\log M_1=A'\delta^{-1293}+O(1)$.  \Cref{ar:model-positivity}
therefore supplies an integer $U$ depending only on $\delta$ such that
\begin{equation*}
 2\le U\le C\delta^{-3},
\end{equation*}
and its conclusion applies whenever
\begin{equation*}
 y\ge M_1,\qquad X\ge(M_1y)^U,
\end{equation*}
with $q_\chi\le M_1$, $|t|\le M_1$.

Let $b_\delta\ge1$ be the constant in \Cref{con:comparison}.  Its
bound and $B=(200/\delta)^2$ give
\begin{equation*}
 b_\delta^2\le C\log(2B)\le C\ell.
\end{equation*}
Choose
\begin{equation*}
 k:=\left\lceil10^4b_\delta^2C_1\delta^{-2}\right\rceil
\end{equation*}
so that
\begin{equation}
 k\le C\delta^{-1295}\ell.
 \label{main:number-of-scales-bound}
\end{equation}
Fix $a_0=2$ and define integer scales by
\begin{equation}
 Z_0:=M_1^{a_0},\qquad
 X_j:=\left\lceil(M_1Z_{j-1})^U\right\rceil,\qquad
 Z_j:=C_0X_j\quad(1\le j\le k).
 \label{main:scale-sequence}
\end{equation}
These scales depend only on $\delta$ and the fixed, $g$-independent parameters $D_0(\delta),W_{0,\delta},w,C_0$.  If
$y_j:=Z_{j-1}$, then $y_j\ge M_1$, $X_j\ge(M_1y_j)^U$, and
$X_j\ge M_1\ge X_\delta,C_wD_0^2$.  Since $M_1\ge1$ and $U\ge2$,
$X_j\ge y_j$, so $y_j\le Z_j=C_0X_j$.  Thus the hypotheses of both
\Cref{ar:model-positivity} and \Cref{con:comparison} will hold at every
scale once the character chosen below satisfies $q_\chi\le Q\le M_1$.

\par
\medskip
\noindent\textbf{Step 3. Estimate one scale by an energy dichotomy.}
Fix $g\in\mathcal M$ and $1\le j\le k$.  If
$E_g(Z_j)>C_1$, then $Z_j\ge N_1$ and
\eqref{main:inverse-contrapositive} gives
$\|g\|_{U^3[Z_j]}<\epsilon_\delta$.  Since $X_j\ge X_\delta$,
\Cref{tr:uniform-estimate} yields
\begin{equation*}
 |K_{X_j}(g)|\le\delta/50.
\end{equation*}

Suppose instead that $E_g(Z_j)\le C_1$.  Choose an admissible pair
$(\chi,t)$ minimizing $E_g(Z_j)$ and put $y:=Z_{j-1}$.  Define the
truncated completely multiplicative function
\begin{equation*}
 \widetilde g(p)=
 \begin{cases}
 g(p),&p\le y,\\
 \chi(p)p^{it},&p>y.
 \end{cases}
\end{equation*}
Since $y\ge M_1\ge Q$, \Cref{main:minimized-energy} gives
$\widetilde g\in\mathcal M$.  Its hypotheses for
\Cref{ar:model-positivity} are explicit: $q_\chi\le Q\le M_1$,
$|t|\le T_1\le M_1$, $y\ge M_1$, and
$X_j\ge(M_1y)^U$.  Hence
\begin{equation}
 \Re K_{X_j}(\widetilde g)\ge-\delta/50.
 \label{main:finite-model-lower-bound}
\end{equation}

\par
\medskip
\noindent\textbf{Step 4. Compare $g$ with the truncated model.}
Let $h:=g/\widetilde g$.  Then $h(p)=1$ for every prime $p\le y$,
including the primes dividing $D_0$, while for $p>y$,
$h(p)=g(p)\overline{\chi(p)}p^{-it}$.  \Cref{con:comparison} applies
because $y\ge M_1\ge D_0,q_\chi,2$, $X_j\ge C_wD_0^2$, and
$y\le Z_j=C_0X_j$.  On the support of the kernel, all four positive
linear forms $m-n,m+n,m,n$ are at most $Z_j$.
\Cref{con:comparison} approximates $h$ on these four forms by the same
unit complex number $\zeta$.  Their product is
$\zeta^2\overline{\zeta}^{\,2}=1$, so no extra phase remains in the
four-linear kernel.  Therefore
\begin{align*}
 |K_{X_j}(g)-K_{X_j}(\widetilde g)|
 &\le b_\delta
   \left(\sum_{\substack{y<p\le Z_j\\p\ \mathrm{prime}}}
       \frac{1-\Re(g(p)\overline{\chi(p)}p^{-it})}{p}\right)^{1/2}
       \notag\\
 &\le b_\delta\sqrt{E_g(Z_j)-E_g(Z_{j-1})}.
\end{align*}
The final inequality is exactly \eqref{main:tail-minimum}: it bounds the prime sum over $(Z_{j-1},Z_j]$ by $E_g(Z_j)-E_g(Z_{j-1})$ for the pair minimizing at $Z_j$.  Combining this estimate with
\eqref{main:finite-model-lower-bound} gives
\[
 \Re K_{X_j}(g)\ge-\delta/50
 -b_\delta\sqrt{E_g(Z_j)-E_g(Z_{j-1})}.
\]

\par
\medskip
\noindent\textbf{Step 5. Average over the fixed scales.}
For $0\le j\le k$ set $E_j:=E_g(Z_j)$, and for $1\le j\le k$ set
$z_j:=K_{X_j}(g)$.  \Cref{main:minimized-energy} gives
$E_0\le E_1\le\cdots\le E_k$.  The large-energy estimate and the last
inequality in Step~4 verify the two alternatives of
\Cref{main:capped-average} with $a=\delta/50$ and $b=b_\delta$.
Since $k\ge10^4b_\delta^2C_1\delta^{-2}$,
\[
 b_\delta\sqrt{C_1/k}\le\delta/100.
\]
Therefore
\begin{equation}
 \Re\left(\frac1k\sum_{j=1}^kK_{X_j}(g)\right)
 \ge-\frac{\delta}{50}-\frac{\delta}{100}
 =-\frac{3\delta}{100}\ge-\delta.
 \label{main:averaged-lower-bound}
\end{equation}
The scales used here are the same for every $g$.

\par
\medskip
\noindent\textbf{Step 6. Define the common measure and bound its height.}
Define
\begin{equation*}
\mu_\delta(r):=\frac1k\sum_{j=1}^k
   \frac{X_j^{-2}}{Z_{X_j}}
   \sum_{\substack{m>n>0\\R(m,n)=r}}
      w(m/X_j,n/X_j)W_0(m,n).
\end{equation*}
Each inner sum is finite and nonnegative.  Moreover,
\[
 \sum_r\mu_\delta(r)=\frac1k\sum_{j=1}^kK_{X_j}(1)=1,
\]
so $\mu_\delta$ is a finitely supported probability measure independent
of $g$.  \eqref{main:averaged-lower-bound} is precisely
\eqref{main:universal-lower-bound} for this measure.

For every support pair $(m,n)$,
\[
 (m^2-n^2)^2+(2mn)^2=(m^2+n^2)^2.
\]
Thus $R(m,n)$ is a Pythagorean ratio.  It is not equal to $1$, since
$R(m,n)=1$ would imply $(m-n)^2=2n^2$, impossible for positive integers
$m,n$.  Since $m+n<Z_j\le Z_k$ on the support, both unreduced legs
$m^2-n^2$ and $2mn$ are at most $Z_k^2$.  Reducing the two legs to lowest
terms can only decrease them, so we may take
\[
 H_\delta:=Z_k^2.
\]

It remains to estimate $H_\delta$.  Write
$z_j:=\log Z_j$ and $m_1:=\log M_1$.  Since
$\lceil u\rceil\le2u$ for $u\ge1$, \eqref{main:scale-sequence} gives
\begin{equation*}
 z_j\le Uz_{j-1}+Um_1+\log(2C_0).
\end{equation*}
As $U\ge2$, summing this recurrence yields
\begin{align}
 z_k
 &\le U^kz_0+(Um_1+\log(2C_0))\frac{U^k-1}{U-1}\notag\\
 &\le U^k\{z_0+2m_1+2\log(2C_0)\}.
 \label{main:solved-recurrence}
\end{align}
Here $z_0=2m_1$.  From \eqref{main:initial-parameter},
\begin{equation}
 \log m_1\le A'\delta^{-1293}+C,
 \qquad \log U\le C\ell.
 \label{main:initial-logarithms}
\end{equation}
Taking logarithms in \eqref{main:solved-recurrence}, then using
\eqref{main:number-of-scales-bound} and \eqref{main:initial-logarithms},
we obtain
\[
 \log\log Z_k=\log z_k
 \le k\log U+O(\log m_1)
 \le C\delta^{-1295}\ell^2.
\]
Finally,
$\log\log H_\delta=\log(2z_k)=\log2+\log\log Z_k$, so increasing the
absolute constant gives
\[
 H_\delta\le\operatorname{Exp}_2
 \left(C\delta^{-1295}\log^2(2/\delta)\right).
\]
This proves the proposition.
\end{proof}
\subsection{The density bound}

We first combine \Cref{main:finite-measure} with
\Cref{tr:finite-transfer} to obtain a threshold implication.
We then choose $\delta$ as a function of $N$ and invert that implication.

\begin{proof}[Proof of \Cref{thm:main}]
We divide the proof into three steps.	
	
\par
\medskip
\noindent\textbf{Step 1. Prove the threshold implication.}
For $0<\delta<1/10$, set
\begin{equation*}
 F_\delta=\delta^{-1295}\log^2(2/\delta).
\end{equation*}
\Cref{main:finite-measure} provides a measure satisfying
the hypotheses of \Cref{tr:finite-transfer}, with
$H_\delta\le\operatorname{Exp}_2(C_HF_\delta)$ for an absolute
$C_H\ge1$.  Choose an absolute $C_N>C_H+2$ so large that, for every
$F\ge1$,
\[
 \frac2{\log2}\exp((C_H+1)F)\le\operatorname{Exp}_2(C_NF),
\]
and set $C_{\rm main}:=C_N$.
Then
\begin{equation}
 N\ge\operatorname{Exp}_2(C_NF_\delta)
 \label{main:final-threshold}
\end{equation}
ensures both size conditions in \eqref{tr:cube-size}.  We verify them
separately.  The first one follows from
$H_\delta\le\operatorname{Exp}_2(C_HF_\delta)\le
\operatorname{Exp}_2(C_NF_\delta)\le N$.  For the second one, note that
on the stated range of $\delta$, $F_\delta\ge1$ and
$\delta^{-2}\le e^{F_\delta}$.  Therefore
\begin{equation*}
 \frac{2\delta^{-2}\log H_\delta}{\log2}
 \le\frac2{\log2}\exp((C_H+1)F_\delta)
 \le\operatorname{Exp}_2(C_NF_\delta),
\end{equation*}
Thus
\Cref{tr:finite-transfer} applies and gives
$f(N)\le2\delta N^N$.  Since $C_{\rm main}=C_N$, this proves
\eqref{eq:main-threshold}.

\par
\medskip
\noindent\textbf{Step 2. Choose $\delta$ and retain the logarithmic correction.}
Put $t=\log_2N$ and $s=1295$.  Choose an absolute constant
$A_0\ge1$ sufficiently large, and then choose $t_0$ sufficiently large
so that for every $t\ge t_0$ the choice
\begin{equation*}
 \delta=A_0\left(\frac{(\log t)^2}{t}\right)^{1/s}
\end{equation*}
lies in $(0,1/10)$ and satisfies
\begin{align*}
 \log(2/\delta)
 &=\log(2/A_0)+\frac{\log t-2\log\log t}{s}
 \le\frac{2\log t}{s},\\
 C_N\delta^{-s}\log^2(2/\delta)
 &\le\frac{4C_N}{s^2A_0^s}\,t\le t.
\end{align*}
where the final inequality follows from fixing
$A_0^s\ge4C_N/s^2$.  Since $N=\operatorname{Exp}_2(t)$, the inequality
$C_N\delta^{-s}\log^2(2/\delta)\le t$ is exactly the hypothesis in
\eqref{main:final-threshold}.  Thus
\begin{equation}
 \frac{f(N)}{N^N}
 \le 2A_0\left(\frac{(\log t)^2}{t}\right)^{1/s}
 =2A_0\left(\frac{(\log_3N)^2}{\log_2N}\right)^{1/1295}.
 \label{main:logarithmic-rate}
\end{equation}
Choose an integer $N_0$ so large that $\log_2N\ge t_0$ and all
logarithms in \eqref{main:logarithmic-rate} are defined for $N\ge N_0$.
After enlarging an absolute constant if necessary, set $C_1:=2A_0$.
Then the first inequality in \eqref{eq:main-rate} follows.

\par
\medskip
\noindent\textbf{Step 3. Deduce the pure-power bounds.}
Let $0<b<1/s$ and put $\kappa=1/s-b>0$.
For $x=\log t>0$, differentiation shows that the maximum of
$x^{2/s}e^{-\kappa x}$ is attained at $x=2/(s\kappa)$.
Consequently,
\begin{equation*}
 t^{-1/s}(\log t)^{2/s}
 \le t^{-b}\left(\frac{2}{e s\kappa}\right)^{2/s}.
\end{equation*}
Combining this inequality with \eqref{main:logarithmic-rate}, and taking
$b_*=1/(s+1)=1/1296$, gives
\begin{equation*}
 \frac{f(N)}{N^N}\le C_2(\log_2N)^{-1/1296}
 \qquad(N\ge N_0),
\end{equation*}
after setting
$C_2:=C_1\bigl(2/(e s(1/s-b_*))\bigr)^{2/s}$ and enlarging it if needed.
This proves the second inequality in \eqref{eq:main-rate}.  For a general
fixed $0<b<1/1295$, set
\[
 C_b:=C_1\left(\frac{2}{e s(1/s-b)}\right)^{2/s},
 \qquad N_b:=N_0.
\]
If necessary enlarge $N_b$ so that the choice of $\delta$ in Step~2 lies
in $(0,1/10)$ for every $N\ge N_b$.  The same calculation then yields
\eqref{eq:main-other-powers}.
\end{proof}

\section*{AI-use disclosure}
The author used OpenAI Codex in GPT-6 Astra Ultra Fast mode during
the preparation of this manuscript. After obtaining
the bound
\begin{equation*}
 \frac{f(N)}{N^N}
 \ll\left(\frac{(\log\log\log N)^2}{\log\log N}\right)^{1/s},
\end{equation*}
the author used Codex to
calculate an admissible value of $s$, to reorganize the proofs
around the value $s=1295$, and to polish the proofs. Codex also assisted with mathematical
auditing and English-language editing. The author has checked all
mathematical statements and takes responsibility for the final manuscript.
\providecommand{\bysame}{\leavevmode\hbox to3em{\hrulefill}\thinspace}
\providecommand{\noopsort}[1]{}
\providecommand{\mr}[1]{\href{http://www.ams.org/mathscinet-getitem?mr=#1}{MR~#1}}
\providecommand{\zbl}[1]{\href{http://www.zentralblatt-math.org/zmath/en/search/?q=an:#1}{Zbl~#1}}
\providecommand{\jfm}[1]{\href{http://www.emis.de/cgi-bin/JFM-item?#1}{JFM~#1}}
\providecommand{\arxiv}[1]{\href{http://www.arxiv.org/abs/#1}{arXiv~#1}}
\providecommand{\doi}[1]{\url{https://doi.org/#1}}
\providecommand{\MR}{\relax\ifhmode\unskip\space\fi MR }
\providecommand{\MRhref}[2]{%
  \href{http://www.ams.org/mathscinet-getitem?mr=#1}{#2}
}
\providecommand{\href}[2]{#2}

\end{document}